\documentclass[11pt,fullpage]{article}
\usepackage{amsmath}
\usepackage{amsfonts,amssymb}
\usepackage{color}	
\usepackage{amsthm}
\usepackage{graphics,graphicx}
\usepackage{hyperref}
\usepackage{titlesec}
\usepackage[right = 2.5cm, left=2.5cm, top = 2.5cm, bottom =2.5cm]{geometry}
\numberwithin{equation}{section}

\titleformat{\paragraph}
{\normalfont\normalsize\bfseries}{\theparagraph}{1em}{}
\titlespacing*{\paragraph}
{0pt}{3.25ex plus 1ex minus .2ex}{1.5ex plus .2ex}

\titleformat{\subparagraph}
{\normalfont\normalsize\bfseries}{\theparagraph}{1em}{}
\titlespacing*{\paragraph}
{0pt}{3.25ex plus 1ex minus .2ex}{1.5ex plus .2ex}

\newcommand{\Real}{\mathbb R}

\newcommand{\Torus}{\mathbb T}
\newcommand{\Integers}{\mathbb Z}
\newcommand{\Integer}{\mathbb Z}
\newcommand{\norm}[1]{\left\lVert #1 \right\rVert}
\newcommand{\abs}[1]{\left\vert#1\right\vert}
\newcommand{\set}[1]{\left\{#1\right\}}

\newcommand{\grad}{\nabla}

\newcommand{\G}{\mathcal{G}}

\newcommand{\I}{\mathbf{I}}

\newcommand{\jap}[1]{\left\langle #1 \right\rangle} 

\newtheorem{theorem}{Theorem}
\theoremstyle{definition}
\newtheorem{remark}{Remark}

\theoremstyle{lemma}

\newtheorem{proposition}{Proposition}[section]

\theoremstyle{definition}

\theoremstyle{lemma}
\newtheorem{lemma}{Lemma}[section]

\numberwithin{remark}{section}

\begin{document}

\title{Dynamics near the subcritical transition of the 3D Couette flow I: Below threshold case}
\author{Jacob Bedrossian\footnote{\textit{jacob@cscamm.umd.edu}, University of Maryland, College Park} \, and Pierre Germain\footnote{\textit{pgermain@cims.nyu.edu}, Courant Institute of Mathematical Sciences} \, and Nader Masmoudi\footnote{\textit{masmoudi@cims.nyu.edu}, Courant Institute of Mathematical Sciences}}
\date{\today}
\maketitle

\begin{abstract} 
We study small disturbances to the periodic, plane Couette flow in the 3D incompressible Navier-Stokes equations at high Reynolds number \textbf{Re}. 
We prove that for sufficiently regular initial data of size $\epsilon \leq c_0\textbf{Re}^{-1}$ for some universal $c_0 > 0$, 
the solution is global, remains within $O(c_0)$ of the Couette flow in $L^2$, and returns to the Couette flow as $t \rightarrow \infty$. 
For times $t \gtrsim \textbf{Re}^{1/3}$, the streamwise dependence is damped by a mixing-enhanced dissipation effect and the solution is rapidly attracted to the class of ``2.5 dimensional'' streamwise-independent solutions referred to as \emph{streaks}. 
Our analysis contains perturbations that  experience a transient growth of kinetic energy from $O(\textbf{Re}^{-1})$ to $O(c_0)$ due to the algebraic linear instability known as the \emph{lift-up effect}. 
Furthermore, solutions can exhibit a direct cascade of energy to small scales. 
The behavior is very different from the 2D Couette flow, in which stability is independent of $\textbf{Re}$, enstrophy experiences a direct cascade, and inviscid damping is dominant (resulting in a kind of \emph{inverse} energy cascade). 
In 3D, inviscid damping will play a role on one component of the velocity, but the primary stability mechanism is the mixing-enhanced dissipation. 
Central to the proof is a detailed analysis of the interplay between  the stabilizing effects of the mixing and enhanced dissipation and the destabilizing effects of the lift-up effect, vortex stretching, and weakly nonlinear instabilities connected to the non-normal nature of the linearization.
\end{abstract}

\setcounter{tocdepth}{1}
{\small\tableofcontents}

\section{Introduction}
We study the 3D Navier-Stokes equations near the Couette flow in the idealized domain $(x,y,z) \in \Torus \times \Real \times \Torus$: if $u + (y,0,0)^{T}$  solves the Navier-Stokes equation, then the disturbance $u$ solves 
\begin{subequations}\label{def:3DNSE}
\begin{align} 
\partial_t u + y \partial_x u + u\cdot \grad u + \grad p^{NL} & = \begin{pmatrix} - u^2 \\ 0 \\ 0 \end{pmatrix} - \grad p^L  + \nu \Delta u \\ 
\Delta p^{NL} & = -\partial_i u^j \partial_j u^i \\ 
\Delta p^L & = -2\partial_x u^2 \\ 
\grad \cdot u & = 0,
\end{align}
\end{subequations}
where $\nu = \textbf{Re}^{-1}$ denotes the inverse Reynolds number, $p^{NL}$ is the nonlinear contribution to the pressure due to the disturbance 
and $p^L$ is the linear contribution to the pressure due to the interaction between the disturbance and the Couette flow.
A vast effort in the applied mathematics and physics community has been made towards understanding the stability of laminar shear flows at high Reynolds number for over 130 years (see \S\ref{sec:History} and \cite{DrazinReid81,SchmidHenningson2001,Yaglom12} for references) and \eqref{def:3DNSE} is the simplest example, representing a fundamental, canonical problem in the field. The goal of this work, and the companion work \cite{BGM15II}, is to advance the mathematically rigorous  understanding of \eqref{def:3DNSE} and, for sufficiently regular initial data, resolve several long-standing questions regarding the (in)stability of \eqref{def:3DNSE} at high Reynolds number. 

\subsection{History and context of new results} \label{sec:History}
Understanding the stability of laminar flows and the transition to turbulence is one of the main objectives of hydrodynamic stability theory (see e.g. the texts \cite{DrazinReid81,SchmidHenningson2001,Yaglom12} and the references therein). 
One of the first and most influential experiments in the field were those of Reynolds \cite{Reynolds83} in 1883, which demonstrated the instability of laminar flow in a pipe for sufficiently high Reynolds number (in fact, this work is the origin of the name \emph{Reynolds number}).
However, such instabilities appeared inconsistent with theoretical studies, which suggested spectral stability independent of Reynolds number for a variety of simple laminar flows, including variations of the Couette flow \eqref{def:3DNSE} studied here \cite{Rayleigh80,Kelvin87,DrazinReid81}.  
Moreover, this spectral stability can indeed translate to nonlinear asymptotic stability, as has been shown in some cases \cite{Romanov73,DrazinReid81,KreissEtAl94,Liefvendahl2002}. 
Sometimes this apparent paradox is referred to as the ``Sommerfeld paradox'' or the ``turbulence paradox''\cite{LiLin11}.
In other cases, even when there are spectral instabilities for high Reynolds number, the flow is sometimes observed to transition at much lower Reynolds number than what the eigenvalue theory predicts and/or exhibits a completely different kind of instability \cite{SchmidHenningson2001,Yaglom12}. 
This behavior is ubiquitous in 3D hydrodynamics and is often referred to as \emph{subcritical transition} or \emph{by-pass transition} in the fluid mechanics literature. 
Since the work of Reynolds, many other influential experiments (see e.g. \cite{Nishioka1975,klebanoff1962,Tillmark92,Daviaud92,Elofsson1999,bottin98,HofJuelMullin2003,Mullin2011,LemoultEtAl2012}) and computer simulations (see e.g. \cite{Orszag80,HLJ93,ReddySchmidEtAl98,DuguetEtAl2010} and the references therein) on subcritical transition phenomena have been performed.

It is natural to suggest that while the flow is technically stable for all finite Reynolds number, the set of stable perturbations shrinks as the Reynolds number increases, leading to transition in any real system at some finite Reynolds number (this suggestion goes back to Lord Kelvin \cite{Kelvin87}, or arguably Reynolds \cite{Reynolds83}).
It is of practical importance to determine, for a specific physical system, the Reynolds number at which transition will likely occur. 
A similar question, formulated better for theoretical analysis, is to determine how the maximal size of stable perturbations in a given norm, the ``transition threshold'', will scale with respect to the Reynolds number (see e.g. \cite{TTRD93}). 
For example, given a norm $\norm{\cdot}_N$, find a $\gamma = \gamma(N)$ such that $\norm{u_{in}}_N \ll \nu^\gamma$ implies stability and perhaps such that $\norm{u_{in}}_N \gg \nu^\gamma$ permits instability. 
Notice that the transition threshold \emph{depends on the norm} and that different norms may result in different answers \cite{ReddySchmidEtAl98}; at asymptotically high Reynolds numbers, vanishing viscosity will not suppress the high frequencies and the Couette flow can move information from the small scales to the large (see Remark \ref{rmk:RegularityInstab} below and \cite{BGM15II} for further discussion). 
It is also of practical interest to determine how the instability will occur if the perturbation is too large. 
To state more precisely: given a norm on the initial perturbation, (A) how large can the perturbation be and still result in an asymptotically stable solution and (B) if the perturbation is large enough, what kind of instabilities are observed? 
For contrast, we emphasize that for sufficiently regular perturbations, the 2D Couette flow does \emph{not} undergo subcritical transition, and instead is nonlinearly, asymptotically stable (in a suitable sense) uniformly at high Reynolds number \cite{BMV14} and also infinite Reynolds number \cite{BM13}. 

A great deal of effort has been spent on trying to determine the transition threshold and the nature of the instabilities for simple laminar flows (see e.g. the texts \cite{SchmidHenningson2001,Yaglom12} and the references therein). 
The linearization of \eqref{def:3DNSE} is non-normal, which means a large transient growth before eventual decay is possible even on the linear level. The suggestion that this is the source of the observed instability goes all the way back to Orr \cite{Orr07} in 1907, even though he was thinking about a 2D non-normal effect, called the \emph{Orr mechanism}, which will not be the main cause of transient growth here (although it will be absolutely crucial for understanding \eqref{def:3DNSE}!). 
Instead, the main mechanism here for transient kinetic energy growth is the 3D non-normal effect known as the \emph{lift-up effect}; see \cite{EllingsenPalm75,landahl80} and \S\ref{sec:lin} below. 
The work of Trefethen et. al. \cite{TTRD93} forwarded the idea that the nonlinearity could interact poorly 
with the non-normal behavior by repeatedly re-exciting growing linear modes, producing a ``nonlinear bootstrap'' scenario. 
The authors discussed a low-dimensional toy model meant to capture certain aspects of this idea and used it to 
conjecture a stability threshold of $\epsilon \sim \nu^{\gamma}$ with $\gamma > 1$ for \eqref{def:3DNSE} (where $\epsilon$ is the size of the initial data). 
A number of works used variations of this idea to understand the threshold via combinations of simplified ODE models, asymptotic analysis, and computation \cite{Gebhardt1994,BDT95,Waleffe95,BT97,LHRS94,Chapman02}.
Various predictions have been made, ranging generally from $1 \leq \gamma \leq 7/4$; for the infinite channel, the mathematically rigorous bound $\gamma \leq 4$ is known \cite{Liefvendahl2002}, see also the earlier work of \cite{Romanov73,KreissEtAl94}.  
We also would like to emphasize that not all of these works consider exactly the same problem. For example, some consider boundaries in $y$ and/or consider a domain which is unbounded in $x$.
Both could potentially alter the answers. 
Boundary layers are notoriously problematic in fluid mechanics, and are known to introduce instabilities in most channel flows (see \cite{DrazinReid81,GrenierGuoNguyen14a} and the references therein).  
The prospect of removing the periodicity assumptions is discussed further in Remark \ref{rmk:periodicity} below.  

In this work, we prove that there exists a universal constant $c_0 > 0$ such that if the initial data is of size $\epsilon < c_0 \nu$ (in a sufficiently regular sense), then the solution is global in time and converges back to the Couette flow as $t \rightarrow \infty$. 
Further, we demonstrate that solutions to \eqref{def:3DNSE} which are $O(\epsilon)$ initially can grow to be as large as $O(c_0)$ before eventually decaying back to Couette flow. 
Note that the supremum in time of these solutions remains $O(c_0)$ uniformly as $\nu \rightarrow 0$. 
Hence, for sufficiently regular perturbations, we are essentially proving that $\gamma = 1$.  
That we can still obtain global solutions despite of this large growth depends crucially on the stabilizing effects of the mixing combined with a detailed weakly nonlinear study. 
Due to this mixing, the $x$-dependence of the solution is damped for $t \gtrsim \nu^{-1/3}$ and all solutions converge to the class of ``streak'' solutions (see \S\ref{sec:streak} below).   
Furthermore, due to the mixing and vortex stretching, the solutions can also exhibit a direct cascade of energy to the small scales where it is subsequently dissipated at the time $\sim \nu^{-1/3}$. To our knowledge, this kind of behavior in the 3D Navier-Stokes equations has not previously been confirmed in a mathematically rigorous setting. 
The class of initial data we consider is the sum of a sufficiently smooth function (Gevrey-$\frac{1}{s}$ for $s > 1/2$ \cite{Gevrey18}) and a much smaller (relative to $\nu$) $H^3$ function. 
It should be possible to improve $H^3$ to $H^{1/2}$ or $L^3$, but it was not our goal to be optimal there; actually this question is totally independent of the work we undertake here.

The stability mechanisms which make our results possible are \emph{mixing-enhanced dissipation} and \emph{inviscid damping}.  
Both effects arise from the mixing effect of the background Couette flow. 
Inviscid damping was first derived on the linear level by Orr \cite{Orr07} in 1907 and was later noticed to be the hydrodynamic analogue of Landau damping in plasma physics \cite{Landau46,Ryutov99,MouhotVillani11,BMM13} (hence the name, see e.g. \cite{BM13,Briggs70,BouchetMorita10,BM95,SchecterEtAl00} for discussions on the relationship between the effects). 
Both are special cases of a more general effect known as \emph{phase mixing} (see e.g. \cite{BMT13,BM13} and the references therein).  
Inviscid damping causes decay of the velocity field in inviscid fluids via mixing of the vorticity.
Indeed, mixing is characterized by a transfer of enstrophy to higher frequencies which then yields strong convergence of the velocity field via the negative order Biot-Savart law (see e.g. \cite{LinZeng11,BM13}).  
In 2D, the effect leads to the asymptotic stability (in the correct sense) of the 2D Couette flow even with no viscosity at all \cite{BM13}.
It is also expected to be relevant in other 2D contexts, for example, in understanding the stability of general 2D shear flows \cite{Zillinger2014} and the axi-symmetrization of 2D vortices \cite{GilbertBassom98}, hurricanes \cite{MontgomeryKallenback1997,SmithMontgomery1995}, and cyclotron beams \cite{CerfonEtAl13}. 
However in 3D, due to the vortex stretching caused by the Couette flow (see \S\ref{sec:lin} below), it only results in the decay of $u^2$, the second component of the velocity.
Due to the special structure of the nonlinearity in \eqref{def:3DNSE}, this is still key for suppressing certain, specific nonlinear effects (see \S\ref{sec:NonlinHeuristics} for more discussion).   

Inviscid damping will play an important role, but enhanced dissipation via mixing is the primary stability mechanism at work. It was first derived in the context of \eqref{def:3DNSE} by Lord Kelvin \cite{Kelvin87} (at least in 2D) and has been subsequently observed or studied by numerous authors in fluid mechanics in various settings (see e.g \cite{RhinesYoung83,DubrulleNazarenko94,LatiniBernoff01,BernoffLingevitch94,Bajer2001,Gilbert88,Lundgren82} as well as the rigorous works \cite{BeckWayne11,ConstantinEtAl08,BMV14}). 
In the 2D analogue of \eqref{def:3DNSE}, this effect causes rapid convergence of the solution on time scales $t \gtrsim \nu^{-1/3}$ to a slowly evolving shear flow, which only relaxes on time scales like $t \gtrsim \nu^{-1}$ (these time scales are explained in more detail below). 
The general intuition is that as information is mixed to smaller scales, the effectiveness of the viscosity is greatly enhanced in streamwise dependent modes.
Here it will imply that the $x$ dependence of the perturbation is damped out for $t \gtrsim \nu^{-1/3}$, but in 3D, the solution does not converge to a shear flow, but rather to the class of ``2.5-dimensional'' $x$-independent solutions to \eqref{def:3DNSE}, sometimes referred to as \emph{streaks} (see \S\ref{sec:streak} below).

\subsection{Linearized equations} \label{sec:lin}
First, we detail the behavior observed on the linear level and derive the lift-up effect, linear vortex stretching, inviscid damping, and enhanced dissipation in the linearization of \eqref{def:3DNSE}. 
The linear behavior will naturally serve as an important guide for the full dynamics.

\subsubsection{Linearized inviscid equations: lift-up effect, vortex stretching, and inviscid damping} 
Since we are interested in asymptotically high Reynolds number, it is sensible to consider first the linearized 3D Euler equations, which read
\begin{subequations}\label{def:3DEuler_Linear}
\begin{align} 
\partial_t u + y \partial_x u  & = \begin{pmatrix} - u^2 \\ 0 \\ 0 \end{pmatrix} - \grad p^L \\ 
\Delta p^L & = -2\partial_x u^2 \\ 
\grad \cdot u & = 0. 
\end{align}
\end{subequations}
It is often natural to consider the vorticity form of the equations, but the stretching nonlinearity destroys much of the simple structure of the vorticity formulation in dimension 3. 
However, it has long been known that the quantity
$$
q^2 = \Delta u^2,
$$
plays in dimension 3 for \eqref{def:3DNSE} a similar role to that played by the vorticity in dimension 2. 
This unknown dates back at least to Lord Kelvin \cite{Kelvin87} and is a standard tool in considering the stability of planar shear flows (see e.g. \cite{SchmidHenningson2001,Yaglom12,Chapman02} and the references therein). 
The distinguished role is due to it being a conservation law of the linear problem (for other shear flows it is no longer conserved, but it solves a special PDE): upon taking the Laplacian of the second component of \eqref{def:3DEuler_Linear}, we derive
\begin{align} 
\partial_t q^2 + y \partial_x q^2 = 0. 
\end{align} 
If we rewind by the action of the Couette flow and define $X = x-ty$, write $U^i(t,X,y,z) = u^i (t,x,y,z)$, and $Q^2(t,X,y,z) = q^2(t,x,y,z)$ and $P^L(t,X,y,z) = p^L(t,x,y,z)$, then we derive
\begin{subequations}\label{def:3DEuler_Linear_CouetteAction}
\begin{align} 
\partial_t U  & = \begin{pmatrix} - U^2 \\ 0 \\ 0 \end{pmatrix} - \grad^L P^L \\ 
\partial_t Q^2 & = 0 \\
\Delta_L U^2 & = Q^2 \\ 
\Delta_L P^L & = -2\partial_X U^2 \\ 
\grad^L \cdot U & = 0,  
\end{align}
\end{subequations}
where we are using
\begin{subequations} \label{def:gradDelL}
\begin{align} 
\grad^L & = (\partial_X, \partial_y - t\partial_X, \partial_z) \\ 
\Delta_L & = \partial_{XX} + (\partial_y - t\partial_X)^2 + \partial_{zz}.
\end{align}
\end{subequations} 
Here `L' stands for `linear'. 
For any sufficiently smooth quantity $f$ we have from the elementary inequality $\frac{1}{k^2 + (\eta - tk)^2} \lesssim \frac{\jap{\eta}^2}{\jap{k t}^2}$ for any non-zero integer $k$, the following fundamental \emph{inviscid damping} estimate for any $\sigma \in [0,\infty)$ and $\beta \in [0,2]$,
 \begin{align}
\norm{\Delta_L^{-1} f_{\neq}}_{H^\sigma} = \left(\sum_{l, k \neq 0} \int \frac{ \jap{k,\eta,l}^{2\sigma} \abs{\hat{f}(k,\eta,l)}^2}{\left(k^2 + l^2 + \abs{\eta-kt}^2\right)^2} d\eta \right)^{1/2} & \lesssim \frac{1}{\jap{t}^{\beta}} \norm{f_{\neq}}_{H^{\sigma+\beta}}, \label{ineq:IDfundamental}
\end{align} 
where we are using $H^{\sigma}$ to denote the $L^2$ Sobolev norm of order $\sigma$ and we are using the notation 
\begin{subequations} 
\begin{align} 
f_{0}(y,z) & = \frac{1}{2\pi}\int f(x,y,z)dx, \label{def:ProjectionNotesZero} \\ 
f_{\neq} & = f - f_{0}, \label{def:ProjectionNotes}
\end{align}
\end{subequations} 
where then `$\neq$' refers to the projection to non-zero Fourier frequencies in $x$; see Appendix \ref{apx:Gev} for the Fourier analysis conventions we are taking here. 

Inviscid damping was first observed by Orr \cite{Orr07} in the context of 2D inviscid flows. 
He also pointed out the potential for transient growth before decay: indeed,  
\begin{align} 
\widehat{\Delta_L^{-1} f}(k,\eta,l) = -\frac{\hat{f}(k,\eta,l)}{k^2 + l^2 + \abs{\eta-kt}^2}, 
\end{align} 
predicts transient growth for modes with $\eta k > 0$: modes that are tilted against the shear.  
These modes are first \emph{un-mixed} to larger scales before being subsequently mixed, eventually yielding the decay in \eqref{ineq:IDfundamental}. 
This manifests itself in the loss of ellipticity for modes with $\eta = kt$ in $\Delta_L$ and it is clear that in order to get a decay estimate like that stated in \eqref{ineq:IDfundamental},  one cannot gain powers of $t^{-1}$ in \eqref{ineq:IDfundamental} without paying regularity. 
We refer to a time $t \sim \eta/k$ as a \emph{critical time} (Orr's original terminology \cite{Orr07}) or a \emph{resonant time} (modern terminology \cite{Craik1971,YuDriscoll02,YuDriscollONeil,SchmidHenningson2001}). 
This loss of regularity is to control the amount of information in the small scales which is to be subsequently unmixed in the future. 
See \cite{BM13} for more discussion of the Orr mechanism. 

Returning to \eqref{def:3DEuler_Linear_CouetteAction}, since $Q^2$ is conserved and $U^2 = \Delta_L^{-1} Q^2$, \eqref{ineq:IDfundamental} implies the inviscid damping of $U^2$: 
\begin{align*} 
\norm{U^2_{\neq}}_{H^N} & \lesssim \jap{t}^{-2} \norm{(Q^2_{in})_{\neq}}_{H^{N+2}}.  
\end{align*}
In particular, this shows that the background shear flow suppresses $x$ variations in $u^2$ even at infinite Reynolds number. 
In turn, this implies the inviscid damping of the linear pressure $P^L$: 
\begin{align*} 
\norm{P_L}_{H^{N}} & \lesssim \jap{t}^{-2} \norm{U^2_{\neq}}_{H^{N+3}} \lesssim \jap{t}^{-4} \norm{(Q^2_{in})_{\neq}}_{H^{N+5}}.  
\end{align*} 
Hence, we see that $U^1_{\neq}$ and $U^3_{\neq}$ actually converge strongly as $t \rightarrow \infty$. 
We can therefore infer that in general, there is no inviscid damping on $u^1_{\neq}$ and $u^3_{\neq}$. 
In fact, this is due to vortex stretching: the vorticity is being mixed along with $q^2$ and the negative-order Biot-Savart law would imply inviscid damping on all components in the absence of vortex stretching, which is precisely what happens in 2D \cite{LinZeng11,BM13}.    
This is also why there is a direct cascade of kinetic energy in 3D but not in 2D. 

Next, we observe that the only contribution on the RHS of \eqref{def:3DEuler_Linear_CouetteAction} which is \emph{not}
integrable in time is the $X$ average of $U^2$. Indeed, upon taking $X$ averages of \eqref{def:3DEuler_Linear_CouetteAction} we derive the degenerate Jordan block-type system
\begin{subequations}\label{def:3DEuler_Liftup}
\begin{align} 
\partial_t U^1_0  & = -U_0^2 \\ 
\partial_t U^2_0 & = \partial_t U^3_0 = 0.  
\end{align}
\end{subequations}
By \eqref{def:3DEuler_Liftup}, $U_0^1$ grows linearly in time, and therefore the 3D Couette flow is linearly (algebraically) unstable in the 3D Euler equations (although classically known to be \emph{spectrally stable} in the sense that there are no unstable eigenvalues \cite{DrazinReid81,Yaglom12}). 
Hence, we see that the instability is ``non-modal''. 
To summarize the behavior of the 3D linear Euler equations we state the following (without making any effort to be optimal in terms of regularity):   
\begin{proposition}[Linearized Euler] \label{prop:LinEuler} 
Let $u_{in}$ be a divergence free vector field with $u_{in} \in H^7$. Then the solution $u(t)$ to the linearized Euler equations \eqref{def:3DEuler_Linear} with initial data $u_{in}$ satisfies the following for some final state $u_\infty = (u_\infty^1,0,u_\infty^3)$: 
\begin{subequations} \label{ineq:LinearID}
\begin{align} 
\norm{u^{2}_{\neq}(t)}_{2} + \norm{u^{2}_{\neq}(t,x+ty,y,z)}_{H^3}  & \lesssim \jap{t}^{-2} \norm{u^2_{in}}_{H^7} \label{ineq:U2LinearID} \\ 
\norm{u^1_{\neq}(t,x+ty,y,z) - u_\infty^1(x,y,z)}_{H^1} & \lesssim \jap{t}^{-1}\norm{u_{in}}_{H^7} \label{ineq:U1LinearID} \\ 
\norm{u^3_{\neq}(t,x+ty,y,z) - u_\infty^3(x,y,z)}_{H^1} & \lesssim \jap{t}^{-3}\norm{u_{in}}_{H^7}, \label{ineq:U3LinearID}  
\end{align} 
\end{subequations} 
and the formulas
\begin{subequations} \label{def:EulerStreak} 
\begin{align}
u^1_0(t,y,z) & = u^1_{in \; 0}(y,z) - tu_{in \; 0}^2(y,z) \label{eq:liftup} \\ 
u^2_0(t,y,z) & = u^2_{in \; 0}(y,z) \\
u^3_0(t,y,z) & = u^3_{in \; 0}(y,z). 
\end{align}
\end{subequations} 
\end{proposition} 
The mechanism described in \eqref{eq:liftup} is the \emph{lift-up effect}, named so as it is caused by faster fluid moving down and slower fluid moving up (hence `lift-up') in the shear flow and so inducing linear-in-time growth of the perturbation \cite{landahl80,EllingsenPalm75,TTRD93,Trefethen2005,SchmidHenningson2001}.   
The same instability is alternatively known as the \emph{streamwise vortex/streak instability} since it is due to streamwise directed vorticity  mixing fluid in the $yz$ planes, which due to the background shear flow, induces the formation of perturbations in $u_0^1$ known as ``streaks'', due the streak-like appearance of the relatively fast fluid \cite{TTRD93,SchmidHenningson2001,Trefethen2005,bottin98}.
Due to the periodicity in $x$, there exists a class of global, exact nonlinear solutions to \eqref{def:3DNSE} (with or without viscosity) with precisely this behavior, which is the class of solutions we are referring to as ``streaks'' (see \S\ref{sec:streak} below).
 
Another property of the behavior in Proposition \ref{prop:LinEuler} worth emphasizing again is the inviscid damping in \eqref{ineq:U2LinearID}, 
which suppresses completely all the $x$ dependence of the $u^2$ component of the velocity field and leads to the convergence of the other components of the velocity field in \eqref{ineq:U1LinearID} and \eqref{ineq:U3LinearID}. 
Note that the convergence of $u^1$ and $u^3$ in \eqref{ineq:U1LinearID} shows that there is generally no inviscid damping of these components due to the vortex stretching and that there is a direct cascade of kinetic energy to high frequencies (at an approximately linear rate). 
This is in contrast to the behavior of 2D Euler, in which both components of the velocity experience inviscid damping and a strong convergence back to a shear flow, a behavior which is more akin to an inverse cascade \cite{BM13}.  

\subsubsection{Linearized viscous equations: enhanced dissipation} \label{sec:NSELin}
We saw above that the linearized 3D Euler equations are indeed unstable, consistent with the experimental observation that 
laminar flows become unstable for sufficiently high Reynolds number.
When accounting for finite Reynolds number we are now considering the linearized Navier-Stokes equations
\begin{subequations}\label{def:3DNSE_Linear}
\begin{align} 
\partial_t u + y \partial_x u  & = \begin{pmatrix} - u^2 \\ 0 \\ 0 \end{pmatrix} - \grad p^L + \nu \Delta u \\ 
\Delta p^L & = -2\partial_x u^2 \\ 
\grad \cdot u & = 0. 
\end{align}
\end{subequations}
In \eqref{def:3DNSE_Linear} we will find the mixing enhanced dissipation, as derived by Lord Kelvin \cite{Kelvin87}. 
We will see that as the Couette flow mixes information to small scales, the viscous dissipation has an increasing effect on the solution, ultimately yielding the enhanced homogenization effect. 
To understand the origins of this effect, consider the evolution of $q^2 = \Delta u^2$: 
\begin{align*} 
\partial_t q^2 + y\partial_x q^2 = \nu \Delta q^2, 
\end{align*} 
which, after re-writing in the variables $(X,y,z)$ with $X = x-ty$ and $Q^2(t,X,y,z) = q^2(t,X+ty,y,z)$, becomes
\begin{align*} 
\partial_t Q^2  & = - \nu \Delta_L Q^2 \\ 
\partial_t \widehat{Q}^2(k,\eta,l) & = -\nu (k^2 + (\eta-kt)^2 + l^2) \widehat{Q}^2(k,\eta,l),
\end{align*}
which integrates to 
\begin{align*} 
\widehat{Q}^2(t,k,\eta,l) = \exp\left[-\nu \int_0^t (k^2 + (\eta-k\tau)^2 + l^2)\,d\tau \right] \widehat{Q^2_{in}}(k,\eta,l). 
\end{align*} 
The elementary inequality $\int_0^t (k^2 + (\eta-k \tau)^2 + l^2)\,d\tau \gtrsim t^3$ for $k$ a non-zero integer gives a decay $\sim e^{-c\nu t^3}$ for some $c > 0$ for all modes which depend on $X$. 
Essentially, since the Couette flow induces something like a linear-in-time transfer of information to high frequencies, the second order viscous dissipation behaves like an $O(\nu t^{2})$ damping.  
Combining this effect with the inviscid case (Proposition~\ref{prop:LinEuler}) gives the following: 

\begin{proposition}[Linearized Navier-Stokes] \label{prop:LinNSE} 
Let $u_{in}$ be a divergence free vector field with $u_{in} \in H^7$. The solution to the linearized Navier-Stokes equations $u(t)$ with initial data $u_{in}$ satisfies the following for some $c \in (0,1/3)$
\begin{subequations} 
\begin{align} 
\norm{u^{2}_{\neq}(t)}_{2} + \norm{u^{2}_{\neq}(t,x+ty,y,z)}_{H^3}  & \lesssim \jap{t}^{-2} e^{-c\nu t^3} \norm{u^2_{in}}_{H^7} \label{ineq:U2LinearID_vs} \\ 
\norm{u^1_{\neq}(t,x+ty,y,z)}_{H^1} & \lesssim e^{-c\nu t^3}  \norm{u_{in}}_{H^7} \label{ineq:U1LinearID_vs} \\ 
\norm{u^3_{\neq}(t,x+ty,y,z)}_{H^1} & \lesssim e^{-c\nu t^3} \norm{u_{in}}_{H^7}, \label{ineq:U3LinearID_vs} 
\end{align} 
\end{subequations} 
and the formulas
\begin{subequations} \label{def:NSEstreak}
\begin{align}
u^1_0(t,y,z) & = e^{\nu t \Delta}\left(u^1_{in \; 0} -  t u_{in \; 0}^2\right) \label{eq:liftup_vs} \\ 
u^2_0(t,y,z) & = e^{\nu t \Delta} u^2_{in \; 0} \\
u^3_0(t,y,z) & = e^{\nu t \Delta} u^3_{in \; 0}. 
\end{align}
\end{subequations}
\end{proposition} 
 Proposition \ref{prop:LinNSE} introduces two important time-scales: the mixing dissipation time scale $O(\nu^{-1/3})$ and the slow dissipation time scale $O(\nu^{-1})$. 
After $O(\nu^{-1/3})$, the $x$ dependence of the solution has essentially been completely damped, and the evolution is dominated by the simpler (linearized) streak evolution \eqref{def:NSEstreak}. 

\subsection{Streaks} \label{sec:streak}
The streaks are solutions to \eqref{def:3DNSE} which do not depend on $x$; one can verify that in this case, \eqref{def:3DNSE} reduces to 
\begin{subequations}\label{def:streak}
\begin{align} 
\partial_t u^i + u^2 \partial_y u^i + u^3 \partial_z u^i + \mathbf{1}_{i \neq 1}\partial_i p^{NL} & = \begin{pmatrix} - u^2 \\ 0 \\ 0 \end{pmatrix} + \nu \Delta u^i \\ 
\Delta p^{NL} & = -\mathbf{1}_{i \neq 1, j \neq 1}\partial_i u^j \partial_j u^i \\ 
\partial_y u^2 + \partial_z u^3 & = 0.   
\end{align}
\end{subequations}
Therefore, if $(u^2(t,y,z),u^3(t,y,z))$ solve the 2D Navier-Stokes (or 2D Euler) equations, then we may simply solve the forced, linear advection-diffusion equation for $u^1(t,y,z)$ and get an exact, global solution of the nonlinear dynamics \eqref{def:3DNSE} (since 2D Navier-Stokes and Euler are globally well-posed for reasonable initial data \cite{MajdaBertozzi}). 
Solutions of this general type are sometimes called ``2.5 dimensional'' \cite{MajdaBertozzi} and we will refer to this particular family as \emph{streaks}. 
That is: 
\begin{proposition}[Streak solutions] \label{prop:Streak} 
Let $\nu \in [0,\infty)$, $u_{in} \in H^{5/2+}$ be divergence free and independent of $x$, that is, $u_{in}(x,y,z) = u_{in}(y,z)$, and denote by $u(t)$ the corresponding unique strong solution to \eqref{def:3DNSE} with initial data $u_{in}$. Then $u(t)$ is global in time and for all $T > 0$, $u(t) \in L^\infty( (0,T);H^{5/2+}(\Real^3))$. 
Moreover, the pair $(u^2(t),u^3(t))$ solves the 2D Navier-Stokes/Euler equations on $(y,z) \in \Real \times \Torus$:
\begin{subequations} \label{def:2DNSEStreak} 
\begin{align} 
\partial_t u^i + (u^2,u^3)\cdot \grad u^i &  = -\partial_i p + \nu \Delta u^i \\ 
\partial_y u^2 + \partial_z u^3 & = 0,  
\end{align} 
\end{subequations} 
and $u^1$ solves the (linear) forced advection-diffusion equation
\begin{align} 
\partial_t u^1 + (u^2,u^3)\cdot \grad u^1 = -u^2 + \nu \Delta u^1. \label{eq:u1streak}
\end{align}  
\end{proposition}
Once a streak becomes $O(1)$ relative to the Couette flow, the solution will generally induce an unstable shear flow and therefore is expected to develop a secondary instability and transition either to a more complicated, nonlinear time-dependent state or directly to turbulence (see e.g. \cite{klebanoff1962,ReddySchmidEtAl98,Chapman02,SchmidHenningson2001} and \cite{BGM15II} for more discussion).
The instability predicted by formal arguments and numerical or physical experiments tends to be a transverse instability that involves a large growth in the $x$-dependent modes \cite{klebanoff1962,Orszag80,ReddySchmidEtAl98,SchmidHenningson2001}. 
Therefore, if one restricts initial data to be independent of $x$, the subcritical transition threshold for \eqref{def:3DNSE} is $\epsilon \sim \nu$ (that is $\gamma = 1$), as any larger perturbation will produce an unstable streak that would immediately transition in a real physical setting.  

Notice that if we do not have periodicity in $x$, then we cannot conveniently write down exact streak solutions unless we make a global, infinite energy perturbation.
It was shown in \cite{landahl80} that the same kind of kinetic energy growth is expected also for localized disturbances in the infinite channel case.
However, the long-time dynamics of the localized perturbations is a little bit more complicated. 
Streaks (or localized approximations of them) have been observed experimentally \cite{klebanoff1962,Elofsson1999,bottin98} and in computer simulations \cite{ReddySchmidEtAl98} and have been noted to be essentially optimal excitations of the linearized Navier-Stokes equations, and, more or less equivalently,  are expected to be on the leading edges of the pseudo-spectrum of the linearized Navier-Stokes equations \cite{TTRD93,Trefethen2005}. 
They are widely believed to be crucial to understanding the transition of \eqref{def:3DNSE} and of other related laminar flows \cite{SchmidHenningson2001}.

\subsection{Statement of result} 
The enhanced dissipation effect observed in Proposition \ref{prop:LinNSE} suggests that the streak solutions may be \emph{attractors} of the nonlinear dynamics near the transition threshold. 
This is the essential content of our results here and in \cite{BGM15II}, for initial data which is \emph{not too rough}. 

Our theorem requires the use of Gevrey regularity class \cite{Gevrey18}, defined on the Fourier side for $\lambda > 0$ and $s \in (0,1]$
\begin{align} 
\norm{f}_{\G^{\lambda;s}}^2 = \sum_{k,l}\int \abs{\hat{f}(k,\eta,l)}^2e^{2\lambda\abs{k,\eta,l}^s} d\eta.  
\end{align}
For $s = 1$ the class coincides with real analytic, however, for $s < 1$ it is less restrictive, for example, compactly supported functions can still be Gevrey class with $s < 1$.  
This class has appeared in most proofs involving inviscid damping \cite{BM13,BMV14} or Landau damping \cite{CagliotiMaffei98,HwangVelazquez09,MouhotVillani11,BMM13,Young14} in nonlinear PDE and in these previous works is associated with the nonlinear \emph{echo resonance} (see e.g. \cite{MalmbergWharton68,Vanneste02,VMW98,YuDriscollONeil,MouhotVillani11} and \S\ref{sec:Disc} and \S\ref{sec:Toy} -- the exception is \cite{FaouRousset14}, but the model considered therein satisfies a strong non-resonance condition which is not satisfied by most other physically relevant models). 
 We will have to deal with 3D variants of the echoes, so the presence of Gevrey class here is not surprising, although it is not obvious that the 2D work \cite{BM13,BMV14} and this work should both require the same Gevrey-2, since 2D nonlinear effects are much too weak to play a role in this work.

\begin{theorem} \label{thm:Threshold} 
For all $s \in (1/2,1)$, all $\lambda_0 > \lambda^\prime > 0$, all integers $\alpha \geq 10$, all $\delta_1>0$, and all $\nu \in (0,1]$, 
there exists constants $c_{00} = c_{00}(s,\lambda_0,\lambda^\prime,\alpha,\delta_1)$ and $K_0 = K_0(s,\lambda_0,\lambda^\prime)$ (both independent of $\nu$), such that for all $c_{0} \leq c_{00}$ and $\epsilon < c_{0} \nu$, if $u_{in} \in L^2$ is a divergence-free vector field that can be written $u_{in} = u_{S} + u_R$ (both also divergence-free) with
\begin{align} 
\norm{u_S}_{\mathcal{G}^{\lambda;s}} + e^{K_0\nu^{-\frac{3s}{2(1-s)}}}\norm{u_R}_{H^{3}} & < \epsilon, \label{ineq:QuantGev2}
\end{align} 
then the unique, classical solution $u(t)$ to \eqref{def:3DNSE} with initial data $u_{in}$ is global in time and the following estimates hold with all implicit constants independent of $\nu$, $\epsilon$, $t$ and $c_{0}$: 
\begin{itemize} 
\item[(i)] transient growth of the streak: if $t < \frac{1}{\nu}$, 
\begin{subequations} \label{ineq:trans}
\begin{align}
\norm{u^1_0(t) -  \left(e^{\nu t \Delta} \left(u^{1}_{in \; 0} - t u_{in \; 0}^2\right) \right) }_{\G^{\lambda^\prime;s}} & \lesssim c_{0}^2 \label{ineq:u01grwth}\\
\norm{u^{2}_0(t) -  e^{\nu t \Delta} u^{2}_{in \; 0} }_{\G^{\lambda^\prime;s}} + \norm{u^{3}_0(t) -  e^{\nu t \Delta} u^{3}_{in \; 0} }_{\G^{\lambda^\prime;s}} & \lesssim c_{0}\epsilon \label{ineq:u023} 
\end{align}
\end{subequations} 
\item[(ii)] uniform bounds and decay of the background streak 
\begin{subequations} 
\begin{align} 
\norm{u^1_0(t)}_{\G^{\lambda^\prime;s}} & \lesssim \min\left(\epsilon \jap{t},c_{0}\right) \label{ineq:finalu1} \\ 
\norm{u^2_0(t)}_{\G^{\lambda^\prime;s}} & \lesssim \frac{\epsilon}{\jap{\nu t}^{\alpha}} \label{ineq:finalu2} \\ 
\norm{u^3_0(t)}_{\G^{\lambda^\prime;s}} & \lesssim \epsilon \label{ineq:finalu3} \\  
\norm{u^1_0(t)}_4 & \lesssim \frac{c_{0}}{\jap{\nu t}^{1/4}} \label{ineq:finalu1L4}\\ 
\norm{u^3_0(t)}_4 & \lesssim \frac{\epsilon}{\jap{\nu t}^{1/4}}; \label{ineq:finalu3L4} 
\end{align}
\end{subequations}
\item[(iii)] the rapid convergence to a streak 
\begin{subequations} \label{ineq:uidamping}
\begin{align} 
\norm{u^1_{\neq}(t,x + ty + t\psi(t,y,z),y,z)}_{\G^{\lambda^\prime;s}} & \lesssim \frac{\epsilon \jap{t}^{\delta_1}}{\jap{\nu t^3}^\alpha} \label{ineq:u1damping}\\ 
\norm{u^2_{\neq}(t,x + ty + t\psi(t,y,z),y,z)}_{\G^{\lambda^\prime;s}} & \lesssim \frac{\epsilon}{\jap{t}^{2-\delta_1}\jap{\nu t^3}^\alpha}, \label{ineq:u2damping} \\ 
\norm{u^3_{\neq}(t,x + ty + t\psi(t,y,z),y,z)}_{\G^{\lambda^\prime;s}} & \lesssim \frac{\epsilon}{\jap{\nu t^3}^\alpha}. 
\end{align}  
\end{subequations}
Here $\psi(t,y,z)$ is an $O(c_{0})$ correction to the mixing which depends on the disturbance 
(defined below in \eqref{def:psiyz}), and satisfies
\begin{align} 
\norm{\psi(t) -  u_0^1(t)}_{\G^{\lambda^\prime;s}} \lesssim \epsilon \jap{t}^{-1}.   \label{ineq:psiest}
\end{align}
\end{itemize}
\end{theorem}

\begin{remark} 
We are only interested in $\nu$ small, so henceforth we will without loss of generality assume $c_0 \nu^{-1} \gg 1$. 
\end{remark} 

\begin{remark}
If $u_{in\;0}^2$ is such that $\norm{u_{in \; 0}^2}_{\G^{\lambda^\prime;s}} \geq \frac{1}{4}\epsilon = \frac{1}{16}c_{0}\nu$ then \eqref{ineq:u01grwth} shows that for $c_{0}$ small (but independent of $\epsilon$ and $\nu$), the streak $u_0^1(t)$ reaches the maximal amplitude of $\norm{u_0^1(t)}_2 \gtrsim c_{0}$. In this sense, Theorem \ref{thm:Threshold} includes perturbations which undergo dramatic growth of kinetic energy before decaying, so much in fact, that the solutions go from $O(\epsilon)$ to $O(c_0)$ before eventually decaying.  
Hence, we are far beyond the realm of monotonically stable perturbations \cite{DrazinReid81,SchmidHenningson2001} and are effectively on the edge of the linear/weakly nonlinear regime (as $c_0$ is independent of $\nu$).
\end{remark}

\begin{remark} 
By the previous remark, we note that the parameter $c_0$ is essentially the maximum size of $u_0^1(t)$. 
\end{remark}

\begin{remark}  \label{rmk:RegularityInstab}
That we need a regularity requirement beyond, for example $H^1$ (finite kinetic energy and finite enstrophy) or $L^3$ (for local well-posedness) is qualitatively consistent with the experimental and computer simulation observations that the nature of the disturbance can be important for determining the threshold or pathway to transition (see \cite{ReddySchmidEtAl98,SchmidHenningson2001,Yaglom12} and the references therein -- in fact, the sensitivity of subcritical transition was noted by Reynolds \cite{Reynolds83}). There are many instabilities observed in subcritical transition processes, and at lower regularity it is likely that other behaviors are possible besides that predicted by Theorem \ref{thm:Threshold} and \cite{BGM15II} (for example, see \cite{LinZeng11,LiLin11} and \cite{BGM15II} for more discussion). 
Further, the transition threshold at lower regularities may also be different, as is observed in computer experiments \cite{ReddySchmidEtAl98}.
That being said, we make no conjecture either way about whether or not Gevrey-2 (in the sense of \eqref{ineq:QuantGev2}) is the lowest regularity class for which $\gamma =1$ and one \emph{only} sees the streamwise vortex/streak instability. 
\end{remark}

\begin{remark} \label{rmk:periodicity} 
One would also like to be able to study unbounded channels in $x$. However, this is significantly more difficult, and the dynamics could potentially be different. 
The main reason for this is that mixing is very weak at the low $x$-frequencies in unbounded shear flows. 
For example, the Landau damping effect in plasma physics is known to be very weak in $\Real_x \times \Real_v$ \cite{glassey94,glassey95} (and can sometimes fail completely). 
However, some formal analyses suggest that at higher frequencies at least, the behavior on the infinite channel will be similar to that predicted by Theorem \ref{thm:Threshold} \cite{Chapman02}. 
\end{remark} 

An interesting by-product of the proof of Theorem \ref{thm:Threshold} is the demonstration of an open set of solutions to \eqref{def:3DNSE} which exhibit a linear-in-time flux of kinetic energy to high frequencies for times $1 \lesssim t \lesssim \nu^{-\frac{1}{3+\delta_1}}$.  
In particular, Proposition \ref{prop:KEflux} below quantifies that there are solutions with an $O(\epsilon)$ packet of kinetic energy in $u^1$ and/or $u^3$ which
for $1 \ll t \ll \nu^{-\frac{1}{3+\delta_1}}$ is roughly at length-scale $O(t^{-1})$ (after $t \sim \nu^{-1/3}$, Theorem \ref{thm:Threshold} shows that it is dissipated by the viscosity). 
As mentioned above, to our knowledge, this is the first rigorous confirmation of a direct energy cascade in the 3D Navier-Stokes equations in any setting. 
See \S\ref{sec:PropBootThm} for a proof: it is proved by taking data which satisfies $u^1_{\neq}, u^3_{\neq} = O(\epsilon)$ and $u^2_{\neq} = O(\epsilon^2)$.   
The time derivatives in \eqref{ineq:inviscidlinear} are showing that $u^1$ and $u^3$ are being moved to small scales while retaining their original profile to leading order.

\begin{proposition}[Direct cascade of kinetic energy]\label{prop:KEflux}
For all $\delta_1 > 0$, there exists an open set of global solutions to \eqref{def:3DNSE} such that for times $1 \lesssim t \ll \nu^{-\frac{1}{3+\delta_1}} \ll \epsilon^{-1/2}$ there holds (with constant independent of $t$, $\nu$, and $\epsilon$),
for all $ 0 < \sigma < \infty$,  
\begin{align}
\norm{u^1(t)}_{H^\sigma} + \norm{u^3(t)}_{H^\sigma} \gtrsim \epsilon \jap{t}^\sigma.  \label{ineq:nrmexplode}
\end{align}
More precisely, for $1 \lesssim t \ll \nu^{-\frac{1}{3+\delta_1}} \ll \epsilon^{-1/2}$ there holds (with constants independent of $t$, $\nu$, and $\epsilon$),
\begin{subequations} \label{ineq:inviscidlinear}
\begin{align}
\norm{\frac{d}{dt}\left(u^1_{\neq}(t,x+ty + t\psi(t,y,z),y,z)\right)}_{\G^{\lambda^\prime;s}} & \lesssim \epsilon \nu t^{2+\delta_1} + \epsilon^2 t^{1+\delta_1} \\ 
\norm{\frac{d}{dt}\left(u^3_{\neq}(t,x+ty + t\psi(t,y,z),y,z)\right)}_{\G^{\lambda^\prime;s}} & \lesssim \epsilon \nu t^2 + \epsilon^2 t^{1+\delta_1} \\ 
\norm{u^2_{\neq}(t,x+ty+t\psi(t,y,z),y,z)}_{\G^{\lambda^\prime;s}} & \lesssim  \epsilon \nu t, \label{ineq:u2grwKE}
\end{align}
\end{subequations} 
with $\psi$ satisfying $\norm{\psi}_{\G^{\lambda^\prime;s}} \lesssim \epsilon t$. 
\end{proposition}

\begin{remark}
Proposition \ref{prop:KEflux} is consistent with the qualitative behavior predicted by the linearized 3D Euler equations for $1 \lesssim t \ll \nu^{-\frac{1}{3+\delta_1}}$.  
\end{remark} 

\subsection{Discussion of Theorem \ref{thm:Threshold}} \label{sec:Disc}

The idea that enhanced dissipation could suppress nonlinear effects in \eqref{def:3DNSE}, and hence influence the stability threshold, seems to go back at least to \cite{DubrulleNazarenko94} but is also present, at least implicitly, in many works, and the expectation that a large mean shear should contribute to flow stability is expressed in many works with varying levels of precision (for example \cite{DrazinReid81,ReddySchmidEtAl98,Lundgren82,SchmidHenningson2001,Chapman02} and the references therein). 
However, without a detailed analysis of the interaction between the transient behaviors associated with non-normality and the nonlinear structure, 
it is not clear how to use these mechanisms to obtain results without any approximations. 
Carrying out and utilizing such an analysis is central to our proof (see \S\ref{sec:Toy}). 
 
The main focus of the toy models in \cite{TTRD93,Gebhardt1994,BDT95,Waleffe95,BT97,LHRS94} mentioned above was to study the interaction of the nonlinearity and the 3D lift-up effect and examine the possibility of transient growths bootstrapping to create instability. 
It is also possible to produce nonlinear bootstraps with the 2D Orr mechanism; see for example in 2D Euler/Navier-Stokes \cite{VMW98,Vanneste02,BM13} and 
in 3D \cite{Craik1971,SchmidHenningson2001}. 
This weakly nonlinear effect is usually referred to as an \emph{echo resonance} (more accurately a ``pseudo-resonance'' \cite{TTRD93,Trefethen2005}), and occurs when two modes interact to excite a mode which is unmixing, producing a strong response later in the future (hence ``echo'').  
Such echoes were isolated experimentally in 2D Euler near a vortex \cite{YuDriscoll02,YuDriscollONeil}. 
These echoes are analogous to the plasma echo phenomenon associated with Landau damping in Vlasov-Poisson, observed earlier in \cite{MalmbergWharton68}. 
The echo resonance is a central difficulty in the proofs of stability for 2D Couette flow in \cite{BM13,BMV14} and in the proofs of nonlinear Landau damping \cite{MouhotVillani11,BMM13}.  
In our work, we derive a new toy model which estimates the worst possible growth due to the leading order weakly nonlinear effects; see \S\ref{sec:Toy} below.
The toy model is used to understand the nonlinear interactions between the three components of the solution, mode-by-mode, accounting for the transient unmixing effects, the lift-up effect, the vortex stretching, and the stabilizing effects of the inviscid damping and enhanced dissipation.
From the model, we derive a set of norms in order the measure the solution that are well-adapted to the nonlinear structure, allowing us to retain more information about the solution than would be possible with standard norms.    

The idea of using toy models to design special norms was inspired by the previous work on the 2D Couette flow in \cite{BM13,BMV14} (one can also find analogies with Alinhac's ``ghost weight'' energy method \cite{Alinhac01}), 
 and although the leading order nonlinear effects are different in 3D, the structure of the Orr mechanism will ensure that the norms we arrive at here share similarities with those in \cite{BM13,BMV14} (and \cite{Zillinger2014}). 
Several other methods from the 2D works will find use here after adaptation to 3D (see \S\ref{sec:proof}), however, the proofs are ultimately fundamentally different. 
Here, the enhanced dissipation is the primary stabilizing mechanism and is used to control most nonlinear terms (because $\epsilon \lesssim \nu$); in 2D, the dissipation plays no such role and instead the inviscid damping is the primary stabilizing mechanism (hence why stability is possible at infinite Reynolds number \cite{BM13}). 
Further, regularity imbalancing at the critical times is far less important here than in \cite{BM13,BMV14} -- in this work it is not done at all and it is only done on the $z$ component of the solution in \cite{BGM15II} (and it is the unbalancing between different \emph{components} which is key).  
Finally, in 3D, the nonlinearity is much more complicated than it was in 2D, and we will need to understand all of the available cancellations. 
Of particular interest is structure which provides a kind of ``non-resonance'' so that the resonant frequencies and problematic components do not directly force each other too strongly (one may also draw analogy to the ``null condition'' from quasilinear wave equations \cite{Klainerman1986,Christodoulou1986,Alinhac01}); see \S\ref{sec:NonlinHeuristics} and \S\ref{sec:Toy} for a discussion of some of the more obvious structures, others we might point out when they arise in the proof.  
There is also a kind of null structure in \cite{BM13,BMV14}, but here it is much more subtle and complex.     

\subsection{Notations and Conventions} \label{sec:Notation}
We use superscripts to denote vector components and subscripts such as $\partial_i$ to denote derivatives with respect to the components $x,y,z$ (or $X,Y,Z$) with the obvious identification $\partial_1 = \partial_X$, $\partial_2 = \partial_Y$, and $\partial_3 = \partial_Z$. Summation notation is assumed: in a product, repeated vector and differentiation indices are always summed over all possible values.   

See Appendix \ref{apx:Gev} for the Fourier analysis conventions we are taking.
A convention we generally use is to denote the discrete $x$ (or $X$) frequencies as subscripts.   
By convention we always use Greek letters such as $\eta$ and $\xi$ to denote frequencies in the $y$ or $Y$ direction, frequencies in the $x$ or $X$ direction as $k$ or $k^\prime$ etc, and frequencies in the $z$ or $Z$ direction as $l$ or $l^\prime$ etc.
Another convention we use is to denote dyadic integers by $M,N \in 2^{\Integer}$ where  
\begin{align*} 
2^\Integer & = \set{...,2^{-j},...,\frac{1}{4},\frac{1}{2},1,2,...,2^j,...}. 
\end{align*}
This will be useful when defining Littlewood-Paley projections and paraproduct decompositions. See \S\ref{sec:paranote} for more information on the paraproduct decomposition and the associated short-hand notations we employ. 
Given a function $m \in L^\infty_{loc}$, we define the Fourier multiplier $m(\grad) f$ by 
\begin{align*} 
(\widehat{m(\grad)f})_k(\eta) =  m( (ik,i\eta,il) ) \hat{f}_k(\eta,l). 
\end{align*}   
We use the notation $f \lesssim g$ when there exists a constant $C > 0$ independent of the parameters of interest 
such that $f \leq Cg$ (we analogously define $f \gtrsim g$). 
Similarly, we use the notation $f \approx g$ when there exists $C > 0$ such that $C^{-1}g \leq f \leq Cg$. 
We sometimes use the notation $f \lesssim_{\alpha} g$ if we want to emphasize that the implicit constant depends on some parameter $\alpha$.
We also employ the shorthand $t^{\alpha+}$ when we mean that there is some small parameter $\gamma >0$ such that $t^{\alpha+\gamma}$ and that we can choose $\gamma$ as small as we want at the price of a constant (e.g. $\norm{f}_{L^\infty} \lesssim \norm{f}_{H^{3/2+}}$). 
We will denote the $\ell^1$ vector norm $\abs{k,\eta,l} = \abs{k} + \abs{\eta} + \abs{l}$, which by convention is the norm taken in our work. 
Similarly, given a scalar or vector in $\Real^n$ we denote
\begin{align*} 
\jap{v} = \left( 1 + \abs{v}^2 \right)^{1/2}. 
\end{align*} 
We denote the standard $L^p$ norms by $\norm{f}_{p}$ and Sobolev norms $\norm{f}_{H^\sigma} := \norm{\jap{\grad}^\sigma f}_2$.  
We make common use of the Gevrey-$\frac{1}{s}$ norm with Sobolev correction defined by 
\begin{align*} 
\norm{f}_{\G^{\lambda,\sigma;s}}^2 = \sum_{k,l}\int \abs{\hat{f}_k(\eta,l)}^2 e^{2\lambda\abs{k,\eta,l}^s}\jap{k,\eta,l}^{2\sigma} d\eta. 
\end{align*} 
Since most of the paper we are taking $s$ as a fixed constant, it is normally omitted. Also, if 
$\sigma =0$, it is omitted. 
We refer to this norm as the $\mathcal{G}^{\lambda,\sigma;s}$ norm and occasionally refer to the space of functions
\begin{align*} 
\mathcal{G}^{\lambda,\sigma;s} = \set{f \in L^2 :\norm{f}_{\G^{\lambda,\sigma;s}}<\infty}. 
\end{align*}
See Appendix \ref{apx:Gev} for a discussion of the basic properties of this norm. 

For $\eta \geq 0$, we define $E(\eta)\in \Integer$ to be the integer part.  
We define for $\eta \in \Real$ and $1 \leq \abs{k} \leq E(\sqrt{\abs{\eta}})$ with $\eta k > 0$, 
$t_{k,\eta} = \abs{\frac{\eta}{k}} - \frac{\abs{\eta}}{2\abs{k}(\abs{k}+1)} =  \frac{\abs{\eta}}{\abs{k}+1} +  \frac{\abs{\eta}}{2\abs{k}(\abs{k}+1)}$ and $t_{0,\eta} = 2 \abs{\eta}$ and 
the critical intervals  
\begin{align*} 
I_{k,\eta} = \left\{
\begin{array}{lr}
[t_{\abs{k},\eta},t_{\abs{k}-1,\eta}] & \textup{ if } \eta k \geq 0 \textup{ and } 1 \leq \abs{k} \leq E(\sqrt{\abs{\eta}}), \\ 
\emptyset & otherwise.
\end{array} 
\right. 
\end{align*} 
For minor technical reasons, we define a slightly restricted subset as the \emph{resonant intervals}
\begin{align*} 
\mathbf I_{k,\eta} = \left\{
\begin{array}{lr} 
I_{k,\eta} & 2\sqrt{\abs{\eta}} \leq t_{k,\eta}, \\ 
\emptyset & otherwise.
\end{array} 
\right. 
\end{align*} 
Note this is the same as putting a slightly more stringent requirement on $k$: $k \leq \frac{1}{2}\sqrt{\abs{\eta}}$.

\section{Outline of the proof}\label{sec:proof} 
In this section we give an outline of the main steps of the proof of Theorem \ref{thm:Threshold} and set up the main energy estimates, focusing on exposition, intuition, and organization. 
After \S\ref{sec:proof}, the remainder of the paper is dedicated to the proof of the energy estimates required and the analysis of the various norms and Fourier analysis tools being employed. 

\subsection{Summary and weakly nonlinear heuristics}

\subsubsection{New dependent variables}
It is common in linear and formal weakly nonlinear studies to work with $(q,\omega^2)$ (the second component of the vorticity) coupled with momentum equations for the $x$-independent velocities (see e.g. \cite{Chapman02,SchmidHenningson2001}) 
as these satisfy good equations and can be used to recover the other unknowns from the divergence free condition. 
However, we found this approach untenable with the new coordinate system we will employ below for several reasons, the primary one being that it significantly distorts the derivatives and hence the procedure for recovering the other unknowns while still retaining precise information on the regularities is far from clear (we will see that very precise information is necessary for the proof). 
Instead, we found it much more natural to define the full set of auxiliary unknowns (which in some sense entirely replace the role of the vorticity) $q^i = \Delta u^i$ for $i = 1,2,3$. A computation shows that $(q^i)$ solves
\begin{equation} \label{def:qi} 
\left\{ \begin{array}{l} 
\partial_t q^1 + y \partial_x q^1 +  2\partial_{xy} u^1 + u \cdot \grad q^1  = -q^2 + 2\partial_{xx} u^2 - q^j \partial_j u^1 + \partial_x\left(\partial_i u^j \partial_j u^i\right) - 2\partial_{i} u^j \partial_{ij}u^1 + \nu \Delta q^1  \\
\partial_t q^2 + y \partial_x q^2 + u \cdot \grad q^2  = -q^j \partial_j u^2 + \partial_y\left(\partial_i u^j \partial_j u^i\right) - 2\partial_{i} u^j \partial_{ij}u^2 + \nu \Delta q^2  \\
\partial_t q^3 + y \partial_x q^3 +  2\partial_{xy} u^3 + u \cdot \grad q^3  =  2\partial_{zx} u^2 -q^j \partial_j u^3 + \partial_z\left(\partial_i u^j \partial_j u^i\right) - 2\partial_{i} u^j \partial_{ij}u^3 + \nu \Delta q^3 .
\end{array}\right.
\end{equation}
As seen in \S\ref{sec:lin}, the linear terms have disappeared from the $q^2$ equation, leaving only the nonlinear terms on the RHS. 
 Note that the equations on $q^1$ and $q^3$ are far less favorable in that they contain linear terms which are associated with the vortex stretching.   
\subsubsection{New independent variables} \label{sec:indepC} 
The need for a change of independent variables can be understood by considering the convection term $y\partial_x q^i + u \cdot \nabla q^i$ which appears in \eqref{def:qi} above. 
Due to the mixing of the Couette flow, any classical energy estimates on $q$ in (say) Sobolev spaces will rapidly grow (as seen in \eqref{ineq:nrmexplode}). 
Moreover, $u_0^1$ will grow like $O(\epsilon \jap{t})$ via the lift-up effect and hence could be $\sim c_0$ by $t \sim \frac{c_0}{\nu}$. 
Therefore, the $u_0^1$ component of the streak will have a major effect on the mixing and will need to be accounted for. 
A full study of the coordinate transformation is done in \S\ref{sec:coordinates} below, but let us just make a quick summary here. 
We start with the ansatz 
\begin{align*}
\left\{ \begin{array}{l} X  = x - ty - t \psi(t,y,z) \\ Y  = y + \psi(t,y,z) \\ Z  = z, \end{array} \right.
\end{align*}
(this is motivated by a further requirement that $\Delta^{-1}$ have good properties in the new coordinates, see \S\ref{sec:coordinates} below). 
Consider the simple convection diffusion equation on a passive scalar $f(t,x,y,z)$
$$
\partial_t f + y \partial_x f + u \cdot \nabla f = \nu \Delta f.
$$
Denoting $F(t,X,Y,Z) = f(t,x,y,z)$ and $U(t,X,Y,Z) = u(t,x,y,z)$, and $\Delta_t$ and $\nabla^t$ for the expressions for $\Delta$ and $\nabla$ in the new coordinates, this simple equation becomes
\begin{align} 
\partial_t F + \begin{pmatrix}u^1 - t (1+\partial_y\psi) u^2 - t \partial_z\psi u^3 - \frac{d}{dt}(t\psi) + \nu t \Delta \psi \\ (1+\partial_y\psi) u^2 + \partial_z \psi u^3 + \partial_t \psi - \nu\Delta \psi  \\ u^3 \end{pmatrix} \cdot \grad_{X,Y,Z} F = \nu \tilde{\Delta_t} F,  \label{ineq:transf}
\end{align}
where $\tilde{\Delta_t}$ is a variant of $\Delta_t$ without lower order terms; it is given below in \eqref{def:tildeDel1}. 
Eliminating $u^1_0$ leads to the equation $u^1_0 - t (1+\partial_y\psi) u^2_0 - t \partial_z\psi u^3_0 - \frac{d}{dt}(t\psi) + \nu t\Delta \psi = 0$. 
We now recast this equation on $\psi$ in terms of $C(t,Y,Z)=\psi(t,y,z)$ and an auxiliary unknown $g = \frac{1}{t}(U_0^1 - C)$ (this roughly measures the time-oscillations of $C$).
A variety of cancellations which take advantage of the precise structures give
\begin{equation} \label{def:Cgintro}
\left\{
\begin{array}{l}
\partial_t C + \tilde U_0 \cdot \grad_{Y,Z} C = g - U_0^2 + \nu \tilde{\Delta_t} C, \\
\partial_t g  + \tilde U_0 \cdot \grad_{Y,Z}g = -\frac{2}{t}g -\frac{1}{t} \left(U_{\neq} \cdot \grad^t U^1_{\neq}\right)_0 + \nu \tilde{\Delta_t}  g,
\end{array}
\right.
\end{equation}
where $\tilde U = \begin{pmatrix} U_{\neq}^1 - t(1 + \psi_y ) U^2_{\neq} - t \psi_z U^3_{\neq}  \\ (1 + \psi_y)U^2_{\neq} + \psi_z U^3_{\neq}+ g \\ U^3 \end{pmatrix}$. 
Coming back to \eqref{def:qi}, we further derive in the new coordinates ($Q(t,X,Y,Z) = q(t,x,y,z)$).
\begin{align} \label{def:Qiintro}
\left\{
\begin{array}{l}
Q^1_t + \tilde U \cdot \grad_{X,Y,Z} Q^1 = -Q^2 - 2\partial_{XY}^t U^1 + 2\partial_{XX} U^2 - Q^j \partial_j^t U^1 - 2\partial_i^t U^j \partial_{ij}^t U^1 + \partial_X(\partial_i^t U^j \partial_j^t U^i) + \nu \tilde{\Delta_t} Q^1 \\ 
Q^2_t + \tilde U \cdot \grad_{X,Y,Z} Q^2 = -Q^j \partial_j^t U^2 - 2\partial_i^t U^j \partial_{ij}^t U^2 + \partial_Y^t(\partial_i^t U^j \partial_j^t U^i) + \nu \tilde{\Delta_t} Q^2 \\ 
Q^3_t + \tilde U \cdot \grad_{X,Y,Z} Q^3 = -2\partial_{XY}^t U^3 + 2\partial_{XZ}^t U^2 - Q^j \partial_j^t U^3 - 2\partial_i^t U^j \partial_{ij}^t U^3 + \partial_Z^t(\partial_i^t U^j \partial_j^t U^i) + \nu \tilde{\Delta_t} Q^3, \\  
\end{array}
\right.
\end{align} 
We will perform most of our estimates on the coupled systems \eqref{def:Qiintro} and \eqref{def:Cgintro}, recovering $U^i$ from $Q^i$ using $\Delta^{-1}$ in the new coordinates.  
The one exception is in controlling very low frequencies of the velocity fields, for which we go back to the momentum formulation.  

\subsubsection{Choice of the norms}
The proof of Theorem \ref{thm:Threshold} relies on a bootstrap argument, or a priori estimate, for a certain set of norms of the solution. The choice of the norms is extremely delicate and amounts to describing precisely the possible distribution of information in Fourier space for $Q$ and $C$. 
The highest norms are derived from the toy model explained in \S\ref{sec:Toy}; the full rigorous definition of the norms is carried out in Appendix \ref{sec:def_nrm}. 
Each $Q^i$ is measured with a slightly different norm, of the form $\norm{A^i(t,\grad) Q^i(t)}_2$ where $A^i(t,\grad)$ are special Fourier multipliers.   
Let us just describe the  norm used to measure $Q^3$, the rest are defined below in \eqref{def:A},  
\begin{align*}
A^3_k(t,\eta,\ell) = e^{\lambda(t)\abs{k,\eta,l}^s}\jap{k,\eta,l}^\sigma \frac{e^{\mu \abs{\eta}^{1/2}}}{w(t,\eta) w_L(t,k,\eta,l)}\left(\mathbf{1}_{k \neq 0} \min\left(1, \frac{\jap{\eta,l}^2}{t^2}\right) + \mathbf{1}_{k = 0} \right). 
\end{align*} 
We now comment on the different components: $e^{\lambda(t)|k,\eta,\ell|^s}$ corresponds to a Gevrey-$\frac{1}{s}$ norm, with decreasing radius, while $\langle k,\eta,\sigma \rangle^\sigma$ gives a Sobolev correction (mainly for technical convenience). 
The next factors correspond to important physical effects and are derived in \S\ref{sec:Toy}.     
The factor $w$ is an estimate on the ``worst-possible'' growth of high frequencies due to weakly nonlinear effects. 
Roughly speaking, it is taken to satisfy the following for $\abs{k}^2 \lesssim \abs{\eta}$ (hence $\sqrt{\abs{\eta}} \lesssim t \lesssim \abs{\eta}$), 
\begin{align*}
\frac{\partial_t {w(t,\eta)}}{w(t,\eta)} \sim \frac{1}{1 + |t-\frac{\eta}{k}|}, \qquad \mbox{when $\left| t-\frac{\eta}{k} \right| \lesssim \frac{\eta}{k^2}$} \quad \mbox{and} \quad w(1,\eta) = 1.
\end{align*}
By requiring that $w$ be increasing, from Stirling's formula this relation leads to a growth of the type $\frac{w(2\abs{\eta},\eta)}{w(1,\eta)} \approx e^{\frac{\mu}{2}|\eta|^{1/2}}$ (up to a small polynomial correction), hence the choice of the numerator in $ \frac{e^{\mu|\eta|^{1/2}}}{w(t,\eta)}$, and the Gevrey-2 regularity requirement (see \S\ref{sec:Toy} below for more details). 
The upshot of including the $\frac{1}{w}$ factor in the definition of the norm is that it allows to estimate $\frac{1}{\sqrt{1+\left| t-\frac{\eta}{k} \right|}}A^i Q^i$ in spaces of the type $L^2_t L^2_x$, close to any resonant time $\frac{\eta}{k}$. In other words, this gives additional time integrability around the times at which the Orr mechanism is the strongest. 
The multiplier $w_L$ is a uniformly bounded multiplier that corrects for the anisotropy of the bounded growth experienced due to linear pressure effects (the $L$ stands for `linear'). 

The last factor corresponds to a growth occurring for times large compared to the frequency due to the linear vortex stretching. That $Q^1$ and $Q^3$ ultimately grow at least quadratically is evident on the linear level: 
 we saw in Proposition \ref{prop:LinEuler} that, due to vortex stretching, no inviscid damping occurs in general. 
That the growth can (and should) be taken only for frequencies small relative to time is less clear, and is explained more in \S\ref{sec:Toy}. 

While the norm which was just sketched corresponds to the highest regularity estimate, estimates at lower regularity are also needed, in particular to quantify the enhanced dissipation. For this, we use an approach similar to that employed for 2D NSE in \cite{BMV14}, and define a set of semi-norms of the type $\norm{A^{\nu;i}(t,\grad) Q^i}_2$ for specially designed Fourier multipliers $A^{\nu;i}(t,\grad)$ (see \eqref{def:Anu} below). For example, for $Q^3$:  
\begin{align*}
A^{\nu;3} = e^{\lambda(t)|k,\eta,\ell|^s} \langle k,\eta,\sigma \rangle^\beta \jap{D(t,\eta)}^\alpha \frac{1}{w_L(t,k,\eta,l)} \min\left(1, \frac{\jap{\eta,l}^2}{t^2}\right) \mathbf{1}_{k \neq 0}, 
\end{align*} 
where $D(t,\eta) \gtrsim \nu t^3$, hence giving the $\nu^{-1/3}$ time scale of enhanced dissipation. 
The multiplier $D$ is adapted to the fact that the enhanced dissipation slows down near critical times, another effect of the Orr mechanism (see \ref{def:D} below). 

\subsubsection{Basic weakly nonlinear heuristics} \label{sec:NonlinHeuristics}
While the actual derivation of the estimates occupies most of this paper, and is of great complexity, we are guided by a few basic principles.
The behavior of the solution is roughly divided into three phases. 
During the early times, $t \lesssim  \nu^{-1/3}$, 
the solution has fully 3D nonlinear effects until the enhanced dissipation eventually dominates. 
During the middle times, $\nu^{-1/3} \lesssim t \lesssim \nu^{-1}$, the solution is mostly in $x$-independent modes and is slowly growing via the lift-up effect. 
 By the time $t \gtrsim \nu^{-1}$ the solution has, in general, become extremely large relative to $\nu$ but it is also very close to a globally regular $x$-independent streak and eventually returns to Couette.  
At the middle and later times, the goal is to prove that the solution retains this special structure.

There are several nonlinear mechanisms which have the potential to cause instability and many have been proposed as important in the applied mathematics and physics literature for understanding transition, see e.g. \cite{Craik1971,TTRD93,ReddySchmidEtAl98,SchmidHenningson2001} and the references therein. 
We are particularly worried about so called ``bootstrap'' mechanisms \cite{TTRD93,Trefethen2005,Waleffe95,BM13}: nonlinear interactions 
that repeatedly excite growing linear modes. 
We will classify the main effects by the $x$ frequency of the interacting functions: denote for instance $0 \cdot \neq \,\to\, \neq$ for the interaction of a zero mode (in $x$) and a non-zero mode (in $x$) giving a non-zero mode (in $x$), and similarly, with obvious notations, $0 \cdot 0 \to 0$, $\neq \cdot \neq \, \to \,\neq$, and $\neq \cdot \neq \,\to 0$.

\begin{itemize} 
\item[\textbf{(2.5NS)}]  ($0 \cdot 0 \to 0$)  For  \emph{2.5D Navier-Stokes}, this corresponds to self-interactions of the streak. 
This interaction is not the most worrisome, however, we will need to get reasonable decay estimates on the streak. 
\item[\textbf{(SI)}] ($0 \cdot \neq \,\to\, \neq$) For \emph{secondary instability}, this effect is the transfer of momentum from the large $u_0^1$ mode to other modes (actually $u_0^2$ and $u_0^3$ also matter, but less). 
These involve a zero frequency and non-zero frequency $k$ interacting to amplify the same mode $k$, or the $k$ mode of a different component, e.g. $u_0^1$ and $u_k^3$ together force $u_k^2$. These interactions are those that arise when linearizing an $x$-dependent perturbation of a streak and so are what ultimately give rise to the secondary instabilities observed in larger streaks  (hence the terminology) \cite{ReddySchmidEtAl98,Chapman02}. 
\item[\textbf{(3DE)}] ($\neq \cdot \neq \,\to \,\neq$) For \emph{three dimensional echoes}, these effects are 3D variants of the 2D hydrodynamic echo phenomenon as observed in \cite{YuDriscoll02,YuDriscollONeil}: nonlinear interactions of $x$-dependent modes forcing unmixing modes \cite{Morrison98,Vanneste02,BM13} -- a nonlinear manifestation of the Orr mechanism (also recall from \S\ref{sec:lin} that this mostly involves frequencies in $x$ and $y$). 
In 3D, this is still important and the range of possible interactions is much wider (see e.g. \cite{Craik1971,SchmidHenningson2001} and \S\ref{sec:Toy} below). 
This involves two non-zero frequencies $k_1$, $k_2$ interacting to force mode $k_1 + k_2$ with $k_{1},k_2,k_1 + k_2 \neq 0$. 
Since these involve the interaction of only non-zero frequencies, they should only be problematic for short times: for $t \gtrsim \nu^{-1/3}$, this effect should be wiped out by the enhanced dissipation. 
\item[\textbf{(F)}] ($\neq \cdot \neq \,\to 0$) For \emph{nonlinear forcing}, this is the effect of the forcing from $x$-dependent modes back into $x$-independent modes.
This involves two non-zero frequencies $k$ and $-k$ interacting to force a zero frequency (in general this could involve a variety of the components). Similar to \textbf{(3DE)}, this effect is over-powered by the enhanced dissipation after $t \gtrsim \nu^{-1/3}$. 
\end{itemize} 
All of these are coupled to one another, and one can imagine bootstrap mechanisms involving several of them (e.g. $u_0^1$ forces a non-zero mode which unmixes and then strongly forces $u_0^2$ which strongly forces $u_0^1$ via the lift-up effect  and repeat: $\textbf{(SI)} \Rightarrow \textbf{(3DE)} \Rightarrow \textbf{(F)} \Rightarrow \textbf{(SI)}$). 
It is the need to consider exactly these kinds of nonlinear bootstraps that eventually precipitates the Gevrey-$2$ regularity requirement, as we will derive formally in \S\ref{sec:Toy} (although one can derive the regularity requirement using just the interaction between $\textbf{(SI)}$ and $\textbf{(3DE)}$). 
In \S\ref{sec:Toy} we derive a special ``toy model'' meant to model the worst possible growth in $Q$ caused by these effects, and then use it to derive the norms $A^i$. 
These specially adapted norms allow us to retain the maximal amount of control over the solution, despite the possible de-stabilizing effects. 

There are various ways in which the special structure of the nonlinearity is key to the proof of Theorem \ref{thm:Threshold}. 
The simplest such structure can be seen in the ubiquitous  $U^j \partial_j^t$ in \eqref{def:Qiintro}. This will balance the bad $U^1$ with the best derivative $\partial_X$ whereas $U^2$ will pair with the $\partial_Y^t$, so that the inviscid damping in $U^2$ will balance the time-growth (indeed, this is reason inviscid damping is important in the proof of Theorem \ref{thm:Threshold}). 
Much more subtle structural properties, such as the way different resonant and non-resonant frequencies interact, will become apparent in \S\ref{sec:Toy} and elsewhere in the proof (and \cite{BGM15II}). 

Now that we have some intuition about the approach, we now begin a detailed outline of the proof of Theorem \ref{thm:Threshold} and set up the main energy estimates.   

\subsection{Instantaneous regularization and continuation of solutions}
To begin the full proof of Theorem \ref{thm:Threshold}, the first step is to see that our initial data becomes small in $\G^{\lambda;s}$ (for an appropriate $\lambda$) after a short time. 
The proof is a straightforward variant of existing parabolic regularizations, see \S\ref{sec:RegCont} for a brief sketch.
 
\begin{lemma}[Local existence and instanteous regularization] \label{lem:Loc} 
Let $u_{in} \in L^2$ satisfy \eqref{ineq:QuantGev2}. 
Then for all $\nu\in (0,1]$, $c_{0}$ sufficiently small, $K_0$ sufficiently large, and all $\lambda_0 > \lambda^\prime > 0$, if $u_{in}$ satisfies \eqref{ineq:QuantGev2}, then there exists a time $t_\star = t_\star(s,K_0,\lambda_0,\lambda^\prime) > 0$ and a unique classical solution to \eqref{def:3DNSE} with initial data $u_{in}$ on $[0,t_\star]$ which is real analytic on $(0,t_\star]$, and satisfies
\begin{align} 
\sup_{t \in [t_\star/2,t_\star]} \norm{u(t)}_{\G^{\bar{\lambda}}} \leq \frac{10}{9}\epsilon, 
\end{align} 
where $\bar{\lambda} = \frac{9\lambda_0}{10} + \frac{\lambda^\prime}{10}$. 
\end{lemma}  

Once such a solution exists, we will be able to continue it and ensure also that it propagates real analyticity as long as a certain Sobolev norm remains finite. 
This will allow us to rigorously justify our a priori estimates and make sure the solution to our transformed system corresponds to the solution of \eqref{def:3DNSE}.   
Analyticity itself is not important, the main point is that the solution propagates some regularity a few derivatives higher than the Gevrey-$\frac{1}{s}$ that we are working in below so that we may easily justify that the norms we are considering take values continuously in time.  

\begin{lemma}[Continuation and propagation of analyticity] \label{lem:Cont} 
Let $T > 0$ be such that the classical solution $u(t)$ to \eqref{def:3DNSE} constructed in Lemma \ref{lem:Loc} exists on $[0,T]$ and is real analytic for $t \in (0,T]$. 
Then there exists a maximal time of existence $T_0$ with $T < T_0 \leq \infty$ such that the solution $u(t)$ remains unique and real analytic on $(0,T_0)$. 
Moreover, if for some $\tau \leq T_0$ and $\sigma > 5/2$ we have $\limsup_{t \nearrow \tau} \norm{u(t)}_{H^{\sigma}} < \infty$, then $\tau < T_0$.  
\end{lemma} 

\subsection{$Q^i$ formulation, the coordinate transformation, and some key cancellations}  \label{sec:coordinates}
The next step in the proof of Theorem \ref{thm:Threshold} 
is the observation that if we are to get good estimates, then we need to remove the fast mixing action of \emph{both} the Couette flow \emph{and} the background streak.
The approach is inspired by the coordinate transforms of \cite{BM13,BMV14}, 
however, here it will be significantly different due to the lift-up effect, the $z$ dependence, and the different kind of estimates we require. 

Our coordinate system needs to satisfy three features for the proof to work:
\begin{enumerate} 
\item the relative velocity field in the new variables must be time integrable or $O(\epsilon)$ (so that the dissipation can control it) as otherwise one cannot propagate uniform estimates for $t \rightarrow \infty$;  
\item we need to be able to treat the Laplacian in the new coordinates as a perturbation from $\Delta_L$, so that we can take advantage of the inviscid damping and enhanced dissipation effects;
\item we need to be able to make practical estimates on the behavior of the coordinate system and the coordinate transformation needs to treat the dissipation in a natural way. 
\end{enumerate}
The second requirement will translate to a need to get derivatives of the form $(\partial_Y - t\partial_X)$, as opposed to something like $(\partial_Y - t C(Y,Z) \partial_X)$, which we do not know how to treat (in particular, it is not obvious how to isolate the critical times). For this to happen, we need the coordinate transform to be of the form:  
\begin{subequations} \label{def:XYZ}
\begin{align} 
X & = x - ty - t \psi(t,y,z) \\ 
Y & = y + \psi(t,y,z) \\ 
Z & = z. 
\end{align} 
\end{subequations}
We could imagine also changing the $z$ variable, but we want our coordinate system to be as simple as possible (although we will do this in \cite{BGM15II}). 
Notice that $Y$ has units of velocity, indeed, we have transformed this coordinate not to measure the vertical location but instead to measure a sort of time-average of the shear velocity.
Provided $\psi$ is sufficiently small in a suitable sense, one can invert \eqref{def:XYZ} for $x,y,z$ as functions of $X,Y,Z$ (see \S\ref{sec:RegCont} below for more information).  
In what follows, denote the Jacobian factors (by slight abuse of notation), which, in some sense, measure space and time variations in the background mixing flow $(y + u_0^1,0,0)$ (note that $Z = z$ in \eqref{def:XYZ}): 
\begin{align*} 
\psi_t(t,Y,Z) & = \partial_t \psi(t,y(t,Y,Z),Z) \\ 
\psi_y(t,Y,Z) & = \partial_y \psi(t,y(t,Y,Z),Z) \\ 
\psi_z(t,Y,Z) & = \partial_z \psi(t,y(t,Y,Z),Z). 
\end{align*}
In what follows we will usually omit the arguments of $y(t,Y,Z)$ and use a more informal notation, such as $\psi_t(t,Y,Z) = \partial_t \psi(t,y,z)$.  

Next we determine $\psi$. 
As in \S\ref{sec:indepC}, for motivation, imagine now a passive scalar transport equation in a perturbation of the shear flow:
\begin{align} 
\partial_t f + y\partial_x f + u\cdot \grad f = \nu \Delta f. \label{def:ftrans} 
\end{align}
Denoting $F(t,X,Y,Z) = f(t,x,y,z)$ and $U(t,X,Y,Z) = u(t,x,y,z)$, we can check that the derivatives transform according to
\begin{equation*}
\left\{ \begin{array}{l}
\partial_t f = ( \partial_t + \psi_t\partial_Y - y\partial_X - \frac{d}{dt}(t\psi) \partial_X) F \\
\partial_x f = \partial_X F \\
\partial_y f = (1 + \psi_y)(\partial_Y - t\partial_X) F\\
\partial_z f = (\partial_Z + \psi_z (\partial_Y - t\partial_X) ) F.
\end{array} \right.
\end{equation*}
It will be useful to define notation for the $(x,y,z)$ derivatives in the new coordinate system
\begin{subequations} \label{def:transD}  
\begin{align} 
\partial_X^t & = \partial_X \\ 
\partial_Y^t & = (1 + \psi_y)(\partial_Y - t\partial_X) \\ 
\partial_Z^t & = \partial_Z  + \psi_z(\partial_Y - t\partial_X) \\ 
\grad^t & = (\partial_X, \partial_Y^t, \partial_Z^t)^{T}.  
\end{align}
\end{subequations}
Note that necessarily these derivatives commute. 
The transport equation \eqref{def:ftrans} in the new coordinate system is given by (it will be more clear to keep $\psi$ and its derivatives in $(y,z)$ for a moment)
\begin{align}  
\partial_t F + \begin{pmatrix}u^1 - t (1+\partial_y\psi) u^2 - t \partial_z\psi u^3 - \frac{d}{dt}(t\psi) \\ (1+\partial_y\psi) u^2 + \partial_z \psi u^3 + \partial_t \psi  \\ u^3 \end{pmatrix} \cdot \grad_{X,Y,Z} F = \nu \Delta_t F,  \label{ineq:transf}
\end{align}
where the upper-case letters are evaluated at $(X,Y,Z)$ and the lower case letters (including $\psi$ and its derivatives) are evaluated at $(x,y,z)$, 
and we are denoting 
\begin{align} 
\Delta_t F = \partial_{XX} F+ \partial_Y^t\left(\partial_Y^t F\right) + \partial_Z^t\left(\partial_Z^t F\right).
\end{align}
Note that
\begin{align*}
\Delta_t F&  = \partial_{XX}F + \left((1+\partial_y\psi)^2 + (\partial_z\psi)^2\right)(\partial_Y - t\partial_X)^2 F + \partial_{ZZ} F\\ & \quad  + 2\partial_z \psi \partial_Z (\partial_Y - t\partial_X) F + \Delta \psi (\partial_Y - t\partial_X) F, 
\end{align*}
and denote
\begin{align} 
\tilde{\Delta_t}F & = \Delta_t F - \Delta \psi (\partial_Y - t\partial_X)F. \label{def:tildeDel1}
\end{align} 
Then, \eqref{ineq:transf} becomes
\begin{align} 
\partial_t F + \begin{pmatrix}u^1 - t (1+\partial_y\psi) u^2 - t \partial_z\psi u^3 - \frac{d}{dt}(t\psi) + \nu t \Delta \psi \\ (1+\partial_y\psi) u^2 + \partial_z \psi u^3 + \partial_t \psi - \nu\Delta \psi  \\ u^3 \end{pmatrix} \cdot \grad_{X,Y,Z} F = \nu \tilde{\Delta_t} F.  \label{ineq:transf2}
\end{align}
Velocity fields of the type appearing in \eqref{ineq:transf2} will henceforth be referred to as ``relative velocity fields''.  
With the lift-up effect in mind, it is the $X$ average of the first component which can grow to $O(1)$ sizes and does not decay appreciably, whereas $x$ dependent modes will be suppressed by inviscid damping and enhanced dissipation.  
Hence, we derive
\begin{subequations} \label{def:psiyz}
\begin{align} 
u_0^1 - t\left( 1 + \partial_y \psi \right)u_0^2 - t \partial_z \psi u_0^3 - \frac{d}{dt}(t\psi) & = -\nu t\Delta \psi \label{def:psi2} \\
\lim_{t \rightarrow 0} t\psi(t,y,z) = 0.  
\end{align} 
\end{subequations} 
The form of the dissipation on the RHS of \eqref{def:psi2} will also ensure that the evolution equations we derive next have a natural dissipative term, rather than an unnatural loss of derivatives. 

In what follows, we denote
\begin{align} 
C(t,Y,Z) & = \psi(t,y,z). 
\end{align} 
From the chain rule we derive:
\begin{subequations} \label{eq:psiyzt}    
\begin{align} 
\psi_y & = \partial_Y^tC = \left(1 + \psi_y\right) \partial_Y C \\  
\psi_z & = \partial_Z^t C =\partial_Z C + \psi_z \partial_YC \\ 
\psi_t & = \partial_t C + \psi_t \partial_YC, \label{ineq:psit} 
\end{align} 
\end{subequations}
and hence, if $C$ is small in the appropriate norm, then we can write
\begin{subequations} \label{def:psi2sqrBrack} 
\begin{align} 
\psi_y & =  \frac{\partial_Y C}{1 - \partial_YC} = \partial_Y C\sum_{j = 0}^\infty (\partial_YC)^j \\ 
\psi_z & = \frac{\partial_Z C}{1 - \partial_YC} = \partial_Z C\sum_{j = 0}^\infty (\partial_YC)^j  \\ 
\psi_t & = \frac{\partial_t C}{1 - \partial_YC} = \partial_t C \sum_{j = 0}^\infty (\partial_YC)^j.  
\end{align} 
\end{subequations} 
From the chain rule together with \eqref{def:psi2}, we derive 
\begin{align}
\partial_t C + \begin{pmatrix} \left(1 + \psi_y\right)U_0^2 + \psi_z U^3_0 + \psi_t - \nu \Delta \psi \\ U_0^3 \end{pmatrix} \cdot \grad_{Y,Z} C = \frac{1}{t}\left(U_0^1 - tU_0^2 - C\right) + \nu \tilde{\Delta_t} C. \label{def:C}
\end{align}
Estimates on $C$ via \eqref{def:C} will enable us to get estimates on the Jacobian factors through \eqref{def:psi2sqrBrack}. 
However, based on the behavior suggested by the linear problem in \S\ref{sec:lin}, it is not clear that the linear terms on the right hand side of \eqref{def:C} are decaying fast enough. 
A similar issue arose in \cite{BM13,BMV14} where good estimates on the coordinate system required additional estimates on the background shear flow. 
In a similar spirit, we will define another auxiliary unknown $g$ which, roughly speaking, is a measure of the time oscillations of the $x$ component of the streak:   
\begin{align} 
g = \frac{1}{t}\left(U_0^1 - C\right). \label{def:g}
\end{align}
Taking $x$ averages of \eqref{def:3DNSE}  and using \eqref{def:transD}, we get  
\begin{align}  
\partial_t U_0^1 + \begin{pmatrix} (1  + \psi_y)U_0^2 + \psi_zU_0^3 + \psi_t - \nu \Delta \psi \\ U_0^3 \end{pmatrix} \cdot \grad_{Y,Z}U_0^1 = - U_0^2 - \left(U_{\neq} \cdot \grad^t U^1_{\neq}\right)_0 + \nu \tilde{\Delta_t} U_0^1.  
\end{align}
Therefore, together with \eqref{def:C}, we derive the following evolution for $g$: 
\begin{align}
\partial_t g  + \begin{pmatrix} (1  + \psi_y)U_0^2 + \psi_z U_0^3 + \psi_t - \nu \Delta \psi \\ U_0^3 \end{pmatrix} \cdot \grad_{Y,Z}g = -\frac{2}{t}g -\frac{1}{t}\left(U_{\neq} \cdot \grad^t U^1_{\neq}\right)_0 + \nu \tilde{\Delta_t} g.  \label{def:gPDE1}
\end{align} 
The most crucial observation here is that the lift-up effect terms canceled.  
This allows us to make good estimates on the regularity of $g$ and, in particular, a decay estimate using the linear damping term in \eqref{def:gPDE1}. 
In turn, this gets us good control on $C$ via \eqref{def:C}.
Note that $\Delta \psi = \Delta_t C$, and hence it will be the following relative velocity field that will govern our equations: 
\begin{align} 
\tilde U = \begin{pmatrix} U_{\neq}^1 - t(1 + \psi_y ) U^2_{\neq} - t \psi_z U^3_{\neq}  \\ (1 + \psi_y)U^2 + \psi_z U^3+ \psi_t - \nu \Delta_t C \\ U^3 \end{pmatrix}, \label{def:tildeU}
\end{align}
From \eqref{def:C}, \eqref{def:g} and \eqref{ineq:psit} we derive
\begin{align} 
\psi_t - \nu\Delta_t C & = g - U_0^2  - \begin{pmatrix} \left(1 + \psi_y \right)U_0^2 + \psi_z U^3_0  \\ U_0^3 \end{pmatrix} \cdot \grad_{Y,Z} C.   \label{def:dtpsi}
\end{align}
We see that the $\nu \Delta_t C$ factor cancels with the similar term in $\tilde U^2$ in \eqref{def:tildeU}. 
Further, from \eqref{eq:psiyzt}, we have the following convenient cancellation:  
\begin{align} 
\tilde U_0^2 = (\psi_y - \partial_YC - \psi_y \partial_YC) U^2_0 + g + (\psi_z - \psi_z \partial_YC - \partial_ZC)U_0^3 = g.   
\end{align} 
It follows that the relative velocity is very different depending on whether we are considering the zero mode or not: 
\begin{align}
\tilde U = \tilde U_0 + \tilde U_{\neq } = \begin{pmatrix} 0 \\ g \\ U^3_0 \end{pmatrix}  + \begin{pmatrix} U_{\neq}^1 - t(1 + \psi_y ) U^2_{\neq} - t \psi_z U^3_{\neq}  \\ (1 + \psi_y)U^2_{\neq} + \psi_z U^3_{\neq}\\ U^3_{\neq} \end{pmatrix} \label{def:tildeU2} 
\end{align}  
(it is worth noting that $\tilde U_{\neq} \cdot \nabla_{X,Y,Z} = U_{\neq} \cdot \nabla^t$).
To see that $\tilde U$ is really the relative velocity, we first begin with $C$. 
Putting \eqref{def:dtpsi} together with \eqref{def:g} and \eqref{def:C} we derive,  
\begin{align} 
\partial_t C + \tilde U_0 \cdot \grad_{Y,Z} C = g - U_0^2 + \nu \tilde{\Delta_t} C. \label{def:CReal}
\end{align} 
The advantage of the cancellation that transfers $\Delta_t$ to $\tilde{\Delta_t}$ is that (A) it eliminates error terms that involve two derivatives of $C$, thereby eliminating some serious regularity issues and (B) the associated terms were slowly decaying. 
Similarly, we derive
\begin{align}
\partial_t g  + \tilde U_0 \cdot \grad_{Y,Z}g = -\frac{2g}{t} -\frac{1}{t} \left(U_{\neq} \cdot \grad^t U^1_{\neq}\right)_0 + \nu \tilde{\Delta_t}  g.   \label{def:gPDE2} 
\end{align}
Recall that $\left(\tilde{U}_{\neq} \cdot \grad_{X,Y,Z} U^1_{\neq}\right)_0 = \left(U_{\neq} \cdot \grad^t U^1_{\neq}\right)_0$.  

Next we change coordinates in \eqref{def:qi}. 
Notice that from \eqref{def:psi2} and \eqref{def:tildeDel1}, for any $f$ solving the passive scalar \eqref{def:ftrans}, we have for $F(t,X,Y,Z) = f(t,x,y,z)$: 
\begin{align*}   
\partial_t F + \tilde U \cdot \grad_{X,Y,Z} F & = \nu \tilde{\Delta_t} F.  
\end{align*} 
Applying this to $q^i = \Delta u^i$, which solve \eqref{def:qi},  gives the following, denoting $Q^i(t,X,Y,Z) = q^i(t,x,y,z)$: 
\begin{align} \label{def:MainSys}
\left\{
\begin{array}{l}
Q^1_t + \tilde U \cdot \grad_{X,Y,Z} Q^1 = -Q^2 - 2\partial_{XY}^t U^1 + 2\partial_{XX} U^2 - Q^j \partial_j^t U^1 - 2\partial_i^t U^j \partial_{ij}^t U^1 + \partial_X(\partial_i^t U^j \partial_j^t U^i) + \nu \tilde{\Delta_t} Q^1 \\  
Q^2_t + \tilde U \cdot \grad_{X,Y,Z} Q^2 = -Q^j \partial_j^t U^2 -  2\partial_i^t U^j \partial_{ij}^t U^2 + \partial_Y^t(\partial_i^t U^j \partial_j^t U^i) + \nu \tilde{\Delta_t} Q^2 \\ 
Q^3_t + \tilde U \cdot \grad_{X,Y,Z} Q^3 = -2\partial_{XY}^t U^3 + 2\partial_{XZ}^t U^2 - Q^j \partial_j^t U^3 - 2\partial_i^t U^j \partial_{ij}^t U^3 + \partial_Z^t(\partial_i^t U^j \partial_j^t U^i) + \nu \tilde{\Delta_t} Q^3,
\end{array}
\right.
\end{align} 
where we use the following to recover the velocity fields: 
\begin{subequations} \label{eq:Ui}
\begin{align} 
U^i & = \Delta_t^{-1} Q^i \label{eq:UiFromQi} \\ 
\partial_i^t U^i & = 0. \label{eq:divfree}
\end{align} 
\end{subequations} 
For the majority of the remainder of the proof, \eqref{def:MainSys}, together with \eqref{def:CReal}, \eqref{def:gPDE2} and \eqref{eq:Ui}, will be the main governing equations. 
The one exception will be in the treatment of the low frequencies of $X$ independent modes, where the use of \eqref{eq:UiFromQi} can be problematic. 
For the remainder of the proof we will use $\grad$ to, by default, denote $\grad_{X,Y,Z}$ or $\grad_{Y,Z}$ (where appropriate) unless otherwise noted. 

From now on we will use the following vocabulary and shorthands 
\begin{subequations} 
\begin{align} 
\tilde U \cdot \grad Q^{\alpha} & = \textup{``transport nonlinearity''} & \mathcal{T} \\  
-Q^j \partial_j^t U^\alpha - 2\partial_i^t U^j \partial_{ij}^t U^\alpha & = \textup{``nonlinear stretching''} & NLS\\  
\partial_\alpha^t(\partial_i^t U^j \partial_j^t U^i)  & = \textup{``nonlinear pressure''} & NLP\\
-2\partial_{XY}^t U^\alpha & = \textup{``linear stretching''} & LS  \\
2\partial_{X\alpha}^t U^2 & = \textup{``linear pressure''} & LP \\ 
\left(\tilde{\Delta_t} - \Delta_L\right)Q^\alpha & =  \textup{``dissipation error''} & \mathcal{D}_E. 
\end{align} 
\end{subequations} 
The pressure terms are named due to the fact that they arise originally from $p^{NL}$ (in the nonlinear case) and $p^L$ (in the linear case) in \eqref{def:3DNSE}.
The stretching terms originally arose from $\Delta(u\cdot \grad u^{\alpha})$ (in the nonlinear case) and $\Delta(y\partial_x u^{\alpha})$ (in the linear case). 
The linear stretching terms are the analogue of the linear vortex stretching which inhibits the inviscid damping, whereas the $NLS$ terms
are analogous to the nonlinear vortex stretching in the 3D Navier-Stokes equations, although the  geometrical interpretation is less clear.     
We remark further that each of the nonlinear terms will be further sub-divided into as many as four pieces in accordance with the different types of nonlinear effects described in \S\ref{sec:NonlinHeuristics}, as each one requires a different treatment. 
The technical complexity is furthered  by the fact that each of the three components of the solution is qualitatively different, which means one should view \eqref{def:MainSys} as a system of three distinct unknowns and some specific combinations of $i$ and $j$ need to be treated specially.  

Before concluding this section, there is one additional, crucial cancellation we would like to emphasize.
As discussed in \S\ref{sec:NonlinHeuristics}, the effect \textbf{(F)}, feedback of $X$-dependent modes on $X$-independent modes, is a potentially important mechanism for streak growth and subsequent transition (see e.g. \cite{Chapman02}).
We need to take advantage of a simple, but useful, cancellation to reduce the strength of this effect (another `non-resonant' structure). 
This is the 3D analogue of the cancellation behind the rapid convergence of the $x$-independent vorticity observed in 2D Euler near Couette flow in \cite{BM13}.     

Consider the evolution of $Q^\alpha$ for $\alpha \in \set{1,2,3}$. 
First, denoting $U_k$ the projection of $U$ on the $k$-th Fourier mode in $x$, we notice the following for $k \neq 0$:
\begin{align*} 
\Delta_t\left( U^j_{-k} \partial_j^t U_k^\alpha \right) = U_{-k} \cdot \grad^t Q^\alpha_k + \partial_i^t(\partial_i^t U^j_{-k} \partial_j^t U^\alpha_{k})  + \partial_i^t U_{-k}^j \partial_i^t \partial_j^t U_{k}^\alpha.
\end{align*}
By the divergence free condition \eqref{eq:divfree}, which holds mode-by-mode in $X$ (again due to the fact that $\psi$ does not depend on $x$), we derive 
\begin{align} 
\Delta_t\left( U^j_{-k} \partial_j^t U^\alpha_k \right) & = \left(\partial_{Y}^t \partial_Y^t + \partial_Z^t\partial_Z^t\right)\partial_j^t \left(U^j_{-k}  U_k^\alpha\right). \label{eq:DeltatCanc2}  
\end{align} 
Similarly, 
\begin{align*} 
\partial_\alpha^t \left(\partial_j^t U^i_{-k} \partial_i^t U^j_k\right) = \partial_\alpha^t \partial_j^t \partial_i^t \left(U^i_{-k} U^j_k\right).
\end{align*} 
Therefore, combining the interaction of two non-zero frequencies, $k$ and $-k$, in transport, stretching and nonlinear pressure and taking $X$ averages, we get the terms  
which we refer to as \emph{forcing}, corresponding to the nonlinear interactions of type \textbf{(F)}: 
\begin{align} 
\mathcal{F}^\alpha & := -\Delta_t \left(U^j_{\neq} \partial_j^t U^\alpha_{\neq}\right)_0 - \partial_\alpha^t \left(\partial_i^t U^j_{\neq} \partial_j^t U^i_{\neq}\right)_0  = \partial_i^t \partial_i^t \partial_j^t \left(U^j_{\neq} U^\alpha_{\neq}\right)_0 - \partial_{\alpha}^t \partial_j^t \partial_i^t \left(U^i_{\neq} U^j_{\neq}\right)_0,  \label{eq:XavgCanc}
\end{align}
where necessarily the sum only runs over $i,j \neq 1$ due to the $X$ average. 
The main point of \eqref{eq:XavgCanc} is that, due to the $X$ averages, the $-t\partial_X$ in the $\partial_{Y}^t$ and $\partial_Z^t$ disappear, which is crucial for limiting the overall size of $\mathcal{F}^\alpha$.

\subsection{The toy model and design of the norms} \label{sec:Toy}

In this section we discuss the toy model for the weakly nonlinear effects. This toy model is then employed to design the norms used to measure the solution to \eqref{def:MainSys} and is the origin for the requirement of Gevrey regularity with $s>1/2$ in Theorem \ref{thm:Threshold}. 
The goal will be to account for the leading order contributions of the nonlinear interactions suggested in \S\ref{sec:NonlinHeuristics}. 
The toy model here will consist of six ODEs (although not fully coupled), which is significantly more complicated than the 2x2 system used for the 2D Couette flow \cite{BM13,BMV14}. 

It is important to emphasize that we are concerned with estimating the \emph{worst possible} behavior, and so are interested in approximate super-solutions of approximate models. 
Moreover, we are interested in only representative, leading order terms:  weaker terms will be ignored as will terms which may be leading order but have the same basic structure as the terms we examine below.   
Note that, due to the ubiquitous presence of $U^j \partial_j$ and $\partial_i U^j \partial_j U^i$, many terms in \eqref{def:MainSys} have a vaguely similar structure. 

Denote the Fourier dual variables of $(X,Y,Z)$ as $(k,\eta,l)$. 
Recall the definition of $\I_{k,\eta}$ from \S\ref{sec:Notation}, which denotes the resonant intervals $t \approx \frac{\eta}{k}$ with $k^2 \lesssim \abs{\eta}$. 
This latter restriction is possible because $\Delta_L$ is elliptic in $k$ which implies that the larger the $k$, the weaker the effect of the resonance. 
Fix $\eta$ and $l$ and $t \in \I_{k,\eta}$ and for simplicity of exposition, assume $\eta,l,k > 0$ for the remainder of \S\ref{sec:Toy}. Here we are only concerned with the case $\eta \gg \max(1,l)$; other contributions are not handled via the toy model. 
This distinguishes a specific $X$-dependent mode as resonant: where the unmixing has brought the information in the mode $(k,\eta,l)$ to $O(1)$ length-scales and the transient growth in the Orr mechanism is reaching its peak.  .
In particular, the velocity field associated with $\widehat{Q^i_k}(\eta,l)$ will be large if $\abs{k,l} \lesssim \abs{\eta}$ (see \S\ref{sec:lin}).

\subsubsection{Leading order \textbf{(SI)} interactions}

The nonlinear interaction \textbf{(SI)}, suggested in \S\ref{sec:NonlinHeuristics}, is the transfer of energy from $U_0^1$ to the $X$-dependent modes. 
Of these, one of the leading order contributions is the nonlinear pressure term in the $Q^2_k$ equation: $\partial_Z^t U_0^1 \partial_{YX}^t U_k^3$.
The leading order contributions to the $Q^3_k$ evolution are the linear pressure and stretching terms. 
Coupling these together isolates the following subsystem between the resonant modes $Q^2_k$ and $Q^3_k$ (notice that this does not happen with $Q^1$; this is a kind of null structure):  
\begin{align*} 
\partial_t Q^2_k & = (\partial^t_Z U_0^1)_{Lo} (\partial^t_{YX} U^3_k)_{Hi} + ... \\  
\partial_t Q^3_k & = - \partial^{t}_{XY}U^3_k + \partial^t_{ZX}U^2_k + ...  
\end{align*} 
where `Hi' and `Lo' denote frequency localization and the ellipses denote temporarily neglected terms. 
By considering $U^j \approx \Delta_L^{-1} Q^j$, we expect that these effects will be greatly amplified when the high frequency modes on the left-hand side are near a critical time. 

To get a better estimate, we will ignore the coefficients $\psi$, replace everything with absolute values, replace $\widehat{(\partial^t_Z U_0^1)_{Lo}} \ast \widehat{(\partial^t_{YX} U^3_k)_{Hi}}$ with $\norm{U_0^1} \widehat{(\partial^t_{YX} U^3_k)_{Hi}}$ and finally replace $U^j$ by $\Delta_L^{-1}Q^j$.
Therefore, using the expected bound $\norm{U_0^1} \lesssim \max(\epsilon \jap{t}, c_0)$, we have
\begin{subequations} \label{def:Resonant} 
\begin{align}
\partial_t Q^2_k & = \max(\epsilon t, c_0) (\partial_Y - t \partial_X)\partial_{X}\Delta_L^{-1} Q_k^3 \\  
\partial_t Q^3_k & = -(\partial_Y - t \partial_X)\partial_{X}\Delta_L^{-1}Q_k^3 + \partial_{ZX}\Delta_L^{-1} Q_k^2.  
\end{align}  
\end{subequations}
On the Fourier side, this becomes
\begin{align*} 
\partial_t \widehat{Q^2_k}(\eta,l) & =  \frac{\max(\epsilon t,c_0) (\eta - kt) k}{k^2+l^2 +\abs{\eta-kt}^2} \widehat{Q^3_k}(\eta,l) \\   
\partial_t \widehat{Q^3_k}(\eta,l) & = \frac{(\eta - kt) k}{k^2+l^2 +\abs{\eta-kt}^2} \widehat{Q^3_k}(\eta,l) + \frac{kl}{k^2 + l^2 + \abs{\eta - kt}^2} \widehat{Q^2_k} (\eta,l). 
\end{align*} 
As $\Delta_L$ is elliptic in $Z$, we will not have to consider very carefully the interaction between the $Z$ frequency and the Orr mechanism, and so we reduce further to
\begin{subequations} \label{def:Resonant_Fourier}
\begin{align}
\partial_t \widehat{Q^2_k}(\eta,l) & =\max(\epsilon t, c_0) \frac{ k}{k +\abs{\eta-kt}} \widehat{Q^3_k}(\eta,l) \\
\partial_t \widehat{Q^3_k}(\eta,l) & =\frac{k}{k +\abs{\eta-kt}} \widehat{Q^3_k}(\eta,l) + \frac{k}{k + \abs{\eta - kt}} \widehat{Q^2_k}(\eta,l). 
\end{align}
\end{subequations}
Since $c_{0}$ is essentially $O(1)$ compared to the size of the non-zero frequencies, \eqref{def:Resonant_Fourier} alone would suggest 
that the components of $Q^2_k$ and $Q^3_k$  could \emph{grow roughly the same way near the critical time} and that, near the critical time, the growth can be approximated by 
\begin{align*} 
\partial_t \widehat{Q^2_k}(\eta,l) \approx \partial_t \widehat{Q^3_k}(\eta,l)  \approx \frac{1}{1 + \abs{\frac{\eta}{k} -t} } \widehat{Q^3_k} \approx \frac{1}{1 + \abs{\frac{\eta}{k} -t} } \widehat{Q^2_k}.  
\end{align*}
Notice that this appears at odds with the linearization! 
Indeed, recall that \eqref{ineq:U2LinearID} predicts that $Q^2$ should remain bounded whereas \eqref{ineq:U1LinearID} and \eqref{ineq:U3LinearID} show that $Q^1$ and $Q^3$ will necessarily grow quadratically by vortex stretching.  
Hence, \eqref{def:Resonant_Fourier} predicts that the quadratic growth of $Q^3_k$ relative to $Q^2_k$ will not occur until after the critical time. 
We will see that due to interaction with non-resonant frequencies, we will actually need to delay the relative quadratic growth of $Q^3$
until after \emph{all} of the critical times. 
For $Q^1_k$, we have to leading order (as above, dropping signs) from the lift-up effect and the linear stretching term:
\begin{align*} 
\partial_t \widehat{Q^1_k} = \widehat{Q^2_k} + \frac{1}{1 + \abs{t - \frac{\eta}{k}}} \widehat{Q^1_k}. 
\end{align*}
From the first term, we see that (in terms of amplitude) we can consider $Q^1_k \approx t Q^2_k$ near the critical times. 
As was the case for $Q^3$, this is at odds with the predictions of \eqref{ineq:LinearID}, which suggests a quadratic growth relative to $Q^2$. 
Instead, we will need to have $Q^1$ grow linearly relative to $Q^2$ until after the critical times, and then eventually quadratically. 
Hence, for the nonlinear problem, we see that each of the three components must be qualitatively different. 

\subsubsection{Leading order \textbf{(3DE)} interactions}
The behavior of \eqref{def:Resonant} suggests how the \emph{resonant} frequency $k$ would grow near $kt \approx \eta$ 
due to a leading order nonlinear interaction of type \textbf{(SI)}.  
Hence, next we add in the effect of the nonlinear interactions between a resonant frequency (denoted $k$) and other $X$ dependent frequencies (denoted $k^\prime$).
As nearby frequencies interact much more, we will assume $\abs{k - k^\prime} \lesssim 1$ (it is enough to consider $k^\prime = k \pm 1$).

We will use the heuristic discussed above, involving taking $Q^1_k \approx tQ^2_k$ and $Q^1_{k^\prime} \approx tQ^2_{k^\prime}$, which allows us to focus only on $Q^2$ and $Q^3$ (along with the observation that many of the nonlinear terms in the $Q^1$ equation of \eqref{def:MainSys} are generally weaker than the other components).    
In estimating the size of the nonlinear terms, it is important to understand that while $Q^3$ will grow quadratically at ``low frequencies'', we saw above that near the critical times, which is ``high frequency'', $Q^3$ is expected to be roughly the same size as $Q^2$.
This creates a large imbalance between low and high frequencies in $Q^3$, which suggests that the worst \textbf{(3DE)} terms likely involve $(Q^3_{\neq})_{Lo}$.
The term that stands out as the worst is from the transport nonlinearity, when the velocity field appears in high frequency near the critical time, which only occurs in the $Q^3_{k^\prime}$ equation; in the $Q^3_k$ equation, the velocity will be away from the critical time and hence less dangerous.    
Therefore, we derive the following (using the same nonlinear pressure term for the $Q^2_{k^\prime}$ equation):  
\begin{subequations} \label{eq:ToyIntermed3DE}
\begin{align}
\partial_t Q^2_{k} & = \left(U_0^1(t)\right)_{Lo} (\partial_Y - t\partial_X) \partial_X \Delta_L^{-1} (Q^3_k)_{Hi}  + ... \\ 
\partial_t Q^2_{k^\prime} & = \left(U_0^1(t)\right)_{Lo} (\partial_Y - t\partial_X) \partial_X \Delta_L^{-1} (Q^3_{k^\prime})_{Hi} + ... \\  
\partial_t Q^3_{k^\prime} & = -t\left(U^{2}_{k}\right)_{Hi}  \left(\partial_X Q^3_{k^\prime-k}\right)_{Lo} + ... \\ 
\partial_t Q^3_{k} & = -(\partial_Y - t \partial_X)\partial_{X}\Delta_L^{-1}Q_k^3 + \partial_{ZX}\Delta_L^{-1} Q^2_k + ... 
\end{align}
\end{subequations}  
Obviously, we have neglected many terms in the nonlinearities. 
Indeed, it will turn out that we have neglected terms which are the same ``size'' as the leading order terms, however, they have essentially the same structure as the leading ones we have identified and hence do not need to be separately considered. 
On the other hand, many other terms which we have neglected have a rather different structure but are instead smaller.

Using the same approximations that we applied to derive \eqref{def:Resonant_Fourier} and substituting in for the low frequency terms the amplitudes predicted by the linear theory in \S\ref{sec:NSELin}, \eqref{eq:ToyIntermed3DE} becomes on the Fourier side, 
\begin{subequations} \label{def:Q2Q3_Toy3DE}
\begin{align} 
\partial_t \widehat{Q^2_{k}}(t,\eta,l) & = \max(\epsilon t, c_0) \frac{k}{k +\abs{\eta-kt}} \widehat{Q^3_k} \\
\partial_t \widehat{Q^2_{k^\prime}}(t,\eta,l) & = \max(\epsilon t, c_0) \frac{k^\prime}{\abs{k^\prime} +\abs{\eta-k^\prime t}} \widehat{Q^3_{k^\prime}} \\  
\partial_t \widehat{Q^3_{k^\prime}}(t,\eta,l) & = \frac{\epsilon t^3}{\jap{\nu t^3}^\alpha} \frac{\widehat{Q^{2}_{k}} }{k^2 +\abs{\eta-kt}^2} \\  
\partial_t \widehat{Q^3_{k}}(t,\eta,l) & = \frac{k}{k +\abs{\eta-kt}}\widehat{Q^3_k} + \frac{k}{k +\abs{\eta-kt}} \widehat{Q^2_k}.  
\end{align}
\end{subequations} 
where all unknowns are evaluated at $(\eta,l)$. 
Notice that by non-resonance we have $\abs{\eta - k^\prime t} \gtrsim t$ and so we can derive 
\begin{subequations} \label{def:Q2Q3_Toy3DE2}
\begin{align} 
\partial_t \widehat{Q^2_{k}}(t,\eta,l) & = \max(\epsilon t, c_0)\frac{k}{k +\abs{\eta-kt}} \widehat{Q^3_k} \\ 
\partial_t \widehat{Q^2_{k^\prime}}(t,\eta,l) & = \max(\epsilon t, c_0) \frac{k^\prime}{\jap{k^\prime, t}} \widehat{Q^3_{k^\prime}} \\  
\partial_t \widehat{Q^3_{k^\prime}}(t,\eta,l) & = \frac{\epsilon t^3}{\jap{\nu t^3}^\alpha} \frac{\widehat{Q^{2}_{k}} }{k^2 +\abs{\eta-kt}^2} \\
\partial_t \widehat{Q^3_{k}}(t,\eta,l) & = \frac{k}{k +\abs{\eta-kt}}\widehat{Q^3_k} + \frac{k}{k +\abs{\eta-kt}} \widehat{Q^2_k}.  
\end{align}
\end{subequations} 
This toy model accounts for the leading order interactions between \textbf{(3DE)} and \textbf{(SI)}. From this model alone one can derive the Gevrey-2 regularity requirement (see below for how to do this). 

\subsubsection{Leading order \textbf{(F)} and \textbf{(2.5NS)} interactions} 

At the zero frequencies, there is another challenge, which are interactions of type \textbf{(F)}: the forcing from non-zero frequencies into the zero frequency. Thanks to the cancellations in \eqref{eq:XavgCanc}, we can narrow this effect down to a pair of forcing terms from the resonant frequencies. We will also account for the leading order interaction in the zero-mode nonlinear pressure interactions of type \textbf{(2.5NS)}, giving (note we are including the dissipation terms here)
\begin{subequations} 
\begin{align} 
\partial_tQ^2_0 & = \partial_{Y}( (\partial_Y U^3_0)_{Hi} (\partial_Z U^2_0)_{Lo}) +  \partial_{YYZ}\left(\left(U^2_{k}\right)_{Hi} \left(U^3_{-k}\right)_{Lo}\right) - \nu(\partial_{YY} + \partial_{ZZ})Q_0^2 \\ 
 \partial_tQ^3_0 & = \partial_{Z}( (\partial_Y U^3_0)_{Hi} (\partial_Z U^2_0)_{Lo}) +  \partial_{YYY}\left(\left(U^2_{k}\right)_{Hi} \left(U^3_{-k}\right)_{Lo}\right) - \nu(\partial_{YY} + \partial_{ZZ})Q_0^3,  
\end{align} 
\end{subequations}
where again we are considering $t \sim \frac{\eta}{k}$. 
Using the same approximations as above together with $\Delta U_0^i = Q_0^i$ and the assumption that $\eta \gg \max(1,l)$, we have 
\begin{subequations} \label{def:Q2Q30_Fourier}
\begin{align} 
\partial_t\widehat{Q^2_0} & = \epsilon \widehat{Q^3_0} +  \frac{\epsilon t^2}{\jap{\nu t^3}^{\alpha}}\frac{\widehat{Q^2_k}}{k^2 +\abs{\eta-kt}^2} - \nu \eta^2 \widehat{Q^2_0} \\ 
 \partial_t\widehat{Q^3_0} & = \epsilon \widehat{Q^3_0} + \frac{\epsilon t^3}{\jap{\nu t^3}^{\alpha}}\frac{\widehat{Q^2_k}}{k^2 +\abs{\eta-kt}^2} - \nu \eta^2 \widehat{Q^3_0}. \label{def:Q03ForceToy}  
\end{align} 
\end{subequations}
Notice that the cancellations in \eqref{eq:XavgCanc} imply the forcing terms on $Q^2_0$ are weaker than $Q^3_0$. This special structure will be crucial in our follow-up work \cite{BGM15II}. 

\subsubsection{Final toy model and resulting predictions}
Next we put together everything into one main model. 
As $Q^3_0$ is strongly forced near the critical times, one should worry that this will couple back to the non-zero frequencies modeled in \eqref{def:Q2Q3_Toy3DE2} and precipitate strong growth. However, these terms turn out not to be leading order due to the overwhelming enhanced dissipation, and so can be omitted from the toy model. 
Finally, it also makes sense to include the dissipation terms in all the equations, as for large $t$ these are dominant.
Putting everything together yields the final toy model, 
\begin{subequations} \label{def:Q2Q3_ToyFinal}
\begin{align} 
\partial_t \widehat{Q^2_{k}}(t,\eta,l) & = \max(\epsilon t, c_0) \frac{k}{k +\abs{\eta-kt}} \widehat{Q^3_k} - \nu\left(k^2 + \abs{\eta-kt}^2\right)\widehat{Q^2_k}   \\  
\partial_t \widehat{Q^2_{k^\prime}}(t,\eta,l) & = \max(\epsilon t, c_0) \frac{k^\prime}{\jap{k^\prime, t}} \widehat{Q^3_{k^\prime}} - \nu\left(k^2 + \abs{\eta-kt}^2\right)\widehat{Q^2_{k^\prime}} \label{eq:Q2kp}  \\  
\partial_t \widehat{Q^3_{k^\prime}}(t,\eta,l) & = \frac{\epsilon t^3}{\jap{\nu t^3}^\alpha} \frac{\widehat{Q^{2}_{k}} }{k^2 +\abs{\eta-kt}^2} - \nu\left(k^2 + \abs{\eta-kt}^2\right)\widehat{Q^3_{k^\prime}}   \\ 
\partial_t \widehat{Q^3_{k}}(t,\eta,l) & = \frac{k}{k +\abs{\eta-kt}}\widehat{Q^3_k} + \frac{k}{k +\abs{\eta-kt}} \widehat{Q^2_k} - \nu\left(k^2 + \abs{\eta-kt}^2\right)\widehat{Q^3_k} \\ 
\partial_t\widehat{Q^2_0}(t,\eta,l) & = \epsilon \widehat{Q^3_0} +  \frac{\epsilon t^2}{\jap{\nu t^3}^{\alpha}}\frac{\widehat{Q^2_k}}{k^2 +\abs{\eta-kt}^2} - \nu \eta^2 \widehat{Q^2_0} \\ 
 \partial_t\widehat{Q^3_0}(t,\eta,l) & = \epsilon \widehat{Q^3_0} + \frac{\epsilon t^3}{\jap{\nu t^3}^{\alpha}}\frac{\widehat{Q^2_k}}{k^2 +\abs{\eta-kt}^2} - \nu \eta^2 \widehat{Q^3_0},
\end{align}
\end{subequations} 
 where all unknowns are evaluated at frequency $(\eta,l)$
From \eqref{def:Q2Q3_ToyFinal} (together with the long-time analyses similar to what is carried out below in \eqref{eq:wL}), one can predict, and design the norms necessary to prove, the results both here and in the above threshold case \cite{BGM15II}. 

In this work, we would like to get control for all time with the requirement $\nu \gtrsim \epsilon$.
Under the assumption $\nu \gtrsim \epsilon$, the following simple choice provides an approximate super-solution to \eqref{def:Q2Q3_ToyFinal} (over just the single critical interval $\I_{k,\eta}$): 
\begin{subequations} \label{def:stablesuper}
\begin{align} 
\partial_t w(t,\eta) & \approx \frac{1}{1 + \abs{t - \frac{\eta}{k}}}w(t,\eta) \label{eq:wapprox} \\ 
Q^2_k \approx Q^2_{k^\prime} & \approx Q^3_k \approx Q^3_{k^\prime} \approx Q^2_0 \approx Q^3_{0} \approx w(t,\eta) \\ 
Q^1_k \approx Q^1_{k^\prime} & \approx tQ^2_{k^\prime}  \approx tQ^2_k. 
\end{align}
\end{subequations} 
By integrating \eqref{eq:wapprox}, we have
\begin{align} 
\frac{w(\frac{\eta}{k} + \frac{\eta}{k^2},\eta)}{w(\frac{\eta}{k} - \frac{\eta}{k^2},\eta)} \approx \left(\frac{\eta}{k^2}\right)^{C}, \label{ineq:wapproxloss} 
\end{align}
where $C$ is a fixed constant depending on the implicit constant in \eqref{eq:wapprox}. 
Hence, \eqref{ineq:wapproxloss} predicts that \emph{both} the resonant (mode $k$) and non-resonant (mode $k^\prime$) modes will lose the same amount of Sobolev regularity near \emph{each} critical time. 
Therefore, instead of just losing regularity once at the critical time, the regularity loss is repeatedly amplified as $t$ goes through successive resonant times $\eta/k$, $\eta/(k-1)$, $\eta/(k-2)$ etc (to see this, take $k^\prime = k-1$).  
The loss \eqref{ineq:wapproxloss} is amplified multiplicatively in each critical time, so counting over all the possible critical times which satisfy $\eta \gtrsim k^2$ gives a regularity loss like
\begin{align}
\frac{Q^2(2\eta,\eta)}{Q^2(\sqrt{\eta},\eta)} \approx \left(\frac{\eta^{\sqrt{\eta}}}{\left((\eta^{1/2})!\right)^{2}}\right)^{C},
\end{align}
which by Stirling's formula is size $O(e^{2C\sqrt{\eta}})$ (up to a polynomial correction) and hence \eqref{def:stablesuper} predicts Gevrey-2 regularity loss; see Appendix \ref{sec:def_nrm} for more information. 
This is the origin of the requirement $s > 1/2$ in Theorem \ref{thm:Threshold}. 
We emphasize that the nonlinear effects from \textbf{(3DE)} are the main cause of the infinite regularity class -- one needs to account for the repeated resonant/non-resonant interactions in order to predict this potential instability. In particular, this cannot be predicted from linear theory.  
In the companion work \cite{BGM15II}, we use a different super-solution which does not require $\epsilon \lesssim \nu$. 
Using \eqref{def:Q2Q3_ToyFinal}, one can find a super-solution for only $\epsilon \lesssim \nu^{2/3}$ by taking advantage of more of the available structure, but this also results in more complicated norms and only works until $t \sim \epsilon^{-1}$ (which if $\epsilon \gg \nu$ is approximately the time when the secondary instability will set in; see \cite{BGM15II} for more discussion). 
Finally, we mention that the growth in \eqref{def:stablesuper} is sharply peaked near the critical times and it will turn out to be useful for unifying and simplifying estimates below to modify $w$ by including additional steady, gradual losses of Gevrey-2 regularity over $1 \leq t \leq 2\abs{\eta}$  (see \eqref{def:wextraloss} in Appendix \ref{sec:defw}). 

There is one remaining detail that must be addressed, which is the behavior of the components for times very large. Notice that in $Q^3$, for $t \gg \abs{\eta}$ the linear terms are to leading order: 
\begin{align} 
\partial_t Q^3_k = \frac{2}{t}Q^3_k + \frac{kl}{k^2 + l^2 + \abs{\eta-kt}^2} Q^2_k. 
\end{align} 
The first term gives the quadratic growth from the stretching, as we are expecting. The second term appears small (as $Q^3$ is growing quadratically whereas $Q^2$ is not) but we cannot so easily neglect its contribution for $l$ large. However, the pre-factor is actually integrable in time \emph{uniformly} in $(k,\eta,l)$.  
To deal with the anisotropy carefully, we will incorporate this factor into the norm via a multiplier of the form: 
\begin{align}
\partial_t w_L(t,k,\eta,l) & \approx \frac{kl}{k^2 + l^2 + \abs{\eta-kt}^2} w_L(t,k,\eta,l), \label{eq:wL}
\end{align}   
which, unlike $w$, is order one due to the uniform integrability (see Appendix \ref{sec:Nmult}).    

\subsubsection{Design of the norms based on the toy model}
Now let us put all of the observations together to design the norms that we will use. 
 As the details are somewhat tedious and mostly involve technicalities similar to those that already appear in \cite{BM13}, we reserve the details for Appendix \ref{sec:def_nrm} and simply give the overview here based on the heuristics above. 
Fix $\beta > 3\alpha + 6$, $\gamma > \beta + 3\alpha + 12$ and $\sigma > \gamma + 6$. 
We will use a hierarchy of regularities, and the high norms will be of the form $\norm{A^i(t,\grad) Q^i(t)}_2$ for Fourier multipliers $A^i$ defined by the following, for a time-varying $\lambda(t)$ defined below, $s > 1/2$, $0 < \delta_1 \ll 1$, and corrector multipliers $w$ and $w_L$ (here $(t,k,\eta,l)$ are now arbitrary):
\begin{subequations} \label{def:A}
\begin{align} 
A^Q_k(t,\eta,l) & = e^{\lambda(t)\abs{k,\eta,l}^s}\jap{k,\eta,l}^\sigma \frac{e^{\mu \abs{\eta}^{1/2}}}{w(t,\eta) w_L(t,k,\eta,l)} \\ 
A^{1}_k(t,\eta,l) & = \frac{1}{\jap{t}}\left(\mathbf{1}_{k \neq 0} \min\left(1, \frac{\jap{\eta,l}^{1+\delta_1}}{t^{1+\delta_1}}\right) + \mathbf{1}_{k = 0} \right) A^Q_k(t,\eta,l) \\ 
A^{2}_k(t,\eta,l) & = \left(\mathbf{1}_{k \neq 0} \min\left(1, \frac{\jap{\eta,l}}{t}\right) + \mathbf{1}_{k = 0} \right) A^Q_k(t,\eta,l) \\ 
A^{3}_k(t,\eta,l) & = \left(\mathbf{1}_{k \neq 0} \min\left(1, \frac{\jap{\eta,l}^2}{t^2}\right) + \mathbf{1}_{k = 0} \right) A^Q_k(t,\eta,l) \\ 
A(t,\eta,l) & = \jap{\eta,l}^2 A^Q_0(t,\eta,l), 
\end{align}
\end{subequations}
where $\mu$ and $w$ are defined precisely in Appendix \ref{sec:def_nrm} and $w_L$ is defined in Appendix \ref{sec:Nmult}.  
The multiplier $A$ is used to measure $C$ and $g$ whereas $A^{i}$ will measure $Q^i$.  
Note that for technical reasons, we are allowing $A^2$ to decay linearly at `low' frequencies, which is not predicted by the linear theory in \S\ref{sec:lin} nor in \S\ref{sec:Toy}. 
For separate technical reasons, we are also allowing $A^1$ to decay slightly faster than quadratically.
We will choose the radius of Gevrey-$\frac{1}{s}$ regularity to satisfy, for some $\delta_\lambda$ small,   
\begin{align*} 
\dot{\lambda}(t) & = - \frac{\delta_\lambda}{\jap{t}^{\min(2s,3/2)}} \\ 
\lambda(1) & = \frac{3\lambda}{4} +  \frac{\lambda^\prime}{4}.  
\end{align*} 
Fix $\delta_\lambda \ll \min(1,\lambda_0 - \lambda^\prime)$ small such that $\lambda(t) > (\lambda_0 + \lambda^\prime)/2$. 

To quantify the enhanced dissipation, we build on the scheme used in \cite{BMV14}.
The general idea is to use semi-norms $\norm{A^{\nu;i}(t,\grad)Q^i(t)}_2$ for Fourier multipliers $A^{\nu;i}(t,\grad)$ which gradually trade regularity for decay of lower frequencies, a kind of parabolic analogue to the regularity one pays to deduce inviscid damping. 
Define $D$ as in \cite{BMV14}: 
\begin{align} 
D(t,\eta) & = \frac{1}{3\alpha}\nu \abs{\eta}^3 + \frac{1}{24 \alpha} \nu\left(t^3 - 8\abs{\eta}^3\right)_+, \label{def:D}
\end{align} 
which satisfies 
\begin{align} 
\max\left( \nu \abs{\eta}^3, \nu t^3\right) \lesssim \alpha D(t,\eta).  \label{ineq:DLowB}
\end{align}
For $\beta$ chosen above, we define the enhanced dissipation multipliers  
\begin{subequations} \label{def:Anu} 
\begin{align}
A^{\nu}_k(t,\eta,l) & = e^{\lambda(t)\abs{k,\eta,l}^s}\jap{k,\eta,l}^{\beta} \jap{D(t,\eta)}^\alpha \frac{1}{w_L(t,k,\eta,l)} \mathbf{1}_{k \neq 0} \\ 
A^{\nu;1}_k(t,\eta,l) & = \frac{1}{\jap{t}}\min\left(1, \frac{\jap{\eta,l}^{1+\delta_1}}{t^{1+\delta_1}}\right) A^{\nu}_k(t,\eta,l)  \\ 
A^{\nu;2}_k(t,\eta,l) & = \min\left(1, \frac{\jap{\eta,l}^{\delta_1}}{t^{\delta_1}}\right) A^{\nu}_k(t,\eta,l) \\
A^{\nu;3}_k(t,\eta,l) & = \min\left(1, \frac{\jap{\eta,l}^2}{t^2}\right) A^{\nu}_k(t,\eta,l). 
\end{align} 
\end{subequations}
Note that $A^{\nu;2}$ now matches the linear behavior predicted in Proposition \ref{prop:LinNSE} up to a deviation by a small $t^{\delta_1}$.  
We believe this is purely technical.
Note that we do not need $w$ in \eqref{def:Anu}. The Orr mechanism (and associated nonlinear effects) do not play a major role in the enhanced dissipation estimates; they are instead mainly determined by how the vortex stretching manifests in the nonlinearity.  
 
Finally, the following inequalities are useful to keep in mind
\begin{subequations} 
\begin{align} 
\int_0^\infty \frac{1}{\jap{\nu t}^{2+}} dt & \approx \nu^{-1} \lesssim \frac{c_{0}}{\epsilon}, \label{ineq:nutint} \\ 
\int_0^\infty \frac{1}{\jap{\nu t^3}^{1/3+}} dt & \approx \nu^{-1/3} \lesssim \frac{c_{0}^{1/3}}{\epsilon^{1/3}} \label{ineq:nut3int} \\ 
 \int_0^\infty \frac{\epsilon t^2}{\jap{\nu t^3}^{1/3+}} dt & \approx c_0. \label{ineq:nut4int}
\end{align} 
\end{subequations}

\subsection{Main energy estimates} \label{sec:energy}
Equipped with the norms defined in \eqref{def:Anu} and \eqref{def:A}, we will be able to propagate estimates from local-in-time (provided by Lemma \ref{lem:Loc}) to global-in-time via a bootstrap argument for as long as the solution to \eqref{def:3DNSE} exists and remains analytic; by un-doing the coordinate transformation, this in turn allows us to continue the solution of \eqref{def:3DNSE} via Lemma \ref{lem:Cont}. See \S\ref{sec:ContCoord} below for more details on this procedure.    

In addition to the norm controls, we have a number of ``dissipation energies'' which arise both due the dissipation itself and due to the fact that the norms are getting weaker in time; we will refer to the dissipation-like terms that arise due to this effect as `CK' terms (for `Cauchy-Kovalevskaya' due to the appearance of similar terms in the classical proofs of abstract Cauchy-Kovalevskaya-type results, e.g. \cite{Nishida77,Nirenberg72}). 
We define the following dissipation energies, for $i \in \set{2,3}$, 
\begin{subequations} \label{def:DissipEnergy}
\begin{align} 
\mathcal{D}Q^i & = \nu \norm{\sqrt{-\Delta_{L}}A^{i} Q^i}_2^2 + CK_\lambda^i + CK_w^i + CK_{wL}^i \nonumber  \\ 
& := \nu \norm{\sqrt{-\Delta_{L}}A^{i} Q^i}_2^2 + \dot{\lambda}\norm{\abs{\grad}^{s/2}A^i Q^i}_2^2 + \norm{\sqrt{\frac{\partial_t w}{w}} A^i Q^i}_2^2 + \norm{\sqrt{\frac{\partial_t w_L}{w_L}}A^i Q^i}_2^2 \\ 
\mathcal{D}Q^1_{\neq} & =  \nu \norm{\sqrt{-\Delta_{L}}A^{1} Q^1_{\neq} }_2^2 + CK_{\lambda;\neq}^{1} + CK_{w;\neq}^{1} + CK_{wL;\neq}^{1} \nonumber  \\ 
& :=  \nu \norm{\sqrt{-\Delta_{L}}A^{1} Q^1_{\neq}}_2^2 + \dot{\lambda}\norm{\abs{\grad}^{s/2}A^1 Q^1_{\neq}}_2^2 + \norm{\sqrt{\frac{\partial_t w}{w}} A^1 Q^1_{\neq}}_2^2 + \norm{\sqrt{\frac{\partial_t w_L}{w_L}}A^1 Q^1_{\neq}}_2^2 \\ 
\mathcal{D}g & = \nu \norm{\sqrt{-\Delta_{L}}A g }_2^2 + CK_L^g + CK_{\lambda}^{g} + CK_{w}^{g} \nonumber \\
& := \nu \norm{\sqrt{-\Delta_{L}}A g }_2^2 + \frac{2}{t}\norm{Ag}_2^2 + \dot{\lambda}\norm{\abs{\grad}^{s/2}A g}_2^2 + \norm{\sqrt{\frac{\partial_t w}{w}} A g}_2^2 \\ 
\mathcal{D}C & = \nu \norm{\sqrt{-\Delta_{L}}A C }_2^2 + CK_{\lambda}^{C} + CK_{w}^{C} \nonumber \\
& := \nu \norm{\sqrt{-\Delta_{L}}A C }_2^2 + \dot{\lambda}\norm{\abs{\grad}^{s/2}A C}_2^2 + \norm{\sqrt{\frac{\partial_t w}{w}} A C}_2^2 \\ 
CK_L^i & := \frac{1}{t}\norm{\mathbf{1}_{t \geq \jap{\grad_{Y,Z}}} A^iQ^i_{\neq}}_2^2 \\ 
\mathcal{D}Q^{\nu;i} & = \nu \norm{\sqrt{-\Delta_{L}}A^{\nu;i} Q^i}_2^2 + CK_\lambda^{\nu;i} + CK_{wL}^{\nu;i} \nonumber  \\ 
& := \nu \norm{\sqrt{-\Delta_{L}}A^{\nu;i} Q^i}_2^2 + \dot{\lambda}\norm{\abs{\grad}^{s/2}A^{\nu;i} Q^i}_2^2 + \norm{\sqrt{\frac{\partial_t w_L}{w_L}}A^{\nu;i} Q^i}_2^2 \\ 
\mathcal{D}Q^{\nu;1} & =  \nu \norm{\sqrt{-\Delta_{L}}A^{\nu;1} Q^{\nu;1} }_2^2 + CK_{\lambda}^{\nu;1} + CK_{wL}^{\nu;1} \nonumber  \\ 
& :=  \nu \norm{\sqrt{-\Delta_{L}}A^{\nu;1} Q^1}_2^2 + \dot{\lambda}\norm{\abs{\grad}^{s/2}A^{\nu;1} Q^1}_2^2 + \norm{\sqrt{\frac{\partial_t w_L}{w_L}}A^{\nu;1} Q^1}_2^2 \\ 
CK_L^{\nu;i} & := \frac{1}{t}\norm{\mathbf{1}_{t \geq \jap{\grad_{Y,Z}}} A^{\nu;i}Q^i}_2^2. 
\end{align} 
\end{subequations} 
Fix constants $K_{Hi}, K_{M1}, K_{H1\neq}, K_{HC1}, K_{HC2}, K_{LC}, K_{EDi}, K_{ED2}, K_{GL3}, K_{U1}, K_{U12}, K_{U3}, K_L$ for $i \in \set{1, 3}$ determined by the proof depending only on $s,\lambda^\prime,\delta_1,\alpha$ and $\lambda_0$, as well as the somewhat arbitrary parameters such as $\sigma$ and $\beta$. 
These are mostly necessary due to the linear terms present in \eqref{def:CReal} and \eqref{def:MainSys}.
Further fix $\sigma^\prime > 3$.
Let $1 \leq T^\star < T^0$ be the largest time such that the following \emph{bootstrap hypotheses} hold (that we can take $T^\star \geq 1$ will be discussed below): 
the high norm controls on $Q^i$,
\begin{subequations} \label{ineq:Boot_Hi} 
\begin{align} 
\norm{A^{1} Q^1_{0}(t)}_2^2 & \leq 4K_{H1} \epsilon^2 \label{ineq:Boot_Q1Hi1} \\
\norm{Q_0^1(t)}_{\G^{\lambda,\gamma}} & \leq 4K_{M1} c_0^2 \label{ineq:Boot_Q1Mid} \\ 
\norm{A^{1} Q^1_{\neq}(t)}_2^2 + \frac{1}{2}\int_1^t \mathcal{D}Q^1_{\neq}(\tau) d\tau  & \leq 4K_{H1\neq}\epsilon^2 \label{ineq:Boot_Q1Hi2} \\   
\norm{A^{2} Q^2}^2_2 + \int_1^t\frac{1}{2}\mathcal{D}Q^2(\tau) + CK_L^2(\tau) d\tau & \leq 4\epsilon^2 \label{ineq:Boot_Q2Hi} \\
\norm{A^{3} Q^3}^2_2 + \frac{1}{2}\int_1^t \mathcal{D}Q^3(\tau) d\tau & \leq 4K_{H3}\epsilon^2; \label{ineq:Boot_Q3Hi}
\end{align} 
\end{subequations}
the coordinate system controls, 
\begin{subequations} \label{ineq:Boot_CgHi}
\begin{align} 
\norm{A C}_2^2 + \frac{1}{2}\int_1^t \mathcal{D}C(\tau) d\tau  &\leq 4 K_{HC1} c^2_{0} \label{ineq:Boot_ACC}\\ 
\jap{t}^{-2}\norm{A C}_2^2 + \frac{1}{2}\int_1^t \jap{\tau}^{-2} \mathcal{D}C(\tau) d\tau &\leq 4K_{HC2}\epsilon^2 \log \epsilon^{-1} \label{ineq:Boot_ACC2}\\ 
\norm{Ag}_2^2 + \frac{1}{2}\int_1^t \mathcal{D}g  d\tau &\leq 4 \epsilon^2 \label{ineq:Boot_Ag} \\ 
\norm{g}_{\G^{\lambda,\gamma}} & \leq  4  \frac{\epsilon}{\jap{t}^{2}} \label{ineq:Boot_gLow} \\
\norm{C}_{\G^{\lambda,\gamma}} & \leq  4K_{LC} \min\left(\epsilon \jap{t}, c_{0}\right); \label{ineq:Boot_LowC} 
\end{align} 
\end{subequations}
the enhanced dissipation estimates,
\begin{subequations} \label{ineq:Boot_ED}
\begin{align} 
\norm{A^{\nu;1} Q^1}_2^2 + \frac{1}{10} \int_1^t \mathcal{D}Q^{\nu;1}(\tau) d\tau & \leq 4K_{ED1}\epsilon^2 \label{ineq:Boot_ED1} \\ 
\norm{A^{\nu;2} Q^2}_2^2 + \int_1^t\frac{1}{10}\mathcal{D}Q^{\nu;2}(\tau) + \delta_1 CK_{L}^{\nu;2}(\tau) d\tau & \leq 4K_{ED2}\epsilon^2 \\ 
\norm{A^{\nu;3} Q^3}_2^2 + \frac{1}{10}\int_1^t \mathcal{D}Q^{\nu;3}(\tau) d\tau & \leq  4K_{ED3}\epsilon^2; \label{ineq:Boot_ED3}
\end{align} 
\end{subequations}
and the additional low frequency controls on the background streak  
\begin{subequations} \label{ineq:Boot_LowFreq} 
\begin{align} 
\norm{Q_0^2}_{H^{\sigma^\prime}} & \leq 4  \frac{\epsilon}{\jap{\nu t}^{\alpha}} \label{ineq:Boot_Q02_Low} \\ 
\norm{U_0^2}_{H^{\sigma^\prime}} & \leq  4 \frac{\epsilon}{\jap{\nu t}^{\alpha}} \label{ineq:Boot_U02_Low} \\ 
\norm{\partial_Z U_0^3}_{H^{\sigma^\prime}} & \leq 4K_{GL3} \frac{\epsilon}{\jap{\nu t}^\alpha} \label{ineq:Boot_partZU03_Low} \\ 
\norm{U_0^1}_{H^{\sigma^\prime}} & \leq 4K_{U1} \epsilon \jap{t} \label{ineq:Boot_U01_Low1} \\ 
\norm{U_0^1}_{H^{\sigma^\prime}}^2 + \frac{\nu}{2} \int_1^t \norm{\grad U_0^1}_{H^{\sigma^\prime}}^2 & \leq 4 K_{U12}^2 c_{0}^2 \label{ineq:Boot_U01_Low2} \\  
\norm{U_0^3}_{H^{\sigma^\prime}}^2 + \frac{\nu}{2} \int_1^{t} \norm{\grad U_0^3}_{H^{\sigma^\prime}}^2 d \tau   & \leq 4K^2_{U3} \epsilon^2 \label{ineq:Boot_U03_Low} \\ 
\norm{U_0^3}_4 & \leq  4K_L\frac{\epsilon}{\jap{ \nu t }^{1/4}} \label{ineq:L43} \\ 
\norm{U_0^1}_4 & \leq 4K_L\frac{c_{0}}{\jap{ \nu t }^{1/4}}. \label{ineq:L41} 
\end{align} 
\end{subequations}
For most steps of the proof we do not need to differentiate so precisely between different bootstrap constants so we define 
\begin{align} 
K_B =  \max\left(K_{Hi}, K_{M1}, K_{H1\neq}, K_{HC1}, K_{HC2}, K_{LC}, K_{EDi},K_{ED2},K_{GL3}, K_{U1}, K_{U12},K_{U3} \right).  \label{def:KB}
\end{align}

It is a consequence of Lemma \ref{lem:Loc} that $T^\star > t_\star > 0$ and it is a consequence 
of Lemma \ref{lem:Cont} that $T^\star < T^0$. 
It is relatively easy to prove that for $\epsilon$ sufficiently small, we have $1 \leq T^\star$; see \S\ref{sec:RegCont} for further details. 
Due to the real analyticity of the solution on $(0, T^0)$, it will follow from the ensuing proof that the quantities in the bootstrap hypotheses take values continuously in time for as long as the solution exists.  
Therefore, we may deduce $T^\star = +\infty$ via the following proposition, the proof of which is the main focus of the remainder of the paper. 

\begin{proposition} [Bootstrap] \label{prop:Boot}
For the constants appearing in the right-hand side of \eqref{def:KB} chosen sufficiently large and $c_0$ chosen sufficiently small (depending only on $s,\lambda_0,\lambda^\prime,\alpha,\delta_1$ and arbitrary parameters such as $\sigma,\beta, \ldots$), if the bootstrap hypotheses \eqref{ineq:Boot_Hi}, \eqref{ineq:Boot_CgHi}, \eqref{ineq:Boot_ED}, and \eqref{ineq:Boot_LowFreq} hold on $[1,T^\star]$, then on the same time interval all the inequalities in \eqref{ineq:Boot_Hi}, \eqref{ineq:Boot_CgHi}, \eqref{ineq:Boot_ED}, and \eqref{ineq:Boot_LowFreq} hold with constant `$2$' instead of `$4$'.  
\end{proposition}

In Lemma \ref{lem:PropBootThm} in \S\ref{sec:PropBootThm} below, it is proved that Proposition \ref{prop:Boot} implies Theorem \ref{thm:Threshold} and hence the majority of the remainder of the paper is dedicated to the proof of Proposition \ref{prop:Boot}.  

First, in \S\ref{sec:BootConst} below we discuss how the bootstrap constants are chosen, then in \S\ref{sec:bootapriori} we detail some immediate a priori estimates which follow from the bootstrap hypotheses and afterwards in \S\ref{sec:summary3}, give a brief discussion of the main principles behind the proof of Proposition \ref{prop:Boot} as well as discuss in more detail the motivations for some 
of the finer details, such as the particular hierarchy of norms being employed.
The next two sections are further preliminaries. First, in \S\ref{sec:RegCont} we deal with the relatively minor technicalities such as changing coordinate systems and the local well-posedness issues. Then, we lay down the main technical tools used in the proof of Proposition \ref{prop:Boot} in \S\ref{sec:nrmuse}, including lemmas involving the norms 
we are using as well as a summary of the paraproduct decompositions that are used heavily in the sequel. 
From there, the remainder of the paper is dedicated to the energy estimates stated in Proposition \ref{prop:Boot}. 

\subsubsection{Bootstrap constants} \label{sec:BootConst}
First, $K_{GL3}, K_{U3},K_{U1}, K_{U12}$ are chosen sufficiently large relative to a universal constant depending only on $\sigma^\prime$. 
From there, $K_{Hi},K_{H1\neq}$, and $K_{M1}$ are chosen independently, sufficiently large relative to constants depending only on the parameters $\lambda_0,\lambda^\prime,\alpha,s$, and $\delta_1$ (and arbitrary parameters such as $\sigma,\beta,\ldots$).  
From there,  $K_{HC1}$, $K_{HC2}$, and $K_{LC}$ are independently chosen small relative to constants depending on the same parameters. 
Then, $K_{ED2}$ is chosen small relative to the same parameters, followed by $K_{ED1}$ and $K_{ED3}$ which are chosen small relative also to $K_{ED2}$ and $K_{Hi}$. After these constants are chosen, $c_0$ is chosen sufficiently small with respect to $K_B$, the max of all the bootstrap constants, as well as the parameters $s,\lambda_0,\lambda^\prime,\alpha,\delta_1$ (and arbitrary parameters such as $\sigma,\beta,\ldots$).  

\subsubsection{A priori estimates from the bootstrap hypotheses} \label{sec:bootapriori}
The motivation for the enhanced dissipation estimates \eqref{ineq:Boot_ED} is the following observation (which follows from \eqref{ineq:DLowB}): for any $f$,  
\begin{subequations} \label{ineq:AnuDecay}
\begin{align}
\norm{f_{\neq} (t)}_{\mathcal{G}^{\lambda(t),\beta}} & \lesssim_\alpha  \jap{t}^{2+\delta_1}\jap{\nu t^3}^{-\alpha} \norm{A^{\nu;1} f(t)}_2  \\ 
\norm{f_{\neq}(t)}_{\mathcal{G}^{\lambda(t),\beta}} & \lesssim_\alpha \jap{t}^{\delta_1} \jap{\nu t^3}^{-\alpha} \norm{A^{\nu;2} f(t)}_2, \\
\norm{f_{\neq}(t)}_{\mathcal{G}^{\lambda(t),\beta}} & \lesssim_\alpha \jap{t}^{2} \jap{\nu t^3}^{-\alpha} \norm{A^{\nu;3} f(t)}_2. 
\end{align} 
\end{subequations} 
Hence, \eqref{ineq:Boot_ED} expresses a rapid decay of $Q^i_{\neq}$ for $t \gtrsim \nu^{-1/3}$.  
Together with the ``lossy elliptic lemma'', Lemma \ref{lem:LossyElliptic}, we then get
\begin{subequations}  \label{ineq:AprioriUneq}
\begin{align} 
\norm{U^1_{\neq} (t)}_{\mathcal{G}^{\lambda(t),\beta-2}} & \lesssim  \frac{\epsilon \jap{t}^{\delta_1}}{\jap{\nu t^3}^{\alpha}} \\ 
\norm{U^2_{\neq}(t)}_{\mathcal{G}^{\lambda(t),\beta-2}} & \lesssim \frac{\epsilon}{\jap{t}^{2-\delta_1}\jap{\nu t^3}^{\alpha}} \\ 
\norm{U^3_{\neq}(t)}_{\mathcal{G}^{\lambda(t),\beta-2}} & \lesssim \frac{\epsilon}{\jap{\nu t^3}^{\alpha}}.  
\end{align}
\end{subequations} 
Lemma \ref{lem:LossyElliptic} is what allows to deduce a similar estimate on $\Delta_t^{-1}$ as we have on $\Delta_L^{-1}$ in \eqref{ineq:IDfundamental}, assuming the coordinate transform is not too large (which is implied by \eqref{ineq:Boot_LowC}). 
 
For the zero frequencies of the velocity field we get from \eqref{ineq:Boot_Hi}, \eqref{ineq:Boot_LowFreq} and Lemma \ref{lem:PELbasicZero} (which allows to understand $\Delta_t^{-1}$ at zero frequencies) the matching a priori estimates
\begin{subequations} \label{ineq:AprioriU0}  
\begin{align} 
\norm{A U^1_0 (t)}_{2} & \lesssim \epsilon \jap{t} \\ 
\norm{U^1_0 (t)}_{\G^{\lambda,\gamma}} & \lesssim c_0 \\ 
\norm{A U^2_0 (t)}_{2} + \norm{A U^3_0 (t)}_{2} & \lesssim \epsilon. 
\end{align} 
\end{subequations} 
In \S\ref{sec:XYZtoxyz}, we prove that \eqref{ineq:AprioriUneq} and \eqref{ineq:AprioriU0} together imply Theorem \ref{thm:Threshold}.

\subsubsection{Short summary of the proof of Proposition \ref{prop:Boot}} \label{sec:summary3}
Let us now quickly discuss some basic principles and tools that go into the proof of Proposition \ref{prop:Boot}.
 
\paragraph{Frequency decomposition tools and regularity levels}
From Proposition \ref{prop:Boot}, \S\ref{sec:bootapriori}, and the definition of the norms in \eqref{def:A} and \eqref{def:Anu}, it is clear that the 
information we have depends a lot on the relationship between frequency and time. 
To take advantage of such details, we need to decompose the nonlinear terms based on these relationships. 
A basic tool of nonlinear Fourier analysis is the paraproduct, introduced by Bony \cite{Bony81}, which allows to 
make decompositions of the type 
\begin{align*}
fg = f_{Hi} g_{Lo} + f_{Lo} g_{Hi} + (fg)_{\mathcal{R}}, 
\end{align*}
that is, the first term is the contribution to the product where the frequencies of $fg$ are comparable to $f$, in the second term the frequencies of $fg$ are comparable to $g$ and the last term is the contribution from where the frequencies of $f$ and $g$ are comparable (see \S\ref{sec:nrmuse} for more information). 
It essentially linearizes the evolution of higher frequencies around lower frequencies -- see \cite{Hormander1990,BMM13} for discussions about the connection between paradifferential calculus and the Nash-Moser iteration. 
Paraproduct tools have been very useful in several works that involve mixing in fluids and plasmas \cite{BM13,BMV14,BMM13,Young14}. 

If we are deducing a high norm estimate, then we can control the contributions of `low frequencies' with the lower norm controls we have; on the other hand if we are getting a low norm estimate then we can absorb all kinds of regularity loss by using the high norm control, which allows us to often get better decay estimates, e.g. the enhanced dissipation \eqref{ineq:Boot_ED} and the decay estimates in \eqref{ineq:Boot_gLow} and  \eqref{ineq:Boot_LowFreq}. 
Hence it becomes more clear why we generally have at least two regularity levels for most of the unknowns. 
 Moreover, we also note that if we are estimating a high norm, then the interaction of two non-zero frequencies in $X$ will be rapidly damped by the use of \eqref{ineq:AprioriUneq} but the same is not quite true of the interaction between a zero and non-zero frequency at the high norm. 

\paragraph{Elliptic estimates}
An important set of tools we need are those pertaining to $\Delta_t^{-1}$, as this is how we recover $U^i$ from $Q^i$. 
These are detailed in Appendix \ref{sec:Elliptic}, as they are basically a generalization of the ideas employed in \cite{BM13,BMV14}.  
One set of lemmas are the so-called `lossy' estimates, detailed in Appendix \ref{sec:Lossy}. These are straightforward and allow us to deduce roughly the equivalent of \eqref{ineq:IDfundamental} for $\Delta_t^{-1}$ given the control \eqref{ineq:Boot_LowC} for $c_0$ sufficiently small. 
In Appendix \ref{sec:PEL}, we detail the more subtle lemmas that arise when putting $\Delta_t^{-1}Q^i$ in a high norm. 
We refer to these as `precision elliptic lemmas' since they treat the loss of ellipticity due to the Orr mechanism in a precise way.  
In \cite{BM13}, one such lemma was needed, but here we need several. The general principle is to use $\Delta_L^{-1}$ as an approximate inverse and coordinate system estimates in \eqref{ineq:Boot_CgHi} to control the errors. 
Unlike \cite{BM13,BMV14}, the enhanced dissipation from \eqref{ineq:AprioriUneq} is crucial for controlling the error terms. 

\paragraph{Control of $Q^i$ in the high norm, \eqref{ineq:Boot_Hi}} \label{sec:XindepHigh}
The general scheme for essentially all  of the high norm energy estimates in Proposition \ref{prop:Boot} is to get an estimate more or less like the following 
\begin{align}
\frac{1}{2}\frac{d}{dt}\norm{A^i Q^i}_2^2 + \mathcal{D}Q^i \lesssim \max(c_0,\epsilon^{\delta})\mathcal{D}Q^i + \max(c_0,\epsilon^{\delta})\epsilon^2 \mathcal{I}(t), \label{ineq:GenScheme}
\end{align}
where $\delta > 0$ and $\mathcal{I}(t)$ is an integrable function, uniformly in $\epsilon$, $\nu$, and $c_0$. In general, $\mathcal{I}$ will involve also the dissipation energies. 
For $c_0$ sufficiently small (hence also $\epsilon$, since $\epsilon < c_0\nu \leq c_0$), estimates of this type will then imply the high norm estimates in Proposition \ref{prop:Boot}.   
A similar method is used to get every estimate in Proposition \ref{prop:Boot} with the exception of \eqref{ineq:Boot_LowC}, \eqref{ineq:Boot_Q1Hi1}, and \eqref{ineq:Boot_Q1Mid}, which are done with a slight variation on this approach.
The issue with these estimates is to avoid losing a $\abs{\log \epsilon}$ (fatal to the proof of Theorem \ref{thm:Threshold}) which requires a bit of extra work as the linear lift-up effect terms in the equations for $C$ and $Q^1_0$ are borderline (because $c_0$ is chosen independent of $\nu$).  
The Sobolev scale estimates are discussed further below. 

The high norm estimates are divided into zero and non-zero (in $X$) modes as the behaviors are starkly different. 
First, are the linear stretching and pressure terms in $Q^1$ and $Q^3$, which cannot be treated perturbatively, although the effect of the coordinate system on them is treated perturbatively (crucially). The norms get weaker at low frequencies in order to absorb the effect of these linear terms and this explains why there is a `$CK_L^2$' term controlled in \eqref{ineq:Boot_Q2Hi} but not in the \eqref{ineq:Boot_Q1Hi2} or \eqref{ineq:Boot_Q3Hi} (the entire $CK_L^{1,3}$ terms are used  in the estimate). 
The control on the $CK_L^2$ term in \eqref{ineq:Boot_Q2Hi} turns out to be useful for controlling the lift-up effect term in the proof of \eqref{ineq:Boot_Q1Hi2}.  
The other complications at the high norm are the nonlinear interactions corresponding to \textbf{(SI)}, \textbf{(F)}, and \textbf{(3DE)}, which are controlled using the norms devised via the toy model. In particular, these contributions make heavy use of the $CK_w$ terms, which make sure that the norms lose regularity in a way which matches the toy model in \S\ref{sec:Toy}. 

\paragraph{Control of the coordinate system}
These estimates are deduced in \S\ref{sec:Coord}. 
Estimate \eqref{ineq:Boot_ACC} is complementary to \eqref{ineq:Boot_ACC2}: at long times, \eqref{ineq:Boot_ACC2} is obviously better, but at shorter times \eqref{ineq:Boot_ACC2} is much better. 
For these estimates, controlling $g$ is quite important, and the decay estimates on $g$ and $U_0^2$ at low frequencies allow uniform estimates on $C$ (high frequencies are generally absorbed by the dissipation). 

\paragraph{Enhanced dissipation estimates}
The enhanced dissipation estimates \eqref{ineq:Boot_ED} are improved in \S\ref{sec:ED}.  
The first important fact to notice regarding the proof of \eqref{ineq:Boot_ED} is that for $t \leq 2\abs{\partial_Y}$ 
we can gain a large power of $\jap{t}^{-1}$ by exchanging $A^{\nu;i}$ for $A^i$, or some regularity class in between (see \eqref{ineq:AnuHiLowSep} below), one of several reasons to get these decay estimates separately from the high norm estimates. 
This means that the main difficulties in obtaining the enhanced dissipation estimates will not really be associated with the Orr mechanism. 
Instead, the main difficulties in \S\ref{sec:ED} arise from the linear vortex stretching and the \textbf{(SI)} interactions: terms that are linear in $X$ dependent frequencies.  

As in the proof of \eqref{ineq:Boot_Hi}, in these estimates, the special structure of the nonlinearity is key, and is combined with \eqref{ineq:AnuLossyED} to high effect. 
A recurring theme is to gain in $t$ from Lemma \ref{lem:AnuLossy} when $\partial_X$ derivatives are present, a kind of ``null'' structure that allows to balance the large growth of $U_0^1$ by extra decay since it often appears as $U_0^1 \partial_X$ in the nonlinear terms.

\paragraph{Low norm estimates on the $X$-independent modes}
The last set of estimates are those in \eqref{ineq:Boot_LowFreq}. 
The purpose of these are two-fold: first, to control the low frequencies of the velocity field (which are not well-controlled by $Q^i$); 
second, to deduce extra decay estimates.
In contrast to the other estimates in Proposition \ref{prop:Boot}, improving \eqref{ineq:Boot_LowFreq} lends itself to a proof in $(X,y,z)$ rather than $(X,Y,Z)$, as it uses the divergence free condition and the dissipation together in a way which is easier in the standard coordinates. 
Since the regularity class in \eqref{ineq:Boot_LowFreq} is finite, there is no problem transferring estimates from one coordinate system to another (see \S\ref{sec:RegCont}). 

\section{Regularization and continuation} \label{sec:RegCont} 
This section contains three topics: (A) comments on the local-wellposedness for classical solutions of \eqref{def:3DNSE} and the instantaneous analytic regularization with initial data of the type \eqref{ineq:QuantGev2} (B) how to move estimates on these classical solutions between coordinate systems, and (C) the proof that Proposition \ref{prop:Boot} implies Theorem \ref{thm:Threshold}.
All of these steps are very similar to the 2D works \cite{BM13,BMV14}, so we include only a brief summary. 

\subsection{Regularization and short-time regularity} \label{sec:ContCoord}

\begin{proof}[\textbf{Proof of  Lemma \ref{lem:Loc}}]
By using an easy variant of standard local well-posedness theory for strong solutions of 3D Navier-Stokes (e.g. \cite{MajdaBertozzi}), it is easy to prove that there is a $t_\star > 0$ such that there is a unique strong solution $u(t)$ with initial data $u_{in}$ to \eqref{def:3DNSE}. 
Standard parabolic regularity theory shows that $u(t)$ is real analytic for $t \in (0,t_\star)$. 
 
The next step is to verify that for $K_0$ chosen sufficiently large, there is some time $t_G \in (0,1)$ such that $\norm{u(t_G)}_{\G^{\bar{\lambda},s}} \leq \frac{10}{9}\epsilon$.
This is a matter of making more quantitative the local regularity estimates from initial data of the type \eqref{ineq:QuantGev2} and 
we can apply a straightforward variant of the method used in \cite{BMV14} for 2D Navier-Stokes.
This consists of several steps. First, we solve \eqref{def:3DNSE} (locally in time) with $u_S$ as the initial data and obtain quantitative estimates on the analyticity of the solution $u_S(t)$ for $t \ll \nu$. Next, we get short-time analyticity results on the rough perturbation $u_R(t)$ with initial data $u_R(0)$ for $t \ll \nu$ ($u_R(t)$ satisfies a PDE such that $u_S(t) + u_R(t)$ together solve \eqref{def:3DNSE}). Finally, we make a third estimate propagating the correct Gevrey-$\frac{1}{s}$ regularity forward to a time independent of $\nu$. We omit the details for brevity as they are similar to those in \cite{BMV14}.  
\end{proof} 

\begin{proof} [\textbf{Proof of  Lemma \ref{lem:Cont}}]
By a straightforward variant of the local existence argument, it is easy to prove that the unique, strong solution exists (uniformly in $\nu$) as long as an $H^{5/2+}$ norm remains finite.  
Propagation of analyticity uniformly in $\nu$ can be proved via variants of e.g. \cite{FoiasTemam89,LevermoreOliver97,KukavicaVicol09}. The details are omitted for brevity. 
\end{proof} 

In order to get the bootstrap argument started, we need to prove that estimates in $(x,y,z)$ transfer naturally to estimates in $(X,Y,Z)$ and give us a buffer zone away from the mild coordinate singularity at $t  = 0$.
\begin{lemma} \label{lem:BootStart}
For $c_0$ sufficiently small, we may take $2 \leq T^\star$ (defined in \S\ref{sec:energy} above) and for $t \leq 2$, the bootstrap estimates in \eqref{ineq:Boot_Hi}, \eqref{ineq:Boot_CgHi}, \eqref{ineq:Boot_ED}, and \eqref{ineq:Boot_LowFreq}, all hold with constant $5/4$ instead of $4$. 
\end{lemma} 
\begin{proof}
The approach we take is that of \cite{BM13}, which is to first get an estimate until time $t = 2$ in the coordinate system used when studying the linearization in \S\ref{sec:lin} and then change coordinates into those defined in \S\ref{sec:coordinates}. We only give a brief sketch. 
Define
\begin{align*} 
\bar{x} & = x-ty \\ 
h^i(t,\bar{x},y,z) & = q^i(t,\bar{x} + ty,y,z) \\ 
v^i(t,\bar{x},y,z) & = u^i(t,\bar{x} + ty,y,z);
\end{align*}
note that $v^i = \Delta_L^{-1}h^i$. 
These unknowns satisfy natural analogues of \eqref{def:MainSys} and \eqref{def:3DNSE}. 
Moreover, for $j \in \set{0,1,2}$,  
\begin{align*} 
\norm{v^i_{\neq}}_{\G^{\mu,\gamma}} & \lesssim \norm{h^i_{\neq}}_{\G^{\mu,\gamma}}  \\ 
\norm{\partial_{y}^j v^i_{\neq}}_{\G^{\mu,\gamma}} & \lesssim \jap{t}^{j} \norm{h^i_{\neq}}_{\G^{\mu,\gamma}} \\ 
\norm{\jap{\partial_{y} - t\partial_{\bar{x}}}^j v^i_{\neq}}_{\G^{\mu,\gamma}} & \lesssim \norm{h^i_{\neq}}_{\G^{\mu,\gamma}} \\ 
\norm{\jap{\partial_{\bar{x}}}^j v^i_{\neq}}_{\G^{\mu,\gamma}} + \norm{\jap{\partial_{z}}^j v^i_{\neq}}_{\G^{\mu,\gamma}} & \lesssim \norm{h^i_{\neq}}_{\G^{\mu,\gamma}} 
\end{align*} 
With these estimates, it is relatively easy to verify using the standard techniques for getting Gevrey regularity estimates on transport equations, see e.g. \cite{BM13,FoiasTemam89,LevermoreOliver97,KukavicaVicol09}, that the following holds: for $\epsilon < c_0 \nu$ sufficiently small and $\mu(t)$ 
chosen as
\begin{align*}
\dot{\mu}(t) = -\epsilon^{1/2} \mu(t), \quad\quad \mu(t_\star) = \bar{\lambda} = \frac{9\lambda_0}{10} + \frac{\lambda^\prime}{10}, 
\end{align*}
such that $\mu(2) \geq \frac{7 \lambda_0}{8} + \frac{\lambda^{\prime}}{8}$, 
we have, for $t \leq 2$, the estimates  
\begin{align} 
\norm{h^{i}}_{\G^{\mu(t)}}  + \norm{v^{i}}_{\G^{\mu(t)}}  & \lesssim \epsilon. \label{ineq:shorttimexyz}
\end{align} 
Next we convert estimates on $h^i$ and $v^i$ to estimates on $C,Q^i$ and $U^i$. 
In order to estimate the norm of $Q^i(t,X,Y,Z)$ in terms of $h^i(\bar{x},y,z)$ we need to solve for $(\bar{x},y,z)$
in terms of $(X,Y,Z)$ and then use a lemma for estimating Gevrey regularity under composition (see Lemma \ref{lem:GevComp}).
From \eqref{def:psi2}, we can use the same methods used to deduce \eqref{ineq:shorttimexyz} to get $\norm{t\psi}_{\G^{\mu(t)}} \lesssim \epsilon$,  
which for $\frac{1}{2} \leq t \leq 2$, yields good estimates on $\psi(t,y,z) = Y(t,y,z) - y$ and on $X(t,x,y,z) = \bar{x}(t,x,y) - t \psi(t,y,z)$. 
Hence, we can write
\begin{align*} 
\bar{x}(t,X,Y,Z) & = X + t\psi(t,y(t,X,Y,Z),z(t,X,Y,Z)) \\ 
y(t,X,Y,Z) & = Y + \psi(t,y(t,X,Y,Z),z(t,X,Y,Z)) \\ 
z(t,X,Y,Z) & = Z,  
\end{align*}
and then for $\epsilon$ sufficiently small, a suitable Gevrey inverse function theorem (see e.g. \cite{BM13}), 
allows to solve for $(\bar{x},y,z)$ with estimates on the Gevrey regularity. 
Lemma \ref{lem:GevComp} on Gevrey composition then allows to get the estimates (using $\epsilon$ sufficiently small)
for some $\delta > 0$ and for $1/2 < t \leq 2$: 
\begin{align} 
\norm{C}_{\G^{(1+\delta)\lambda(t),\sigma}} + \norm{Q^i}_{\G^{(1+\delta)\lambda(t),\sigma}} + \norm{U^i}_{\G^{(1+\delta)\lambda(t),\sigma}} \lesssim \epsilon. \label{ineq:inintCQU} 
\end{align}
By replacing the initial $\epsilon$ with some smaller $\epsilon^\prime$, we can change the norms to the $A^i$ and $A$ (paying the $\delta \lambda$ regularity in exchange for the factor involving $w^{-1}$ from \eqref{ineq:IncExp}) and drop the constants in \eqref{ineq:inintCQU}. Hence, the lemma follows up to adjusting the bootstrap constants.  
\end{proof} 

\subsection{Moving from $(X,Y,Z)$ to $(x,y,z)$} \label{sec:XYZtoxyz}

In order to prove Theorem \ref{thm:Threshold}, we need to be able to transfer information in $(X,Y,Z)$ coordinates back to information in $(x,y,z)$. 
This is also used in \S\ref{sec:LowNrmVel}. 
Similar to the method used in \cite{BM13,BMV14}, we will first move to the coordinate system $(X,y,z)$. 
Writing $\bar{q}^i(t,X,y,z)  = Q^i(t,X,Y,Z) = q(t,x,y,z)$ and $\bar{u}^i(t,X,y,z) = U^i(t,X,Y,Z) = u^i(t,x,y,z)$ we derive the following, noting that $\bar{u}^i_0 = u^i_0$: 
\begin{align} 
\partial_t u_0^i + (u_0^2,u_0^3) \cdot \grad u_0^i & = (-u_0^2,0,0)^T - (0,\partial_yp_0^{NL}, \partial_z p^{NL}_0)^T + \nu \Delta u^i_0 + \mathcal{F}^i, \label{eq:u0i}
\end{align} 
where 
\begin{align} 
\mathcal{F}^i =  - \left(\bar{u}^1_{\neq} \partial_X \bar{u}_{\neq}^i\right)_0  -\left(\bar{u}^2_{\neq} (\partial_y - \partial_y\psi t\partial_X) \bar{u}_{\neq}^i\right)_0  - \left(\bar{u}^3_{\neq} (\partial_z - t\partial_z \psi \partial_X) \bar{u}_{\neq}^i\right)_0. 
\end{align} 
Since the divergence free condition transforms into
\begin{align} 
\partial_X \bar{u}^{1} + (\partial_y - \partial_y\psi t\partial_X) \bar{u}^2 + (\partial_z - t\partial_z \psi \partial_X)\bar{u}^3,
\end{align} 
then we get for $\mathcal{F}$ the following form, analogous to the cancellations in \eqref{eq:XavgCanc}, 
\begin{align} 
\mathcal{F}^i =   -\partial_y \left(\bar{u}^2_{\neq}\bar{u}_{\neq}^i\right)_0  - \partial_z\left(\bar{u}^3_{\neq} \bar{u}_{\neq}^i \right)_0. \label{eq:Fbaru}
\end{align}
\begin{lemma} \label{lem:intermedSob}
For $\epsilon < c_0\nu$ and $c_0$ sufficiently small, the bootstrap hypotheses imply the following for some $c \in (0,1)$ chosen such that $c\lambda(t) \in (\lambda^\prime,\lambda(t))$ for all $t$:   
\begin{subequations} \label{ineq:Xyzubds}
\begin{align} 
\norm{\bar{u}^1_{\neq}}_{\G^{c\lambda(t)}} & \lesssim \epsilon \jap{t}^{\delta_1}\jap{\nu t^3}^{-\alpha}  \\ 
\norm{\bar{u}_{\neq}^2}_{\G^{c\lambda(t)}} & \lesssim \epsilon \jap{t}^{-2+\delta_1} \jap{\nu t^3}^{-\alpha}  \\ 
\norm{\bar{u}_{\neq}^3}_{\G^{c\lambda(t)}} & \lesssim \epsilon  \jap{\nu t^3}^{-\alpha}, 
\end{align}  
\end{subequations}
and
\begin{subequations} \label{ineq:uzAPriori}
\begin{align} 
\norm{u^1_0(t)}_{\G^{c\lambda(t)}} & \lesssim \max(\epsilon \jap{t},c_0) \label{ineq:uzApriori1} \\ 
\norm{u^2_0(t)}_{\G^{c\lambda(t)}} + \norm{u^3_0(t)}_{\G^{c\lambda(t)}} & \lesssim \epsilon \label{ineq:uzApriori23} \\ 
\norm{q_{0}^2(t)}_{H^{\sigma^\prime}} + \norm{u_{0}^2(t)}_{H^{\sigma^\prime}} & \lesssim \frac{\epsilon}{\jap{\nu t}^\alpha} \\ 
\norm{u_{0}^1(t)}_{H^{\sigma^\prime}}^2 + \nu \int_1^t \norm{\grad u_0^1(\tau)}_{H^{\sigma^\prime}}^2 d\tau  & \lesssim c_0^2 \\ 
\norm{u_{0}^3(t)}_{H^{\sigma^\prime}}^2 + \nu \int_1^t \norm{\grad u_0^3(\tau)}_{H^{\sigma^\prime}}^2 d\tau  & \lesssim \epsilon^2 \\ 
\norm{u_0^3}_4 & \lesssim \frac{\epsilon}{\jap{ \nu t }^{1/4}} \label{ineq:u03L4} \\ 
\norm{u_0^1}_4 & \lesssim \frac{c_{0}}{\jap{ \nu t }^{1/4}}. \label{ineq:u01L4} 
\end{align}  
\end{subequations} 
\end{lemma} 
\begin{proof} 
We want to quantify the Gevrey regularity of $\bar{q}^i(X,y,z) = Q^i(X,Y(y,z),Z(y,z))$ and $\bar{u}^i(X,y,z) = U^i(t,X,Y(y,z),Z(y,z))$ via the composition inequality Lemma \ref{lem:GevComp} (and also Sobolev regularity via Sobolev composition).  
Notice that $Z(y,z) = z$ and $Y(y,z) - y = \psi(y,z)$, and hence what we need is a Gevrey estimate on $\psi(t,y,z)$. 
However, the bootstrap hypotheses give only direct estimates on $C$, $\psi_y$ and $\psi_z$, in $(Y,Z)$ coordinates. 
Recall $(\psi_y(t,Y(t,y,z),Z), \psi_z(t,Y(t,y,z),Z) = (\partial_y \psi(t,y,z), \partial_z\psi(t,y,z))$.
By the $C^\infty$ inverse function theorem, we can solve $Y(t,y,z) = y + \psi(t,y,z)$ for $y = y(t,Y,z)$ as a function of $Y,z$.
For $c_0$ sufficiently small we can write the derivative as a convergent power series (by Lemma \ref{lem:CoefCtrl} below, $\norm{\psi_y}_\infty \lesssim \norm{\psi_y}_{\G^{\lambda,\gamma-1}} \lesssim c_0 < 1/2$): 
\begin{align} 
\partial_Y y(t,Y,z) = \frac{1}{\partial_y Y(t,y(t,Y,z),z)} = \frac{1}{1 + \psi_y(t,Y,z) } = \sum_{j = 0}^\infty (-\psi_y(t,Y,z))^j. 
\end{align} 
By Lemma \ref{lem:GevProdAlg} it follows that 
\begin{align} 
\norm{\partial_Y y - 1}_{\G^{\lambda}} \lesssim c_0. 
\end{align} 
Moreover, we get from this and the controls on $C$, 
\begin{align*}
\norm{Y(t,y,z) - y}_2 & = \norm{\psi}_2  \lesssim \norm{\partial_Y y}_\infty \norm{C}_2  \lesssim c_0. \\
\norm{y(t,Y,z)  - Y}_2 & = \norm{C(t,Y,z)}_2  \lesssim c_0. 
\end{align*} 
Together these imply
\begin{align*} 
\norm{y(t,Y,z) - Y}_{\G^{\lambda}} & \lesssim c_0. 
\end{align*}
Finally, by a Gevrey inverse function theorem (see e.g. the one used in \cite{BM13} for this same purpose) and choosing $c_0$ sufficiently small,  
we derive for a constant $c^\prime > 0$ such that $c^\prime\lambda \in (\lambda^\prime,\lambda)$, 
\begin{align*} 
\norm{Y(t,y,z) - y}_{\G^{c^\prime\lambda}} \lesssim c_0. 
\end{align*}
Then by Lemma \ref{lem:GevComp}, for $c_0$ sufficiently small, we derive for a constant $c < c^\prime$ such that $c\lambda \in (\lambda^\prime,c^\prime\lambda)$, 
\begin{subequations} \label{ineq:barqu} 
\begin{align} 
\norm{\bar{q}^i}_{\G^{c\lambda}} & \lesssim \norm{Q^i}_{\G^{\lambda}} \\ 
\norm{\bar{u}^i}_{\G^{c\lambda}} & \lesssim \norm{U^i}_{\G^{\lambda}}.  
\end{align}
\end{subequations} 
From here \eqref{ineq:Xyzubds} follows from \eqref{ineq:AprioriUneq}. 
Similarly, the first two inequalities in \eqref{ineq:uzAPriori} follow from \eqref{ineq:AprioriU0}. 
By Sobolev composition the remainder of the inequalities in \eqref{ineq:uzAPriori} follow from \eqref{ineq:AprioriU0} or \eqref{ineq:Boot_LowFreq}.
\end{proof} 

\subsection{Proposition \ref{prop:Boot} implies Theorem \ref{thm:Threshold} and Proposition \ref{prop:KEflux}} \label{sec:PropBootThm}
In this section we prove 
\begin{lemma} \label{lem:PropBootThm}
For $c_0$ sufficiently small (depending only on $s,\lambda,\lambda^\prime,\alpha$ and $\delta_1$) and for all $\epsilon < c_0 \nu$ with $\nu \in (0,1]$, Proposition \ref{prop:Boot} implies Theorem \ref{thm:Threshold}.  
\end{lemma} 
\begin{proof} 
From the definition of $X$ in \eqref{def:XYZ}, \eqref{ineq:uidamping} follows from \eqref{ineq:Xyzubds}. 
Similarly, \eqref{ineq:uzApriori1} and \eqref{ineq:uzApriori23} imply \eqref{ineq:finalu1} and \eqref{ineq:finalu3}. 
By refining the proof of \eqref{ineq:Boot_Q02_Low}, it is not hard to deduce \eqref{ineq:finalu2} (all of the Sobolev estimates in \eqref{ineq:Boot_LowFreq} could be deduced in a Gevrey class, but this was not necessary for the proof of Proposition \ref{prop:Boot}).   
Naturally, \eqref{ineq:finalu3L4} and \eqref{ineq:finalu1L4} are the same as \eqref{ineq:u03L4} and \eqref{ineq:u01L4}. 
To see \eqref{ineq:psiest}, note by \eqref{def:g} 
\begin{align*}
C = U_0^1 - tg, 
\end{align*}
and so \eqref{ineq:psiest} follows from the Gevrey composition arguments in Lemma \ref{lem:intermedSob} above and \eqref{ineq:Boot_gLow}. 

It remains to deduce \eqref{ineq:trans}. By Duhamel's formula we have for $i \in \set{2,3}$
\begin{subequations} 
\begin{align}
u_0^1(t) & = e^{\nu t \Delta}u^{1}_{0 \; in} - \int_0^t e^{\nu(t-\tau)\Delta}u_0^2(\tau) d\tau - \int_0^t e^{\nu(t-s)\Delta}\left(u_0 \cdot \grad u_0^1(\tau) + <\bar{u}_{\neq}\cdot \grad^t \bar{u}_{\neq}^1>_X \right) d\tau \label{eq:Duhamel1} \\
u_0^{i}(t) & = e^{\nu t \Delta}u^{i}_{0 \; in} - \int_0^t e^{\nu(t-s)\Delta}\left(u_0 \cdot \grad u_0^{i}(\tau) +\partial_i p_0^{NL} + \left(\bar{u}_{\neq}\cdot \grad^t \bar{u}_{\neq}^i\right)_0 \right) d\tau. \label{eq:Duhamel23}
\end{align}
\end{subequations} 
We see that \eqref{ineq:u023} follows easily from Lemma \ref{lem:intermedSob} and \eqref{eq:Duhamel23}. Then \eqref{ineq:u01grwth} follows from \eqref{eq:Duhamel1}, Lemma \ref{lem:intermedSob}, and \eqref{ineq:u023}. 
\end{proof} 

In addition, we also prove Proposition \ref{prop:KEflux}. 
\begin{proof}[\textbf{Proof of Proposition \ref{prop:KEflux}}] 
Consider solutions with initial data $\norm{u^1_{\neq}}_2 + \norm{u^3_{\neq}}_2 \gtrsim \epsilon$ but $\norm{u^2_{\neq}}_2 \lesssim \epsilon^2$. 
By integrating the equation for $Q^2$ in \eqref{def:MainSys} using the a priori estimates from Proposition \ref{prop:Boot}, we see the bound 
\begin{align*}
\norm{Q^2}_{\G^{\lambda(t),\beta-5}} & \lesssim \epsilon \nu t^3 + \epsilon^2 + \int_1^t \frac{\epsilon^2 \tau^2}{\jap{\nu \tau^3}^\alpha} d\tau,  
\end{align*} 
which for $\nu t^3 \ll 1$ gives $\norm{Q^2}_{\G^{\lambda(t),\beta-5}} \lesssim \epsilon \nu \jap{t}^3$. Then, by Lemma \ref{lem:LossyElliptic}, we have 
the bound stated in \eqref{ineq:u2grwKE} after moving to the $(x,y,z)$ coordinates. Then the other bounds in \eqref{ineq:inviscidlinear} follow from estimating
the corresponding equations for $U^1$ and $U^3$ in the $(X,Y,Z)$ coordinates using the stronger a priori estimate on $Q^2$ and $U^2$ and then transferring the information back to $(x,y,z)$.
The norm growth \eqref{ineq:nrmexplode} then follows from \eqref{ineq:inviscidlinear}, using that the deformations from $\psi$ , the nonlinear effects, and the dissipation are smaller than the contribution of the initial data.  
\end{proof}

\section{Multiplier and paraproduct tools} \label{sec:nrmuse}

In this section we outline some basic general inequalities regarding the multipliers which are used in the sequel as well as introduce and explain the paraproduct decomposition. 
The purpose is to set up a general framework that will  
make the large number of energy estimates later in the paper as painless as possible.   

\subsection{Basic inequalities regarding the multipliers} \label{sec:basicmult}

This section covers the key properties of the multipliers we are using, however the statements and proofs can be rather technical and will likely appear unmotivated on first reading. 
A reader could potentially skip this to start and refer back to it whenever specific inequalities are needed. 

In the lemmas which follow, one should imagine that frequencies $(k^\prime,\xi,l^\prime)$ and $(k-k^\prime,\eta-\xi,l-l^\prime)$ are interacting to force $(k,\eta,l)$, as will be occurring in the quadratic energy estimates.  

The first two lemmas explain how to compare the $A^i$ and $A$ to each other at various frequencies. 
The first of the pair deals with using the multipliers $A^i$ and $A$ in estimating the contributions from the leading order $f_{Hi}g_{Lo}$-type contributions to the paraproducts (see \eqref{def:parapp} below) and is hence the most important. 
\begin{lemma}[Frequency ratios for $A$ and $A^i$] \label{lem:ABasic}
Let $\theta < 1/2$ and suppose 
\begin{align} 
\abs{k-k^\prime,\eta-\xi,l-l^\prime} \leq \theta\abs{k,\eta,l}. \label{ineq:AFreqLoc}
\end{align}
Define, for $i,j \in \set{1,2,3}$ and $a,b \in \set{0,\neq}$, the weight $\Gamma(i,j,a,b)$ given by, 
\begin{align*}
\Gamma(i,i,a,a) & = 1, & \Gamma(i,j,a,b)  & = \Gamma(j,i,b,a)^{-1}, \\
\Gamma(1,2,0,0) & = \jap{t}^{-1}, & \Gamma(1,2,\neq,\neq) & = \jap{t}^{-1} \jap{\frac{t}{\jap{\xi,l^\prime}}}^{-\delta_1}, \\ 
\Gamma(1,2,0,\neq) & = \jap{t}^{-1} \jap{\frac{t}{\jap{\xi,l^\prime}}}, & \Gamma(1,2,\neq,0) & = \jap{t}^{-1} \jap{\frac{t}{\jap{\xi,l^\prime}}}^{-1-\delta_1}, \\
\Gamma(1,3,0,0) & = \jap{t}^{-1},   & \Gamma(1,3,\neq,\neq) & = \jap{t}^{-1} \jap{\frac{t}{\jap{\xi,l^\prime}}}^{1-\delta_1}, \\ 
\Gamma(1,3,0,\neq) & = \jap{t}^{-1} \jap{\frac{t}{\jap{\xi,l^\prime}}}^2, & \Gamma(1,3,\neq,0) & = \jap{t}^{-1} \jap{\frac{t}{\jap{\xi,l^\prime}}}^{-1-\delta_1}, \\
\Gamma(2,3,\neq,\neq) & = \jap{\frac{t}{\jap{\xi,l^\prime}}}, & \Gamma(2,3,0,\neq) & = \jap{\frac{t}{\jap{\xi,l^\prime}}}^2, \\
\Gamma(2,3,\neq,0) & = \jap{\frac{t}{\jap{\xi,l^\prime}}}^{-1}, & \Gamma(2,3,0,0) & = 1, \\ 
\Gamma(1,1,0,\neq)&  = \jap{\frac{t}{\jap{\xi,l^\prime}}}^{1+\delta_1}, & \Gamma(2,2,0,\neq) & = \jap{\frac{t}{\jap{\xi,l^\prime}}}, \\ 
\Gamma(3,3,0,\neq) & = \jap{\frac{t}{\jap{\xi,l^\prime}}}^2. 
\end{align*}
Then there exists a $c = c(s) \in (0,1)$ such that for all $t$ the following holds for $i \in \set{1,2,3}$ and $a = \neq$ if $k \neq 0$ (otherwise $a=0$) and $b = \neq$ if $k^\prime \neq 0$ (otherwise $b = 0$),  
\begin{align}
A^{i}_k(t,\eta,l) & \lesssim \Gamma(i,j,a,b) A^{j}_{k^\prime}(t,\xi,l^\prime) e^{c\lambda\abs{k - k^\prime,\eta-\xi,l-l^\prime}^s}, \label{ineq:ABasic}
\end{align}
where the implicit constant depends only on $s,\lambda,\sigma$ and $\kappa$. 
Analogous inequalities hold also with $A(t,\eta,l)$ using that $A(t,\eta,l) = \jap{\eta,l}^2 A_0^2(t,\eta,l)$. 
\end{lemma} 
\begin{proof} 
Let us simply explain the proof of \eqref{ineq:ABasic} when $i = j$ and $a = b = \neq$ as the others are the same except with the weight $\Gamma$. 
By \eqref{ineq:AFreqLoc} and \eqref{lem:scon},  there exists a $c^\prime = c^\prime(s,\theta) \in (0,1)$ such that 
\begin{align*}
A_k^i(t,\eta,l) & \lesssim \jap{k^\prime,\xi,l^\prime}^{\sigma} \frac{e^{\mu \abs{\xi}^{1/2}}}{w(t,\eta) w_L(t,k,\eta,l)} e^{\lambda \abs{k^\prime,\xi,l^\prime}^s} e^{c^\prime \lambda \abs{k-k^\prime,\eta-\xi,l-l^\prime}^s + c^\prime\mu\abs{\eta-\xi}^{s}}.  
\end{align*}
Then, by Lemma \ref{lem:wRat} (and using that $w_L$ is $O(1)$ by \eqref{ineq:unifN} and \eqref{def:wL}), 
\begin{align*}
A_k^i(t,\eta,l) & \lesssim A_{k^\prime}^i(t,\xi,l^\prime) e^{c^\prime \lambda \abs{k-k^\prime,\eta-\xi,l-l^\prime}^s + (c^\prime\mu+K)\abs{\eta-\xi}^{s}}. 
\end{align*}
Finally, \eqref{ineq:ABasic} follows by \eqref{ineq:IncExp} and taking $c^\prime < c < 1$. 
\end{proof}

The next lemma deals with the remainders in the paraproduct decompositions (see \eqref{def:parapp} below). 
\begin{lemma} \label{lem:Arem}
For all $K > 0$ there exists a $c = c(s,K) \in (0,1)$ such that if 
\begin{align} 
\frac{1}{K}\abs{k^\prime,\xi,l^\prime} \leq \abs{k - k^\prime,\eta-\xi,l-l^\prime} \leq K\abs{k^\prime,\xi,l^\prime},\label{ineq:FreqLocRemainder}
\end{align}
then 
\begin{subequations} \label{ineq:ARemainderBasic}
\begin{align} 
A^1_k(t,\eta,l) & \lesssim  \jap{t}^{-2-\delta_1} e^{c\lambda \abs{k^\prime,\xi,l^\prime}^s}e^{c\abs{k-k^\prime,\eta-\xi,l-l^\prime}^s} \\ 
A^2_k(t,\eta,l) & \lesssim  \jap{t}^{-1} e^{c \lambda \abs{k^\prime,\xi,l^\prime}^s}e^{c\abs{k-k^\prime,\eta-\xi,l-l^\prime}^s} \\
A^3_k(t,\eta,l) & \lesssim  \jap{t}^{-2} e^{c \lambda \abs{k^\prime,\xi,l^\prime}^s}e^{c\abs{k-k^\prime,\eta-\xi,l-l^\prime}^s}, \label{ineq:A3remainderBasic} 
\end{align}  
\end{subequations}
and if $k = k^\prime = 0$ then
\begin{align} 
A(t,\eta,l) & \lesssim e^{c \lambda \abs{\xi,l^\prime}^s}e^{c \lambda \abs{\eta-\xi,l-l^\prime}^s}. \label{ineq:AARemainderBasic}
\end{align}
All implicit constants depend only on $\kappa, \lambda, \sigma$ and $s$. 
\end{lemma} 
\begin{proof}
 All of the inequalities in \eqref{ineq:ARemainderBasic} and \eqref{ineq:AARemainderBasic} are basically the same, so let us just prove \eqref{ineq:A3remainderBasic}. 
By \eqref{ineq:FreqLocRemainder}, \eqref{lem:strivial}, and the triangle inequality for $\jap{\cdot}$, there holds for some $c^\prime = c^\prime(s,K) \in (0,1)$: 
\begin{align*} 
A^3_k(t,\eta,l) & \lesssim \frac{e^{\mu \abs{\eta}^{1/2}}}{w_L(t,k,\eta,l) w(t,\eta)} \jap{\frac{t}{\jap{\eta,l}}}^{-2} e^{c^\prime \lambda \abs{k^\prime,\xi,l^\prime}^s}e^{c^\prime \lambda \abs{k-k^\prime,\eta-\xi,l-l^\prime}^s} \jap{k^\prime,\xi,l^\prime}^{\sigma} \\ 
& \lesssim \frac{e^{\mu \abs{\eta}^{1/2}}}{w_L(t,k,\eta,l) w(t,\eta)} \jap{t}^{-2} e^{c^\prime \lambda \abs{k^\prime,\xi,l^\prime}^s}e^{c^\prime \lambda \abs{k-k^\prime,\eta-\xi,l-l^\prime}^s} \jap{k^\prime,\xi,l^\prime}^{\sigma+2}. 
\end{align*}
By \eqref{lem:totalGrowthw} (and that $w_L$ is $O(1)$) and applying \eqref{lem:strivial} again we have 
\begin{align*}
A^3_k(t,\eta,l) & \lesssim e^{2\mu \abs{\xi}^{1/2}} e^{2\mu \abs{\eta-\xi}^{1/2}} \jap{t}^{-2} e^{c^\prime \lambda \abs{k^\prime,\xi,l^\prime}^s}e^{c^\prime \lambda \abs{k-k^\prime,\eta-\xi,l-l^\prime}^s} \jap{k^\prime,\xi,l^\prime}^{\sigma+2}.  
\end{align*}
Then the result follows by \eqref{ineq:IncExp} and \eqref{ineq:SobExp}, choosing $c^\prime < c <1$. 
\end{proof} 

We also have similar inequalities regarding the CK multipliers. 

\begin{lemma} [Frequency ratios for $\partial_t w$ and $\partial_t w_L$] \label{lem:CKwFreqRat} 
For all $t \geq 1$ we have
\begin{subequations} 
\begin{align} 
\left(\sqrt{\frac{\partial_t w(t,\eta)}{w(t,\eta)}} + \frac{\abs{k,\eta,l}^{s/2}}{\jap{t}^{s}}\right) & \lesssim \left(\sqrt{\frac{\partial_t w(t,\xi)}{w(t,\xi)}} + \frac{\jap{k^\prime,\xi,l^\prime}^{s/2}}{\jap{t}^{s}}\right) \jap{k-k^\prime,\eta-\xi,l-l^\prime}^2 \label{ineq:dtwBasicBrack} \\ 
\left(\sqrt{\frac{\partial_t w(t,\eta)}{w(t,\eta)}} + \frac{\abs{k,\eta,l}^{s/2}}{\jap{t}^{s}}\right) & \lesssim \left(\sqrt{\frac{\partial_t w(t,\xi)}{w(t,\xi)}} + \frac{\abs{k^\prime,\xi,l^\prime}^{s/2} + \abs{k-k^\prime,\eta-\xi,l-l^\prime}^{s/2}}{\jap{t}^{s}}\right) \nonumber \\ & \quad\quad \times \jap{k-k^\prime,\eta-\xi,l-l^\prime}^2 \label{ineq:dtwBasicBrack2} \\ 
\sqrt{\frac{\partial_t w_L(t,k,\eta,l)}{w_L(t,k,\eta,l)}} & \lesssim \sqrt{\frac{\partial_t w_L(t,k,\xi,l^\prime)}{w_L(t,k,\xi,l^\prime)}} \jap{\eta-\xi,l-l^\prime}^{3/2}. \label{ineq:dtNBasic}
\end{align}
\end{subequations}
Further, if $\abs{k^\prime,\xi,l^\prime} \gtrsim 1$ then \eqref{ineq:dtwBasicBrack} implies
\begin{align}
\left(\sqrt{\frac{\partial_t w(t,\eta)}{w(t,\eta)}} + \frac{\abs{k,\eta,l}^{s/2}}{\jap{t}^{s}}\right) & \lesssim \left(\sqrt{\frac{\partial_t w(t,\xi)}{w(t,\xi)}} + \frac{\abs{k^\prime,\xi,l^\prime}^{s/2}}{\jap{t}^{s}}\right) \jap{k-k^\prime,\eta-\xi,l-l^\prime}^2.  \label{ineq:dtwBasic}
\end{align}
Moreover, both \eqref{ineq:dtwBasicBrack} and \eqref{ineq:dtwBasic} hold if we replace $\abs{k,\eta,l}$ and $\abs{k,\xi,l^\prime}$ by $\abs{\eta}$ and $\abs{\xi}$ (respectively). 
\end{lemma} 
\begin{proof} 
First, for \eqref{ineq:dtNBasic} simply note that by $\abs{\xi - kt} \leq \abs{\eta-\xi} + \abs{\eta-kt}$ and the fact that $k \neq 0$ (as otherwise there is nothing to prove),
\begin{align*}
\sqrt{\frac{\partial_t w_L(t,k,\eta,l)}{w_L(t,k,\eta,l)}} = \frac{\sqrt{\abs{k}\jap{l}}}{\abs{k} + \abs{l} + \abs{\eta-kt}} & \lesssim \frac{\sqrt{\abs{k}\jap{l^\prime}}}{\abs{k} + \abs{l^\prime} + \abs{\xi-kt}} \jap{\eta-\xi,l-l^\prime}^{1/2} \left(\frac{\abs{k} + \abs{l^\prime} + \abs{\xi-kt}}{\abs{k} + \abs{l} + \abs{\eta-kt}}\right) \\
& \lesssim \sqrt{\frac{\partial_t w_L(t,k,\xi,l^\prime)}{w_L(t,k,\xi,l^\prime)}} \jap{\eta-\xi,l-l^\prime}^2. 
\end{align*} 

Let us next consider \eqref{ineq:dtwBasicBrack} and  \eqref{ineq:dtwBasicBrack2}. 
First, if $\abs{\eta} \leq 1/2$ then it follows that 
\begin{align*}
\sqrt{\frac{\partial_t w(t,\eta)}{w(t,\eta)}} + \frac{\abs{k,\eta,l}^{s/2}}{\jap{t}^{s}} & = \frac{\abs{k,\eta,l}^{s/2}}{\jap{t}^{s}}  \lesssim \frac{\abs{k^\prime,\xi,l^\prime}^{s/2}}{\jap{t}^{s}} + \frac{\abs{k-k^\prime,\eta-\xi,l-l^\prime}^{s/2}}{\jap{t}^{s}}, 
\end{align*}
from which \eqref{ineq:dtwBasicBrack2} and \eqref{ineq:dtwBasicBrack} both follow. 

Next consider the case $\abs{\eta} > 1/2$. 
In this case we still have 
\begin{align*}
\frac{\abs{k,\eta,l}^{s/2}}{\jap{t}^{s}} &  \lesssim \frac{\abs{k^\prime,\xi,l^\prime}^{s/2}}{\jap{t}^{s}} + \frac{\abs{k-k^\prime,\eta-\xi,l-l^\prime}^{s/2}}{\jap{t}^{s}},   
\end{align*}
which is again consistent with both \eqref{ineq:dtwBasicBrack2} and \eqref{ineq:dtwBasicBrack}. 
To treat the term involving $\partial_t w$, first since $\partial_t w = 0$ for $t \geq 2\abs{\eta}$, we may assume that $\abs{\eta} \geq t/2 \gtrsim 1$. 
Consider the case that $\abs{\eta - \xi} \geq \frac{1}{2}\max(\abs{\eta},\abs{\xi})$. 
Then \eqref{ineq:dtwBasicBrack} follows immediately from 
\begin{align*}
\sqrt{\frac{\partial_t w(t,\eta)}{w(t,\eta)}} & \lesssim 1 \lesssim \frac{\abs{\eta}^s}{\jap{t}^s}  \lesssim \frac{\jap{k^\prime,\xi,l^\prime}^{s/2}}{\jap{t}^s} \jap{\eta-\xi}^{s/2}, 
\end{align*}
whereas \eqref{ineq:dtwBasicBrack2} follows from 
\begin{align*}
\sqrt{\frac{\partial_t w(t,\eta)}{w(t,\eta)}} & \lesssim 1 \lesssim \frac{\abs{\eta-\xi}^s}{\jap{t}^s}  \lesssim \frac{\abs{\eta-\xi}^{s/2}}{\jap{t}^s} \jap{\eta-\xi}^{s/2}.  
\end{align*}
Hence, it suffices to consider the case $\abs{\eta - \xi} \leq \frac{1}{2}\max(\abs{\eta},\abs{\xi})$, which implies that $\abs{\eta} \approx \abs{\xi}$. 
From here we can apply \eqref{ineq:partialtw_endpt}, and so this concludes the proof of both \eqref{ineq:dtwBasicBrack} and \eqref{ineq:dtwBasicBrack2}. 
\end{proof}

The next inequality is immediate from the definition of $D$ \eqref{def:D}, but useful for separating the critical times from the enhanced dissipation estimates. 
\begin{lemma} 
For all $p \geq 0$ and $(t,k,\eta,l)$ such that $t \geq 1$, there holds
\begin{subequations} 
\begin{align}
A^{\nu;i}_k(t,\eta,l) & \lesssim \jap{t}^{-p}\jap{k,\eta,l}^{\beta + 3\alpha + p}e^{\lambda\abs{k,\eta,l}^{s}} + A^{\nu;i}_k(t,\eta,l) \mathbf{1}_{t \geq 2\abs{\eta}} \label{ineq:AnuHiLowSep} \\ 
A^{\nu;i}_k(t,\eta,l) & \lesssim \jap{t}^{-p}\jap{k,\eta,l}^{\beta + 3\alpha + p}e^{\lambda\abs{k,\eta,l}^{s}} + \jap{t}^{-1}\left(\abs{k} + \abs{\eta-kt}\right)A^{\nu;i}_k(t,\eta,l) \mathbf{1}_{t \geq 2\abs{\eta}}.  \label{ineq:AnuHiLowSep2}
\end{align}
\end{subequations}
\end{lemma} 

The next lemma tells us how to treat ratios involving $\Delta_L$ and, when combined with the precision elliptic lemmas in Appendix \ref{sec:PEL}, forms the core of the technical tools at our disposal. 

\begin{lemma}[Frequency ratios for $\Delta_L$] \label{lem:MainFreqRat}
If $t \geq 1$, then for all $\eta,\xi,l,l^\prime,k^\prime$ and $k$, we have the following:
\begin{itemize}
\item approximate integration by parts: for all $k \neq 0$, 
\begin{align} 
\abs{\eta-kt} \lesssim \jap{\eta-\xi}\left(\abs{k} + \abs{\xi-kt}\right); \label{ineq:TriTriv}
\end{align}
\item for absorbing long-time losses: for all $k \neq 0$, 
\begin{align} 
\frac{1}{\abs{k,\eta-kt,l}} \jap{\frac{t}{\jap{\xi,l^\prime}}} & \lesssim \jap{\eta-\xi,l-l^\prime}; \label{ineq:ratlongtime}
\end{align}
\item for the linear stretching terms 
\begin{align}
\frac{\abs{k} \mathbf{1}_{t \leq 2\abs{\eta}}}{\abs{k} + \abs{l} + \abs{\eta-kt}} & \lesssim \kappa^{-1}\frac{\partial_t w(t,\eta)}{w(t,\eta)}; \label{ineq:CKwLS}
\end{align}
\item for \textbf{(SI)} terms: if $p \in \mathbb{R}$, 
\begin{align} 
\frac{\abs{k,\eta-kt,l} \abs{k}}{k^2 + (l^\prime)^2 + \abs{\xi-kt}^2} \jap{\frac{t}{\jap{\xi,l^\prime}}}^p & \lesssim \left(\left(\sqrt{\frac{\partial_t w(t,\eta)}{w(t,\eta)}} + \frac{\abs{k,\eta}^{s/2}}{\jap{t}^s} \right)\left(\sqrt{\frac{\partial_t w(t,\xi)}{w(t,\xi)}} + \frac{\abs{k,\xi}^{s/2}}{\jap{t}^s} \right) \right. \nonumber \\ & \left. \quad +  \frac{1}{\jap{t}}\min\left(1,\frac{\abs{k, \eta-kt,l}}{\jap{kt}}  \right) \jap{\frac{t}{\jap{\xi,l^\prime}}}^p \right) \jap{\eta-\xi,l-l^\prime}^{4}; \label{ineq:AiPartX}
\end{align}  
\item For \textbf{(3DE)} terms: if $p \in \mathbb{R}$ and $k^\prime,k \neq 0$, for $a \in \set{1,2}$, 
\begin{align} 
\frac{1}{ \abs{k^\prime, \xi-k^\prime t,l^\prime}^a} \jap{\frac{t}{\jap{\xi,l^\prime}}}^p & \lesssim \nonumber \\ & \hspace{-3cm} \left(\sqrt{\frac{\partial_t w(t,\eta)}{w(t,\eta)}} + \frac{\abs{k,\eta}^{s/2}}{\jap{t}^s} \right)\left(\sqrt{\frac{\partial_t w(t,\xi)}{w(t,\xi)}} + \frac{\abs{k^\prime, \xi}^{s/2}}{\jap{t}^s} \right) \jap{k-k^\prime, \eta-\xi,l-l^\prime}^3  \nonumber \\ & \hspace{-3cm} \quad\quad  + \frac{1}{\jap{t}^a}\jap{\frac{t}{\jap{\xi,l^\prime}}}^p \jap{k-k^\prime, \eta-\xi,l-l^\prime}^3.  \label{ineq:AikDelLNoD}
\end{align} 
\item For \textbf{(3DE)} terms in the nonlinear pressure and stretching: if $p \in \mathbb{R}$ and $k^\prime,k \neq 0$,  
\begin{subequations} \label{ineq:AikDelL2D}
\begin{align}
\frac{\abs{k,\eta-k t,l}\abs{k,\xi-k^\prime t,l^\prime}}{(k^\prime)^2 + (l^\prime)^2 + \abs{\xi-k^\prime t}^2} \jap{\frac{t}{\jap{\xi,l^\prime}}}^p & \lesssim \left(\jap{t} + \jap{\frac{t}{\jap{\xi,l^\prime}}}^p\right)\jap{k-k^\prime,\eta-\xi,l-l^\prime}^{2}  \\ 
\frac{\abs{k,\eta-k t,l}\abs{k^\prime,\xi-k^\prime t,l^\prime}}{(k^\prime)^2 + (l^\prime)^2 + \abs{\xi-k^\prime t}^2} \jap{\frac{t}{\jap{\xi,l^\prime}}}^p & \lesssim \left(\jap{t}\left(\sqrt{\frac{\partial_t w(t,\eta)}{w(t,\eta)}} + \frac{\abs{k,\eta}^{s/2}}{\jap{t}^s} \right)\left(\sqrt{\frac{\partial_t w(t,\xi)}{w(t,\xi)}} + \frac{\abs{k^\prime \xi}^{s/2}}{\jap{t}^s} \right) \right. \nonumber  \\  & \left. \quad\quad + \jap{\frac{t}{\jap{\xi,l^\prime}}}^p\right)\jap{k-k^\prime,\eta-\xi,l-l^\prime}^{2}. \label{ineq:AikDelL2D_CKw}  \\ 
\frac{\abs{l^\prime}\abs{k,\eta-kt,l}}{(k^\prime)^2 + (l^\prime)^2 + \abs{\xi-k^\prime t}^2} \jap{\frac{t}{\jap{\xi,l^\prime}}} & \lesssim \left(\jap{t}\left(\sqrt{\frac{\partial_t w(t,\eta)}{w(t,\eta)}} + \frac{\abs{k,\eta}^{s/2}}{\jap{t}^s} \right)\left(\sqrt{\frac{\partial_t w(t,\xi)}{w(t,\xi)}} + \frac{\abs{k^\prime, \xi}^{s/2}}{\jap{t}^s} \right) \right. \nonumber  \\  & \left. \quad\quad + 1\right) \jap{k-k^\prime,\eta-\xi,l-l^\prime}^{2}. \label{ineq:AikDelL2D_CKw2}  
\end{align} 
\end{subequations}
\item For triple derivative terms (these arise in the treatment of \textbf{(F)} terms): if $p \in \mathbb{R}$,   
\begin{subequations} \label{ineq:AdelLij} 
\begin{align} 
\frac{\abs{l}^3}{(k)^2 + (l^\prime)^2 + \abs{\xi-k t}^2} \jap{\frac{t}{\jap{\xi,l^\prime}}}^p & \lesssim \abs{l}\left(\jap{l-l^\prime}^2  + \frac{\abs{l}^2}{\jap{\xi,l^\prime,t}^2} \jap{\frac{t}{\jap{\xi,l^\prime}}}^p\right)  \label{ineq:AdeZZZ} \\ 
\frac{\abs{\eta}\abs{l}^2 + \abs{\eta}^2 \abs{l}}{k^2 + (l^\prime)^2 + \abs{\xi-k t}^2} \jap{\frac{t}{\jap{\xi,l^\prime}}}^p & \lesssim \left(\jap{t}^2\left(\sqrt{\frac{\partial_t w(t,\eta)}{w(t,\eta)}} + \frac{\abs{\eta}^{s/2}}{\jap{t}^s}\right)\left(\sqrt{\frac{\partial_t w(t,\xi)}{w(t,\xi)}} + \frac{\abs{\xi}^{s/2}}{\jap{t}^s}\right) \right. \nonumber \\ & \left. \quad + \abs{l}\left(1  + \frac{\abs{\eta}\abs{l} + \abs{\eta}^2}{\jap{\xi,t,l^\prime}^2} \jap{\frac{t}{\jap{\xi,l^\prime}}}^p\right)\right) \jap{k,\eta-\xi,l-l^\prime}^3  \label{ineq:AdeYZZ}  \\ 
\frac{\abs{\eta}^3}{k^2 + (l^\prime)^2 + \abs{\xi-k t}^2} \jap{\frac{t}{\jap{\xi,l^\prime}}}^p & \lesssim \left(\jap{t}^3\left(\sqrt{\frac{\partial_t w(t,\eta)}{w(t,\eta)}} + \frac{\abs{\eta}^{s/2}}{\jap{t}^s}\right)\left(\sqrt{\frac{\partial_t w(t,\xi)}{w(t,\xi)}} + \frac{\abs{\xi}^{s/2}}{\jap{t}^s}\right) \right. \nonumber \\ & \left. \quad  + \min(\abs{\eta},\jap{\xi-kt})\left(1  + \frac{\abs{\eta}^2}{\jap{\xi,t,l^\prime}^2} \jap{\frac{t}{\jap{\xi,l^\prime}}}^p\right) \right) \jap{k,\eta-\xi,l-l^\prime}^3. \label{ineq:AdeYYY}
\end{align} 
\end{subequations} 
\end{itemize} 
\end{lemma}
\begin{remark}
Although many of the above inequalities appear very similar, none exactly implies another. 
\end{remark} 
\begin{remark} 
Note that \eqref{ineq:AdelLij} implies 
\begin{subequations}
\begin{align}
\frac{\abs{\eta,l}\jap{\eta,l}^2}{k^2 + (l^\prime)^2 + \abs{\xi-k t}^2} \jap{\frac{t}{\jap{\xi,l^\prime}}}^p & \lesssim \left(\jap{t}^3\left(\sqrt{\frac{\partial_t w(t,\eta)}{w(t,\eta)}} + \frac{\abs{\eta}^{s/2}}{\jap{t}^s}\right)\left(\sqrt{\frac{\partial_t w(t,\xi)}{w(t,\xi)}} + \frac{\abs{\xi}^{s/2}}{\jap{t}^s}\right) \right. \nonumber \\ & \left. \quad  + \abs{\eta,l}\left(1  + \frac{\jap{\eta,l}^2}{\jap{\xi,t,l^\prime}^2} \jap{\frac{t}{\jap{\xi,l^\prime}}}^p\right)\right)\jap{k,\eta-\xi,l-l^\prime}^3  \label{ineq:AdeGen} \\ 
\frac{\abs{l}\jap{\eta,l}^2}{k^2 + (l^\prime)^2 + \abs{\xi-k t}^2} \jap{\frac{t}{\jap{\xi,l^\prime}}}^p & \lesssim \left(\jap{t}^2\left(\sqrt{\frac{\partial_t w(t,\eta)}{w(t,\eta)}} + \frac{\abs{\eta}^{s/2}}{\jap{t}^s}\right)\left(\sqrt{\frac{\partial_t w(t,\xi)}{w(t,\xi)}} + \frac{\abs{\xi}^{s/2}}{\jap{t}^s}\right) \right. \nonumber \\ & \left. \quad  + \abs{l}\left(1  + \frac{\jap{\eta,l}^2}{\jap{\xi,t,l^\prime}^2} \jap{\frac{t}{\jap{\xi,l^\prime}}}^p\right)\right)\jap{k,\eta-\xi,l-l^\prime}^3. \label{ineq:AdeGen2}
\end{align}
\end{subequations} 
\end{remark} 
\begin{proof} 
Most of these inequalities use a similar set of ideas based on time-frequency decompositions. 
First, \eqref{ineq:TriTriv} is immediate from the triangle inequality. 
Inequality \eqref{ineq:ratlongtime} is similarly straightforward. 

Next, consider \eqref{ineq:CKwLS}. If $t \in \I_{k,\eta}$, then the result is immediate from Lemma \ref{dtw}. If $t \lesssim \sqrt{\abs{\eta}}$, then it is similarly straightforward from Lemma \ref{dtw}. If $t \in \I_{r,\eta}$, and $r \neq k$ then 
\begin{align*}
\frac{\abs{k} \mathbf{1}_{t \leq 2\abs{\eta}}}{\abs{k} + \abs{l} + \abs{\eta-kt}} & \lesssim \frac{\abs{k}}{t\abs{k-r}} \lesssim \frac{\abs{r}}{t}, 
\end{align*}
from which \eqref{ineq:CKwLS} now follows from Lemma \ref{dtw}. 

Consider \eqref{ineq:AiPartX}.  
First, if $t \leq 2 |\xi|$, $\langle \frac{t}{\langle \xi,l' \rangle} \rangle \sim 1$ and we have by the same argument used to deduce \eqref{ineq:CKwLS}, 
\begin{align*}
\frac{\abs{k} \abs{k,\eta - kt,l}}{k^2 + (l')^2 + |\xi-kt|^2} & \lesssim \frac{\langle \xi - \eta,l-l' \rangle}{1 + |t - \frac{\xi}{k}|} \lesssim \kappa^{-1} \frac{\partial_t w(t,\xi)}{w(t,\xi)}\langle \xi - \eta,l-l' \rangle. 
\end{align*}  
If $t \geq 2 |\xi|$ then $|\xi - kt| \approx \abs{kt}$, so that it is straightforward to obtain
\begin{align*}
\frac{\abs{k} \abs{k,\eta - kt,l}}{k^2 + (l')^2 + |\xi-kt|^2} \jap{\frac{t}{\jap{\xi,l'}} }^p & \lesssim \frac{\abs{k} \abs{k,\eta - kt,l}}{k^2 + (l')^2 + |kt|^2} \jap{ \xi - \eta,l-l'}  \jap{ \frac{t}{\jap{\xi,l'}} }^p, 
\end{align*}
from which \eqref{ineq:AiPartX} follows immediately. Inequality \eqref{ineq:AikDelLNoD} is similar.  

Let us next consider \eqref{ineq:AikDelL2D}; specifically we will just discuss \eqref{ineq:AikDelL2D_CKw}, as the others are similar or easier. 
If $\xi > 2 k^\prime t$ or $\xi < \frac{1}{2} k^\prime t$ or $\abs{l^\prime} \gtrsim \abs{\xi}$, then we immediately have 
\begin{align*}
\frac{\abs{k,\eta-k t,l} \abs{k^\prime,\xi-k^\prime t,l^\prime}}{(k^\prime)^2 + (l^\prime)^2 + \abs{\xi-k^\prime t}^2}\jap{\frac{t}{\jap{\xi,l^\prime}}}^p & \lesssim \frac{\abs{k,kt,\eta,l}}{\abs{k^\prime,l^\prime,\xi,k^\prime t}}\jap{\frac{t}{\jap{\xi,l^\prime}}}^p  \lesssim \jap{k-k^\prime,l-l^\prime} \jap{\frac{t}{\jap{\xi,l^\prime}}}^p, 
\end{align*} 
which is consistent with \eqref{ineq:AikDelL2D_CKw}. Hence, we only need to consider the case $\xi \approx k^\prime t$ and $\abs{l^\prime} \lesssim \abs{\xi}$.
If $t \lesssim \sqrt{\abs{\xi}}$ then by $\abs{\eta-kt} \lesssim \abs{\xi-k^\prime t} \jap{t} \jap{\eta-\xi,k-k^\prime}$, we have 
\begin{align*}
\frac{\abs{k,\eta-k t,l}}{\abs{k^\prime,\xi-k^\prime t,l^\prime}} & \lesssim \jap{t}\jap{\eta-\xi,l-l^\prime}, 
\end{align*}
which is consistent with \eqref{ineq:AikDelL2D_CKw} by Lemma \ref{lem:dtw}.   
If $t \in \I_{k^\prime,\eta}\cap \I_{k^\prime,\xi}$, then it follows that 
\begin{align*}
\frac{\abs{k,\eta-k t,l} \abs{k^\prime,\xi-k^\prime t,l^\prime}}{(k^\prime)^2 + (l^\prime)^2 + \abs{\xi-k^\prime t}^2} & \lesssim \frac{\abs{k,t\abs{k-k^\prime},l}}{\abs{k^\prime,l^\prime,\xi-k^\prime t}}\jap{\eta-\xi} \lesssim \frac{t }{1 + \abs{t - \frac{\xi}{k^\prime}}} \jap{k-k^\prime,\eta-\xi,l-l^\prime}^2, 
\end{align*}
from which the result follows now from Lemma \ref{lem:dtw}.
If $t \in \I_{k^\prime,\xi}$,  but $t \not\in \I_{k^\prime,\eta}$, then by Lemma \ref{lem:wellsep}, it follows that 
\begin{align*}
\frac{\abs{k,\eta-k t,l} \abs{k^\prime,\xi-k^\prime t,l^\prime}}{(k^\prime)^2 + (l^\prime)^2 + \abs{\xi-k^\prime t}^2} & \lesssim \frac{\abs{k,t\abs{k-k^\prime},l}}{\jap{k^\prime,t,l^\prime}}  \lesssim  \jap{k-k^\prime,l-l^\prime}^2, 
\end{align*}
which is consistent with \eqref{ineq:AikDelL2D_CKw}.  
Finally, if $t \in \I_{r,\xi}$ with $r \neq k^\prime$, then 
\begin{align*}
\frac{\abs{k,\eta-k t,l} \abs{k^\prime,\xi-k^\prime t,l^\prime}}{(k^\prime)^2 + (l^\prime)^2 + \abs{\xi-k^\prime t}^2} & \lesssim \frac{\abs{k,t\jap{r-k},l}}{\jap{k^\prime,l^\prime,t\abs{k^\prime - r}}}  \lesssim \jap{k-k^\prime,l-l^\prime}^2, 
\end{align*}
which is also consistent with \eqref{ineq:AikDelL2D_CKw}.  

Let us next turn to the inequalities in \eqref{ineq:AdelLij}. 
These follow the same general pattern, so let us focus on one which is especially not obvious, \eqref{ineq:AdeYZZ}, and omit the others for brevity. 
First, if $t > 2\frac{\xi}{k}$, or $t < \frac{\xi}{2k}$, then $|\xi - kt| \gtrsim kt$ which implies
\begin{align*} 
\frac{\abs{\eta}l^2}{k^2 + (l')^2 + \abs{\xi-kt}^2} \jap{\frac{t}{\jap{\xi,l'}}}^p \lesssim \frac{\abs{\eta}l^2}{k^2 + (l')^2 + \abs{\xi}^2 + \abs{kt}^2} \jap{\frac{t}{\jap{\xi,l'}} }^p, 
\end{align*}
which is consistent with \eqref{ineq:AdeYZZ}. 
If $t \approx \frac{\xi}{k}$  and $\abs{l} < \abs{\xi}$,
\begin{align*}
\frac{\abs{\eta}l^2}{k^2 + (l')^2 + \abs{\xi-kt}^2} \jap{\frac{t}{\jap{\xi,l'}}}^p & \lesssim \frac{\abs{kt} l^2}{k^2 + (l')^2 + \abs{\xi- kt}^2}\jap{\eta-\xi} \\ 
& \lesssim \frac{t^2}{\abs{k,\xi-kt,l}} \jap{k,\eta-\xi, l - l'}^4,
\end{align*}
which is consistent with \eqref{ineq:AdeYZZ} by Lemmas \ref{lem:dtw} and Lemma \ref{lem:CKwFreqRat}. 
Finally, if $t \approx \frac{\xi}{k}$  and $\abs{l} > \abs{\xi}$, we see immediately that
$$
\frac{\abs{\eta}l^2}{k^2 + (l')^2 + \abs{\xi-kt}^2} \jap{\frac{t}{\jap{\xi,l'}}}^p \lesssim \abs{l} \jap{\eta-\xi, l-l'}^2, 
$$
which is consistent with \eqref{ineq:AdeYZZ}. 
\end{proof} 

\subsection{Paraproducts and related notations} \label{sec:paranote}
In this section we discuss the type of paraproducts decompositions we will be using (introduced by Bony \cite{Bony81}).  
Due to the vast number of terms (many cubic order or higher due to the coordinate system) 
we make use of several shorthand notations to reduce the length of the technical details. 

By convention, we will use the homogeneous variant of the paraproduct and utilize the following short-hand to suppress the appearance of Littlewood-Paley projections:
\begin{align} 
fg & = f_{Hi} g_{Lo} + f_{Lo} g_{Hi} + (fg)_{\mathcal{R}} \nonumber \\ 
& = \sum_{M \in 2^\Integers} f_{M} g_{<M/8} + \sum_{M \in 2^{\Integers}} f_{<M/8} g_{M} + \sum_{M \in 2^{\Integers}} \sum_{M/8 \leq M^\prime \leq 8M} f_{M} g_{M^\prime}. \label{def:parapp}
\end{align}  
Next we explain how to go directly from the shorthand to $L^2$ energy estimates.
\begin{lemma}[Paraproducts for quadratic nonlinearities] \label{gevreyparaproductlemma}
Let $s\in[0,1)$, $\mu \geq 0$, $p \geq 0$.  Then, there exists a $c = c(s) \in (0,1)$ such that the following holds,   
\begin{subequations} \label{ineq:paraquad} 
\begin{align}
\norm{f_{Hi} g_{Lo}}_{\G^{\mu,p}} & \lesssim \norm{f}_{\G^{\mu,p}}\norm{g}_{\G^{c\mu,3/2+}}  \label{ineq:quadHL} \\ 
\norm{(fg)_{\mathcal{R}}}_{\G^{\mu,p}} & \lesssim \norm{f}_{\G^{c\mu,p}} \norm{g}_{\G^{c\mu,3/2+}} \label{ineq:quadR} \\ 
\int e^{\mu\abs{\grad}^s}\jap{\grad}^{p} h \, e^{\mu\abs{\grad}^s}\jap{\grad}^{p} \left(f_{Hi} g_{Lo}\right) dV & \lesssim \norm{h}_{\G^{\mu,p}}\norm{f}_{\G^{\mu,p}}\norm{g}_{\G^{c\mu,3/2+}}.   \label{ineq:triQuadHL}
\end{align}
\end{subequations}  
\end{lemma} 
\begin{remark} 
In most places in the proof, $\mu = 0$ as normally the multipliers $A^i$ or $A^{\nu;i}$ are playing the role of the norm.
\end{remark} 
\begin{proof}
Let us first prove \eqref{ineq:quadHL}.  
From  almost orthogonality \eqref{ineq:GeneralOrtho},
\begin{align*} 
\norm{f_{Hi} g_{Lo}}_{\G^{\mu,p}}^2 & \approx \sum_{N \in 2^{\Integer}} \norm{f_N g_{<N/8}}_{\G^{\mu,p}}^2 \\ 
& \approx  \sum_{N \in 2^{\Integers}} \sum_{k,l} \int \abs{ \sum_{k^\prime,l^\prime} \int e^{\mu\abs{k,\eta,l}^s} \jap{k,\eta,l}^{p} \hat{f}_{k^\prime}(\xi,l^\prime)_{N} \hat{g}_{k-k^\prime}(\eta-\xi,l-l^\prime)_{<N/8} d\xi }^2 d\eta. 
\end{align*}
By \eqref{lem:scon} and the frequency localizations due to the Littlewood-Paley projections (see Appendix \ref{apx:Gev}), there is some $c = c(s) \in (0,1)$ such that 
\begin{align*} 
\norm{f_{Hi} g_{Lo}}_{\G^{\mu,p}}^2 & \lesssim  \sum_{N \in 2^{\Integers}} \sum_{k,l} \int_\eta \abs{\sum_{k^\prime,l^\prime} \int_{\xi} e^{\mu\abs{k^\prime,\xi,l^\prime}^s} \jap{k^\prime,\xi,l^\prime}^{p} \abs{\hat{f}_{k^\prime}(\xi,l^\prime)_{N}}  \right. \\ & \left. \quad\quad \times e^{c\mu\abs{k-k^\prime,\eta-\xi,l-l^\prime}^s}\abs{\hat{g}_{k-k^\prime}(\eta-\xi,l-l^\prime)_{<N/8}}  d\xi }^2 d\eta 
\end{align*}
By \eqref{ineq:L2L1} and almost orthogonality again,  
\begin{align*} 
\norm{f_{Hi} g_{Lo}}_{\G^{\mu,p}}^2 & \lesssim \sum_{N \in 2^{\Integers}} \norm{f_N}^2_{\G^{\mu,p}} \norm{g_{<N/8}}^2_{\G^{c\mu,3/2+}} \lesssim \norm{f}^2_{\G^{\mu,p}}\norm{g}^2_{\G^{c\mu,3/2+}}, 
\end{align*} 
which completes \eqref{ineq:quadHL}. 

Next, for \eqref{ineq:quadR}, by the triangle inequality, the frequency localizations due to the Littlewood-Paley projections (see Appendix \ref{apx:Gev}), and \eqref{lem:strivial}, there holds for some $c = c(s) \in (0,1)$ (not necessarily the same as above), 
\begin{align*} 
\norm{(fg)_{\mathcal{R}}}_{\G^{\mu,p}} & = \norm{\sum_{N \in 2^{\Integers}} f_{N} g_{\sim N}}_{\G^{\mu,p}} \\ & \leq \sum_{N \in 2^{\Integers}} \norm{f_{N} g_{\sim N}}_{\G^{\mu,p}} \\  
& \lesssim \left(\sum_{N\in 2^{\Integers} : N\leq 1} + \sum_{N\in 2^{\Integers} : N > 1}\right) \left( \sum_{k^\prime,l^\prime} \int_\eta \left( \int_\xi e^{c\mu\abs{k^\prime,\xi,l^\prime}^s}\jap{k^\prime,\xi,l^\prime}^{p} \abs{ \hat{f}_{k^\prime}(\xi,l^\prime)_N} \right. \right. \\ & \quad\quad \left. \left. \times e^{c\mu\abs{k-k^\prime,\eta-\xi,l-l^\prime}^s} \abs{\hat{g}_{k-k^\prime}(\eta-\xi,l-l^\prime)_{\sim N}} d\xi \right)^2 d\eta  \right)^{1/2} \\ 
& = L + H. 
\end{align*} 
For $H$ we have by \eqref{ineq:L2L1}, 
\begin{align*} 
H  & \lesssim_{\delta} \sum_{1 < N \in 2^{\Integers}} N^{-\delta} \norm{f_N}_{\G^{c\mu,p}} \norm{g_{\sim N}}_{\G^{c\mu,3/2+2\delta}} \lesssim \norm{f}_{\G^{c\mu,p}} \norm{g}_{\G^{c\mu,3/2+2\delta}}, 
\end{align*} 
which suffices for \eqref{ineq:quadR}. 
For $L$ we have by Bernstein's inequalities, 
\begin{align*} 
L  & \lesssim \sum_{N \leq 1 } \norm{f_N g_{\sim N}}_{2}  \lesssim \sum_{N \leq 1}  \norm{f_N}_\infty \norm{g_{\sim N}}_{2}  \lesssim \sum_{N \leq 1}  N^{3/2} \norm{f_N}_2 \norm{g_{\sim N}}_{2} \lesssim \norm{f}_{2} \norm{g}_{2}, 
\end{align*} 
which completes the proof of \eqref{ineq:quadR}. 

Finally, \eqref{ineq:triQuadHL} follows from \eqref{ineq:quadHL} and Cauchy-Schwarz.  
\end{proof}  

Due to the coordinate system, many of the nonlinear terms we are working with are higher order (up to quintic). 
For dealing with cubic nonlinear terms, we have 
\begin{align} 
fgh & = \sum_{N \in 2^{\Integers}} f_N g_{<N/8}h_{<N/8} + g_{N} f_{<N/8} h_{<N/8} + f_{< N/8} g_{<N/8} h_N  \nonumber \\ 
& \quad + \left( \sum f_N g_{<N/8} h_{\sim N} + \sum f_N g_{\sim N} h_{<N/8} + \sum f_{\sim N} g_{\sim N} h_{\sim N} \nonumber \right. \\ 
& \quad +\left. \sum g_N f_{<N/8} h_{\sim N} + \sum g_N f_{\sim N} h_{<N/8} \nonumber \right. \\ 
& \quad + \left. \sum h_N g_{<N/8} f_{\sim N} + \sum h_N g_{\sim N} f_{<N/8}\right)  \nonumber \\
& := f_{Hi}(gh)_{Lo} + g_{Hi}(fh)_{Lo}  + h_{Hi}(gf)_{Lo} + (fgh)_{\mathcal{R}}. \label{def:cubic}
 \end{align} 
Note the short-hand $(gh)_{Lo} = g_{Lo} h_{Lo}$. 
We may iterate this idea and derive decompositions for quartic and quintic nonlinear terms as well (now applying the above short-hand). 
For quartic terms, 
\begin{align} 
fghk & = \sum_N f_N g_{<N/8} h_{<N/8} k_{<N/8} + \sum_N f_{<N/8} g_{N} h_{<N/8} k_{<N/8} \nonumber \\ & \quad + \sum_N f_{<N/8} g_{<N/8} h_{N} k_{<N/8} + \sum_N f_{<N/8} g_{<N/8} h_{<N/8} k_{N} + (fghk)_{\mathcal{R}} \nonumber \\ 
& := f_{Hi}(ghk)_{Lo} + g_{Hi}(fgk)_{Lo}  + h_{Hi}(gfk)_{Lo} + k_{Hi}(fgh)_{Lo} + (fghk)_{\mathcal{R}}, \label{def:quartic}    
\end{align}  
where the remainder term $(fghk)_{\mathcal{R}}$, includes all of the frequency contributions not included in the leading order terms. 
We make an analogous decomposition also for quintic terms.  
We have the equivalents of \eqref{ineq:quadHL}, \eqref{ineq:quadR}, and \eqref{ineq:triQuadHL}. 
\begin{lemma}[Paraproducts for higher order nonlinear terms] \label{lem:ParaHighOrder}
For all $\mu \geq 0$ and $p \geq 0$, there is some $c = c(s) \in (0,1)$ such that 
\begin{subequations}  \label{ineq:quin}
\begin{align} 
\norm{g_{Hi} (fhkj)_{Lo}}_{\G^{\mu,p}} & \lesssim_{p} \norm{g}_{\G^{\mu,p}}\norm{f}_{\G^{c\mu,3/2+}} \nonumber \\ & \quad\quad \times \norm{h}_{\G^{c\mu,3/2+}} \norm{k}_{\G^{c\mu,3/2+}} \norm{j}_{\G^{c\mu,3/2+}} \label{ineq:quinHL} \\ 
\norm{(fghkj)_{\mathcal{R}}}_{\G^{\mu,p}} & \lesssim_{p} \norm{g}_{\G^{c\mu,3/2+}}\norm{f}_{\G^{c\mu,3/2+}} \norm{h}_{\G^{c\mu,3/2+}} \nonumber \\ & \quad\quad \times \norm{k}_{\G^{c\mu,3/2+}} \norm{j}_{\G^{c\mu,3/2+}} \label{ineq:quinR} \\  
\int e^{\mu \abs{\grad}^s}\jap{\grad}^p q e^{\mu\abs{\grad}^s} \jap{\grad}^p (g_{Hi}(fhkj)_{Lo}) dV & \lesssim_{p} \norm{q}_{\G^{\mu,p}}\norm{g}_{\G^{\mu,p}} \norm{f}_{\G^{c\mu,3/2+}} \nonumber \\ & \quad\quad \times \norm{h}_{\G^{c\mu,3/2+}} \norm{k}_{\G^{c\mu,3/2+}} \norm{j}_{\G^{c\mu,3/2+}}. \label{ineq:triQuinHL}
\end{align}
\end{subequations}
Analogous estimates hold also for the cubic and quartic decompositions in \eqref{def:cubic} and \eqref{def:quartic}. 
\end{lemma} 

One final shorthand we would like to introduce involves the inner products that appear in the energy estimates below. 
Consider, for example, a typical Gevrey energy estimate involving three quantities $f,g,h$, where generally $h$ will be a product of several low frequency terms: 
\begin{align*} 
\int e^{\lambda\abs{\grad}^s}f  e^{\lambda\abs{\grad}^s}\left(g_{Hi} h_{Lo}\right) dV & = \sum_{k,l,k^\prime,l^\prime} \int_{\eta,\xi} e^{\lambda\abs{k,\eta,l}^s}\overline{\hat{f}}_k(\eta,l)  e^{\lambda\abs{k,\eta,l}^s} \hat{g}_{k^\prime}(\xi,l^\prime)_{Hi} \hat{h}_{k-k^\prime}(\eta-\xi,l-l^\prime)_{Lo} d\eta d\xi. 
\end{align*} 
By the frequency localizations of the paraproduct and \eqref{lem:scon}, for some $c = c(s) \in (0,1)$ we have (by \eqref{ineq:triQuadHL}), 
\begin{align*}
\int e^{\lambda\abs{\grad}^s}f  e^{\lambda\abs{\grad}^s}\left(g_{Hi} h_{Lo}\right) dV & \lesssim  \sum_{k,l,k^\prime,l^\prime} \int_{\eta,\xi} e^{\lambda\abs{k,\eta,l}^s}\abs{\hat{f}_k(\eta,l)}  e^{\lambda\abs{k^\prime,\xi,l^\prime}^s} \abs{\hat{g}_{k^\prime}(\xi,l^\prime)_{Hi}} \\  & \quad\quad \times e^{c\lambda\abs{k-k^\prime,\eta-\xi,l-l^\prime}^s}\abs{\hat{h}_{k-k^\prime}(\eta-\xi,l-l^\prime)_{Lo}} d\eta d\xi \\ 
& \lesssim \norm{f}_{\G^{\lambda}}\norm{g}_{\G^{\lambda}} \norm{h}_{\G^{c\lambda,3/2+}}.
\end{align*}  
The low frequency factors will generally all be put in the norm $\G^{c\lambda,3/2+}$ and hence it makes sense 
to use a short-hand for the low-frequency factor as $\norm{h}_{\G^{c\lambda,3/2+}}Low(k-k^\prime,\eta-\xi,l-l^\prime)$ where the function $Low$ is taken as an $O(1)$ function in $\G^{\lambda,3/2+}$ (and which may change line-to-line like the implicit constants). 
For example,
\begin{align} 
\int e^{\lambda\abs{\grad}^s}f  e^{\lambda\abs{\grad}^s}\left(g_{Hi} h_{Lo}\right) dV & := \norm{h}_{\G^{\lambda,3/2+}}\sum_{k,l,k^\prime,l^\prime} \int_{\eta,\xi} e^{\lambda\abs{k,\eta,l}^s}\overline{\hat{f}}_k(\eta,l)  e^{\lambda\abs{k,\eta,l}^s} \hat{g}_{k^\prime}(\xi,l^\prime)_{Hi} \nonumber \\ & \quad\quad \times Low(k-k^\prime,\eta-\xi,l-l^\prime) d\eta d\xi \nonumber \\ 
 & \lesssim \norm{h}_{\G^{\lambda,3/2+}}\sum_{k,l,k^\prime,l^\prime} \int_{\eta,\xi} e^{\lambda\abs{k,\eta,l}^s}\abs{\hat{f}_k(\eta,l)}  e^{\lambda\abs{k^\prime,\xi,l^\prime}^s} \abs{\hat{g}_{k^\prime}(\xi,l^\prime)_{Hi}}\nonumber \\ & \quad\quad \times Low(k-k^\prime,\eta-\xi,l-l^\prime) d\eta d\xi \nonumber \\
& \lesssim \norm{f}_{\G^{\lambda}}\norm{g}_{\G^{\lambda}} \norm{h}_{\G^{c\lambda,3/2+}}. \label{def:Low}
\end{align} 
The utility of this short-hand will quickly become clear in the course of the proof.

One final shorthand we would like to introduce involves the inner products that appear naturally in energy estimates. Consider for example, a typical energy estimate involving three quantities $f,g,h$, where generally $h$ will be a product of several low frequency terms: 
\begin{align*} 
\int e^{\lambda\abs{\grad}^s}f  e^{\lambda\abs{\grad}^s}\left(g_{Hi} h_{Lo}\right) dV & = \frac{1}{(2\pi)^{3/2}}\sum_{k,l,k^\prime,l^\prime} \int_{\eta,\xi} e^{\lambda\abs{k,\eta,l}^s}\overline{\hat{f}}_k(\eta,l)  e^{\lambda\abs{k,\eta,l}^s} \hat{g}_{k^\prime}(\xi,l^\prime)_{Hi} \hat{h}_{k-k^\prime}(\eta-\xi,l-l^\prime)_{Lo} d\eta d\xi. 
\end{align*} 
By the frequency localizations inherent in the shorthand and \eqref{lem:scon}, for some $c = c(s) \in (0,1)$ we have 
\begin{align*}
\int e^{\lambda\abs{\grad}^s}f  e^{\lambda\abs{\grad}^s}\left(g_{Hi} h_{Lo}\right) dV & \lesssim  \sum_{k,l,k^\prime,l^\prime} \int_{\eta,\xi} e^{\lambda\abs{k,\eta,l}^s}\abs{\hat{f}_k(\eta,l)}  e^{\lambda\abs{k^\prime,\xi,l^\prime}^s} \abs{\hat{g}_{k^\prime}(\xi,l^\prime)_{Hi}} \\  & \quad\quad \times e^{c\lambda\abs{k-k^\prime,\eta-\xi,l-l^\prime}^s}\abs{\hat{h}_{k-k^\prime}(\eta-\xi,l-l^\prime)_{Lo}} d\eta d\xi.
\end{align*}   
However, this can get quite tedious, especially when multiple computations like this are being done. Instead we refer to all ``low frequency'' garbage by an $O(1)$ function $Low(k-k^\prime,\eta-\xi,l-l^\prime)$ which is taken to absorb any additional derivative losses that arise, that is for example,
\begin{align} 
\int e^{\lambda\abs{\grad}^s}f  e^{\lambda\abs{\grad}^s}\left(g_{Hi} h_{Lo}\right) dV & := \norm{h}_{\G^{\lambda,3/2+}}\sum_{k,l,k^\prime,l^\prime} \int_{\eta,\xi} e^{\lambda\abs{k,\eta,l}^s}\overline{\hat{f}}_k(\eta,l)  e^{\lambda\abs{k,\eta,l}^s} \hat{g}_{k^\prime}(\xi,l^\prime)_{Hi} \nonumber \\ & \quad\quad \times Low(k-k^\prime,\eta-\xi,l-l^\prime) d\eta d\xi \nonumber \\ 
 & \lesssim \norm{h}_{\G^{\lambda,3/2+}}\sum_{k,l,k^\prime,l^\prime} \int_{\eta,\xi} e^{\lambda\abs{k,\eta,l}^s}\abs{\hat{f}_k(\eta,l)}  e^{\lambda\abs{k^\prime,\xi,l^\prime}^s} \abs{\hat{g}_{k^\prime}(\xi,l^\prime)_{Hi}}\nonumber \\ & \quad\quad \times Low(k-k^\prime,\eta-\xi,l-l^\prime) d\eta d\xi, \label{def:Low}
\end{align} 
the advantage being that it shortens the notation when $h$ is a complicated expression or when there are multiple sources of derivative loss on $h$ involved.
The utility of this short-hand will quickly become clear in the course of the proof. 

\subsection{Product lemmas and a few immediate consequences} 
The first product lemma is an immediate consequence of Lemma \ref{gevreyparaproductlemma}. 
\begin{lemma}[Gevrey Product lemma] \label{lem:GevProdAlg}
For all $s \in (0,1)$, $\mu > 0$, and $p \geq 0$, there exists $c = c(s) \in (0,1)$ such that the following holds for all $f,g \in \mathcal{G}^{\mu,p}$:
\begin{subequations}
\begin{align} 
\norm{fg}_{\G^{\mu,p}} & \lesssim_{p} \norm{f}_{\G^{c\mu,3/2+}} \norm{g}_{\G^{\mu,p}} + \norm{g}_{\G^{c\mu,3/2+}} \norm{f}_{\G^{\mu,p}}, \label{ineq:GProduct}
\end{align}
\end{subequations}
in particular, $\mathcal{G}^{\mu,p}$ is an algebra for all $p \geq 0$ by \eqref{ineq:SobExp}:  
\begin{align} 
\norm{f g}_{\G^{\mu,\sigma}} & \lesssim_{p,\mu} \norm{f}_{\G^{\mu,p}} \norm{g}_{\G^{\mu,p}}. \label{ineq:GAlg} 
\end{align} 
The constant in \eqref{ineq:GAlg} can be taken independent of $\mu$ for $p > 3/2$. 
\end{lemma}

Next we state the following product lemma for $A$ and related $CK$ terms. 
\begin{lemma}[Product lemma for $A$ and $A^i$] \label{lem:AAiProd}
Let 
 $p \geq 0$ and $r \geq -\sigma$. Then there exists a $c = c(s) \in (0,1)$ such that for all $f,g$
\begin{subequations} \label{ineq:Aprod}
\begin{align} 
\norm{\abs{\grad}^p \jap{\grad}^{r}A^i(fg)}_2 & \lesssim \norm{f}_{\G^{c\lambda,3/2+}}\norm{\abs{\grad}^p \jap{\grad}^{r} A^ig}_2 \nonumber \\ & \quad + \norm{g}_{\G^{c\lambda,3/2+}}\norm{\abs{\grad}^p \jap{\grad}^{r} A^if}_2 \label{ineq:Aprodi} \\ 
\norm{\left(\sqrt{\frac{\partial_t w}{w}} + \frac{\abs{\grad}^{s/2}}{\jap{t}^s} \right)A^i(fg)}_2 & \lesssim \norm{f}_{\G^{c\lambda,3/2+}}\norm{\left(\sqrt{\frac{\partial_t w}{w}} + \frac{\abs{\grad}^{s/2}}{\jap{t}^s} \right)A^i g}_2 \nonumber \\ & \quad + \norm{g}_{\G^{c\lambda,3/2+}}\norm{\left(\sqrt{\frac{\partial_t w}{w}} + \frac{\abs{\grad}^{s/2}}{\jap{t}^s} \right)A^if}_2. \label{ineq:ACKwProd}
\end{align}
\end{subequations}
Further, if we assume $f$ and $g$ are independent of $X$ then the above also holds with $A^i$ replaced by $A$.    
\end{lemma} 
\begin{proof}
The proof of \eqref{ineq:Aprodi} are immediate from Lemmas \ref{lem:ABasic}, \ref{lem:Arem}, and \ref{gevreyparaproductlemma}. 
Let us briefly comment on how to prove \eqref{ineq:ACKwProd}; let us just prove it in the case of $A$ for $f$ and $g$ independent of $X$ as the more general case is similar.  
We would like to conclude quickly from Lemma \ref{lem:CKwFreqRat}, however, we have to see that we can apply \eqref{ineq:dtwBasic} rather than \eqref{ineq:dtwBasicBrack}.
Define the multiplier 
\begin{align*}
\mathcal{M}(t,\eta,l) = \left(\sqrt{\frac{\partial_t w(t,\eta)}{w(t,\eta)}} + \frac{\abs{\eta,l}^{s/2}}{\jap{t}^{s}}\right)A(t,\eta,l). 
\end{align*} 
Expand with a paraproduct expansion and apply the multiplier  
\begin{align*} 
\norm{\mathcal{M}(fg)}_2 & \leq \norm{\mathcal{M}(f_{Hi}g_{Lo})}_2 + \norm{\mathcal{M}(f_{Lo}g_{Hi})}_2  + \norm{\mathcal{M}(fg)_{\mathcal{R}}}_2 \\ 
& = T_{HL} + T_{LH} + T_{\mathcal{R}}. 
\end{align*}
By symmetry it suffices to treat only one of $T_{HL}$ and $T_{LH}$. 
Analogous to the proof of \eqref{ineq:quadHL}, from  almost orthogonality \eqref{ineq:GeneralOrtho},
\begin{align*} 
(T_{HL})^2  & \approx \sum_{N \in 2^{\Integer} : N \leq 1} \norm{\mathcal{M}(f_N g_{<N/8})}_{2}^2 + \sum_{N \in 2^{\Integer} : N \geq 1} \norm{\mathcal{M}(f_N g_{<N/8})}_{2}^2  = L + H. 
\end{align*} 
In bounding the $H$ term as in \eqref{ineq:quadHL} use Lemma \ref{lem:ABasic} and \eqref{ineq:dtwBasic} by the frequency localizations; then the proof proceeds as in \eqref{ineq:quadHL}.  
In the low frequency terms we may use that  
\begin{align*}
\sum_{N \in 2^{\Integer} : N \leq 1} \norm{\mathcal{M}(f_N g_{<N/8})}_{2}^2 & \lesssim \sum_{N \in 2^{\Integer} : N \leq 1} \norm{\left(\sqrt{\frac{\partial_t w}{w}} + \frac{\abs{\grad}^{s/2}}{\jap{t}^{s}}\right)(f_N g_{<N/8})}_{2}^2, 
\end{align*}
and then apply \eqref{ineq:dtwBasicBrack2} followed by an argument analogous to the proof of \eqref{ineq:quadHL}; we omit the details for brevity. 
This completes the treatment of the $T_{HL}$ term (and hence $T_{LH}$ as well). 

To treat the remainder term we may proceed analogous to the proof of \eqref{ineq:quadR}, applying \eqref{ineq:AARemainderBasic}
and \eqref{ineq:dtwBasicBrack2}; we omit the details for brevity. 
\end{proof}

Together with \eqref{def:psi2sqrBrack}, Lemma \ref{lem:AAiProd} and Lemma \ref{lem:GevProdAlg} imply the following lemma (as long as $C$ remains sufficiently small of course). 

\begin{lemma}[Coefficient control] \label{lem:CoefCtrl} 
Define 
\begin{align} 
G = \left((1+\psi_y)^2 + \psi_z^2\right) - 1. \label{def:G}
\end{align}
Then, under the bootstrap hypotheses for $c_0$ sufficiently small, we have for $h \in \set{\psi_y,\psi_z,G}$, 
\begin{subequations} 
\begin{align} 
\norm{\jap{\grad}^{-1} A h}_2 & \lesssim \norm{AC}_2 \\ 
\norm{A h}_2 & \lesssim \norm{\grad AC}_2 \\ 
\norm{\jap{\grad}^{-1} \sqrt{\frac{\partial_t w}{w}} A h}_2  & \lesssim \norm{\left(\sqrt{\frac{\partial_t w}{w}} + \frac{\abs{\grad}^{s/2}}{\jap{t}^s}\right) AC}_2 \\ 
\norm{\jap{\grad}^{-1} \abs{\grad}^{s/2} Ah}_2 & \lesssim \norm{\abs{\grad}^{s/2}AC}_2. 
\end{align}
\end{subequations}
Furthermore, 
\begin{subequations} 
\begin{align}
\norm{\jap{\grad}^{-2} A \Delta_t C}_2 & \lesssim \norm{ AC}_2 \\ 
\norm{\jap{\grad}^{-1} A \Delta_t C}_2 & \lesssim \norm{\grad AC}_2 \\ 
\norm{\jap{\grad}^{-2} \sqrt{\frac{\partial_t w}{w}} A \Delta_t C}_2  & \lesssim \norm{\left(\sqrt{\frac{\partial_t w}{w}} + \frac{\abs{\grad}^{s/2}}{\jap{t}^s}\right) AC}_2 \\ 
\norm{\jap{\grad}^{-2} \abs{\grad}^{s/2} A \Delta_t }_2 & \lesssim \norm{\abs{\grad}^{s/2}AC}_2.
\end{align}
\end{subequations}
Similarly, for any $\mu > 0$ and $p \geq 0$ (the constant can be taken independent of $\mu$  for $p > 1$): 
\begin{align} 
\norm{\psi_y}_{\G^{\mu,p}} + \norm{\psi_z}_{\G^{\mu,p}} + \norm{G}_{\G^{\mu,p}} + \norm{\Delta_t C}_{\G^{\mu,p-1}} & \lesssim \norm{\grad C}_{\G^{\mu,p}}. \label{ineq:psiCLow}
\end{align}
\end{lemma}  

\begin{remark} \label{rmk:CoefCtrlLow} 
An immediate consequence of \eqref{ineq:psiCLow} together with \eqref{ineq:Boot_LowC} implies that any time the coefficients $\psi_y$,$\psi_z$ or $G$ appear in `low frequency' in a paraproduct, they satisfy the a priori estimates
\begin{align}
\norm{\psi_y}_{\G^{\lambda,\gamma-1}} + \norm{\psi_z}_{\G^{\lambda,\gamma-1}} + \norm{G}_{\G^{\lambda,\gamma-1}} & \lesssim \max(c_0,\epsilon \jap{t}). \label{ineq:CoefCtrlLow}
\end{align}  
Together with $\epsilon t \jap{\nu t^3}^{-1} \lesssim c_{0}\jap{t}^{-2}$, \eqref{ineq:CoefCtrlLow} shows that when there is enhanced dissipation present, we generally need only treat the leading order terms that arise from the approximation $\grad^t \approx \grad^L$ (recall the definition \eqref{def:gradDelL}). 
\end{remark}

\begin{remark} \label{rmk:SIcoefneglect} 
Even when enhanced dissipation is not present, the coefficients do not depend on $X$ and hence do not shift the frequencies in $X$. This is of absolutely crucial importance, since it means the approximations made in \S\ref{sec:Toy} will not be badly disturbed by the presence of the coefficients.
This will mean that even when there are no powers of $\jap{\nu t^3}^{-1}$, terms in which coefficients appear in low frequency are generally treatable with an easy variant of the treatment used on the leading order terms. 
Accordingly, these terms are generally omitted except for a few exceptions when the structure is changed appreciably by the coefficients. 
\end{remark} 

We can deduce the following important lemma, which is an easy variant of an analogous lemma proved also in \cite{BMV14}. 
Hence, the proof is omitted.  
\begin{lemma}[$A^\nu$ Product Lemma] \label{lem:AnuProd} 
The following holds for all $f^1$ and $f^2$ such that $f^2_{\neq} = f^2$, 
\begin{align} 
\norm{A^{\nu;i}(f^1 f^2)}_2 & \lesssim \norm{f^1}_{\G^{\lambda,\beta + 3\alpha+2}}\norm{A^{\nu;i}f^2}_2. \label{ineq:AnuiDistri}  
\end{align} 
Moreover, if also $f^1_{\neq} = f^1$, then we have the product lemma-type inequalities 
\begin{subequations} \label{ineq:AnuiDistriDecay}
\begin{align} 
\norm{A^{\nu;1}(f^1 f^2)}_2 \lesssim \frac{\jap{t}^{2+\delta_1}}{\jap{\nu t^3}^\alpha}\left(\norm{\jap{\grad}^{2-\beta} A^{\nu;1} f^1}_{2}\norm{A^{\nu;1} f^2}_2 + \norm{A^{\nu;1} f^1}_{2}\norm{\jap{\grad}^{2-\beta} A^{\nu;1} f^2}_2 \right) \label{ineq:AnuiDistriDecay1} \\ 
\norm{A^{\nu;2}(f^1 f^2)}_2 \lesssim \frac{\jap{t}^{\delta_1}}{\jap{\nu t^3}^\alpha}\left(\norm{\jap{\grad}^{2-\beta} A^{\nu;2} f^1}_{2}\norm{A^{\nu;2} f^2}_2 + \norm{A^{\nu;2} f^1}_{2}\norm{\jap{\grad}^{2-\beta} A^{\nu;2} f^2}_2 \right) \label{ineq:AnuiDistriDecay2} \\ 
\norm{A^{\nu;3}(f^1 f^2)}_2 \lesssim \frac{\jap{t}^{2}}{\jap{\nu t^3}^\alpha}\left(\norm{\jap{\grad}^{2-\beta} A^{\nu;3} f^1}_{2}\norm{A^{\nu;3} f^2}_2 + \norm{A^{\nu;3} f^1}_{2}\norm{\jap{\grad}^{2-\beta} A^{\nu;3} f^2}_2 \right). \label{ineq:AnuiDistriDecay3}
\end{align} 
\end{subequations}  
\end{lemma}

\section{High norm estimate on $Q^2$}\label{sec:HiQ2}
First compute the time evolution of $A^2 Q^2$ in $L^2$: 
\begin{align}
\frac{1}{2}\frac{d}{dt}\norm{A^{2} Q^2}_2^2 & \leq \dot{\lambda}\norm{\abs{\grad}^{s/2}A^{2} Q^2}_2^2 - \norm{\sqrt{\frac{\partial_t w}{w}}A^{2} Q^2}_2^2  - \frac{1}{t}\norm{\mathbf{1}_{t > \jap{\grad_{Y,Z}}} A^{2}Q^2_{\neq}}_2^2 \nonumber \\
 & \quad -\norm{\sqrt{\frac{\partial_t w_L}{w_L}} A^{2} Q^2}_2^2  + \nu \int A^{2} Q^2 A^{2} \left(\tilde{\Delta_t} Q^2\right) dV -\int A^{2} Q^2 A^{2} \left( \tilde U \cdot \grad Q^2 \right) dV \nonumber \\
& \quad - \int A^{2} Q^2 A^{2}\left( Q^j \partial_j^t U^2  + 2\partial_i^t U^j \partial_{i}^t \partial_{j}^t U^i - \partial_Y^t\left(\partial_j^t U^i \partial_i^t U^j\right) \right) dV \nonumber \\ 
& = - \mathcal{D}Q^2 - CK^2_L + \mathcal{D}_E + \mathcal{T} + NLS1 + NLS2 + NLP, \label{eq:A2Q2Evo}  
\end{align}
where we used the definition
\begin{align*}
\mathcal{D}_E = \nu \int A^{2} Q^{2} A^{2} \left((\tilde{\Delta_t} - \Delta_L) Q^2\right) dV. 
\end{align*}
Let us here introduce the following enumerations. 
For $i,j\in \set{1,2,3}$ and $a,b \in \set{0,\neq}$: 
\begin{subequations} \label{def:Q2Enums}
\begin{align}
NLP(i,j,a,b) &= \int A^2 Q^2_{\neq} A^2\left( \partial_Y^t \left(\partial_j^t U^i_a \partial_i^t U^j_b \right) \right) dV \\ 
NLS1(j,a,b) & = -\int A^2 Q^2_{\neq} A^2\left(Q^j_a\partial_j^t U^2_{b}\right) dV \\
NLS2(i,j,a,b) & = -\int A^2 Q^2_{\neq} A^2\left(\partial_i^t U^j_a \partial_i^t\partial_j^t U^2_{b}\right) dV \\
NLP(i,j,0) & = \int A^2 Q^2_{0} A^2\left( \partial_Y^t \left(\partial_j^t U^i_0 \partial_i^t U^j_0 \right) \right) dV \\ 
NLS1(j,0) & = -\int A^2 Q^2_{0} A^2\left(Q^j_0\partial_j^t U^2_{0}\right) dV \\
NLS2(i,j,0) & = -\int A^2 Q^2_{0} A^2\left(\partial_i^t U^j_0 \partial_i^t\partial_j^t U^2_{0}\right) dV \\ 
\mathcal{F} & = -\int A^2 Q^2_{0} A^2\left(\partial_i^t \partial_i^t \partial_j^t \left( U^j_{\neq} U^2_{\neq}\right)_0 - \partial_{Y}^t \partial_j^t \partial_i^t \left( U^i_{\neq} U^j_{\neq}\right)_0 \right)dV \\ 
\mathcal{T}_0 & = -\int A^{2} Q_0^2 A^{2} \left( \tilde U_{0} \cdot \grad Q_{0}^2 \right) dV \\ 
\mathcal{T}_{\neq} & = -\int A^{2} Q_{\neq}^2 A^{2} \left( \tilde U \cdot \grad Q^2 \right) dV 
\end{align}
\end{subequations} 

We have divided the nonlinearity up based on the heuristics of \S\ref{sec:NonlinHeuristics}, as each type of interaction warrants a different treatment.  
Note that we have split $\mathcal{T}$ into three contributions: $\mathcal{T}_0$ (the \textbf{(2.5NS)} interactions), $\mathcal{T}_{\neq}$ (the \textbf{(SI)} and \textbf{(3DE)} interactions), and a contribution that is grouped with $\mathcal{F}$ (the \textbf{(F)} interactions).  Similarly, we have split the $NLS$ and $NLP$ terms into several contributions: $NLS1(j,0)$, $NLS2(i,j,0)$, and $NLP(i,j,0)$ (the \textbf{(2.5NS)} interactions), the $NLS1(j,a,b)$, $NLS2(i,j,a,b)$, and $NLP(i,j,a,b)$ (the \textbf{(SI)} and \textbf{(3DE)} interactions), and a contribution that is grouped with $\mathcal{F}$ (the \textbf{(F)} interactions). 
This kind of subdivision will be used repeatedly in the sequel.  

\subsection{Zero frequencies} \label{sec:AQ2Zero}
The nonlinear terms associated with zero frequencies tend to have a very different flavor than those 
of the non-zero frequencies. 

\subsubsection{Transport nonlinearity} \label{sec:TransQ20}
Here we consider the nonlinear interaction of two zero frequencies in the transport nonlinearity (hence, these are of type \textbf{(2.5NS)}), defined as $\mathcal{T}_0$ above in \eqref{def:Q2Enums}. We further subdivide by frequency: 
\begin{align*} 
\mathcal{T}_{0} & = -\int \left(A^{2} Q_0^2\right)_{>1} A^{2} \left( \tilde U_{0} \cdot \grad Q_{0}^2 \right)_{>1} dV - \int \left(A^{2} Q^2_0\right)_{\leq 1} A^{2} \left( \tilde U_{0} \cdot \grad Q_{0}^2 \right)_{\leq 1} dV \\ 
& = \mathcal{T}^H_0 + \mathcal{T}^L_0. 
\end{align*} 
On the high frequencies we can use Lemma \ref{lem:AAiProd} and the frequency projection to deduce (recall \eqref{def:tildeU2}): 
\begin{align*} 
\mathcal{T}^H_0 & \leq \norm{A^2Q^2_{>1}}_2 \norm{A^2\left(\tilde U_0 \cdot \grad Q^2_0 \right)_{>1}}_2 \\ 
& \lesssim \left(\norm{A^2 g}_2 + \norm{A^2 U_0^3}_2\right)\norm{\grad A^2 Q^2_0}_2^2 \\ 
& \lesssim \epsilon\norm{\sqrt{-\Delta_L}A^2 Q^2}_2^2, 
\end{align*}  
where the last line used \eqref{ineq:AprioriU0} and \eqref{ineq:Boot_Ag}. 
This term is then absorbed by the dissipation for $c_0$ sufficiently small (depending on $K_B$ of course).   

Turn next to the low frequencies $\mathcal{T}^L_0$. This term requires a slightly more precise treatment as we cannot add $\sqrt{-\Delta_L}$ to the leading factor. 
Separate into two contributions: 
\begin{align*} 
\mathcal{T}^L_0 & = - \int \left(A^{2} Q^2_0\right)_{\leq 1} A^{2} \left(g \partial_Y Q^2_{0} \right)_{\leq 1} dV - \int \left(A^{2} Q^2_0\right)_{\leq 1} A^{2} \left(U_0^3 \partial_Z Q^2_{0} \right)_{\leq 1} dV \\ 
& = \mathcal{T}^L_g + \mathcal{T}^L_U.  
\end{align*} 
For $\mathcal{T}^L_g$ we have by Sobolev embedding and \eqref{ineq:Boot_gLow}, 
\begin{align*} 
\mathcal{T}^L_g & \leq \norm{ (A^2 Q^2_0)_{\leq 1}}_2\norm{A^2\left(g \partial_Y Q^2_0\right)_{\leq 1}}_2 \\ 
& \lesssim \norm{g}_\infty \norm{Q^2_0}_2 \norm{\partial_ Y Q^2_0}_2 \lesssim \frac{\epsilon}{\jap{t}^{2}}\left(\norm{A^2 Q^2}_2^2 + \norm{\sqrt{-\Delta_L}A^2 Q^2}_2^2 \right), 
\end{align*} 
which suffices for Proposition \ref{prop:Boot} for $c_0$ and $\epsilon$ sufficiently small. 
For $\mathcal{T}^L_U$, first note
\begin{align*} 
 \mathcal{T}^L_U & \lesssim \sum_{l,l^\prime} \int \abs{\widehat{Q^2_0}(\eta,l) \widehat{U^3_0}(\xi,l^\prime) (l-l^\prime) \widehat{Q^2_{0}}(\eta-\xi,l-l^\prime)}  d\eta d\xi.
\end{align*} 
Then since $l \neq l^\prime$ on the support of the integrand, at least one of $l$ or $l^\prime$ is non-zero, and therefore we have by \eqref{ineq:L2L2L1}, 
\begin{align*} 
\mathcal{T}^L_U & \lesssim \norm{Q^2}_2 \norm{\partial_Z U^3_0}_2 \norm{\partial_Z Q^2}_{H^{1+}} +  \norm{\partial_Z Q^2}_2 \norm{U^3_0}_2 \norm{\partial_Z Q^2}_{H^{1+}} \\ 
& \lesssim \epsilon \norm{\grad U^3_0}^2_2 +   \epsilon \norm{\sqrt{-\Delta_L} A^2 Q^2}^2_2, 
\end{align*}
where the last line followed from \eqref{ineq:Boot_Q2Hi}, and \eqref{ineq:AprioriU0}.  
By \eqref{ineq:Boot_LowFreq}, this is consistent with Proposition \ref{prop:Boot} for $c_0$ chosen sufficiently small by absorbing terms with the dissipation and integrating, as discussed in \eqref{ineq:GenScheme}.

\subsubsection{Nonlinear pressure and stretching} \label{sec:NLPSQ20}
These terms correspond to the nonlinear zero frequency interactions in the pressure and stretching terms, and so are of type \textbf{(2.5NS)},  
Due to the ellipticity of $\tilde{\Delta}_t$ at the zero frequency and 
the fact that $A^{2}_0 = A^{3}_0$, all of these terms are similar. 

Consider the nonlinear pressure interaction with $i = j = 3$, denoted $NLP(3,3,0)$ (recall \eqref{def:Q2Enums}), as a representative example; the other contributions can all be treated with a very similar approach (or are slightly easier due to low frequency decay estimates on $U^2_0$) and hence these are omitted for the sake of brevity.    
As in the transport nonlinearity above in \S\ref{sec:TransQ20}, we divide into high and low frequencies
\begin{align} 
NLP(3,3,0) & = -\int \left(A^{2}Q^2_0\right)_{>1} \left(A^{2}\partial_Y^t\left( \partial_Z^t U^3_0 \partial_Z^t U^3_0 \right)\right)_{>1} dV + \int \left(A^{2}Q^2_0\right)_{\leq 1} \left(A^{2}\partial_Y^t\left( \partial_Z^t U^3_0 \partial_Z^t U^3_0 \right)\right)_{\leq 1} dV \nonumber \\ 
& = P^H + P^L. \label{def:NLP033_Q2}
\end{align}  
Consider the high frequency term first. We have by Lemma \ref{lem:AAiProd}, Lemma \ref{lem:CoefCtrl}, and \eqref{ineq:Boot_ACC}: 
\begin{align} 
P^H & \leq \norm{\left(A^2 Q^2_0\right)_{>1}}_2 \norm{ \left(A^2 \partial_Y^t\left( \partial_Z^t U^3_0 \partial_Z^t U^3_0 \right)\right)_{>1}}_2 \nonumber \\ 
& \lesssim \norm{\grad A^2 Q^2_{0}}_2(1+\norm{A C}_2)\norm{A^2 \partial_Y\left(\partial_Z^t U^3_0 \partial_Z^t U^3_0\right)}_2 \nonumber \\ 
& \lesssim \norm{\sqrt{-\Delta_L} A^2 Q^2}_2 \left((1 + \norm{AC})^2 \norm{AC}_2\norm{\grad A^2 U^3_0}_2^2 + (1 + \norm{AC})^3\norm{\grad^2 A^2 U_0^3}_2 \norm{\grad A^2 U_0^3}_2 \right) \nonumber \\ 
& \lesssim \norm{\sqrt{-\Delta_L} A^2 Q^2}_2\left(\norm{\grad A^2 U^3_0}_2^2 + \norm{\grad^2 A^2 U_0^3}_2 \norm{\grad A^2 U_0^3}_2 \right). \label{ineq:NLP033H_Q2}
\end{align} 
By \eqref{ineq:AprioriU0} and Lemma \ref{lem:PELbasicZero} (specifically \eqref{ineq:gradAU0i}) we get 
\begin{align*} 
\norm{\grad A^2 U^3_0}_2^2 + \norm{\grad^2 A^2 U_0^3}_2 \norm{\grad A^2 U_0^3}_2  & \lesssim \norm{AU_0^3}_2 \norm{\grad AU_0^3}_2 \lesssim \epsilon \left(\norm{\grad A^3 Q_0^3}_2 + \norm{\grad U_0^3}_2 + \epsilon \norm{\grad AC}_2\right).  
\end{align*} 
Applying this to \eqref{ineq:NLP033H_Q2} implies
\begin{align*} 
P^H & \lesssim \epsilon \norm{\sqrt{-\Delta_L} A^2 Q^2}^2_2 + \epsilon \norm{\sqrt{-\Delta_L} A^3 Q^3}^2_2 + \epsilon \norm{\grad U_0^3}^2_2 + \epsilon^3\norm{\grad AC}_2^2, 
\end{align*}
which can be absorbed by the dissipation and time-integrated to be consistent with Proposition \ref{prop:Boot} by choosing $c_{0}$ sufficiently small via the bootstrap hypotheses \eqref{ineq:Boot_ACC}, \eqref{ineq:Boot_Hi} and \eqref{ineq:Boot_LowFreq}. 

Turn next to the low frequency term in \eqref{def:NLP033_Q2}. 
Unlike the transport term in \S\ref{sec:TransQ20}, there are enough derivatives to absorb this term with the dissipation in a relatively straightforward manner.  
Indeed, by Sobolev embedding and Lemma \ref{lem:CoefCtrl}, we have 
\begin{align*} 
P^L & \lesssim  \norm{Q^2_0}_2\left(1 + \norm{\psi_y}_\infty \right) \norm{\partial_Y\left( \partial_Z^t U^3_0 \partial_Z^t U^3_0\right)}_2 \\ 
& \lesssim \norm{Q^2_0}_2\left(1 + \norm{C}_{H^{7/2+}} \right)^3\norm{\grad U^3_0}_{H^{3/2+}}^2 \\ 
& \lesssim  \epsilon \norm{\grad U^3}_{H^{3/2+}}^2,
\end{align*}
where the last line followed by \eqref{ineq:Boot_LowFreq}. This is consistent with Proposition \ref{prop:Boot} by \eqref{ineq:Boot_LowFreq} for $c_{0}$ sufficiently small. 
This completes the treatment of $NLP(3,3,0)$. 
As mentioned above, the other $NLP$ and $NLS$ terms are treated in the same way and yield similar contributions so are omitted for the sake of brevity (note for $NLS1$, one should use $Q^j_0 = \Delta_t U^j_0$ in order to see $\grad U^j$ so that viscous dissipation can be used). 

\subsubsection{Forcing from non-zero frequencies} \label{sec:NzeroForcing}
We now turn to nonlinear interactions of type \textbf{(F)}: the interaction of two $X$ frequencies $k$ and $-k$. 
Further divide via, 
\begin{align*} 
\mathcal{F}  & = -\int A^2 Q^2_0 A^2\left(\partial_Z^t \partial_Z^t \partial_j^t \left(U^j_{\neq} U^2_{\neq}\right)_0 - \partial_Y^t \partial_Y^t \partial_Z^t \left(U^3_{\neq} U^2_{\neq}\right)_0  - \partial_Y^t \partial_Z^t \partial_Z^t \left(U^3_{\neq} U^3_{\neq}\right)_0 \right) dV  \\ 
& = F^1 + F^2 + F^3.     
\end{align*} 
All three are treated via variants of the same basic approach which will ultimately come down to applying the 
appropriate multiplier estimate in \eqref{ineq:AdelLij} depending on the combination of derivatives present. 
Hence, we will treat the example $F^2$ and omit the others for the sake of brevity.
We begin by expanding via a quintic paraproduct decomposition (see \S\ref{sec:paranote}) (the term is quintic due to the presence of $\psi$ in the derivatives).
However, by Remark \ref{rmk:CoefCtrlLow}, the terms involving the coefficients are much smaller than the leading order contributions unless they appear as the high frequency factor in the paraproduct. 
Therefore, we expand $F^2$ in the following manner: 
\begin{align*}  
F^2 & = \sum_{k\neq 0} \int A^{2}Q^2_0 A_0^{2} \partial_Y\partial_Y\partial_Z\left( \left(U^3_{-k}\right)_{Hi} \left( U^2_k\right)_{Lo}\right)dV \\
& \quad + \sum_{k\neq 0} \int A^{2}Q^2_0 A_0^{2} \partial_Y\partial_Y\partial_Z\left( \left(U^3_{-k}\right)_{Lo} \left( U^2_k\right)_{Hi}\right) dV \\ 
& \quad +\sum_{k\neq 0}  \int A^{2}Q^2_0 A_0^{2} \left( (\psi_y)_{Hi} \partial_Y\partial_Y\partial_Z\left( \left(U^3_{-k}\right)_{Lo} \left( U^2_k\right)_{Lo}\right)\right) dV \\ 
& \quad + \sum_{k\neq 0} \int A^{2}Q^2_0 A_0^{2} \left( \partial_Y (\psi_y)_{Hi}\partial_Y\partial_Z\left( \left(U^3_{-k}\right)_{Lo} \left( U^2_k\right)_{Lo}\right) \right) dV \\
& \quad + \sum_{k\neq 0} \int A^{2}Q^2_0 A_0^{2} \left( \partial_Y \partial_Y (\psi_z)_{Hi} \partial_Z \left( \left(U^3_{-k}\right)_{Lo} \left( U^2_k\right)_{Lo}\right) \right) dV \\ 
& \quad + F^2_{\mathcal{R},C} \\ 
& = F^2_{HL} + F^2_{LH} + F^2_{C1} + F^2_{C2} + F^2_{C3} + F^2_{\mathcal{R},C},    
\end{align*} 
where here $F^2_{\mathcal{R},C}$ includes all of the remainders from the quintic paraproduct as well as the higher order terms involving coefficients as low frequency factors. 
Turn first to $F^2_{HL}$ (recall \eqref{ineq:AprioriUneq} the shorthand discussed in \eqref{def:Low} above), which by \eqref{ineq:ABasic}, can be bounded by
\begin{align*} 
F_{HL}^2 & \lesssim \frac{\epsilon}{\jap{t}^{2-\delta_1} \jap{\nu t^3}^{\alpha}} \sum_{k\neq 0} \sum_{l,l^\prime} \int \abs{A^{2} \widehat{Q^2_0}(\eta,l) A^2_0(\eta,l) \frac{\abs{\eta}^2\abs{l}}{k^2 + (l^\prime)^2 + \abs{\xi-kt}^2} \Delta_L \widehat{U^3_{k}}(\xi,l^\prime)_{Hi}} \\ & \quad\quad \times Low(-k,\eta-\xi,l-l^\prime) d\eta d\xi \\ 
& \lesssim \frac{\epsilon}{\jap{t}^{2-\delta_1} \jap{\nu t^3}^{\alpha}} \sum_{k\neq 0} \sum_{l,l^\prime} \int \abs{A^{2} \widehat{Q^2_0}(\eta,l) \frac{\abs{\eta}^2\abs{l} \jap{\frac{t}{\jap{\xi,l^\prime}}}^2}{k^2 + (l^\prime)^2 + \abs{\xi-kt}^2} \Delta_L A^{3}\widehat{U^3_{k}}(\xi,l^\prime)_{Hi}} Low(-k,\eta-\xi,l-l^\prime) d\eta d\xi. 
\end{align*} 
Hence, by \eqref{ineq:AdeYZZ} (with $p = 2$) and \eqref{ineq:triQuadHL} (with $\mu = p = 0$) this can be estimated via 
 \begin{align*} 
F_{HL}^2 & \lesssim \frac{\epsilon t^{\delta_1}}{\jap{\nu t^3}^{\alpha}}\norm{\left(\sqrt{\frac{\partial_t w}{w}} + \frac{\abs{\grad}^{s/2}}{\jap{t}^s}\right)A^{2}Q^2}_2 \norm{\left(\sqrt{\frac{\partial_t w}{w}} + \frac{\abs{\grad}^{s/2}}{\jap{t}^s}\right)A^{3} \Delta_LU^3_{\neq}}_2
\\ & \quad + \frac{\epsilon}{\jap{t}^{2-\delta_1} \jap{\nu t^3}^\alpha}\norm{\sqrt{-\Delta_L} A^{2}Q^2}_2 \norm{A^{3}\Delta_L U^3_{\neq}}_2,     
\end{align*} 
which, after the application of Lemmas \ref{lem:PEL_NLP120neq} and \ref{lem:SimplePEL}, is consistent with Proposition \ref{prop:Boot} for $c_0$ sufficiently small by 
absorbing the leading terms with the dissipation energies and integrating the others, as discussed in \eqref{ineq:GenScheme}.  

Turn next to $F_{LH}^2$, for which we use a similar approach as $F_{HL}^2$.   
Indeed, by \eqref{ineq:ABasic} followed by \eqref{ineq:AdeYZZ} (with $p = 1$) and \eqref{ineq:triQuadHL} (with $\mu = p = 0$) this is estimated via: 
\begin{align*} 
F_{LH}^{2} & \lesssim \frac{\epsilon}{\jap{\nu t^3}^{\alpha}}\sum_{k\neq 0} \sum_{l,l^\prime} \int \abs{A^{2} \widehat{Q^2_0}(\eta,l) \frac{\abs{\eta}^2\abs{l} \jap{\frac{t}{\jap{\xi,l^\prime}}}  }{k^2 + (l^\prime)^2 + \abs{\xi-kt}^2} \Delta_L A^{2}\widehat{U^2_{k}}(\xi,l^\prime)_{Hi}} Low(-k,\eta-\xi,l-l^\prime) d\eta d\xi \\ 
& \lesssim \frac{\epsilon t^2}{\jap{\nu t^3}^{\alpha}}\norm{\left(\sqrt{\frac{\partial_t w}{w}} + \frac{\abs{\grad}^{s/2}}{\jap{t}^s}\right) A^{2}Q^2}_2 \norm{\left(\sqrt{\frac{\partial_t w}{w}} + \frac{\abs{\grad}^{s/2}}{\jap{t}^s}\right)A^{2} \Delta_L U^2_{\neq}}_2
\\ & \quad + \frac{\epsilon}{\jap{\nu t^3}^\alpha}\norm{\sqrt{-\Delta_L} A^{2}Q^2}_2 \norm{A^{2}\Delta_L U^2_{\neq}}_2,    
\end{align*} 
which, after the application of the precision elliptic lemmas, Lemmas \ref{lem:PEL_NLP120neq} and \ref{lem:SimplePEL}, is consistent with Proposition \ref{prop:Boot} for $\epsilon$ sufficiently small. 

The most difficult coefficient term is $F^2_{C3}$, since two derivatives of the coefficients are being taken. 
By Lemma \ref{lem:ABasic}, \eqref{ineq:triQuadHL} and Lemma \ref{lem:CoefCtrl} we have
\begin{align*} 
F^2_{C3} & \lesssim  \frac{\epsilon^2}{\jap{t}^{2-\delta_1} \jap{\nu t^3}^{2\alpha}}\sum_{k\neq 0} \sum_{l,l^\prime} \int \abs{A^{2} \widehat{Q^2_0}(\eta,l) A_0^{2}(\eta,l) \abs{\eta}^2 \widehat{\psi_z}(\xi,l^\prime)_{Hi}} Low(k,\eta-\xi,l-l^\prime) d\eta d\xi \\ 
& \lesssim \frac{\epsilon^2}{\jap{t}^{2-\delta_1} \jap{\nu t^3}^{2\alpha}} \norm{\sqrt{-\Delta_L} A^{2}Q^2}_2 \norm{AC}_2 \\ 
& \lesssim \epsilon\norm{\sqrt{-\Delta_L} A^{2}Q^2}_2^2 + \frac{\epsilon^3}{\jap{t}^{4-2\delta_1} \jap{\nu t^3}^{2\alpha}}\norm{AC}^2_2,
\end{align*}  
which is consistent with Proposition \ref{prop:Boot} by the bootstrap hypotheses. Note that the inviscid damping which produces $t^{2\delta_1 - 4}$ in the second factor is not necessary to treat this term, and indeed, it is not present in the analogous term in $F^3$. 
All of the other coefficient high frequency terms are easier and give similar contributions and are hence omitted for brevity. 
As discussed previously, the remainder terms $F_{\mathcal{R},C}^2$ are significantly easier and yield similar contributions (except smaller) as the leading order terms (though note some terms involving coefficients in low frequency may require a different inequality in \eqref{ineq:AdelLij} than the leading order term). 

\subsubsection{Dissipation error terms} \label{sec:DEQ02}
Recalling the definitions of the dissipation error terms and the short-hand \eqref{def:G}, we have 
\begin{align} 
\mathcal{D}_E & = \nu \int A^2 Q^2_0A^2\left(\tilde{\Delta_t} - \Delta_L\right)Q^2_0 dV =  \nu\int A_0^2 Q^2_0 A_0^2\left(G \partial_{YY}Q^2_0 + 2\psi_z \partial_{YZ}Q^2_0 \right) dV. \label{def:DEQ2}
\end{align} 
The two error terms are very similar, so just consider the first and expand with a paraproduct
\begin{align} 
\nu\int A_0^2 Q^2_0 A_0^2\left(G\partial_{YY}Q^2_0\right) dV & = \nu\int A_0^2 Q^2_0 A_0^2\left(G_{Hi} (\partial_{YY}Q^2_0)_{Lo} \right) dV + \nu\int A_0^2 Q^2_0 A_0^2\left(G_{Lo} (\partial_{YY}Q^2_0)_{Hi} \right) dV \nonumber \\& \quad
 +  \nu\int A_0^2 Q^2_0 A_0^2\left(\left(G\partial_{YY}Q^2_0\right)_{\mathcal{R}} \right) dV \nonumber \\
 & = D_{HL} + D_{LH} + D_{\mathcal{R}}. \label{ineq:DEQ2}
\end{align} 
By Lemma \ref{lem:ABasic} and \eqref{ineq:triQuadHL} followed by Lemma \ref{lem:CoefCtrl} and \eqref{ineq:Boot_Q2Hi}, 
\begin{align} 
D_{HL} & \lesssim \nu \norm{A^2Q^2_0}_2 \norm{\grad A C}_2 \norm{\grad Q_0^2}_{\G^{\lambda,3/2+}} \lesssim c_0^{-1} \nu \epsilon^2 \norm{\grad A C}_2^2 + c_0 \nu \norm{\sqrt{-\Delta_L }A^2 Q^2}_{2}^2. \label{ineq:DHLQ2} 
\end{align}
Note by \eqref{ineq:Boot_ACC}, 
\begin{align*} 
\int_{1}^{T^\star} c_0^{-1} \nu \epsilon^2 \norm{\grad A C(t)}_2^2 dt & \lesssim c_{0} \epsilon^2 K_B.  
\end{align*} 
Hence, for $c_{0}$ sufficiently small (relative to $K_B$ of course), the first term in \eqref{ineq:DHLQ2} will make a contribution to the energy estimate in \eqref{eq:A2Q2Evo} that is consistent with Proposition \ref{prop:Boot} (naturally, the second term is absorbed by the dissipation in \eqref{eq:A2Q2Evo}). 
This completes the treatment of  $D_{HL}$. 

To treat the second term in \eqref{ineq:DEQ2}, apply Lemma \ref{lem:CoefCtrl}, \eqref{ineq:ABasic} and \eqref{ineq:triQuadHL} to deduce: 
\begin{align*} 
D_{LH} & \lesssim c_0 \nu \sum_{l,l^\prime \in \Integers} \int \abs{A^2 \widehat{Q^2_0}(\eta,l) A_0^2(\eta,l) \abs{\xi}^2 \widehat{Q^2_0}(\xi,l-l^\prime)_{Hi} Low(\eta-\xi,l-l^\prime)} d\eta d\xi \\ 
& \lesssim c_0 \nu \sum_{l,l^\prime \in \Integers} \int \abs{\eta A^2 \widehat{Q^2_0}(\eta,l) A_0^2(\xi,l^\prime) \abs{\xi} \widehat{Q^2_0}(\xi,l-l^\prime)_{Hi} Low(\eta-\xi,l-l^\prime)} d\eta d\xi \\ 
& \lesssim c_{0} \nu\norm{\sqrt{-\Delta_L}A^2 Q^2}_2^2, 
\end{align*} 
which is then absorbed by the dissipation for $c_{0}$ chosen sufficiently small. 

To treat the remainder term in \eqref{ineq:DEQ2}, we essentially apply the same proof as $D_{HL}$. 
Indeed, applying Lemma \ref{lem:CoefCtrl}, \eqref{ineq:ARemainderBasic}, and \eqref{ineq:quadR}, we have
\begin{align*} 
D_{\mathcal{R}} & \lesssim \nu\norm{A^2Q^2}_2\norm{\grad C}_{\G^{\lambda,3/2+}} \norm{\grad Q^2_0}_{\G^{\lambda,3/2+}}  \lesssim c_0^{-1} \nu \epsilon^2\norm{\grad A C}_{2}^2 + c_0\nu\norm{\sqrt{-\Delta_L}A^2Q^2}_{2}^2, 
\end{align*}
after which the treatment follows as in \eqref{ineq:DHLQ2}. 
This completes the treatment of the first error term in \eqref{def:DEQ2}. 
The second error term in \eqref{def:DEQ2} is similar to the first and yields the same contributions so is omitted for the sake of brevity. 
  
\subsection{Non-zero frequencies}
Next we consider the contributions to \eqref{eq:A2Q2Evo} that come from the evolution of non-zero $X$ frequencies.   

\subsubsection{Nonlinear pressure $NLP$} \label{sec:NLPQ2}

\paragraph{Treatment of $NLP(1,j,0,\neq)$} \label{sec:NLP213}
Recall the enumeration \eqref{def:Q2Enums}.  
Here $j \in \set{2,3}$ due to the structure of the nonlinearity (a key null structure).
The case $j = 3$ was singled out in  \S\ref{sec:Toy} as a leading order nonlinear interaction of type $\textbf{(SI)}$ in $Q^2$. 
We will concentrate on this case and omit the treatment of $j=2$, which is treated with the same method. 

This term is quartic (in the sense that the nonlinearity is order $4$) and we use the paraproduct decomposition described in \S\ref{sec:paranote}. 
We will group terms where the coefficients appear in `low frequency' with the remainder, as these are weaker or similar due to Remarks \ref{rmk:CoefCtrlLow} and \ref{rmk:SIcoefneglect}.
Hence, we will expand with the paraproduct but only keep the coefficients around when they are in high frequency:  
\begin{align*} 
NLP(1,3,0,\neq)  & =  \sum_{k \neq 0} \int A^2 Q^2_{k} A^2\left( (\partial_Y - t\partial_X) \left( \left(\partial_Z U^1_0\right)_{Lo} (\partial_XU^3_{k})_{Hi}  \right) \right) dV \\ 
& \quad  +\sum_{k \neq 0} \int A^2 Q^2_{k} A^2\left( (\partial_Y - t\partial_X) \left( (\partial_Z U^1_0)_{Hi} (\partial_XU^3_{k})_{Lo} \right) \right) dV \\ 
& \quad +\sum_{k \neq 0} \int A^2 Q^2_{k} A^2\left((\psi_y)_{Hi}(\partial_Y - t\partial_X)\left( (\partial_XU^3_{k})_{Lo} (\partial_Z U^1_0)_{Lo} \right)\right) dV \\ 
 & \quad  +\sum_{k \neq 0} \int A^2 Q^2_{k} A^2 \left( (\partial_Y - t\partial_X) \left( (\partial_XU^3_{k})_{Lo} (\psi_z)_{Hi}\partial_Y (U^1_0)_{Lo}\right) \right) dV \\
& \quad + P_{\mathcal{R},C} \\  
& = P_{LH} + P_{HL} + P_{C1} + P_{C2} + P_{\mathcal{R},C},  
\end{align*} 
where $P_{\mathcal{R},C}$ includes all of the remainders from the quartic paraproduct as well as the higher order terms involving coefficients as low frequency factors.  

Turn first to $P_{LH}$, which by \eqref{ineq:AprioriU0} and \eqref{ineq:ABasic} is bounded by 
\begin{align*} 
P_{LH} & \lesssim \min(\epsilon t,c_{0}) \sum_{k \neq 0} \int \abs{A^2 \widehat{Q^2_{k}}(\eta,l) A^2_k(\eta,l) \frac{(\eta - tk)k}{k^2 + \abs{l^\prime}^2 + \abs{\xi-tk}^2} \Delta_L\widehat{U^3_{k}}(\xi,l^\prime)_{Hi}} Low(\eta-\xi,l-l^\prime) d\eta d\xi \\ 
& \lesssim \min(\epsilon t,c_{0}) \sum_{k \neq 0} \int \abs{A^2 \widehat{Q^2_{k}}(\eta,l) \frac{(\eta - tk)k}{k^2 + \abs{l^\prime}^2 + \abs{\xi-tk}^2}\jap{\frac{t}{\jap{\xi,l^\prime}}} A^3\Delta_L\widehat{U^3_{k}}(\xi,l^\prime)_{Hi}} Low(\eta-\xi,l-l^\prime) d\eta d\xi. 
\end{align*}
Therefore, by \eqref{ineq:AiPartX} (with $p = 1$) followed by \eqref{ineq:triQuadHL}, 
\begin{align*} 
P_{LH} & \lesssim c_{0}\norm{\left(\sqrt{\frac{\partial_t w}{w}} + \frac{\abs{\grad}^{s/2}}{\jap{t}^s}\right) A^2Q^2}_2\norm{\left(\sqrt{\frac{\partial_t w}{w}} + \frac{\abs{\grad}^{s/2}}{\jap{t}^s}\right)A^3\Delta_L U^3_{\neq}}_2  \\ 
& \quad + \epsilon \norm{\sqrt{-\Delta_L} A^2 Q^2}_2 \norm{A^3 \Delta_L U^3_{\neq}}_2.      
\end{align*} 
By Lemmas \ref{lem:PEL_NLP120neq} and \ref{lem:SimplePEL} together with the bootstrap hypotheses, this is consistent with Proposition \ref{prop:Boot} for $c_0$ and $\epsilon$ sufficiently small.   

Turn next to the contribution of $P_{HL}$, which by \eqref{ineq:AprioriUneq} is estimated via 
\begin{align*} 
P_{HL} & \lesssim \frac{\epsilon}{\jap{\nu t^3}^{\alpha}}\sum_{k \neq 0} \sum_{l, l^\prime} \int \abs{A^2 \widehat{Q^2_{k}}(\eta,l) A^2_k(\eta,l) \abs{\eta - kt} \abs{l^\prime} \widehat{U^1_0}(\xi,l^\prime)_{Hi} Low(k,\eta-\xi,l-l^\prime)} d\eta d\xi,
\end{align*}
Therefore, by Lemma \ref{lem:ABasic} followed by \eqref{ineq:triQuadHL}, 
\begin{align*} 
P_{HL} & \lesssim \frac{\epsilon}{\jap{\nu t^3}^{\alpha}} \sum_{k \neq 0} \sum_{l, l^\prime} \int \abs{ A^2 \widehat{Q^2_{k}} \frac{\abs{\eta - kt}}{\jap{\xi,l^\prime}^2 \jap{\frac{t}{\jap{\xi,l^\prime}}}} \abs{l^\prime} A\widehat{U^1_0}(\xi,l^\prime)_{Hi} Low(k,\eta-\xi,l-l^\prime)} d\eta d\xi \\ 
& \lesssim \frac{\epsilon}{\jap{t} \jap{\nu t^3}^{\alpha}} \norm{\sqrt{-\Delta_L} A^2 Q^2}_2 \norm{AU^1_0}_2 \\ 
& \lesssim \epsilon \norm{\sqrt{-\Delta_L} A^2 Q^2}_2^2 + \frac{\epsilon}{\jap{\nu t^3}^{2\alpha}} \left(\frac{1}{\jap{t}^2}\norm{AU^1_0}_2^2\right). 
\end{align*} 
After applying Lemma \ref{lem:PELbasicZero}, this is consistent with Proposition \ref{prop:Boot} using \eqref{ineq:Boot_ACC2} and \eqref{ineq:Boot_Hi}.  

Turn next to $P_{C1}$. By \eqref{ineq:AprioriUneq}, \eqref{ineq:AprioriU0}, and Lemma \ref{lem:ABasic}, we have
\begin{align*} 
P_{C1} & \lesssim \frac{\epsilon^2 t^2}{\jap{\nu t^3}^{\alpha}} \sum_{k \neq 0} \sum_{l, l^\prime} \int \abs{A^2 \widehat{Q^2_{k}}(\eta,l) A^2_k(\eta,l) \widehat{\psi_y}(\xi,l^\prime)_{Hi}} Low(k,\eta-\xi,l-l^\prime) d\eta d\xi \\ 
& \lesssim \frac{\epsilon^2 t^2}{\jap{\nu t^3}^{\alpha}} \sum_{k \neq 0} \sum_{l, l^\prime} \int \abs{A^2 \widehat{Q^2_{k}}(\eta,l) \frac{1}{\jap{\xi,l^\prime}^2 \jap{\frac{t}{\jap{\xi,l^\prime}}}} A\widehat{\psi_y}(\xi,l^\prime)_{Hi}} Low(k,\eta-\xi,l-l^\prime) d\eta d\xi. 
\end{align*}  
Hence, by \eqref{ineq:triQuadHL} and Lemma \ref{lem:CoefCtrl}, we have
\begin{align*} 
P_{C1} & \lesssim \frac{\epsilon^2 t}{\jap{\nu t^3}^{\alpha}} \norm{A^2 Q^2_{\neq}}_2 \norm{\jap{\grad}^{-1 }A \psi_y}_2 \lesssim \frac{\epsilon}{\jap{\nu t^3}^{\alpha}} \norm{ A^2 Q^2}_2^2 + \frac{\epsilon^{3} t^4}{\jap{\nu t^3}^{\alpha}} \left(\frac{1}{\jap{t}^2} \norm{A C}^2_2\right). 
\end{align*} 
Both terms are consistent with Proposition \ref{prop:Boot} (in particular, by \eqref{ineq:Boot_ACC2}).  
This completes the treatment of $P_{C1}$.  
The second coefficient term, $P_{C2}$, is very similar: there is one extra derivative landing on the coefficient but there is one less power of time from the low frequency factor. By Lemma \ref{lem:ABasic}, we will be able to balance the loss by the gain and apply essentially the same treatment as we did for $P_{C2}$. Hence, this is omitted for the sake of brevity.  

The remainder and coefficient terms $P_{\mathcal{R},C}$ are omitted as they are easier or very similar. 
This completes the treatment of $NLP(1,3,0,\neq)$. Similarly, the treatment of $NLP(1,2,0,\neq)$ is more or less the same and hence this is also omitted for the sake of brevity. 

\paragraph{Treatment of $NLP(i,j,0,\neq)$ with $i \in \set{2,3}$} \label{sec:NLPQ2_0neq_notX} 

These terms constitute nonlinear interactions with the $Y$ and $Z$ components of the streak, which, since they are much smaller than the $X$ component, are expected to be easier to handle. 
On the other hand, due to the nonlinear structure, the $Y$, $Z$ contributions come with 
$\partial_Y^t$ and $\partial_Z^t$ (respectively) derivatives on the non-zero frequencies, which are more difficult to deal with than $\partial_X$ derivatives. 

We will demonstrate how to deal with these terms by the example of $NLP(2,3,0,\neq)$, which is one of the leading order terms. 
The remaining terms yield analogous contributions and are omitted for the sake of brevity. 
This term is quintic, however, as in the treatment of $NLP(1,3,0,\neq)$ above in \S\ref{sec:NLP213}, we will 
only deal with the presence of the variable coefficients when they appear in high frequency.  
Hence, expanding with a quintic paraproduct and focusing only on the leading order contributions gives
\begin{align*} 
NLP(2,3,0,\neq) & = \sum_{k \neq 0} \int A^2 Q^2_{k} A^2\left((\partial_Y-t\partial_X)((\partial_Y - t\partial_X)(U^3_{k})_{Hi} (\partial_Z^t U^2_0)_{Lo}) \right) dV \\ 
& \quad +\sum_{k \neq 0} \int A^2 Q^2_{k} A^2\left((\partial_Y-t\partial_X)((\partial_Y-t\partial_X)(U^3_{k})_{Lo} (\partial_Z^t U^2_0)_{Hi}) \right) dV \\
& \quad +\sum_{k \neq 0} \int A^2 Q^2_{k} A^2\left((\psi_y)_{Hi}(\partial_Y- t\partial_X)(\partial_Y^t (U^3_{k})_{Lo} (\partial_Z^t U^2_0)_{Lo}) \right) dV \\
& \quad +\sum_{k \neq 0} \int A^2 Q^2_{k} A^2\left( (\partial_Y - t\partial_X)((\psi_y)_{Hi}(\partial_Y-t\partial_X) (U^3_{k})_{Lo} (\partial_Z^t U^2_0)_{Lo}) \right) dV \\
& \quad + P_{\mathcal{R},C} \\  
& = P_{HL} + P_{LH} + P_{C1} + P_{C2} + P_{\mathcal{R},C}, 
\end{align*}
where the term $P_{\mathcal{R},C}$ contains the remainders from the quintic paraproducts and the higher order terms where the coefficients are in low frequency. 
Consider first $P_{HL}$. In this case, we have by \eqref{ineq:AprioriU0} and \eqref{ineq:ABasic}, 
\begin{align*} 
P_{HL} & \lesssim \epsilon\sum_{k \neq 0} \int \abs{A^2 \widehat{Q^2_{k}}(\eta,l) A^2_k(\eta,l) \abs{\eta - tk} \abs{\xi-tk} \widehat{U^3_{k}}(\xi,l^\prime)_{Hi} } Low(\eta-\xi,l-l^\prime) d\eta d\xi \\ 
& \lesssim \epsilon\sum_{k \neq 0} \int \abs{A^2 \widehat{Q^2_{k}}(\eta,l) \jap{\frac{t}{\jap{\xi,l^\prime}}} \frac{\abs{\eta - tk} \abs{\xi-tk}}{k^2 + (l^\prime)^2 + \abs{\xi-kt}^2} A^3 \Delta_L \widehat{U^3_{k}}(\xi,l^\prime)_{Hi}} Low(\eta-\xi,l-l^\prime) d\eta d\xi. 
\end{align*} 
Hence, by \eqref{ineq:ratlongtime} and \eqref{ineq:triQuadHL}, we get 
\begin{align*} 
P_{HL} & \lesssim \epsilon\norm{\sqrt{-\Delta_L}A^2Q^2}_2 \norm{\Delta_L A^3 U^3_{\neq}}_2, 
\end{align*} 
which, after Lemma \ref{lem:SimplePEL},  is consistent with Proposition \ref{prop:Boot} for $c_0$ sufficiently small. 

Turn next to $P_{LH}$. This term is treated as in the analogous term in $NLP(1,3,0,\neq)$, using that extra loss of time from the second $\partial_Y^t$ derivative replaces the gain in $t$ from the presence of $U_0^2$ as opposed to $U_0^1$. 
Indeed, by Lemma \ref{lem:ABasic} followed by \eqref{ineq:triQuadHL} we have
\begin{align*} 
P_{LH} & \lesssim \epsilon \norm{\sqrt{-\Delta_L} A^2 Q^2}_2^2 + \frac{\epsilon}{\jap{\nu t^3}^{\alpha}}\norm{AU_0^2}_2^2, 
\end{align*} 
which, after Lemma \ref{lem:PELbasicZero}, is consistent with Proposition \ref{prop:Boot}. 
 
Of the coefficient error terms, $P_{C2}$ is more difficult. 
By \eqref{ineq:AprioriU0}. \eqref{ineq:AprioriUneq}, and Lemma \ref{lem:ABasic} we have, 
\begin{align*}
P_{C2} & \lesssim \frac{\epsilon^2 \jap{t}}{\jap{\nu t^3}^{\alpha}} \sum \int \abs{A^2 \widehat{Q^2_k}(\eta,l) \frac{\abs{\eta-kt}}{\jap{\xi,l^\prime}^2 \jap{\frac{t}{\jap{\xi,l^\prime}}}}A \widehat{\psi_y}(\xi,l^\prime)_{Hi}} Low(k, \eta-\xi,l-l^\prime) d\xi d\eta. 
\end{align*} 
Then, by \eqref{ineq:triQuadHL} and Lemma \ref{lem:CoefCtrl} we get
\begin{align*} 
P_{C2} & \lesssim \frac{\epsilon^2}{\jap{\nu t^3}^\alpha} \norm{\sqrt{-\Delta_L} A^2 Q^2}_2 \norm{AC}_2 \lesssim \epsilon\norm{\sqrt{-\Delta_L} A^2 Q^2}_2^2 + \frac{\epsilon^{3}}{\jap{\nu t^3}^{2\alpha}}\norm{AC}^2_2,
\end{align*} 
which is consistent with Proposition \ref{prop:Boot} using \eqref{ineq:Boot_ACC2}. 
The treatment of $P_{C1}$ is analogous and hence omitted for the sake of brevity. 

The remainder terms are either very easy or similar to the ones we have already treated and are hence omitted. This completes the treatment of $NLP(2,3,0,\neq)$; other $NLP(i,j,0,\neq)$ with $i \neq 1$ terms are treated similarly and are hence omitted as well. 

\paragraph{Treatment of $NLP(i,j,\neq,\neq)$ terms} \label{sec:NLPQ2_neqneq} 
These are pressure interactions now of the type \textbf{(3DE)}. 
All of these terms can be treated in the same fashion, hence to fix ideas let us just focus on the case $i = 1$ and $j =3$ for simplicity.
This term is quartic, but as above, when we expand with the paraproduct we will group terms with the coefficients in low frequency with the remainder. 
Hence, we write
\begin{align*} 
NLP(1,3,\neq,\neq) & = \sum_{k} \int \mathbf{1}_{k k^\prime(k-k^\prime) \neq 0} A^2 Q^2_{k} A^2\left( (\partial_Y-t\partial_X)( (\partial_Z U^1_{k-k^\prime})_{Lo} (\partial_X U^3_{k^\prime})_{Hi} )\right) dV \\ 
& \quad + \sum_{k} \int \mathbf{1}_{k k^\prime(k-k^\prime) \neq 0}  A^2 Q^2_{k} A^2\left( (\partial_Y-t\partial_X)( (\partial_Z U^1_{k-k^\prime})_{Hi} (\partial_X U^3_{k^\prime})_{Lo} )\right) dV \\ 
& \quad + \sum_{k} \int \mathbf{1}_{k k^\prime(k-k^\prime) \neq 0} A^2 Q^2_{k} A^2\left( (\psi_y)_{Hi}(\partial_Y-t\partial_X)( (\partial_Z U^1_{k-k^\prime})_{Lo} (\partial_X U^3_{k^\prime})_{Lo} )\right) dV \\ 
& \quad + \sum_{k} \int \mathbf{1}_{k k^\prime(k-k^\prime) \neq 0} A^2 Q^2_{k} A^2\left( (\partial_Y-t\partial_X)( (\psi_z)_{Hi}(\partial_Y-t\partial_X)(U^1_{k-k^\prime})_{Lo} (\partial_X U^3_{k^\prime})_{Lo} )\right) dV \\ 
& \quad + P_{\mathcal{R},C} \\  
& = P_{LH} + P_{HL} + P_{C1} + P_{C2} + P_{\mathcal{R},C},
\end{align*}
where $P_{\mathcal{R},C}$ contains all of the remainders and the terms where coefficients appear in low frequency. 
By \eqref{ineq:AprioriUneq} and \eqref{ineq:ABasic} followed by \eqref{ineq:ratlongtime} and \eqref{ineq:triQuadHL}, we have
\begin{align*} 
P_{LH} & \lesssim \frac{\epsilon \jap{t}^{\delta_1}}{\jap{\nu t^3}^\alpha} \sum_{k}\int \abs{A^2 \widehat{Q^2_k}(\eta,l)A^2_k(\eta,l) (\eta-k t)k^\prime \widehat{U^3_{k^\prime}}(\xi,l^\prime)} Low(k-k^\prime,\eta-\xi,l-l^\prime) d\eta d\xi \\ 
& \lesssim \frac{\epsilon \jap{t}^{\delta_1}}{\jap{\nu t^3}^\alpha} \sum_{k}\int \abs{(\eta-k t)A^2 \widehat{Q^2_k}(\eta,l) \right. \\ &  \quad\quad\quad \left. \times \jap{\frac{t}{\jap{\xi,l^\prime}}}\frac{k^\prime}{\abs{k^\prime}^2 + \abs{l^\prime}^2 + \abs{\xi-k^\prime t}^2} \Delta_L A^3\widehat{U^3_{k^\prime}}(\xi,l^\prime)} Low(k-k^\prime,\eta-\xi,l-l^\prime) d\eta d\xi \\
& \lesssim \frac{\epsilon \jap{t}^{\delta_1}}{\jap{\nu t^3}^\alpha}\norm{\sqrt{-\Delta_L}A^2 Q^2_{\neq}}_2 \norm{\Delta_L A^3 U^3_{\neq}}_2,
\end{align*} 
which is consistent with Proposition \ref{prop:Boot} by the Lemma \ref{lem:SimplePEL} and the bootstrap hypotheses.
 By \eqref{ineq:AprioriUneq}, \eqref{ineq:ABasic}, and \eqref{ineq:AikDelL2D_CKw2} followed by \eqref{ineq:triQuadHL}, we have
\begin{align*} 
P_{HL} & \lesssim \frac{\epsilon}{\jap{\nu t^3}^\alpha} \sum_{k}\int \abs{A^2 \widehat{Q^2}_k(\eta,l)A^2_k(\eta,l) (\eta-k t)l^\prime \widehat{U^1}_{k^\prime}(\xi,l^\prime)} Low(k-k^\prime,\eta-\xi,l-l^\prime) d\eta d\xi \\ 
& \lesssim \frac{\epsilon}{\jap{\nu t^3}^\alpha} \sum_{k}\int \abs{A^2 \widehat{Q^2}_k(\eta,l) \jap{\frac{t}{\jap{\xi,l^\prime}}}^{\delta_1}\frac{\jap{t} (\eta-k t) l^\prime}{\abs{k^\prime}^2 + \abs{l^\prime}^2 + \abs{\xi-k^\prime t}^2} \Delta_L A^1\widehat{U^1}_{k^\prime}(\xi,l^\prime)} \\ & \quad\quad\quad \times Low(k-k^\prime,\eta-\xi,l-l^\prime) d\eta d\xi \\
& \lesssim \frac{\epsilon \jap{t}^{2}}{\jap{\nu t^3}^{\alpha}}\norm{\left(\sqrt{\frac{\partial_t w}{w}} + \frac{\abs{\grad}^{s/2}}{\jap{t}^{s}}\right)A^2 Q^2}_2\norm{\left(\sqrt{\frac{\partial_t w}{w}} + \frac{\abs{\grad}^{s/2}}{\jap{t}^{s}}\right)\Delta_L A^1 U^1_{\neq}}_2 \\ 
& \quad + \frac{\epsilon \jap{t}^{1+\delta_1}}{\jap{\nu t^3}^\alpha}\norm{A^2 Q^2_{\neq}}_2 \norm{\Delta_L A^1 U^1_{\neq}}_2,
\end{align*} 
which after Lemmas \ref{lem:PEL_NLP120neq} and \ref{lem:SimplePEL}, is consistent with Proposition \ref{prop:Boot} for $\delta_1$ and $\epsilon$ sufficiently small.  

The two coefficient error terms are treated by analogously to the treatment used in \S\ref{sec:NLP213} above. Hence we omit the details and conclude  
\begin{align*} 
P_{C1} + P_{C2} & \lesssim \frac{\epsilon^2 \jap{t}^{\delta_1}}{\jap{\nu t^3}^{2\alpha}}\norm{A^2 Q^2_{\neq}}_2 \norm{AC}_2 \lesssim \frac{\epsilon}{\jap{\nu t^3}^{2\alpha}}\norm{A^2 Q^2_{\neq}}_2^2 + \frac{\epsilon^{3} \jap{t}^{2\delta_1}}{\jap{\nu t^3}^{2\alpha}}\norm{AC}_2^2, 
\end{align*}
which is consistent with Proposition \ref{prop:Boot} for $\epsilon$ sufficiently small. 
As discussed above, the remainder terms $P_{\mathcal{R}.C}$ are much easier than the leading order terms, and these are hence omitted. 
This completes the treatment of $NLP(1,3,\neq,\neq)$. 
Other $i,j$ combinations can be treated via a simple variant of this. 

\subsubsection{Nonlinear stretching $NLS$} \label{sec:NLSQ2}

\paragraph{Treatment of $NLS1(j,0,\neq)$ and $NLS1(j,\neq,0)$} \label{sec:NLS1Q20neq}
The $NLS1(j,0,\neq)$ terms can essentially be treated in the same manner as the $NLP(j,2,0,\neq)$ nonlinear pressure terms in \S\ref{sec:NLPQ2_0neq_notX} and \S\ref{sec:NLP213} (though easier). We omit the details for brevity. 

Turn next to the $NLS1(j,\neq,0)$ terms. 
Notice that the $j = 1$ term disappears. The $j=3$ term is then the most dangerous remaining term. 
As usual we expand the term with a paraproduct and only keep the coefficients to leading order when they appear in high frequency. Hence we have
\begin{align*} 
NLS1(3,\neq,0) & = -\int A^2 Q^2 A^2\left( (Q^3_{\neq})_{Hi} (\partial_Z U^2_{0})_{Lo}\right) dV -  \int A^2 Q^2 A^2\left( (Q^3_{\neq})_{Lo} (\partial_Z U^2_{0})_{Hi}\right) dV  \\ 
& \quad - \int A^2 Q^2 A^2\left( (Q^3_{\neq})_{Lo} (\psi_z)_{Hi} (\partial_Y U^2_{0})_{Lo}\right) dV + S_{\mathcal{R},C} \\ 
& = S_{HL} + S_{LH} + S_{C} + S_{\mathcal{R},C}, 
\end{align*}
where $S_{\mathcal{R},C}$ contains the paraproduct remainders and the terms where the coefficients appear in low frequency. 
By \eqref{ineq:AprioriU0} and \eqref{ineq:ABasic} followed by \eqref{ineq:ratlongtime} and \eqref{ineq:triQuadHL}, we have
\begin{align*} 
S_{HL} & \lesssim \epsilon \sum_{k \neq 0}\int \abs{A^2 \widehat{Q^2_k}(\eta,l) \jap{\frac{t}{\jap{\xi,l^\prime}}}  A^3 \widehat{Q^3_{k}}(\xi,l^\prime)_{Hi}} Low(\eta-\xi,l-l^\prime) d\eta d\xi \\ 
& \lesssim \epsilon\norm{\sqrt{-\Delta_L}A^2 Q^2}_2\norm{A^3 Q^3_{\neq}}_2, 
\end{align*}
which is absorbed by the dissipation due to the non-zero frequencies for $c_0$ sufficiently small.  
For the $S_{LH}$ term we have by \eqref{ineq:Boot_ED}, Lemma \ref{lem:ABasic}, and \eqref{ineq:triQuadHL}
\begin{align*} 
S_{LH} & \lesssim   \frac{\epsilon \jap{t}^2}{\jap{\nu t^3}^\alpha} \sum_{k}\int \abs{A^2 \widehat{Q^2_k}(\eta,l) A^2_k(\eta,l) l^\prime \widehat{U^2_{0}}(\xi,l^\prime)_{Hi}} Low(\eta-\xi,l-l^\prime) d\eta d\xi \\ 
& \lesssim \frac{\epsilon \jap{t}}{\jap{\nu t^3}^\alpha} \sum_{k}\int \abs{A^2 \widehat{Q^2_k}(\eta,l)  A\widehat{U^2_{0}}(\xi,l^\prime)_{Hi}} Low(\eta-\xi,l-l^\prime) d\eta d\xi \\ 
& \lesssim \frac{\epsilon \jap{t}}{\jap{\nu t^3}^\alpha} \norm{A^2 Q^2}_2\norm{AU_0^2}_2, 
\end{align*}
which is consistent with Proposition \ref{prop:Boot} for $\epsilon$ sufficiently small. 
For the coefficient error term $S_C$, by \eqref{ineq:AprioriU0}, \eqref{ineq:AprioriUneq}, and Lemma \ref{lem:ABasic},  followed by \eqref{ineq:triQuadHL}, 
\begin{align*} 
S_C & \lesssim \frac{\epsilon^2 \jap{t}^2}{\jap{\nu t^3}^\alpha} \sum_{k}\int \abs{A^2 \widehat{Q^2_k}(\eta,l) \frac{1}{\jap{t} \jap{\xi,l^\prime} } A\widehat{\psi_y}(\xi,l^\prime)_{Hi}} Low(\eta-\xi,l-l^\prime) d\eta d\xi \\ 
& \lesssim \frac{\epsilon^2 \jap{t}}{\jap{\nu t^3}^\alpha} \norm{A^2 Q^2}_2\norm{\jap{\grad}^{-1} A\psi_y}_2 \\ 
& \lesssim \frac{\epsilon \jap{t}}{\jap{\nu t^3}^\alpha} \norm{A^2 Q^2}_2^2 + \frac{\epsilon^3 \jap{t}}{\jap{\nu t^3}^{\alpha}}\norm{AC}_2^2, 
\end{align*}
where the last line followed from Lemma \ref{lem:CoefCtrl}. This is consistent with Proposition \ref{prop:Boot} for $\epsilon$ sufficiently small by the bootstrap hypotheses. 

As usual, the remainders and coefficient error terms in $S_{\mathcal{R},C}$ are significantly easier to treat and hence 
are omitted for brevity. 
This completes the treatment of $NLS1(3,\neq,0)$; the other term, $NLS1(2,\neq,0)$ is easier and is treated the same way, hence we omit this for brevity.

\paragraph{Treatment of $NLS1(j,\neq,\neq)$} 
All of these terms are treated in essentially the same way; let us focus on $j = 3$ for brevity. 
As usual we expand the term with a paraproduct and only keep the coefficients to leading order when they appear in high frequency. Hence we have
\begin{align*}
NLS1(3,\neq,\neq) & = -\int A^2 Q^2 A^2\left( (Q^3_{\neq})_{Hi} (\partial_Z U^2_{\neq})_{Lo}\right) dV - \int A^2 Q^2 A^2\left( (Q^3_{\neq})_{Lo} (\partial_Z U^2_{\neq})_{Hi}\right) dV \\
& \quad - \int A^2 Q^2 A^2\left( (Q^3_{\neq})_{Lo} (\psi_z)_{Hi} ((\partial_Y - t\partial_X)U^2_{\neq})_{Lo}\right) dV  + S_{\mathcal{R},C} \\
& = S_{HL} + S_{LH} + S_{C} + S_{\mathcal{R},C}, 
\end{align*}
where $S_{\mathcal{R},C}$ contains the paraproduct remainders and the terms where the coefficients appear in low frequency. 
By \eqref{ineq:AprioriUneq}, \eqref{ineq:ABasic}, and \eqref{ineq:triQuadHL} we have
\begin{align*} 
S_{HL} & \lesssim \frac{\epsilon}{\jap{t}^{2-\delta_1}\jap{\nu t^3}^{\alpha}} \sum_{k}\int \abs{A^2 \widehat{Q^2_k}(\eta,l) \jap{\frac{t}{\jap{\xi,l^\prime}}}  A^3\widehat{Q^3_{k^\prime}}(\xi,l^\prime)_{Hi}} Low(k-k^\prime,\eta-\xi,l-l^\prime) d\eta d\xi \\ 
& \lesssim \frac{\epsilon}{\jap{t}^{1-\delta_1}\jap{\nu t^3}^{\alpha}} \norm{A^2Q^2}_2 \norm{A^3 Q^3}_2,
\end{align*}
which is consistent with Proposition \ref{prop:Boot} for $\epsilon$ sufficiently small. 
Turn next to the $S_{LH}$ term. By \eqref{ineq:Boot_Hi}, \eqref{ineq:ABasic}, \eqref{ineq:AiPartX} and \eqref{ineq:triQuadHL} we have
\begin{align*} 
S_{LH} & \lesssim \frac{\epsilon \jap{t}^2}{\jap{\nu t^3}^{\alpha}} \sum_{k}\int \abs{A^2 \widehat{Q^2_k}(\eta,l) \frac{\abs{l^\prime}}{\abs{k^\prime}^2 + \abs{l^\prime}^2 + \abs{\xi - k^\prime t}^2} \Delta_LA^2 \widehat{U^2_{k^\prime}}(\xi,l^\prime)_{Hi}} Low(k-k^\prime,\eta-\xi,l-l^\prime) d\eta d\xi \\ 
& \lesssim \frac{\epsilon \jap{t}^2}{\jap{\nu t^3}^{\alpha}}\norm{\left(\sqrt{\frac{\partial_t w}{w}} + \frac{\abs{\grad}^{s/2}}{\jap{t}^{s}}\right) A^2 Q^2}_2 \norm{\left(\sqrt{\frac{\partial_t w}{w}} + \frac{\abs{\grad}^{s/2}}{\jap{t}^{s}}\right)A^2 \Delta_L U^2_{\neq}}_2 \\ 
& \quad + \frac{\epsilon \jap{t}}{\jap{\nu t^3}^{\alpha}}\norm{A^2 Q^2}_2 \norm{A^2 \Delta_L U^2_{\neq}}_2, 
\end{align*}
which is consistent with Proposition \ref{prop:Boot} for $\epsilon$ small by Lemmas \ref{lem:SimplePEL} and \ref{lem:PEL_NLP120neq}.
For the coefficient error term is treated in the same fashion as the corresponding error term associated with $NLS1(3,\neq,0)$ in \S\ref{sec:NLS1Q20neq} above. 
Hence, the treatment is omitted. Similarly, the remainder and coefficient low frequency terms in $S_{\mathcal{R},C}$ are also omitted. 
This completes the treatment of the $NLS1(3,\neq,\neq)$ term; the other $NLS(i,j,\neq,\neq)$ terms are treated similarly. 

\paragraph{Treatment of $NLS2(i,1,0,\neq)$} \label{sec:NLS2i1neq0}
The non-zero contributions come from $i  = 2$ and $i = 3$; these are essentially identical so 
we will treat just consider the $i=2$ case.
Moreover, we will be able to follow existing treatments; in particular, we can essentially apply the same treatment here as in the treatment of $NLP(1,2,0,\neq)$ in \S\ref{sec:NLP213} (which was omitted since it is easier than the leading order $NLP(1,3,0,\neq)$). 

\paragraph{Treatment of $NLS2(i,j,0,\neq)$ with $j \neq 1$} \label{sec:NLS2ijneq0}
As with $NLS2(i,1,0,\neq)$, we will be able to adapt existing treatments in a straightforward manner. 
Note that $i \neq 1$ by the nonlinear structure. As all of the contributions are similar, so just consider the $i = 3$,  $j = 2$ case (note the $i = j = 2$ term cancels with the $NLP$ term).  
We expand with a paraproduct as usual; when $U^2_{\neq}$ is in high frequency, the treatment is analogous to the analogous $NLP(2,3,0,\neq)$ in \S\ref{sec:NLPQ2_0neq_notX} and when $U_0^2$ or the coefficients are in high frequency, we may adapt the methods in e.g. \S\ref{sec:NLP213}. 

\paragraph{Treatment of $NLS2(i,j,\neq,0)$}
Note that both $j \neq 1$ and $i \neq 1$ by the nonlinear structure.
It is straightforward to adapt the methods applied in e.g. \S\ref{sec:NLS2ijneq0} to this case and hence the details are omitted for brevity.

\paragraph{Treatment of $NLS2(i,j,\neq,\neq)$}
Each of these cases are the same up to small variations hence we focus on one representative example, $NLS2(2,3,\neq,\neq)$ and omit the rest. 
 As usual, expanding with a paraproduct and keeping leaving all of the terms where the coefficients appear in low frequency in the remainder, we have
\begin{align*} 
NLS2(2,3,0,\neq) & = -\int A^2 Q^2 A^2\left( (\partial_Y - t\partial_X)(U^3_{\neq})_{Lo} \partial_Z(\partial_Y - t\partial_X) (U^2_{\neq})_{Hi} \right) dV \\ 
& \quad - \int A^2 Q^2 A^2\left( (\partial_Y - t\partial_X)(U^3_{\neq})_{Hi} \partial_Z(\partial_Y - t\partial_X)(U^2_{\neq})_{Lo} \right) dV \\
& \quad + S_{C1} + S_{C2} + S_{C3} + S_{\mathcal{R},C} \\   
& = S_{LH} + S_{HL} + S_{C1} + S_{C2} + S_{C3} + S_{\mathcal{R},C},  
\end{align*} 
where we are omitting the full definitions of $S_{Ci}$ since they are analogous to previous cases in e.g. $NLP$ and can be treated in similar manners.  
By \eqref{ineq:AprioriUneq}, \eqref{ineq:ABasic}, and \eqref{ineq:triQuadHL}. 
\begin{align*} 
S_{LH} & \lesssim \frac{\epsilon\jap{t}}{\jap{\nu t^3}^{\alpha}}\norm{A^2 Q^2_{\neq}}_2\norm{A^2 \Delta_L U^2_{\neq}}_2,
\end{align*} 
which is consistent with Proposition \ref{prop:Boot} for $\epsilon$ sufficiently small. 
A similar treatment for $S_{HL}$, combined with the additional inviscid damping on $U^2$ and \eqref{ineq:ratlongtime}, gives
\begin{align*} 
S_{HL} & \lesssim \frac{\epsilon}{\jap{t}^{1-\delta_1}\jap{\nu t^3}^{\alpha}}\norm{A^2 Q^2_{\neq}}_2\norm{A^3 \Delta_L U^3_{\neq}}_2,
\end{align*} 
which is consistent with Proposition \ref{prop:Boot} for $\epsilon$ sufficiently small. 
As mentioned above, the coefficient error terms and the remainder terms are simpler and are omitted for the sake of brevity, hence this completes the treatment of $NLS(2,3,\neq,\neq)$. The other $NLS2(i,j,\neq,\neq)$ contributions follow similarly.  

\subsubsection{Transport nonlinearity} \label{sec:Q2_TransNon}
Next we treat the contribution of the transport nonlinearity to the evolution of non-zero frequencies. 
Begin with a paraproduct decomposition: 
\begin{align*} 
\mathcal{T}_{\neq} & = -\int A^{2} Q^2_{\neq} A^{2} \left( \tilde U_{Lo} \cdot \grad Q^2_{Hi} \right) dV  -\int A^{2} Q^2_{\neq} A^{2} \left( \tilde U_{Hi} \cdot \grad Q^2_{Lo} \right) dV - \int A^{2} Q^2_{\neq} A^{2} \left( \tilde U \cdot \grad Q^2 \right)_{\mathcal{R}} dV \\  
& = \mathcal{T}_T + \mathcal{T}_{R} + \mathcal{T}_{\mathcal{R}},
\end{align*}
where `T' and `R' stand for `transport' and `reaction' respectively, keeping with the terminology used in \cite{BM13} (the terminology `reaction' traces back further to \cite{MouhotVillani11}).
Decompose the transport and reaction terms into subcomponents depending on the $X$ frequencies: 
\begin{align*} 
\mathcal{T}_{T}  & = -\int A^{2} Q^2_{\neq} A^{2} \left( (\tilde U_{\neq})_{Lo}  \cdot (\grad Q^2_0)_{Hi} \right) dV - \int A^{2}_{\neq} Q^2 A^{2} \left( (\tilde U_0)_{Lo} \cdot \grad (Q^2_{\neq})_{Hi} \right) dV \\ & \quad - \int A^{2} Q^2_{\neq} A^{2} \left( (\tilde U_{\neq})_{Lo} \cdot \grad (Q^2_{\neq})_{Hi} \right) dV \\ 
& = \mathcal{T}_{T;\neq 0}+ \mathcal{T}_{T;0 \neq}+ \mathcal{T}_{T;\neq \neq}, 
\end{align*} 
and,
\begin{align*} 
\mathcal{T}_{R} & =  -\int A^{2} Q^2_{\neq} A^{2} \left( (\tilde U_{\neq})_{Hi}  \cdot (\grad Q^2_0)_{Lo} \right) dV - \int A^{2} Q^2_{\neq} A^{2} \left( (\tilde U_0)_{Hi} \cdot \grad (Q^2_{\neq})_{Lo} \right) dV \\ & \quad - \int A^{2} Q^2_{\neq} A^{2} \left( (\tilde U_{\neq})_{Hi} \cdot \grad (Q^2_{\neq})_{Lo} \right) dV \\ 
& = \mathcal{T}_{R;\neq 0} + \mathcal{T}_{R;0 \neq}+ \mathcal{T}_{R;\neq \neq}. 
\end{align*} 

\paragraph{Transport by zero frequencies: $\mathcal{T}_{T;0 \neq}$}
Turn first to $\mathcal{T}_{T;0 \neq}$, which is the transport by zero frequencies. 
Write 
\begin{align*} 
\mathcal{T}_{T;0 \neq} & = -\int A^2 Q^2_{\neq} A^2 \left(g_{Lo} \partial_Y (Q^2_{\neq})_{Hi} \right) dV - \int A^2 Q^2_{\neq} A^2 \left( (U_0^3)_{Lo} \partial_Z (Q^2_{\neq})_{Hi} \right) dV \\ 
& = \mathcal{T}_{T;0 \neq}^g + \mathcal{T}_{T;0 \neq}^U. 
\end{align*} 
On the Fourier side, 
\begin{align*} 
\mathcal{T}_{T;0 \neq}^g & \lesssim \sum_{k \neq 0} \sum_{l,l^\prime} \int \abs{A^2 \widehat{Q^2_k}(\eta,l) A_k^2(\eta,l)\hat{g}(\eta-\xi,l-l^\prime)_{Lo} \xi \widehat{Q^2_k}(\xi,l^\prime)_{Hi}} d\eta d\xi. 
\end{align*} 
Therefore by \eqref{ineq:ABasic}, $\abs{\xi} \leq \abs{\xi - kt} + \abs{kt}$, and \eqref{ineq:triQuadHL}, we have
\begin{align*} 
\mathcal{T}_{T;0 \neq}^g & \lesssim \norm{g}_{\G^{\lambda,3/2+}}\norm{A^2 Q^2_{\neq}}_2 \left(\norm{(\partial_Y - t\partial_X) A^2 Q^2}_2 + t\norm{\partial_X A^2 Q^2}_2\right) \\
& \lesssim \jap{t}\norm{g}_{\G^{\lambda,3/2+}}\norm{A^2 Q^2_{\neq}}_2\norm{\sqrt{-\Delta_L} A^2 Q^2_{\neq}}_2 \\ 
& \lesssim \epsilon\norm{\sqrt{-\Delta_L} A^2 Q^2_{\neq}}_2^2 + \frac{\epsilon}{\jap{t}^{2}}\norm{A^2 Q^2_{\neq}}_2^2, 
\end{align*} 
where the last line followed from the low norm control on $g$, \eqref{ineq:Boot_gLow}. 
The first factor is absorbed by the dissipation for $c_0$ small and the second factor integrates to something consistent with Proposition \ref{prop:Boot} for $\epsilon$ sufficiently small. 
To treat $\mathcal{T}_{T;0 \neq}^U$, we have by \eqref{ineq:ABasic} and \eqref{ineq:triQuadHL},
\begin{align*} 
\mathcal{T}_{T;0 \neq}^U & \lesssim \norm{U_0^3}_{\G^{\lambda,3/2+}}\norm{A^2 Q^2_{\neq}}_2\norm{\sqrt{-\Delta_L} A^2 Q^2_{\neq}}_2 \lesssim \epsilon \norm{\sqrt{-\Delta_L} A^2 Q^2_{\neq}}_2^2, 
\end{align*} 
where the last inequality followed from \eqref{ineq:AprioriU0}. This contribution is then absorbed by the dissipation for $c_0$ small.

\paragraph{Transport by non-zero frequencies, $\mathcal{T}_{T;\neq \neq}$ and $\mathcal{T}_{T;\neq 0}$}
Turn next to $\mathcal{T}_{T;\neq \neq}$. 
By \eqref{def:tildeU2}, we can write this term more naturally with a $\grad^t$,
\begin{align*} 
\mathcal{T}_{T;\neq \neq} & =  -\int A^2 Q^2_{\neq} A^2_k \left( \begin{pmatrix} (U_{\neq}^1)_{Lo}  \\ \left((1+\psi_y)U^2_{\neq}\right)_{Lo} + \left(\psi_zU^3_{\neq}\right)_{Lo} \\ (U^3_{\neq})_{Lo} \end{pmatrix} \cdot \begin{pmatrix} \partial_X  \\ \partial_Y - t\partial_X \\ \partial_Z \end{pmatrix} (Q_{\neq}^2)_{Hi} \right)  dV. 
\end{align*} 
The presence of the coefficients is irrelevant by Lemma \ref{lem:GevProdAlg} and Lemma \ref{lem:CoefCtrl} so let us ignore them. 
By \eqref{ineq:ABasic}, \eqref{ineq:triQuadHL}. and \eqref{ineq:AprioriUneq} we have
\begin{align*} 
\mathcal{T}_{T;\neq \neq} & \lesssim \left(\norm{U^1_{\neq}}_{\G^{\lambda,3/2+}} +\norm{U^2_{\neq}}_{\G^{\lambda,3/2+}} + \norm{U^3_{\neq}}_{\G^{\lambda,3/2+}}\right)\norm{A^{2} Q^2}_2 \norm{\sqrt{-\Delta_L}A^{2} Q^2}_2 \\ 
& \lesssim \frac{\epsilon \jap{t}^{2\delta_{1}}}{\jap{\nu t^3}^{2\alpha}} \norm{A^{2} Q^2}^2_2 +  \epsilon\norm{\sqrt{-\Delta_L}A^{2} Q^2}^2_2, 
\end{align*} 
which is consistent with Proposition \ref{prop:Boot} for $c_0$ and $\epsilon$ sufficiently small. 
The contribution from $\mathcal{T}_{T;\neq 0}$ is treated similarly and yields 
\begin{align*} 
\mathcal{T}_{T;\neq 0} & \lesssim \frac{\epsilon \jap{t}^{\delta_{1}}}{\jap{\nu t^3}^{\alpha}}\norm{A^2 Q^2}_2\norm{\grad A^2 Q^2_0}_2  \lesssim \epsilon \| \nabla A^2 Q^2_0 \|_2^2 + \frac{\epsilon \jap{t}^{2 \delta_1}}{\jap{\nu t^3}^{2 \alpha}} \|A^2 Q^2 \|_2^2.
\end{align*}
This completes the treatment of the `transport' contribution to the transport nonlinearity. 

\paragraph{Reaction term $\mathcal{T}_{R;0 \neq}$}
Turn next to $\mathcal{T}_{R;0 \neq}$. 
By Lemma \ref{lem:ABasic} and \eqref{ineq:triQuadHL}, 
\begin{align*} 
\mathcal{T}_{R;0,\neq} & \lesssim \left(\norm{Ag}_2 + \norm{AU_0^3}\right)\norm{A^2 Q^2_{\neq}}_2^2 \lesssim \epsilon\norm{\sqrt{-\Delta_L}A^2 Q^2}_2^2, 
\end{align*} 
where the last inequality followed from \eqref{ineq:Boot_Ag} and \eqref{ineq:AprioriU0}. 
This contribution is then absorbed by the dissipation for $c_0$ sufficiently small.  

\paragraph{Reaction term $\mathcal{T}_{R;\neq 0}$} \label{sec:Q2TRneq0}
Turn next to the non-zero mode reaction contributions, which are the source of several subtle and important difficulties. 
First consider $\mathcal{T}_{R;\neq 0}$, which must be further decomposed on the Fourier side; note we will have to keep more careful track of the low frequency factors here: 
\begin{align*} 
\mathcal{T}_{R;\neq 0} & \lesssim  \sum_{k \neq 0}\int \abs{A^{2} \widehat{Q^2_k}(\eta,l) A^{2}_{k}(\eta,l) \widehat{U_k^2} (\xi,l^\prime)_{Hi} \widehat{\left((1+\psi_y)\partial_Y Q^2_0\right)}(\eta-\xi,l-l^\prime)_{Lo}} d\eta d\xi \\ 
& \quad + \sum_{k \neq 0}\int \abs{A^{2} \widehat{Q^2_k}(\eta,l) A^{2}_{k}(\eta,l) \widehat{U_k^3} (\xi,l^\prime)_{Hi} \widehat{\left( (\psi_z\partial_Y + \partial_Z) Q^2_0\right)}(\eta-\xi,l-l^\prime)_{Lo}} d\eta d\xi \\ 
& \quad + \sum_{k \neq 0}\int \abs{A^{2} \widehat{Q^2_k}(\eta,l) A^{2}_{k}(\eta,l) (\widehat{\psi_y})(\xi,l^\prime)_{Hi} \widehat{ \left(U_k^2\partial_Y Q_0^2 \right) }(\eta-\xi,l-l^\prime)_{Lo}} d\eta d\xi \\ 
& \quad + \sum_{k \neq 0}\int \abs{A^{2} \widehat{Q^2_k}(\eta,l) A^{2}_{k}(\eta,l) (\widehat{\psi_z})(\xi,l^\prime)_{Hi} \widehat{ \left(U_k^3\partial_Y Q_0^2 \right) }(\eta-\xi,l-l^\prime)_{Lo}} d\eta d\xi \\ & \quad + \mathcal{T}_{R;\neq 0;\mathcal{R}} \\ 
& = \mathcal{T}_{R;\neq 0;2} + \mathcal{T}_{R;\neq 0;3} + \mathcal{T}_{R;\neq 0;C1} + \mathcal{T}_{R;\neq 0;C2}+ \mathcal{T}_{R;\neq 0;\mathcal{R}}. 
\end{align*} 
Turn first to $\mathcal{T}_{R;\neq,0;2}$. By \eqref{ineq:ABasic}, \eqref{ineq:quadHL} and \eqref{lem:GevProdAlg} (also Lemma \ref{lem:CoefCtrl} and \eqref{ineq:Boot_LowC}): 
\begin{align*} 
\mathcal{T}_{R;\neq 0;2} & \lesssim \norm{A^2 Q^2_{\neq}}_2 \norm{A^2 U^2_{\neq}}_2 \norm{\grad Q^2_0}_{\G^{\lambda,3/2+}} \\ 
& \lesssim \epsilon \norm{A^2 Q^2_{\neq}}_2 \norm{\grad Q^2_0}_{\G^{\lambda,3/2+}} \\ 
& \lesssim \epsilon\norm{\sqrt{-\Delta_L} A^2 Q^2}_2^2,
\end{align*}  
where the second to last line followed from Lemma \ref{lem:SimplePEL} and the bootstrap hypotheses. This is then absorbed by the dissipation. 
Turn next to $\mathcal{T}_{R;\neq,0;3}$. By \eqref{ineq:ABasic} and \eqref{ineq:ratlongtime}, 
\begin{align*} 
\mathcal{T}_{R;\neq 0;3} & \lesssim \sum_{k \neq 0}\int \abs{A^{2} \widehat{Q^2_k}(\eta,l) \frac{\jap{\frac{t}{\jap{\xi,l^\prime}}}}{k^2 + (l^\prime)^2 + \abs{\xi-kt}^2} \Delta_L A^3\widehat{U_k^3} (\xi,l^\prime)_{Hi} \right. \\ & \quad\quad \times \left. e^{c\lambda\abs{\eta-\xi,l-l^\prime}^s} \widehat{\left( (\psi_z\partial_Y + \partial_Z) Q^2_0\right)}(\eta-\xi,l-l^\prime)_{Lo}} d\eta d\xi \\  
&  \lesssim \sum_{k \neq 0}\int \abs{A^{2} \widehat{Q^2_k}(\eta,l) \Delta_L A^3\widehat{U_k^3} (\xi,l^\prime)_{Hi} \right. \\ & \quad\quad \times \left. e^{c\lambda\abs{\eta-\xi,l-l^\prime}^s} \widehat{\left( (\psi_z\partial_Y + \partial_Z) Q^2_0\right)}(\eta-\xi,l-l^\prime)_{Lo}} d\eta d\xi. 
\end{align*} 
Then, by \eqref{ineq:quadHL} and \eqref{lem:GevProdAlg} (also Lemma \ref{lem:CoefCtrl} and the bootstrap hypotheses) we have
\begin{align*} 
\mathcal{T}_{R;\neq 0;3} & \lesssim \norm{A^2 Q^2_{\neq}}\norm{A^3 \Delta_L U^3_{\neq}}_2 \norm{\grad Q^2_{0}}_{\G^{\lambda,5/2+}} \\  
& \lesssim \norm{A^2 Q^2_{\neq}}\norm{A^3 \Delta_L U^3_{\neq}}_2 \norm{\sqrt{-\Delta_L} A^2 Q^2}_2 \\ 
& \lesssim \epsilon\norm{\sqrt{-\Delta_L} A^2 Q^2}_2^2 + \epsilon\norm{A^3 \Delta_L U^3_{\neq}}_2^2. 
\end{align*} 
After Lemma \ref{lem:SimplePEL}, this is consistent with Proposition \ref{prop:Boot} for $c_0$ sufficiently small. 
Of the two coefficient terms $\mathcal{T}_{R;\neq,0;C1}$ and $\mathcal{T}_{R;\neq,0;C2}$, the second is more difficult as $U^3$ is much larger than $U^2$ due to the inviscid damping on the latter. Hence, we focus only on the latter and omit the former for brevity; it yields similar results. 
From Lemma \ref{lem:ABasic} and \eqref{ineq:quadHL}, we have
\begin{align*} 
\mathcal{T}_{R;\neq 0;C2} & \lesssim \frac{\epsilon^2}{\jap{\nu t^3}^\alpha}\sum_{k \neq 0}\int \abs{A^{2} \widehat{Q^2_k}(\eta,l)\frac{1}{\jap{\xi,l^\prime}^2 \jap{\frac{t}{\jap{\xi,l^\prime}}}} A \hat{\psi_z}(\xi,l^\prime)_{Hi} Low(k,\eta-\xi,l-l^\prime)} d\eta d\xi \\  
& \lesssim \frac{\epsilon^2}{\jap{t}\jap{\nu t^3}^{\alpha}}\norm{A^2 Q^2_{\neq}}_2 \norm{\jap{\grad}^{-1}A\psi_z}_2 \\ 
& \lesssim \frac{\epsilon}{\jap{\nu t^3}^\alpha}\norm{A^2 Q^2_{\neq}}_2^2 + \frac{\epsilon^3}{\jap{t}^2\jap{\nu t^3}^\alpha}\norm{AC}_2^2, 
\end{align*}   
where the last line followed from Lemma \ref{lem:CoefCtrl}. This is consistent with Proposition \ref{prop:Boot} by the bootstrap hypotheses. 
Finally the remainder term $\mathcal{T}_{R;\neq 0;\mathcal{R}}$ is straightforward due to the presence of non-zero frequencies; we omit the treatment for brevity. 
This completes the treatment of $\mathcal{T}_{R;\neq 0}$. 

\paragraph{Reaction term $\mathcal{T}_{R; \neq \neq}$} \label{sec:Q2TRneqneq}
Turn finally to $\mathcal{T}_{R;\neq \neq}$, which are the most difficult. 
As in the treatment of  $\mathcal{T}_{R;\neq 0}$ above in \S\ref{sec:Q2TRneq0} we further decompose in terms of frequency and component: 
\begin{align*} 
\mathcal{T}_{R;\neq \neq} & \lesssim \frac{\epsilon \jap{t}^{\delta_1}}{\jap{\nu t^3}^\alpha} \sum_{k,k^\prime}\int \mathbf{1}_{k k^\prime(k-k^\prime) \neq 0} \abs{A^{2} \widehat{Q^2_k}(\eta,l) A^{2}_{k}(\eta,l) \widehat{U_{k^\prime}^1} (\xi,l^\prime)_{Hi}} Low(k-k^\prime, \eta-\xi, l - l^\prime) d\eta d\xi \\   
& \quad +  \frac{\epsilon \jap{t}^{1+\delta_1}}{\jap{\nu t^3}^\alpha} \sum_{k,k^\prime}\int \mathbf{1}_{k k^\prime(k-k^\prime) \neq 0} \abs{A^{2} \widehat{Q^2_k}(\eta,l) A^{2}_{k}(\eta,l) \widehat{U_{k^\prime}^2} (\xi,l^\prime)_{Hi}} Low(k-k^\prime, \eta-\xi, l - l^\prime) d\eta d\xi \\  
& \quad +  \frac{\epsilon \jap{t}^{\delta_1}}{\jap{\nu t^3}^{\alpha-1}} \sum_{k,k^\prime}\int \mathbf{1}_{k k^\prime(k-k^\prime) \neq 0} \abs{A^{2} \widehat{Q^2_k}(\eta,l) A^{2}_{k}(\eta,l) \widehat{U_{k^\prime}^3} (\xi,l^\prime)_{Hi}} Low(k-k^\prime, \eta-\xi, l - l^\prime) d\eta d\xi \\   
& \quad + \frac{\epsilon^2}{\jap{t}^{1 -2\delta_1} \jap{\nu t^3}^{\alpha}} \sum_{k,k^\prime}\int \mathbf{1}_{k k^\prime(k-k^\prime) \neq 0} \abs{A^{2} \widehat{Q^2_k}(\eta,l) A^{2}_{k}(\eta,l) \widehat{\psi_y}(\xi,l^\prime)_{Hi}} Low(k-k^\prime, \eta-\xi, l - l^\prime) d\eta d\xi \\   
& \quad + \frac{\epsilon^2 \jap{t}^{1+\delta_1}}{\jap{\nu t^3}^{\alpha}} \sum_{k,k^\prime}\int \mathbf{1}_{k k^\prime(k-k^\prime) \neq 0} \abs{A^{2} \widehat{Q^2_k}(\eta,l) A^{2}_{k}(\eta,l) \widehat{\psi_z}(\xi,l^\prime)_{Hi}} Low(k-k^\prime, \eta-\xi, l - l^\prime) d\eta d\xi \\   
& \quad + \mathcal{T}_{R;\neq \neq;\mathcal{R}} \\ 
& = \mathcal{T}_{R;\neq\neq}^{1} + \mathcal{T}_{R;\neq\neq}^2 + \mathcal{T}_{R;\neq\neq}^3  + \mathcal{T}_{R;\neq\neq}^{C1} + \mathcal{T}_{R;\neq\neq}^{C2} + \mathcal{T}_{R;\neq\neq;\mathcal{R}}. 
\end{align*} 
Note that here we have applied $\epsilon t \jap{\nu t^3}^{-1} \lesssim t^{-2}$ to reduce the power of time of the $(U^3)_{Hi} \left(\psi_z(\partial_Y-t\partial_X)Q^2\right)_{Lo}$ term. 

We may treat the leading order terms together. In the case of $\mathcal{T}_{R;\neq\neq}^{1}$, a power of $t$ is lost from Lemma \ref{lem:ABasic} when exchanging $A^2$ for $A^1$ and in $\mathcal{T}_{R;\neq\neq}^{2}$ a power of $t$ is lost from the low frequencies. For $\mathcal{T}_{R;\neq\neq}^{3}$ a power of $t$ is lost exchanging $A^2$ for $A^3$ in Lemma \ref{lem:ABasic} (for $t \gtrsim \abs{\xi,l^\prime}$). 
Together then we have for any $j \in \set{1,2,3}$, from \eqref{ineq:ABasic}, \eqref{ineq:ratlongtime} and \eqref{ineq:triQuadHL}, 
\begin{align*}
\mathcal{T}_{R;\neq \neq}^{j} & \lesssim \frac{\epsilon \jap{t}^{1+\delta_1}}{\jap{\nu t^3}^{\alpha-1}} \sum \int \abs{A^{2} \widehat{Q^2_k}(\eta,l)} \frac{1}{(k^\prime)^2 + (l^\prime)^2 + \abs{\xi - k^\prime t}^2} \jap{\frac{t}{\jap{\xi,l^\prime}}}^{\delta_1} \abs{A^{j} \Delta_L \widehat{U^j_{k^\prime}}(\xi,l^\prime)_{Hi}} \\ & \quad\quad\quad \times Low(k-k^\prime,\eta-\xi,l-l^\prime) d\eta d\xi \\ 
& \lesssim \frac{\epsilon \jap{t}^{1+\delta_1}}{\jap{\nu t^3}^{\alpha-1}} \norm{A^2 Q^2_{\neq}}_2 \norm{A^j \Delta_L U^j_{\neq}}_2. 
\end{align*}
which by Lemma \ref{lem:SimplePEL}, is consistent with Proposition \ref{prop:Boot} by the bootstrap hypotheses for $\epsilon$ and $\delta_1$ sufficiently small. 
The coefficient error terms are treated the same as in \S\ref{sec:Q2TRneq0}; hence we omit the treatments for brevity. 
The remainder terms $\mathcal{T}_{R;\neq\neq}$ are similarly straightforward and are omitted for brevity as well. 
This completes the treatment of the transport nonlinearity for $Q^2$. 

\subsubsection{Dissipation error terms $\mathcal{D}$} \label{sec:DEneqQ2}
Recalling the dissipation error terms and the short-hand \eqref{def:G}, we have 
\begin{align*} 
\mathcal{D}_E & = \nu\sum_{k \neq 0}\int A^2 Q^2_k A^2_k\left( G(\partial_{Y} - t\partial_X)^2 Q^2_k + 2\psi_z(\partial_Y - t \partial_X)\partial_{Z}Q^2_k \right) dV \\ 
& = \mathcal{D}_E^{1} + \mathcal{D}_E^{2}. 
\end{align*} 
The first is strictly harder than the latter, so we focus only on $\mathcal{D}_E^1$; $\mathcal{D}_E^2$ is treated in the same fashion and results in similar contributions.   
We expand via paraproduct 
\begin{align*} 
\mathcal{D}_E^1 & = \nu\sum_{k \neq 0}\int A^2 Q^2_k A^2_k\left(G_{Hi}(\partial_{Y} - t\partial_X)^2 (Q^2_k)_{Lo}\right) dV \\ 
& \quad + \nu\sum_{k \neq 0}\int A^2 Q^2_k A^2_k\left(G_{Lo}(\partial_{Y} - t\partial_X)^2 (Q^2_k)_{Hi}\right) dV \\ 
& \quad + \nu\sum_{k \neq 0}\int A^2 Q^2_k A^2_k\left( \left(G(\partial_{Y} - t\partial_X)^2 Q^2_k\right)_{\mathcal{R}}\right) dV \\ 
& = \mathcal{D}_{E;HL}^1 + \mathcal{D}_{E;LH}^1 + \mathcal{D}_{E;\mathcal{R}}^1. 
\end{align*}
Via \eqref{ineq:TriTriv}, \eqref{ineq:ABasic}, and \eqref{ineq:triQuadHL} we have
\begin{align*} 
\mathcal{D}_{E;LH}^1 & \lesssim \nu\norm{C}_{\G^{\lambda,3/2+}}\norm{\sqrt{-\Delta_L} A^2 Q^2}_2^2 \lesssim c_0\nu\norm{\sqrt{-\Delta_L} A^2 Q^2}_2^2,  
\end{align*}
which is absorbed by the dissipation for $c_0$ sufficiently small. 
The remainder terms are treated similarly (though with \eqref{ineq:ARemainderBasic}, and \eqref{ineq:quadR}) and yield the same contribution. 

Turn next to $\mathcal{D}_{E;HL}^1$, where the coefficient is in high frequency. 
For this term we may use a straightforward treatment: from \eqref{ineq:AprioriUneq} and Lemma \ref{lem:ABasic}, followed by \eqref{ineq:triQuadHL} and Lemma \ref{lem:CoefCtrl},  
\begin{align*}
\mathcal{D}_{E;HL}^{1} & \lesssim \frac{\nu \epsilon \jap{t}^{2+\delta_1}}{\jap{\nu t^3}^{\alpha}} \sum_{k \neq 0} \int \abs{A^2 \widehat{Q^2_k}(\eta,l) \frac{1}{\jap{\xi,l^\prime} \jap{t}} A\widehat{G}(\xi,l^\prime)_{Hi}} Low(\eta-\xi,l-l^\prime) d\eta d\xi \\ 
& \lesssim \frac{\nu \epsilon \jap{t}^{1+\delta_1}}{\jap{\nu t^3}^{\alpha}} \norm{A^2 Q^2_{\neq}}_2 \norm{AC}_2 \\ 
& \lesssim \frac{\epsilon}{\jap{\nu t^3}^\alpha} \norm{A^2 Q^2_{\neq}}_2^2 + \frac{\nu^2 \epsilon \jap{t}^{2+2\delta_1}}{\jap{\nu t^3}^{\alpha}} \norm{AC}_2^2
\end{align*}
which is consistent with Proposition \ref{prop:Boot} for $\epsilon$ sufficiently small. 

\section{High norm estimate on $Q^3$}
Computing the evolution of $A^{3}Q^3$:  
\begin{align} 
\frac{1}{2}\frac{d}{dt}\norm{A^{3} Q^3}_2^2 & \leq \dot{\lambda}\norm{\abs{\grad}^{s/2}A^{3} Q^3}_2^2 - \norm{\sqrt{\frac{\partial_t w}{w}}A^{3} Q^3}_2^2 
-\norm{\sqrt{\frac{\partial_t w_L}{w_L}} A^{3} Q^3}_2^2 
-\frac{2}{t}\norm{\mathbf{1}_{t > \jap{\grad_{Y,Z}}} A^{3} Q^3}_2^2\nonumber \\
& \quad -2 \int A^{3} Q^3 A^{3} \partial_{YX}^t U^3 dV + 2 \int A^{3} Q^3 A^{3} \partial_{ZX}^t U^2 dV \nonumber\\  & \quad
 + \nu \int A^{3} Q^{3} A^{3} \left(\tilde{\Delta_t} Q^3\right) dV 
-\int A^{3} Q^3 A^{3}\left( \tilde U \cdot \grad Q^3 \right) dv \nonumber \\ 
& \quad - \int A^{3} Q^3 A^{3} \left[ Q^j \partial_j^t U^3 + 2\partial_i^t U^j \partial_{ij}^t U^3  - \partial_Z^t\left(\partial_i^t U^j \partial_j^t U^i\right) \right] dV \nonumber \\ 
& = \mathcal{D}Q^3 - CK_{L}^3 + LS + LP + \mathcal{D}_E + \mathcal{T} + NLS1 + NLS2 + NLP,  \label{ineq:AQ3_Evo} 
\end{align} 
where we are again using 
\begin{align*}
\mathcal{D}_E = \nu \int A^{3} Q^{3} A^{3} \left((\tilde{\Delta_t} - \Delta_L) Q^3\right) dV. 
\end{align*}

As in \eqref{def:Q2Enums}, let us here introduce the following enumerations: for $i,j\in \set{1,2,3}$ and $a,b \in \set{0,\neq}$: 
\begin{subequations} \label{def:Q3Enums}
\begin{align}
NLP(i,j,a,b) &= \int A^3 Q^3_{\neq} A^3\left( \partial_Z^t \left(\partial_j^t U^i_a \partial_i^t U^j_b \right) \right) dV \\ 
NLS1(j,a,b) & = -\int A^3 Q^3_{\neq} A^3\left(Q^j_a\partial_j^t U^3_{b}\right) dV \\
NLS2(i,j,a,b) & = -\int A^3 Q^3_{\neq} A^3\left(\partial_i^t U^j_a \partial_i^t\partial_j^t U^3_{b}\right) dV \\
NLP(i,j,0) & = \int A^3 Q^3_{0} A^3\left( \partial_Z^t \left(\partial_j^t U^i_0 \partial_i^t U^j_0 \right) \right) dV \\ 
NLS1(j,0) & = -\int A^3 Q^3_{0} A^3\left(Q^j_0\partial_j^t U^3_{0}\right) dV \\
NLS2(i,j,0) & = -\int A^3 Q^3_{0} A^3\left(\partial_i^t U^j_0 \partial_i^t\partial_j^t U^3_{0}\right) dV \\ 
\mathcal{F} & = -\int A^3 Q^3_{0} A^3\left(\partial_i^t \partial_i^t \partial_j^t \left( U^j_{\neq} U^3_{\neq} \right)_0 - \partial_{Z}^t \partial_j^t \partial_i^t \left( U^i_{\neq} U^j_{\neq}\right)_0 \right)dV \\ 
\mathcal{T}_0 & = -\int A^{3} Q_0^3 A^{3} \left( \tilde U_{0} \cdot \grad Q_{0}^3 \right) dV \\ 
\mathcal{T}_{\neq} & = -\int A^{2} Q_{\neq}^3 A^{3} \left( \tilde U \cdot \grad Q^3 \right) dV. 
\end{align}
\end{subequations} 
 
\subsection{Zero frequencies}
As in treatment of $A^2Q^2$ in \S\ref{sec:AQ2Zero}, the estimate on $Q_0^3$ is very different from the estimate on $Q^3_{\neq}$ and are hence naturally separated. 

\subsubsection{Transport nonlinearity} 
The treatment of the zero frequency transport nonlinearity in \eqref{ineq:AQ3_Evo}, denoted $\mathcal{T}_0$ in \eqref{def:Q3Enums}, goes through exactly the same as the corresponding treatment for $Q_0^2$ in \S\ref{sec:TransQ20} and hence, for the sake of brevity this term is omitted.

\subsubsection{Nonlinear pressure and stretching}
For these terms ($NLS1(j,0), NLS2(i,j,0), NLP(i,j,0)$ in \eqref{def:Q3Enums})we may use the same treatment here as we used in \S\ref{sec:NLPSQ20} without any significant changes.  
The treatment is omitted for the sake of brevity. 

\subsubsection{Forcing from non-zero frequencies} \label{sec:NzeroForcingQ3}
Turn next to the treatment of $\mathcal{F}$ (defined above in \eqref{def:Q3Enums}).  
These are interactions of type \textbf{(F)}, and consistent with \S\ref{sec:Toy}, we will see that the effects on $Q^3_0$ are more extreme than those on $Q_0^2$. 
Write
\begin{align*} 
\mathcal{F} & = -\int A^3 Q^3 A^3 \left(\partial_Y^t \partial_Y^t\partial_j^t \left(U^j_{\neq} U^3_{\neq} \right)_0 - \partial_Z^t \partial_Z^t \partial_Y^t \left( U^2_{\neq} U^3_{\neq} \right)_0 - \partial_Z^t \partial_Y^t \partial_Y^t \left(U^2_{\neq} U^2_{\neq}\right)_0 \right) dV \\ 
& = F^1 + F^2 + F^3.
\end{align*} 
The most dangerous term is $F^1$ due to the three $Y$ derivatives for $j =2$; we omit the other three terms for brevity as they are relatively easy variants. 
As in \S\ref{sec:NzeroForcing}, the treatments all reduce to applying the appropriate version of \eqref{ineq:AdelLij}. 
Write 
\begin{align*}
F^1 = -\int A^3 Q^3 A^3\left( \partial_Y^t \partial_Y^t\partial_Y^t \left( U^2_{\neq} U^3_{\neq} \right)_0 + \partial_Y^t \partial_Y^t\partial_Z^t \left( U^3_{\neq} U^3_{\neq} \right)_0 \right) dV = F^{1;2} + F^{1;3}.
\end{align*}
Let us focus only on $F^{1;2}$, which is the leading order contribution.  
As in \S\ref{sec:NzeroForcing}, we will expand with a quintic paraproduct but we group the terms with low frequency coefficients in with the remainder terms (due to Remark \ref{rmk:CoefCtrlLow}). 
For $F^{1;2}$ this means the decomposition 
\begin{align*}  
F^{1;2} & = -\sum_{k\neq 0} \int A^{3}Q^3_0 A_0^{3} \partial_Y\partial_Y\partial_Y\left( \left(U^3_{-k}\right)_{Hi} \left( U^2_k\right)_{Lo}\right)dV \\
& \quad - \sum_{k\neq 0} \int A^{3}Q^3_0 A_0^{3} \partial_Y \partial_Y\partial_Y\left( \left(U^3_{-k}\right)_{Lo} \left( U^2_k\right)_{Hi}\right) dV \\ 
& \quad -\sum_{k\neq 0}  \int A^{3}Q^3_0 A_0^{3} (\psi_y)_{Hi} \partial_Y\partial_Y\partial_Y\left( \left(U^3_{-k}\right)_{Lo} \left( U^2_k\right)_{Lo}\right) dV \\ 
& \quad - \sum_{k\neq 0} \int A^{3}Q^3_0 A_0^{3} \partial_Y \left((\psi_y)_{Hi}\partial_Y\partial_Y\left( \left(U^3_{-k}\right)_{Lo} \left( U^2_k\right)_{Lo}\right)\right) dV \\
& \quad - \sum_{k\neq 0} \int A^{3}Q^3_0 A_0^{3} \partial_Y \partial_Y \left((\psi_y)_{Hi} \partial_Y \left( \left(U^3_{-k}\right)_{Lo} \left( U^2_k\right)_{Lo}\right)\right) dV \\ 
& \quad + F^1_{\mathcal{R},C} \\ 
& = F_{HL} + F_{LH} + F_{C1} + F_{C2} + F_{C3} + F_{\mathcal{R},C},    
\end{align*} 
where here $F_{\mathcal{R},C}$ includes all of the remainders from the paraproduct and other terms where coefficients appear in low frequency. 
From \eqref{ineq:AprioriUneq}, \eqref{ineq:AdelLij}, and \eqref{ineq:triQuadHL} we have, 
\begin{align*} 
F_{HL} & \lesssim \frac{\epsilon}{\jap{t}^{2-\delta_1} \jap{\nu t^3}^{\alpha}} \sum_{k\neq 0} \sum_{l,l^\prime} \int \abs{A^{3} \widehat{Q^3}_0(\eta,l)} \frac{\abs{\eta}^3 \jap{\frac{t}{\jap{\xi,l^\prime}}}^2}{k^2 + (l^\prime)^2 + \abs{\xi-kt}^2} \abs{ \Delta_L A^{3}\widehat{U^3}_{k}(\xi,l^\prime)} \\ & \quad\quad\quad \times Low\left(-k,\eta-\xi,l-l^\prime \right) d\eta d\xi \\ 
& \lesssim \frac{\epsilon \jap{t}^{1+\delta_1}}{\jap{\nu t^3}^{\alpha}}\norm{\left(\sqrt{\frac{\partial_t w}{w}} + \frac{\abs{\grad}^{s/2}}{\jap{t}^{s}}\right) A^{3}Q^3}_2 \norm{\left(\sqrt{\frac{\partial_t w}{w}} + \frac{\abs{\grad}^{s/2}}{\jap{t}^s}\right)A^{3} \Delta_L U_{\neq}^3}_2
\\ & \quad + \frac{\epsilon}{\jap{t}^{2-\delta_1} \jap{\nu t^3}^\alpha}\norm{\sqrt{-\Delta_L} A^{3}Q^3}_2 \norm{A^{3}\Delta_L U^3_{\neq}}_2,    
\end{align*} 
which, after the application of Lemmas \ref{lem:PEL_NLP120neq} and \ref{lem:SimplePEL}, is consistent with Proposition \ref{prop:Boot}. 

Turn next to $F_{LH}$, which is the leading order effect of \textbf{(F)} isolated in \S\ref{sec:Toy}.  
 By \eqref{ineq:AprioriUneq}, \eqref{ineq:AdelLij}, and \eqref{ineq:triQuadHL}, 
\begin{align*} 
F_{LH} & \lesssim  \frac{\epsilon}{\jap{\nu t^3}^{\alpha}}\sum_{k\neq 0} \sum_{l,l^\prime} \int \abs{A^{3} \widehat{Q^3}_0(\eta,l) \frac{\abs{\eta}^3 \jap{\frac{t}{\jap{\xi,l^\prime}}}}{k^2 + (l^\prime)^2 + \abs{\xi-kt}^2} \Delta_L A^{2}\widehat{U^2}_{k}(\xi,l^\prime)} Low(-k,\eta-\xi,l-l^\prime) d\eta d\xi \\ 
& \lesssim \frac{\epsilon t^3}{\jap{\nu t^3}^{\alpha}}\norm{\left(\sqrt{\frac{\partial_t w}{w}} + \frac{\abs{\grad}^{s/2}}{\jap{t}^{s}}\right)A^{3}Q^3}_2 \norm{\left(\sqrt{\frac{\partial_t w}{w}} + \frac{\abs{\grad}^{s/2}}{\jap{t}^s}\right)A^{2} \Delta_L U^2_{\neq}}_2
\\ & \quad + \frac{\epsilon}{\jap{\nu t^3}^\alpha}\norm{\sqrt{-\Delta_L} A^{3}Q^3}_2 \norm{A^{2}\Delta_L U^2_{\neq}}_2.  
\end{align*} 
After the application of Lemmas \ref{lem:PEL_NLP120neq} and \ref{lem:SimplePEL}, is consistent with Proposition \ref{prop:Boot}. 
Notice the appearance of $\epsilon t^3 \jap{\nu t^3}^{-\alpha}$, which uses the hypothesis $\nu \gtrsim \epsilon$ to control.  

The $F_{Ci}$ terms associated with the coefficients in high frequency are treated the same as the corresponding terms in \S\ref{sec:NzeroForcing} and are hence omitted for brevity. 
The remainder terms are similarly straightforward or easy variants of the other treatments and are hence omitted as well.  
This completes the treatment of $F^{1;2}$. 
As mentioned above, the treatments of $F^{1;3}$, $F^2$, and $F^3$ are similar (but weaker) and hence also omitted. 

\subsubsection{Dissipation error terms}
The treatment of the dissipation error terms for $Q^3_0$ is the same as $Q_0^2$ as outlined in \S\ref{sec:DEQ02}, and therefore is omitted for the sake of brevity. 

\subsection{Non-zero frequencies}

\subsubsection{Nonlinear pressure $NLP$} \label{sec:NLP3}

\paragraph{Treatment of $NLP(1,j,0,\neq)$} \label{sec:NLPQ3ij_0neq}
This term is the analogue of the nonlinear terms treated in \S\ref{sec:NLP213} and are of type \textbf{(SI)}.  
Note that $j \neq 1$. 
We can essentially use the same treatment, although here it is easier since $Y$ derivatives are slightly harder than $Z$ derivatives 
and because we are imposing one less power of time control on $Q^3_{\neq}$ than on $Q^2_{\neq}$. 
For this reason, we omit the treatment for brevity and simply conclude the result: 
\begin{align*} 
NLP(1,j,0,\neq) & \lesssim c_{0}\norm{\left(\sqrt{\frac{\partial_t w}{w}} + \frac{\abs{\grad}^{s/2}}{\jap{t}^{s}}\right)A^3 Q^3}_2\norm{\left(\sqrt{\frac{\partial_t w}{w}} + \frac{\abs{\grad}^{s/2}}{\jap{t}^s} \right) \Delta_L A^j U^j_{\neq}}_2 \\& \quad + \epsilon \norm{\sqrt{-\Delta_L} A^3 Q^3}_2 \norm{A^j \Delta_L U^j_{\neq}}_2 \\ 
& \quad + \frac{\epsilon}{\jap{\nu t^3}^{\alpha}}\norm{A^3 Q^3_{\neq}}_2 \norm{A U_0^1}_2 + \frac{\epsilon^2 t}{\jap{\nu t^3}^{\alpha-1}}\norm{A^3 Q^3_{\neq}}_2 \norm{AC}_2, 
\end{align*}   
 which after Lemmas \ref{lem:PELbasicZero}, \ref{lem:PEL_NLP120neq}, and \ref{lem:SimplePEL}, is consistent with Proposition \ref{prop:Boot} for $\epsilon$ and $c_0$ sufficiently small. 

\paragraph{Treatment of $NLP(i,j,0,\neq)$ with $i \in \set{2,3}$} \label{sec:NLPQ3_0neq_notX} 
This is the analogue of the nonlinear terms treated in \S\ref{sec:NLPQ2_0neq_notX} above.
Note that again $j \neq 1$.   
These can treated analogously to the treatment in \S\ref{sec:NLPQ2_0neq_notX}, but in fact it is much easier here due to the fact that $Q^3$ is growing quadratically at frequencies with $t \gtrsim \jap{\grad_{Y,Z}}$.
In particular we can deduce as above, 
\begin{align*}
NLP(i,j,0,\neq) & \lesssim \epsilon \norm{A^3Q^3_{\neq}}_2 \norm{\Delta_L A^3 U^j_{\neq}}_2 + \frac{\epsilon \jap{t}}{\jap{\nu t^3}^{\alpha}} \norm{A^3Q^3_{\neq}}_2 \norm{A U^i_{0}}_2 \\ 
& \quad + \frac{\epsilon^2\jap{t}}{\jap{\nu t^3}^{\alpha-1}}\norm{A^3 Q^3}_2 \norm{AC}_2,  
\end{align*}   
which, after Lemmas \ref{lem:PELbasicZero} and \ref{lem:SimplePEL}, is consistent with Proposition \ref{prop:Boot}. 

\paragraph{Treatment of $NLP(i,j,\neq,\neq)$} \label{sec:NLPQ3_neqneq}
These terms can all be treated with an easy variant of the treatment in \S\ref{sec:NLPQ2_neqneq}, however as $Q^3$ is growing quadratically at `low' frequencies, 
it is significantly easier here.
Accounting for the worst case (which is realized by $(i,j) = (1,1)$) 
\begin{align*} 
NLP(i,j,\neq,\neq) & \lesssim \frac{\epsilon \jap{t}^{1+\delta_1}}{\jap{\nu t^3}^{\alpha-2}}\norm{A^3 Q^3_{\neq}}_2 \left( \norm{A^j \Delta_L U^j_{\neq}}_2  + \norm{A^i \Delta_L U^i_{\neq}}_2 \right) + \frac{\epsilon^2 \jap{t}^{2\delta_1}}{\jap{\nu t^3}^{\alpha-2}}\norm{A^3 Q^3_{\neq}}_2 \norm{AC}_2, 
\end{align*}
which is consistent with Proposition \ref{prop:Boot}.

\subsubsection{Nonlinear stretching $NLS$} \label{sec:NLSQ3}
Unlike the case of the $NLP$ terms in \S\ref{sec:NLP3}, controlling the $NLS$ terms in the evolution of $Q^3$ is in general 
slightly harder than for $Q^2$ (treated above in \S\ref{sec:NLSQ2}) since $U^3$ is larger than $U^2$. 

\paragraph{Treatment of $NLS1(j,\neq,0)$ and  $NLS1(j,0,\neq)$} \label{sec:NLS1Q3_neq0} 
Consider first the  $NLS1(j,0,\neq)$ terms. 
As in \S\ref{sec:NLS1Q20neq}, these terms are all easier variants of 
the $NLP(j,3,0,\neq)$ terms, and hence can be omitted for brevity.  

Consider next the  $NLS1(j,\neq,0)$ terms. 
Notice that $j \neq 1$ by the nonlinear structure. 
Both the remaining contributions are essentially the same so to fix ideas just consider $j = 3$. 
Expand the term with a paraproduct (as usual, grouping terms where the coefficients appear in low frequency with the remainder) 
\begin{align*}
NLS1(3,\neq,0)) & = -\int A^3 Q^3_{\neq} A^3\left( (Q^3_{\neq})_{Hi} \partial_Z(U_0^3)_{Lo} \right) dV - \int A^3 Q^3_{\neq} A^3\left( (Q^3_{\neq})_{Lo} \partial_Z(U_0^3)_{Hi} \right) dV \\   
& \quad - \int A^3 Q^3_{\neq} A^3\left( (Q^3_{\neq})_{Lo} (\psi_z)_{Hi} (\partial_YU_0^3)_{Lo} \right) dV + S_{\mathcal{R,C}} \\
& = S_{HL} + S_{LH} + S_{C} + S_{\mathcal{R}}.
\end{align*}
By \eqref{ineq:ABasic}, \eqref{ineq:AprioriU0}, and \eqref{ineq:triQuadHL}, 
\begin{align*}
S_{HL} & \lesssim \epsilon\norm{A^3 Q^3_{\neq}}_2^2,
\end{align*}
which is absorbed by the dissipation due to the non-zero frequencies.
By Lemma \ref{lem:ABasic}, \eqref{ineq:AprioriUneq}, and \eqref{ineq:triQuadHL} we have 
\begin{align*}
S_{HL} & \lesssim \frac{\epsilon t}{\jap{\nu t^3}^{\alpha}}\norm{A^3 Q^3_{\neq}}_2\norm{AU_0^3},
\end{align*} 
which is consistent with Proposition \ref{prop:Boot} by Lemma \ref{lem:PELbasicZero}. 
Similarly (using Lemma \ref{lem:CoefCtrl}), 
\begin{align*}
S_{C} & \lesssim \frac{\epsilon^2 t}{\jap{\nu t^3}^{\alpha}}\norm{A^3 Q^3_{\neq}}_2\norm{AC},
\end{align*} 
which is consistent with Proposition \ref{prop:Boot} for $\epsilon$ sufficiently small. The remainder term is similarly straightforward and is hence omitted. 

\paragraph{Treatment of $NLS1(j,\neq,\neq)$} \label{sec:NLS1Q3_neqneq} 
These terms are interactions of type \textbf{(3DE)}. 
All of these terms can be treated in a similar fashion; to fix ideas consider $j = 1$. 
Expand the term with a paraproduct
\begin{align*}
NLS1(1,\neq,\neq)) & = -\int A^3 Q^3_{\neq} A^3\left( (Q^1_{\neq})_{Hi} (\partial_X U_{\neq}^3)_{Lo} \right) dV  - \int A^3 Q^3_{\neq} A^3\left( (Q^1_{\neq})_{Lo} (\partial_X U_{\neq}^3)_{Hi} \right) dV  + S_{\mathcal{R}} \\ 
& = S_{HL} + S_{LH} + S_{\mathcal{R}}.
\end{align*}
By \eqref{ineq:ABasic}, \eqref{ineq:AprioriUneq}, and \eqref{ineq:triQuadHL}, 
\begin{align*}
S_{HL} & \lesssim \frac{\epsilon \jap{t}}{\jap{\nu t^3}^{\alpha}} \norm{A^3 Q^3}_2 \norm{A^1 Q^1}_2,   
\end{align*}
which is consistent with Proposition \ref{prop:Boot}. 
For $S_{LH}$, we use \eqref{ineq:AiPartX}, \eqref{ineq:ABasic}, \eqref{ineq:Boot_ED} and \eqref{ineq:triQuadHL}, 
\begin{align*}
S_{LH} & \lesssim \frac{\epsilon \jap{t}^{2+\delta_1}}{\jap{\nu t^3}^{\alpha}}\norm{\left(\sqrt{\frac{\partial_t w}{w}} + \frac{\abs{\grad}^{s/2}}{\jap{t}^{s}}\right)A^3 Q^3}_2\norm{\left(\sqrt{\frac{\partial_t w}{w}} + \frac{\abs{\grad}^{s/2}}{\jap{t}^{s}}\right)A^3 \Delta_L U^3_{\neq}}_2 \\ 
& \quad + \frac{\epsilon \jap{t}^{1+\delta_1}}{\jap{\nu t^3}^{\alpha}} \norm{A^3 Q^3}_2 \norm{A^3 \Delta_L U^3_{\neq}}_2, 
\end{align*}
which is consistent with Proposition \ref{prop:Boot} after Lemmas \ref{lem:PEL_NLP120neq} and \ref{lem:SimplePEL}. 
The remainder error term $S_{\mathcal{R}}$ is straightforward and is hence omitted. 

\paragraph{Treatment of $NLS2(i,1,0,\neq)$}
Note that $i \neq 1$ due to the zero $X$ frequency. 
This term is again similar to   $NLP(1,3,0,\neq)$ in $Q^2$ treated in \S\ref{sec:NLP213}; hence the treatment is omitted for brevity and we just state the result:  
\begin{align*} 
NLS2(i,1,0,\neq) & \lesssim c_0 \norm{\left(\sqrt{\frac{\partial_t w}{w}} + \frac{\abs{\grad}^{s/2}}{\jap{t}^s}\right)A^3 Q^3_{\neq}}_2\norm{\left(\sqrt{\frac{\partial_t w}{w}} + \frac{\abs{\grad}^{s/2}}{\jap{t}^s}\right) \Delta_L A^3 U^3_{\neq}}_2 \\ 
& \quad +  \epsilon \norm{A^3 Q^3_{\neq}}_2\norm{A^3 \Delta_L U^3_{\neq}}_2 + \frac{\epsilon}{\jap{\nu t^3}^{\alpha-1}} \norm{A^3 Q^3}_2 \norm{AU_0^1}_2 + \frac{\epsilon^2 \jap{t}}{\jap{\nu t^3}^{\alpha-1}}\norm{A^3 Q^3_{\neq}}_2 \norm{AC}_2. 
\end{align*} 

\paragraph{Treatment of $NLS2(i,j,0,\neq)$ for $j\in \set{2,3}$}
Note that $i \neq 1$ due to the zero $X$ frequency. 
By methods similar to those applied in \S\ref{sec:NLPQ3_0neq_notX} and \S\ref{sec:NLPQ2_0neq_notX}, we have 
\begin{align*} 
NLS2(i,j,0,\neq) & \lesssim \norm{A^3 Q^3_{\neq}}_2 \left(\frac{\epsilon \jap{t}}{\jap{\nu t^3}^{\alpha-1}} \norm{AU^j_{0}}_2  + \epsilon \norm{A^3 \Delta_L U^3_{\neq}}_2 + \frac{\epsilon^2 \jap{t}}{\jap{\nu t^3}^{\alpha-1}}\norm{AC}_2\right). 
\end{align*}

\paragraph{Treatment of $NLS2(i,j,\neq,0)$}
Note that necessarily both $i \neq 1$ and $j\neq1$ due to the zero $X$ frequency. 
Again by methods  similar to those applied in \S\ref{sec:NLPQ3_0neq_notX} and \S\ref{sec:NLPQ2_0neq_notX}, 
\begin{align*} 
NLS2(i,j,\neq,0) & \lesssim \norm{A^3 Q^3_{\neq}}_2 \left(\frac{\epsilon \jap{t}}{\jap{\nu t^3}^{\alpha-1}} \norm{AU^3_{0}}_2  + \epsilon \norm{A^j \Delta_L U^j_{\neq}}_2 + \frac{\epsilon^2 \jap{t}}{\jap{\nu t^3}^{\alpha-1}}\norm{AC}_2\right). 
\end{align*} 

\paragraph{Treatment of $NLS2(i,j,\neq,\neq)$}
These are treated in essentially the same manner as $NLP(i,j,\neq,\neq)$ and are hence omitted for the sake of brevity. 

\subsubsection{Transport nonlinearity $\mathcal{T}$} \label{sec:Q3_TransNon}
Begin with a paraproduct decomposition: 
 \begin{align*} 
 \mathcal{T}_{\neq}  & = -\int A^{3} Q^3_{\neq} A^{3} \left( \tilde U_{Lo} \cdot \grad Q^3_{Hi} \right) dV -\int A^{3} Q^3_{\neq} A^{3} \left( \tilde U_{Hi} \cdot \grad Q^3_{Lo} \right) dV + \mathcal{T}_{\mathcal{R}} \\ 
& = \mathcal{T}_T + \mathcal{T}_{R} + \mathcal{T}_{\mathcal{R}}, 
\end{align*}
where $\mathcal{T}_{\mathcal{R}}$ includes the remainder (as above in \S\ref{sec:Q2_TransNon}, we use the terminology `transport' and `reaction' for the first two terms respectively).     
We will divide the reaction term $\mathcal{T}_R$ further below. 
The transport and remainder contributions, $\mathcal{T}_T$ and $\mathcal{T}_{\mathcal{R}}$ respectively, 
are treated as in \S\ref{sec:Q2_TransNon} and are hence omitted here. The resulting terms are given by
\begin{align} 
\mathcal{T}_T + \mathcal{T}_{\mathcal{R}} & \lesssim \epsilon\norm{\sqrt{-\Delta_L}A^3 Q^3}_2^2 + \left(\frac{\epsilon}{t^2} + \frac{\epsilon t^{2\delta_1}}{\jap{\nu t^3}^{2\alpha}}\right)\norm{A^3 Q^3}_2^2.   
\end{align}  

Turn next to the reaction contribution.  
First, decompose the reaction term based on the $X$ dependence of each factor:  
\begin{align*} 
\mathcal{T}_{R} & =  -\int A^{3} Q^3_{\neq} A^{3} \left( (\tilde U_{\neq})_{Hi}  \cdot (\grad Q^3_0)_{Lo} \right) dV - \int A^{3} Q^3_{\neq} A^{3} \left( (\tilde U_0)_{Hi} \cdot \grad (Q^3_{\neq})_{Lo} \right) dV \\ & \quad - \int A^{3} Q^3_{\neq} A^{3} \left( (\tilde U_{\neq})_{Hi} \cdot \grad (Q^3_{\neq})_{Lo} \right) dV \\ 
& = \mathcal{T}_{R;\neq 0} + \mathcal{T}_{R;0 \neq}+ \mathcal{T}_{R;\neq \neq}. 
\end{align*} 
Each will be treated with a slightly different approach. 

\paragraph{Reaction term $\mathcal{T}_{R;0 \neq}$}
Turn first to the easiest, $\mathcal{T}_{R;0 \neq}$. 
By \eqref{ineq:Boot_ED3} and Lemma \ref{lem:ABasic}, we get (also recalling \eqref{def:tildeU2}):  
\begin{align*} 
\mathcal{T}_{R;0\neq} & \lesssim \frac{\epsilon \jap{t}^2}{\jap{\nu t^3}^{\alpha}}\sum_{k \neq 0}\int \abs{A^{3} \widehat{Q_k} (\eta,l) A^{3}_{k}(\eta,l) \widehat{\tilde {U}_{0}} (\xi,l^\prime)_{Hi}} Low(k,\eta-\xi,l-l^\prime) d\xi d\eta \\
& \lesssim \frac{\epsilon }{\jap{\nu t^3}^{\alpha}}\sum_{k \neq 0}\int \abs{A^{3} \hat{Q}_k (\eta,l) A \widehat{\tilde {U}_{0}} (\xi,l^\prime)_{Hi}} Low(k,\eta-\xi,l-l^\prime) d\xi d\eta \\
& \lesssim \frac{\epsilon}{\jap{\nu t^3}^{\alpha}}\norm{A^3 Q^3_{\neq}}_2\left(\norm{Ag}_2 + \norm{AU_0^3}_2 \right). 
\end{align*} 
After \eqref{ineq:AprioriU0} and the bootstrap hypotheses, this is consistent with Proposition \ref{prop:Boot}. 

\paragraph{Reaction terms $\mathcal{T}_{R;\neq 0}$} \label{sec:Q3TRneq0}
Next consider $\mathcal{T}_{R;\neq 0}$.
In fact, since $Q^3_0$ is the same order of magnitude as $Q^2_0$, this term can be treated in the same fashion as was done in \S\ref{sec:Q2TRneq0}. 
Hence, we omit the details for brevity and simply conclude
\begin{align*} 
\mathcal{T}_{R;\neq 0} & \lesssim \norm{A^3 Q^3_{\neq}}_2 \left(\norm{A^2 U^2_{\neq}}_2 + \norm{A^3 U^3_{\neq}}_2 \right) \norm{\grad Q^3_{0}}_{\G^{\lambda}} \lesssim \epsilon \norm{\sqrt{-\Delta_L} A^3 Q^3}_2^2, 
\end{align*}  
where the last inequality followed by Lemma \ref{lem:SimplePEL} and the bootstrap hypotheses. This term is then absorbed by the dissipation for $c_0$ small. 

\paragraph{Reaction term $\mathcal{T}_{R;\neq \neq}$} \label{sec:Q3TRneqneq}
Turn next to $\mathcal{T}_{R;\neq \neq}$. This includes terms isolated in \S\ref{sec:Toy} as leading order contributions to the \textbf{(3DE)} nonlinear interactions (see \S\ref{sec:NonlinHeuristics}).  
As in \S\ref{sec:Q2TRneqneq} above,  we further sub-divide: 
\begin{align*} 
\mathcal{T}_{R;\neq \neq} & \lesssim \frac{\epsilon \jap{t}^{2}}{\jap{\nu t^3}^\alpha} \sum_{k,k^\prime}\int \mathbf{1}_{k,k^\prime,k-k^\prime \neq 0} \abs{A^{3} \hat{Q}^3_k(\eta,l) A^{3}_{k}(\eta,l) \hat{U}_{k^\prime}^1 (\xi,l^\prime)_{Hi}} Low(k-k^\prime, \eta-\xi, l - l^\prime) d\eta d\xi \\   
& \quad +  \frac{\epsilon \jap{t}^{3}}{\jap{\nu t^3}^\alpha} \sum_{k,k^\prime}\int \mathbf{1}_{k,k^\prime,k-k^\prime \neq 0} \abs{A^{3} \hat{Q}^3_k(\eta,l) A^{3}_{k}(\eta,l) \hat{U}_{k^\prime}^2 (\xi,l^\prime)_{Hi}} Low(k-k^\prime, \eta-\xi, l - l^\prime) d\eta d\xi \\  
& \quad +  \frac{\epsilon \jap{t}^{2}}{\jap{\nu t^3}^{\alpha-1}} \sum_{k,k^\prime}\int \mathbf{1}_{k,k^\prime,k-k^\prime \neq 0} \abs{A^{3} \hat{Q}^3_k(\eta,l) A^{3}_{k}(\eta,l) \hat{U}_{k^\prime}^3 (\xi,l^\prime)_{Hi}} Low(k-k^\prime, \eta-\xi, l - l^\prime) d\eta d\xi \\   
& \quad + \frac{\epsilon^2\jap{t}^{1+\delta_1}}{ \jap{\nu t^3}^{\alpha}} \sum_{k,k^\prime}\int \mathbf{1}_{k,k^\prime,k-k^\prime \neq 0} \abs{A^{3} \hat{Q}^3_k(\eta,l) A^{3}_{k}(\eta,l) \hat{\psi_y}(\xi,l^\prime)_{Hi}} Low(k-k^\prime, \eta-\xi, l - l^\prime) d\eta d\xi \\   
& \quad + \frac{\epsilon^2 \jap{t}^{3}}{\jap{\nu t^3}^{\alpha}} \sum_{k,k^\prime}\int \mathbf{1}_{k,k^\prime,k-k^\prime \neq 0} \abs{A^{3} \hat{Q}^3_k(\eta,l) A^{3}_{k}(\eta,l) \hat{\psi_z}(\xi,l^\prime)_{Hi}} Low(k-k^\prime, \eta-\xi, l - l^\prime) d\eta d\xi \\   
& \quad + \mathcal{T}_{R;\neq \neq;\mathcal{R}} \\ 
& = \mathcal{T}_{R;\neq\neq}^{1} + \mathcal{T}_{R;\neq\neq}^2 + \mathcal{T}_{R;\neq\neq}^3  + \mathcal{T}_{R;\neq\neq}^{C1} + \mathcal{T}_{R;\neq\neq}^{C2} + \mathcal{T}_{R;\neq\neq;\mathcal{R}}. 
\end{align*} 
The $\mathcal{T}_{R;\neq\neq}$ term appears in the toy model in \S\ref{sec:Toy}. 
By \eqref{ineq:ABasic}, \eqref{ineq:AikDelLNoD}, and \eqref{ineq:quadHL} we have, 
\begin{align*} 
\mathcal{T}_{R;\neq\neq}^{1} + \mathcal{T}_{R;\neq\neq}^2 + \mathcal{T}_{R;\neq\neq}^3 & \lesssim \norm{\left(\sqrt{\frac{\partial_t w}{w}} + \frac{\abs{\grad}^{s/2}}{\jap{t}^s}\right)A^3 Q^3}_2\left[\frac{\epsilon t^3}{\jap{\nu t^3}^{\alpha}}\norm{\left(\sqrt{\frac{\partial_t w}{w}} + \frac{\abs{\grad}^{s/2}}{\jap{t}^s}\right)\Delta_L A^1 U^1_{\neq}}_2\right. \\ & \hspace{-2cm} \quad \left. + \frac{\epsilon t^3}{\jap{\nu t^3}^{\alpha}}\norm{\left(\sqrt{\frac{\partial_t w}{w}} + \frac{\abs{\grad}^{s/2}}{\jap{t}^s}\right)\Delta_L A^2 U^2_{\neq}}_2 + \frac{\epsilon t^2}{\jap{\nu t^3}^{\alpha-1}}\norm{\left(\sqrt{\frac{\partial_t w}{w}} + \frac{\abs{\grad}^{s/2}}{\jap{t}^s}\right)\Delta_L A^3 U^3_{\neq}}_2 \right]. \\ 
& \hspace{-2cm} \quad + \norm{A^3 Q^3}_2 \left(\frac{\epsilon \jap{t}}{\jap{\nu t^3}^{\alpha}} \norm{\Delta_L A^1 U^1_{\neq}}_2 + \frac{\epsilon \jap{t}}{\jap{\nu t^3}^{\alpha}} \norm{\Delta_L A^2 U^2_{\neq}}_2 + \frac{\epsilon }{\jap{\nu t^3}^{\alpha-1}} \norm{\Delta_L A^3 Q^3_{\neq}}_2\right).
\end{align*} 
By Lemmas \ref{lem:PEL_NLP120neq} and \ref{lem:SimplePEL} the above is consistent with Proposition \ref{prop:Boot} by the bootstrap hypotheses.  
As predicted in \S\ref{sec:Toy}, note the appearance of $\epsilon t^3 \jap{\nu t^3}^{-\alpha}$, which requires the hypothesis $\epsilon \lesssim \nu$ to control.
This completes the treatment of the reaction term $\mathcal{T}_{R;\neq\neq}^{1,2,3}$ for the estimate \eqref{ineq:AQ3_Evo}.

Turn next to $\mathcal{T}_{R;\neq\neq}^{C1}$ and $\mathcal{T}_{R;\neq\neq}^{C2}$. By Lemma \ref{lem:ABasic} and \eqref{ineq:quadHL} (and Lemma \ref{lem:CoefCtrl}), we have
\begin{align*} 
\mathcal{T}_{R;\neq\neq}^{C1}+ \mathcal{T}_{R;\neq\neq}^{C2}  & \lesssim \frac{\epsilon^2 \jap{t}^2}{\jap{\nu t^3}^{\alpha-1}}\norm{A^3 Q^3}_2 \norm{AC}_2 \lesssim \frac{\epsilon \jap{t}^2}{\jap{\nu t^3}^{\alpha-1}}\norm{A^3 Q^3}^2_2 + \frac{\epsilon^3 \jap{t}^2}{\jap{\nu t^3}^{\alpha-1}} \norm{AC}_2^2, 
\end{align*}
which is consistent with Proposition \ref{prop:Boot} by the bootstrap hypotheses. 
This completes the treatment of $\mathcal{T}_{R;\neq\neq}$ and hence all of $\mathcal{T}$. 

\subsubsection{Dissipation error terms  $\mathcal{D}_E$} \label{sec:DEneqQ3}
We will use a refinement of the treatment of the analogous term treated in \S\ref{sec:DEneqQ2}.
Recalling the dissipation error terms and the short-hand \eqref{def:G}, we have 
\begin{align*} 
\mathcal{D}_E & = \nu\sum_{k \neq 0}\int A^3 Q^3_k A^3_k\left(G(\partial_{Y} - t\partial_X)^2 Q^3_k + 2\psi_z (\partial_Y - t \partial_X)\partial_{Z}Q^3_k \right) dV \\ 
& = \mathcal{D}_E^1 + \mathcal{D}_E^2. 
\end{align*} 
We will only treat $\mathcal{D}_E^1$; $\mathcal{D}_E^2$ yields similar contributions and is hence omitted.
Begin by expanding with a paraproduct 
\begin{align*} 
\mathcal{D}_E^1 & = \nu\sum_{k \neq 0}\int A^3 Q^3_k A^3_k \left(G_{Hi} (\partial_{Y} - t\partial_X)^2 (Q^3_k)_{Lo} \right) dV + \nu\sum_{k \neq 0}\int A^3 Q^3_k A^3_k \left(G_{Lo} (\partial_{Y} - t\partial_X)^2 (Q^3_k)_{Hi} \right) dV \\ 
& \quad + \nu\sum_{k \neq 0}\int A^3 Q^3_k A^3_k \left( \left(G(\partial_{Y} - t\partial_X)^2 Q^3_k \right)_{\mathcal{R}} \right) dV \\ 
& = \mathcal{D}_{E;HL}^1 + \mathcal{D}_{E;LH}^1 + \mathcal{D}_{E;\mathcal{R}}^1. 
\end{align*} 
As in \S\ref{sec:DEneqQ2}, we can control the latter two terms by the dissipation; we omit the details. 
Next to turn to the treatment of $\mathcal{D}_{E;HL}^1$. As in \S\ref{sec:DEneqQ2}, from \eqref{ineq:AprioriUneq} and Lemma \ref{lem:ABasic}, followed by \eqref{ineq:triQuadHL} and Lemma \ref{lem:CoefCtrl},  
\begin{align*}
\mathcal{D}_{E;HL}^{1} & \lesssim \frac{\nu \epsilon \jap{t}^4}{\jap{\nu t^3}^{\alpha}} \sum_{k \neq 0} \int \abs{A^3 \widehat{Q^3_k}(\eta,l) \frac{1}{\jap{\xi,l^\prime} \jap{t}} A\widehat{G}(\xi,l^\prime)_{Hi}} Low(\eta-\xi,l-l^\prime) d\eta d\xi \\ 
& \lesssim \frac{\nu \epsilon \jap{t}^3}{\jap{\nu t^3}^{\alpha}} \norm{A^3 Q^3_{\neq}}_2 \norm{AC}_2 \\ 
& \lesssim \frac{\epsilon^{1/3}}{\jap{\nu t^3}^{\alpha}}\norm{A^3 Q^3_{\neq}}_2^2 + \frac{\nu^2 \epsilon^{5/3} \jap{t}^6}{\jap{\nu t^3}^{\alpha}} \norm{AC}_2^2
\end{align*}
which is consistent with Proposition \ref{prop:Boot} for $\epsilon$ sufficiently small. 

\subsubsection{Linear stretching term $LS3$} \label{sec:LS30_Hi}
First separate into two parts (to be sub-divided further below), 
\begin{align*} 
LS3 & = -2\int A^{3} Q^3 A^{3}\partial_X(\partial_Y - t\partial_X) U^3  dV - 2\int A^{3} Q^3 A^{3} \partial_X\left(\psi_y(\partial_Y - t\partial_X) U^3\right) dV \\ 
& = LS3^0 + LS3^{C}. 
\end{align*} 

\paragraph{Treatment of $LS3^C$} \label{sec:LS3C}
Expand with a paraproduct, 
\begin{align*} 
LS3^{C} & = -2\int A^{3} Q^3 A^{3}\left((\psi_y)_{Hi}(\partial_Y - t\partial_X) \partial_X \left(U^3\right)_{Lo}\right) dV \\ & \quad - 2\int A^{3} Q^3 A^{3}\left((\psi_y)_{Lo}(\partial_Y - t\partial_X) \partial_X \left(U^3\right)_{Hi} \right) dV \\ 
& \quad - 2\int A^{3} Q^3 A^{3}\left((\psi_y)(\partial_Y - t\partial_X) \partial_X U^3\right)_{\mathcal{R}} dV \\ 
& = LS3^{C}_{HL} + LS3^{C}_{LH} + LS3^{C}_{\mathcal{R}}.  
\end{align*}
The main issue is $LS3^{C}_{HL}$, where the coefficients appear in `high frequency', so turn to this term first.    
By Lemma \ref{lem:ABasic}, \eqref{ineq:triQuadHL}, and Lemma \ref{lem:CoefCtrl},  
\begin{align*} 
LS3^C_{HL} & \lesssim \frac{\epsilon\jap{t}}{\jap{\nu t^3}^{\alpha}}\sum_{k \neq 0}\int \abs{A^{3} \widehat{Q^3_k}(\eta,l)\frac{1}{\jap{\xi,l^\prime}\jap{t}}A\widehat{\psi_y}(\xi,l^\prime)_{Hi}} Low(k,\eta-\xi,l-l^\prime) d\xi \\ 
& \lesssim \frac{\epsilon}{\jap{\nu t^3}^{\alpha}}\norm{A^3 Q^3}_2 \norm{\jap{\grad}^{-1}A \psi_y}_2 \\ 
& \lesssim \frac{\epsilon}{\jap{\nu t^3}^{\alpha}}\norm{A^3 Q^3}_2 \norm{AC}_2,  
\end{align*}
which is consistent with Proposition \ref{prop:Boot} by \eqref{ineq:Boot_ACC2}. 

Turn next to the $LS3^C_{LH}$, which is reminiscent of $NLP(1,3,0,\neq)$ in \S\ref{sec:NLP213}. 
Indeed, by \eqref{ineq:ABasic}, \eqref{ineq:AiPartX}, and \eqref{ineq:triQuadHL} we have  
\begin{align} 
LS3^C_{LH} & \lesssim \min(\epsilon\jap{t},c_{0})\sum_{k \neq 0} \int \abs{ A^3 \widehat{Q^3_k}(\eta,l) A^3_k(\eta,l) \frac{k\abs{\xi - kt}}{k^2 + (l^\prime)^2 + \abs{\xi-kt}^2} \Delta_L \widehat{U^3_k}(\xi,l^\prime)} Low(\eta-\xi,l-l^\prime) d\eta dx \nonumber \\ 
& \lesssim c_{0}\norm{\left(\sqrt{\frac{\partial_t w}{w}} + \frac{\abs{\grad}^{s/2}}{\jap{t}^{s}}\right) A^3 Q^3_{\neq}}_2 \norm{\left(\sqrt{\frac{\partial_t w}{w}} + \frac{\abs{\grad}^{s/2}}{\jap{t}^{s}}\right) A^3 \Delta_L U^3_{\neq}}_2 \nonumber  \\ 
& \quad + \frac{\epsilon}{\jap{t}}\norm{\sqrt{-\Delta_L}A^3 Q^3}_2 \norm{A^3 \Delta_L U^3_{\neq}}_2, \label{ineq:LS3CLH} 
\end{align} 
which, after the application of Lemmas \ref{lem:PEL_NLP120neq} and \ref{lem:SimplePEL}, is consistent with Proposition \ref{prop:Boot}.

The remainder $LS3^C_{\mathcal{R}}$ follows easily and is hence omitted.  

\paragraph{Leading order term, $LS3^0$} \label{sec:LS30}
Turn to the leading order term, $LS3^0$. 
The $2$ in the leading order term is crucially important and cannot be altered; it is the origin of the quadratic growth of $Q^3$ at low (relative to time) frequencies. For this reason we have to treat this term with a little more precision than we usually treat the $\Delta_t^{-1}$. 
Begin by isolating the leading order contribution: 
\begin{align} 
LS3^0 & =  -2\int A^{3} Q^3 A^{3}\partial_X(\partial_Y - t\partial_X) \Delta_{L}^{-1} \left( Q^3 - G(\partial_Y - t\partial_X)^2 U^3 \right. \nonumber \\ 
& \quad\quad \left. - 2\psi_z\partial_Z (\partial_Y - t\partial_X) U^3 - \Delta_t C (\partial_Y - t\partial_X) U^3 \right)  dV \nonumber \\  
& = LS3^{0;0} + LS3^{0;C1} + LS3^{0;C2} + LS3^{0;C3}.   \label{eq:LS30}
\end{align} 
The idea is that the latter terms can be treated perturbatively since we have additional smallness from the coefficients. 

Turn to the leading order term in \eqref{eq:LS30}.
Begin with separating between short and long times (relative to the critical times): 
\begin{align*} 
LS3^{0;0} & = -2\sum \int \abs{A^{3} \widehat{Q^3_k}(\eta,l)}^2 \frac{k(\eta-kt)}{k^2 + l^2 + \abs{\eta-kt}^2} d\eta \\ 
& = -2\sum \int \left[\mathbf{1}_{t \leq 2\abs{\eta}} + \mathbf{1}_{t > 2\abs{\eta}}\right] \abs{A^{3} \widehat{Q^3_k}(\eta,l)}^2 \frac{k(\eta-kt)}{k^2 + l^2 + \abs{\eta-kt}^2} d\eta \\
& = LS3^{0;0,ST} + LS3^{0;0,LT}, 
\end{align*}
where `ST' and `LT' stand for `short-time'  and `long-time' respectively. 
In the long-time term, divide further based on how time compares with the $Y,Z$ derivatives:  
\begin{align*} 
LS3^{0;0,LT}  & = -2\sum \int \mathbf{1}_{t > 2\abs{\eta}}\left[ \mathbf{1}_{t < \jap{\eta,l}} + \mathbf{1}_{t \geq \jap{\eta,l}} \right] \abs{A^{3} \widehat{Q^3_k}(\eta,l)}^2 \frac{k(\eta-kt)}{k^2 + l^2 + \abs{\eta-kt}^2} d\eta \\
& = LS3^{0;0,LT,Z} + LS3^{0;0,LT,Y}. 
\end{align*} 
In the case of $LS3^{0;0,LT,Z}$, we have
\begin{align} 
LS3^{0;0,LT,Z} & \lesssim \frac{1}{\jap{t}}\sum \int \mathbf{1}_{t > 2\abs{\eta}} \mathbf{1}_{t < \jap{\eta,l}} \abs{A^{3} \widehat{Q^3_k}(\eta,l)}^2 d\eta \lesssim \frac{1}{\jap{t}^{3/2}} \norm{\abs{\grad}^{1/4}A^{3}Q^3}_2^2 \nonumber
\end{align}
so that 
\begin{align}
LS3^{0;0,LT,Z} \leq &  \frac{\delta_\lambda}{10\jap{t}^{3/2}}\norm{\abs{\grad}^{s/2} A^{3}Q^3}_2^2 +  \frac{K}{\delta_\lambda^{\frac{1}{2s-1}}\jap{t}^{3/2}}\norm{ A^{3}Q^3}_2^2; \label{ineq:LS300LTZ}
\end{align}
for some universal constant $K > 0$. The first term is absorbed by the $CK_\lambda$ term and the latter is consistent with Proposition \ref{prop:Boot} provided $K_{H3} \gg \exp\left[C\delta_\lambda^{-\frac{1}{2s-1}}\right]$, for some universal constant $C > 0$ (via integrating factors). 

In the case of $LS3^{0;0,LT,Y}$, we want to absorb to leading order by $CK_{L}^3$; indeed: 
\begin{align*} 
LS3^{0;0,LT,Y} & \leq 2\sum_{k \neq 0} \int \mathbf{1}_{t > 2\abs{\eta}} \mathbf{1}_{t \geq \jap{\eta,l}} \abs{A^{3} \widehat{Q^3}}^2 \frac{1}{\sqrt{1 + \abs{t- \abs{\eta}}^2}} d\eta \\ 
& = CK^3_{L} - 2\sum_{k\neq 0} \int \mathbf{1}_{t > 2\abs{\eta}} \mathbf{1}_{t \geq \jap{\eta,l}} \abs{A^{3} \widehat{Q^3_k}}^2 \left(\frac{1}{t} -  \frac{1}{\sqrt{1 + \left(t- \abs{\eta}\right)^2}} \right) d\eta \\  
& = CK^3_{L} + LS3^{0;0,LT,Y}_{\mathcal{R}}.
\end{align*} 
The first term is then absorbed by the corresponding $CK_L^3$ term in \eqref{ineq:AQ3_Evo}.  
The remainder term is controlled via 
\begin{align} 
LS3^{0;0,LT,Y}_{\mathcal{R}} & = -2\sum \int \mathbf{1}_{t > 2\abs{\eta}} \mathbf{1}_{t \geq \jap{\eta,l}} \abs{A^{3} \widehat{Q^3_k}(\eta,l)}^2 \frac{\sqrt{1 + \left(t- \abs{\eta}\right)^2} - t}{t \sqrt{1 + \left(t- \abs{\eta}\right)^2}} d\eta \nonumber \\ 
& = -2\sum \int \mathbf{1}_{t > 2\abs{\eta}} \mathbf{1}_{t \geq \jap{\eta,l}} \abs{A^{3} \widehat{Q^3_k}(\eta,l)}^2 \frac{1 + \left(t- \abs{\eta}\right)^2 - t^2}{t\left(t + \sqrt{1 + \left(t- \abs{\eta}\right)^2}\right) \sqrt{1 + \left(t- \abs{\eta}\right)^2}} d\eta \nonumber \\  
& \lesssim \sum \int \mathbf{1}_{t > 2\abs{\eta}} \mathbf{1}_{t \geq \jap{\eta,l}} \abs{A^{3} \widehat{Q^3_k}(\eta,l)}^2 \frac{\jap{\eta}^2 + \jap{t}\jap{\eta}}{\jap{t}^3} d\eta \nonumber \\  
& \lesssim \sum \int \mathbf{1}_{t > 2\abs{\eta}} \mathbf{1}_{t \geq \jap{\eta,l}} \abs{A^{3} \widehat{Q^3_k}}^2 \frac{\jap{\eta}^{1/2}}{\jap{t}^{3/2}} d\eta \nonumber \\  
& \lesssim \frac{1}{\jap{t}^{3/2}}\norm{\abs{\grad}^{1/4} A^{3}Q^3}_2^2 + \frac{1}{\jap{t}^{3/2}}\norm{A^{3}Q^3}_2^2 \nonumber \\ 
& \leq \frac{\delta_\lambda}{10\jap{t}^{3/2}}\norm{\abs{\grad}^{s/2} A^{3}Q^3}_2^2 +  \frac{K}{\delta_\lambda^{\frac{1}{2s-1}}\jap{t}^{3/2}}\norm{ A^{3}Q^3}_2^2,  \label{ineq:LS3Remainder}
\end{align} 
for a universal constant $K > 0$. As discussed above after \eqref{ineq:LS300LTZ}, this is consistent with Proposition \ref{prop:Boot} for small $\delta_\lambda$ and large $K_{H3}$.  This completes the treatment of $LS3^{0;0,LT}$. 
 
Turn to the the short-time term, $LS3^{0;0,ST}$. 
By the frequency localization, \eqref{ineq:CKwLS}, and Cauchy-Schwarz
\begin{align*} 
LS3^{0;0,ST} & \lesssim \kappa^{-1}\norm{\sqrt{\frac{\partial_t w}{w}} A^{3} Q^3}_2^2 
\end{align*} 
This term is then absorbed by $CK_w^3$ in \eqref{ineq:AQ3_Evo} for sufficiently large $\kappa$ (notice that the implicit constant does \emph{not} depend on $\kappa$) 
This completes the leading order $LS3^{0;0}$.

Now consider the first error term in \eqref{eq:LS30}, $LS3^{0;C1}$.
The second error term, $LS3^{0;C2}$, is treated in the same fashion and yields similar results. Hence we omit its treatment and focus on $LS3^{0;C1}$.
Expand $LS3^{0;C1}$ with a paraproduct: 
\begin{align*} 
LS3^{0;C1} & = -2\int A^{3} Q^3 A^{3}\partial_X(\partial_Y - t\partial_X) \Delta_{L}^{-1} \left( G_{Hi} (\partial_Y - t\partial_X)^2 U^3_{Lo}\right) dV \\ 
& \quad -2\int A^{3} Q^3 A^{3}\partial_X(\partial_Y - t\partial_X) \Delta_{L}^{-1} \left( G_{Lo} (\partial_Y - t\partial_X)^2 U^3_{Hi}\right) dV \\ 
& \quad -2\int A^{3} Q^3 A^{3}\partial_X(\partial_Y - t\partial_X) \Delta_{L}^{-1} \left( G (\partial_Y - t\partial_X)^2 U^3 \right)_{\mathcal{R}} dV \\ 
& = LS3^{0;C1}_{HL} + LS3^{0;C1}_{LH} + LS3^{0;C1}_{\mathcal{R}}.  
\end{align*} 

Consider the $HL$ term first and divide into resonant and non-resonant frequencies: 
\begin{align*} 
LS3^{0;C1}_{HL} & \lesssim \frac{\epsilon\jap{t}^2}{\jap{\nu t^3}^{\alpha}}\sum_{l,l^\prime,k\neq 0}\int \left[\chi^R + \chi^{NR} \right] \abs{A^{3} \widehat{Q^3_k}(\eta,l) A^{3}_k(\eta,l) \frac{\abs{\eta-kt}}{k^2 + l^2 + \abs{\eta-kt}^2}\widehat{G}(\xi,l^\prime)_{Hi}} \\ & \quad\quad \times Low(k,\eta-\xi,l-l^\prime) d\xi d\eta \\
& =  LS3^{0;C1;R}_{HL} + LS3^{0;C1;NR}_{HL}, 
\end{align*}
where $\chi^R = \mathbf{1}_{t \in \I_{k,\eta}}\mathbf{1}_{t \in \I_{k,\xi}}\mathbf{1}_{\abs{l} \lesssim \abs{\eta}}$ and $\chi^{NR} = 1 - \chi^{R}$. 

By Lemma \ref{lem:wellsep}, we have
\begin{align}
\chi^{NR} \frac{|\eta-kt|}{k^2 + |\eta-kt|^2 + l^2} \lesssim \frac{1}{t} \langle \eta - \xi \rangle. \label{ineq:NRtrick}
\end{align}
Indeed, to see \eqref{ineq:NRtrick}, note that either $t \not\in \I_{k,\eta}$ or $t \not\in \I_{k,\xi}$ (in which case, Lemma \ref{lem:wellsep} applies) or $t \in \I_{k,\eta} \cap \I_{k,\xi}$ but $\abs{l} \gtrsim \abs{\eta} \gtrsim \abs{kt}$. 
Therefore, by \eqref{ineq:NRtrick}, Lemma \ref{lem:ABasic}, \eqref{ineq:triQuadHL}, and Lemma \ref{lem:CoefCtrl}, we have 
\begin{align*} 
LS3^{0;C1;NR}_{HL} & \lesssim \frac{\epsilon\jap{t}^2}{\jap{\nu t^3}^{\alpha}}\int \abs{\chi^{NR} A^{3} \hat{Q}^3(k,\eta,l) \frac{\abs{\eta-kt}}{k^2 + l^2 + \abs{\eta-kt}^2} \frac{1}{\jap{\xi,l^\prime} \jap{t}} A\widehat{G}(\xi,l^\prime) Low(k,\eta-\xi,l-l^\prime)} d\xi d\eta \\
& \lesssim \frac{\epsilon}{\jap{\nu t^3}} \norm{A^{3}Q^3}_2 \norm{AC}_2, 
\end{align*} 
which is consistent with Proposition \ref{prop:Boot} by \eqref{ineq:Boot_ACC2}.  
In the resonant regime, we have by Lemma \ref{lem:ABasic}, Lemma \ref{lem:dtw}, \eqref{ineq:triQuadHL}, and Lemma \ref{lem:CoefCtrl},
\begin{align*} 
LS3^{0;C1;R}_{HL} & \lesssim \frac{\epsilon\jap{t}^2}{\jap{\nu t^3}^{\alpha}}\int \chi^{R} \abs{A^{3} \widehat{Q^3_k}(\eta,l)  \frac{1}{\left(\abs{k} + \abs{\eta-kt}\right) \jap{\xi,l^\prime} } \frac{1}{\jap{t}}A(\xi,l^\prime) \widehat{G}(\xi,l^\prime)} Low(k,\eta-\xi,l-l^\prime) d\xi d\eta \\ 
& \lesssim \frac{\epsilon t}{\jap{\nu t^3}^\alpha } \norm{\sqrt{\frac{\partial_t w}{w}} A^{3}Q^3}_2 \norm{\left(\sqrt{\frac{\partial_t w}{w}} + \frac{\abs{\grad}^{s/2}}{\jap{t}^{s}}\right) A C}_2 \\ 
& \lesssim \frac{c_0}{\jap{\nu t^3}^\alpha } \norm{\sqrt{\frac{\partial_t w}{w}} A^{3}Q^3}^2_2 + \frac{\epsilon^{2} \jap{t}^4 }{c_0\jap{\nu t^3}^\alpha } \jap{t}^{-2}\norm{\left(\sqrt{\frac{\partial_t w}{w}} + \frac{\abs{\grad}^{s/2}}{\jap{t}^{s}}\right) A C}^2_2,  
\end{align*}   
which is time integrable by \eqref{ineq:Boot_ACC2}. This completes $LS3_{HL}^{0;C1}$. 

Turn to $LS3^{0;C1}_{LH}$, which is written on the frequency side as
\begin{align*} 
LS3^{0;C1}_{LH} \lesssim \min(\epsilon t,c_{0}) \sum_{k,l}\int \abs{A^{3} \widehat{Q^3_k}(\eta,l) A^{3}_k(\eta,l) \frac{\abs{k}\abs{\eta-kt}}{k^2 + l^2 + \abs{\eta-kt}^2} \left(\Delta_L \widehat{U^3_k}\right)_{Hi}(\xi,l^\prime)} Low(\eta-\xi,l-l^\prime) d \eta d\xi.   
\end{align*}  
We can treat this term roughly like $NLP(1,3,0,\neq)$ for $Q^2$ in \S\ref{sec:NLP213}. 
By \eqref{ineq:AiPartX} and \eqref{ineq:triQuadHL},  
\begin{align*} 
LS3^{0;C1}_{LH} & \lesssim c_{0}\norm{\left(\sqrt{\frac{\partial_t w}{w}} A^{3}+ \frac{\abs{\grad}^{s/2}}{\jap{t}^s}A^{3}\right) Q^3}_2^2 + c_{0}\norm{\left(\sqrt{\frac{\partial_t w}{w}} A^{3}+ \frac{\abs{\grad}^{s/2}}{\jap{t}^s}A^{3}\right) \Delta_L U^3_{\neq}}_2^2 \\ 
& \quad + \epsilon\norm{A^3 Q^3_{\neq}}_2\norm{\Delta_L A^3 U^3_{\neq}}_2.
\end{align*}
By Lemmas \ref{lem:PEL_NLP120neq} and  \ref{lem:SimplePEL}, this consistent with Proposition \ref{prop:Boot} by the bootstrap hypotheses. 
 
The remainder $LS3^{0;C1}_{\mathcal{R}}$ is straightforward and is omitted for the sake of brevity. 
This completes the first error term in \eqref{eq:LS30}, $LS3^{0;C1}$. 
As mentioned above, $LS3^{0;C2}$ in \eqref{eq:LS30} is very similar, and is hence omitted for the sake of brevity. 

Turn to the third error term in \eqref{eq:LS30}, $LS3^{0;C3}$. 
The treatment is almost the same as $LS3^{0;C1}$; we only briefly sketch the differences. 
As above, expand with a paraproduct: 
\begin{align*} 
LS3^{0;C3} & = -2\int A^{3} Q^3 A^{3}\partial_X(\partial_Y - t\partial_X) \Delta_{L}^{-1} \left( (\Delta_t C)_{Hi}  (\partial_Y - t\partial_X)(U^3)_{Lo}\right) dV \\ 
& \quad -2\int A^{3} Q^3 A^{3}\partial_X(\partial_Y - t\partial_X) \Delta_{L}^{-1} \left( (\Delta_t C)_{Lo} (\partial_Y - t\partial_X)(U^3)_{Hi}\right) dV \\ 
& \quad -2\int A^{3} Q^3 A^{3}\partial_X(\partial_Y - t\partial_X) \Delta_{L}^{-1} \left( (\Delta_t C) (\partial_Y - t\partial_X) (U^3) \right)_{\mathcal{R}} dV \\ 
& = LS3^{0;C3}_{HL} + LS3^{0;C3}_{LH} + LS3^{0;C3}_{\mathcal{R}}.  
\end{align*}  
Consider first the contribution when the coefficients are in high frequency, $LS3^{0;C3}_{HL}$. 
We have (dividing into resonant and non-resonant frequencies as in the treatment of $LS3^{0;C1}_{HL}$):  
\begin{align*} 
LS3^{0;C2}_{HL} & \lesssim \frac{\epsilon\jap{t}}{\jap{\nu t^3}^\alpha}\int \left[\chi^R + \chi^{NR}\right]\abs{A^{3} \widehat{Q^3_k}(\eta,l)} \\ & \quad\quad \times A^{3}_k(\eta,l) \frac{\abs{\eta-kt}}{k^2 + l^2 + \abs{\eta-kt}^2} \abs{\Delta_t\hat{C}(\xi,l^\prime)_{Hi}} Low(k,\eta-\xi,l-l^\prime) d\xi d\eta, 
\end{align*} 
where $\chi^R = \mathbf{1}_{t \in \I_{k,\eta}}\mathbf{1}_{t \in \I_{k,\xi}}\mathbf{1}_{\abs{l} \lesssim \abs{\eta}}$ and $\chi^{NR} = 1 - \chi^{R}$. 
In the non-resonant contributions, we gain one less power of $t$ relative to the treatment of $LS3^{0;C1;NR}_{HL}$ but 
we also lose one less. 
In the resonant contributions there is essentially an exact equivalence between frequency and time, and hence the extra derivative on $C$ is the same as the extra power of time in $LS3^{0;C1;R}_{HL}$ and so the treatment is the same also for $LS3^{0;C3}_{HL}$. 
Moreover, it yields terms which are the same as those given by $LS3^{0;C1}$. 
The treatment of the coefficients in $\Delta_t$ is handled with Lemma \ref{lem:AAiProd} combined with Lemma \ref{lem:CoefCtrl}. 
Hence, we omit the treatment for brevity. 
The treatments and resulting terms from $LS3^{0;C3}_{LH}$ and $LS3^{0;C3}_{\mathcal{R}}$ are essentially the same as for $LS3^{0;C1}$, although easier since we are losing one less power of $t$ here. We omit these contributions as well as they also yield similar contributions as $LS3^{0;C1}$. This concludes the treatment of the linear stretching term $LS3$.  

\subsubsection{Linear pressure term $LP3$} \label{ineq:LP3_Hi} 
As in $LS3$, first separate the coefficient corrections and expand them with a paraproduct:
\begin{align*} 
LP3 & = 2 \int A^{3} Q^3 A^{3} \partial_Z \partial_X U^2 dV \\
& \quad + 2 \int A^{3} Q^3 A^{3} \left((\psi_{z})_{Lo}(\partial_Y - t\partial_X) \left(\partial_X U^2\right)_{Hi}\right) dV \\
& \quad + 2\int A^{3} Q^3 A^{3} \left( (\psi_{z})_{Hi}(\partial_Y - t\partial_X) \partial_X \left(U^2\right)_{Lo} \right) dV \\ 
& \quad + 2\int A^{3} Q^3 A^{3} \left((\psi_{z})(\partial_Y - t\partial_X) \partial_X  U^2\right)_{\mathcal{R}} dV \\ 
& = LP3^{0} + LP3^{C}_{LH} + LP3^{C}_{HL} + LP3^{C}_{\mathcal{R}}. 
\end{align*} 
As in the treatment of $LS3$, the latter three terms can be treated perturbatively. 

\paragraph{Treatment of $LP3^{0}$}
Begin with the leading order term, which is treated as follows, using the definition \eqref{def:wL}
\begin{align*} 
LP3^{0} \leq \frac{1}{2\kappa}\norm{\sqrt{\frac{\partial_t w_L}{w_L}} A^{3} Q^3}_2^2 + \frac{1}{2\kappa}\norm{\sqrt{\frac{\partial_t w_L}{w_L}}A^{3} \Delta_L U^2_{\neq}}^2_2.
\end{align*} 
The first term is absorbed by the $CK_{wL}^3$ term in \eqref{ineq:AQ3_Evo} and for the latter term we apply Lemma \ref{lem:QPELpressureI}.  

\paragraph{Treatment of $LP3^C$}
Turn first to $LP3^{C}_{HL}$, in which the coefficient is in `high frequency'. 
Here we have by \eqref{ineq:AprioriUneq}, followed by Lemma \ref{lem:ABasic}, \eqref{ineq:triQuadHL}, and Lemma \ref{lem:CoefCtrl}, we have,    
\begin{align*} 
LP3^{C}_{HL} &\lesssim \frac{\epsilon}{\jap{t}^{1-\delta_1} \jap{\nu t^3}^{\alpha}}\sum_{k,l}\int \abs{A^{3} \widehat{Q^3_k}(\eta,l) A^{3}_k(\eta,l) \hat{\psi_z}(\xi,l)_{Hi}} Low(k,\eta-\xi,l-l^\prime) d\eta \\ 
& \lesssim \frac{\epsilon}{\jap{t}^{2-\delta_1}\jap{\nu t^3}^{\alpha}}\norm{A^{3}Q^3}_2\norm{A C}_2, 
\end{align*}  
which is consistent with Proposition \ref{prop:Boot} for $\epsilon$ sufficiently small by the bootstrap hypotheses.  

Next turn to $LP3^{C}_{LH}$, which by the bootstrap hypotheses and \eqref{ineq:ABasic} is controlled via
\begin{align*} 
LP3^{C}_{LH} & \lesssim \min(\epsilon \jap{t}, c_{0})\sum_{k \neq 0,l}\int_\eta  \abs{A^{3} \widehat{Q^3_k}(\eta,l)} \frac{k\abs{\xi-kt}}{\left(k^2 + l^{\prime 2} + \abs{\xi-kt}^2\right)} \abs{A^{3}\widehat{\left(\Delta_{L} U^2 \right)}_k(\xi,l')_{Hi}} Low(\eta-\xi,l-l^\prime) d\eta \\ 
 & \lesssim \min(\epsilon \jap{t}, c_{0}) \sum_{k \neq 0,l}\int_\eta  \abs{A^{3} \widehat{Q^3_k}(\eta,l)}  \\ & \quad\quad\quad \times  \frac{\abs{k}\abs{\xi-kt}}{\left(k^2 + l^{\prime 2} + \abs{\xi-kt}^2\right) \jap{\frac{t}{\jap{\xi,l^\prime}}}} \abs{A^{2}\widehat{\left(\Delta_{L} U^2 \right)}_k(\xi,l')_{Hi}} Low(\eta-\xi,l-l^\prime) d\eta. 
\end{align*} 
We may treat this in a manner similar to the canonical $NLP(1,3,0,\neq)$ on $Q^2$ in \S\ref{sec:NLP213}.
Indeed, by \eqref{ineq:AiPartX} and \eqref{ineq:triQuadHL} we have,  
\begin{align*} 
LP3^{C}_{HL} & \lesssim c_{0}\norm{\left(\sqrt{\frac{\partial_t w}{w}} A^{3} + \frac{\abs{\grad}^{s/2}}{\jap{t}^s}A^{3}\right) Q^3}_2^2 + c_{0}\norm{\left(\sqrt{\frac{\partial_t w}{w}} A^{2}+ \frac{\abs{\grad}^{s/2}}{\jap{t}^s}A^{2}\right) \Delta_L U^2_{\neq}}_2^2 \\ 
& \quad + \epsilon\norm{A^3 Q^3_{\neq}}_2\norm{\Delta_L A^2 U^2_{\neq}}_2, 
\end{align*}  
which by Lemmas \ref{lem:PEL_NLP120neq} and  \ref{lem:SimplePEL}, is consistent with Proposition \ref{prop:Boot} by the bootstrap hypotheses. 
The remainder term $LP3^{\mathcal{R}}$ is straightforward and is omitted for the sake of brevity; this completes the treatment of $LP3$. 

\section{High norm estimates on $Q^1_0$} 
\subsection{Improvement of \eqref{ineq:Boot_Q1Hi1}} \label{sec:Q1Hi1}
The proof of \eqref{ineq:Boot_Q1Hi1} (with constant `2') proceeds slightly differently than several other estimates we are making. 
The goal is to obtain exactly $O(\epsilon \jap{t})$ growth for $Q_0^1$, rather than any logarithmic losses in $t$ or $\epsilon$ (which would be fatal to the proof of Theorem \ref{thm:Threshold}). 
We will deduce an estimate like
\begin{align} 
\frac{1}{2}\frac{d}{dt}\norm{A^{1}Q^1_0}^2_2 & \leq -\frac{t}{\jap{t}^2}\norm{A^1 Q^1_0}_2^2 +  \frac{1}{\jap{t}}\norm{A^{1}Q^1_0}_2\norm{A^2 Q^2_0}_2 + c_{0}\epsilon^2\mathcal{I}(t) \nonumber \\  
& \leq -\frac{t}{\jap{t}^2}\norm{A^1Q^1_0}_2^2 +  \frac{4\epsilon}{\jap{t}}\norm{A^{1}Q^1_0}_2 + c_{0}\epsilon^2\mathcal{I}(t),  \label{ineq:basic_Q1}
\end{align} 
where $\mathcal{I}(t)$ is integrable and $O(K_B)$ uniformly in $\epsilon$. 
This yields the desired bound by comparing $X(t) = \norm{A^1 Q^1_0(t)}^2_2$ to the super-solution of the inequality given by $Y(t) = \max\left(\frac{3}{2}K_{H10}, 6\sqrt{2}\right)\epsilon + c_{0}\epsilon^2 \int_{1}^t\mathcal{I}(\tau) d\tau$.
Indeed, for $K_{H10}$ sufficiently large, 
\begin{align*} 
\partial_t Y(t) = c_{0}\epsilon^2\mathcal{I}(t) \geq \left(-\frac{t}{\jap{t}^2}Y(t) +  \frac{4\epsilon}{\jap{t}}\right) Y(t) + c_{0}\epsilon^2\mathcal{I}(t), 
\end{align*}
as the additional two terms on the RHS sum to something negative by the choice of $Y(t)$ (recall $t \geq 1$). 
By Lemma \ref{lem:BootStart} it follows that $X(1) < Y(1)$, and hence we have by comparison and \eqref{ineq:basic_Q1} that $X(t) \leq Y(t)$ for all $t \in [1,T_\star]$.

By the above discussion, improving \eqref{ineq:Boot_Q1Hi1} reduces to proving an estimate like \eqref{ineq:basic_Q1}.
From the evolution equation for $Q^1_0$: 
\begin{align} 
\frac{1}{2}\frac{d}{dt}\norm{A^{1} Q^1_0}_2^2 & \leq \dot{\lambda}\norm{\abs{\grad}^{s/2}A^{1} Q_0^1}_2^2 - \norm{\sqrt{\frac{\partial_t w}{w}} A^{1} Q_0^1}_2^2 - \frac{t}{\jap{t}^2}\norm{A^1 Q_0^1}_2^2 \nonumber \\ 
& \quad - \int A^{1}Q^1_0 A^1 Q^2_0 dV + \nu  \int A^{1} Q^{1}_0 A^{1} \left(\tilde{\Delta_t} Q^1_0\right) dV \nonumber \\  
& \quad - \int A^{1} Q^1_0 A^{1}\left(\tilde U_0 \cdot \grad Q^1_0 \right) dV - \int A^{1} Q^1_0 A^{1} \left( Q^j_0 \partial_j^t U^1_0 + 2\partial_i^t U^j_0 \partial_{ij}^t U^1_0\right) dV \nonumber \\ 
&\quad - \int A^{1} Q^1_0 A^{1}\left(\tilde U_{\neq} \cdot \grad Q^1_{\neq} \right)_0 dV - \int A^{1} Q^1_0 A^{1} \left( Q^j_{\neq} \partial_j^t U^1_{\neq} + 2\partial_i^t U^j_{\neq} \partial_{ij}^t U^1_{\neq}\right)_0 dV \nonumber \\ 
& = -\mathcal{D}Q_0^1 - CK^1_L  + LU + \mathcal{D}_E + \mathcal{T}_0 + NLS1(j,0) + NLS2(i,j,0) + \mathcal{F} \label{eq:AevoQ10}
\end{align} 
where, as above, we write 
\begin{align*}
\mathcal{D}_E  = \nu\int A^1 Q^1_0 A^1 \left((\tilde{\Delta_t} - \Delta)Q_0^1\right) dV. 
\end{align*}
Notice that, due to the $X$ average, the linear pressure and stretching terms both disappear. 
Hence the main growth of $Q^1_0$ is caused by the lift-up effect term, $LU$.
By Cauchy-Schwarz:
\begin{align*}
LU \leq \jap{t}^{-1} \norm{A^{1} Q^1_0}_2 \norm{A^{2} Q^2_0}_2,   
\end{align*}  
which, together with \eqref{ineq:Boot_Q2Hi} is responsible for the leading order linear term in \eqref{ineq:basic_Q1}. 
Hence, it remains to see how to control the nonlinear terms.

\subsubsection{Transport nonlinearity}
This term corresponds to the transport of $Q_0^1$ by $\tilde U_0$ involving only zero frequencies.
The treatment of this term can be made in the same way as the corresponding treatment for $Q^2$ in \S\ref{sec:TransQ20}, 
which yields 
\begin{align*} 
\mathcal{T}_0 = -  \int A^{1} Q^1_0 A^{1}\left(\tilde U_0 \cdot \grad Q^1_0\right) dV& \lesssim \epsilon \norm{\grad A^1 Q^1_0}_2^2 + \frac{\epsilon}{\jap{t}^{4}} \norm{A^1 Q^1_0}_2^2 + \epsilon\norm{\grad U_0^3}_{2}^2, 
\end{align*}
which is consistent with \eqref{ineq:basic_Q1}.

\subsubsection{Nonlinear stretching $NLS$}
This term is the analogue of those treated in \S\ref{sec:NLPSQ20} and corresponds to the nonlinear stretching effects on $Q_0^1$ involving only zero frequencies (the pressure disappears due to the $X$ average), hence they are of type \textbf{(2.5NS)}. 
Notice that these nonlinear terms are linear in $Q^1_0$, which is an important ``null'' structure (and is expected from the form of the streaks \eqref{def:streak}). 
It can be handled by a variation of the corresponding treatment for $Q^2$ in \S\ref{sec:TransQ20} and \S\ref{sec:NLPSQ20}. 

Let us show the treatment of one set of representative terms and leave the others. 
Consider, for example, the $NLS1$ term: 
\begin{align*}
NLS1(j,0,0) & = \int \left(A^1 Q^1_0\right)_{>1} \left(A^1\left(Q^j_0 \partial^t_j U^1_0 \right)\right)_{>1} dX dY + \int \left(A^1 Q^1_0\right)_{\leq 1} \left(A^1\left(Q^j_0 \partial_j^t U^1_0 \right)\right)_{\leq 1} dX dY \\
& = S^H + S^L.  
\end{align*} 
Observing that $j \neq 1$, turn first to the low-frequency term. 
First, 
\begin{align*}
S^L & \lesssim \norm{\left(A^1 Q_0^1\right)_{\leq 1}}_2 \norm{\left(A^1\left( Q_0^j \partial_j^t U^1_0\right)\right)_{\leq 1}}_2 \lesssim \jap{t}^{-2}\norm{Q_0^1}_2 \norm{ Q_0^j \partial_j^t U^1_0}_2. 
\end{align*} 
For $t \leq c_0\epsilon^{-1}$, by Sobolev embedding and $\sigma^\prime > 5/2$ (also Lemma \ref{lem:CoefCtrl}), 
\begin{align*}
S^L & \lesssim \jap{t}^{-2}\norm{Q_0^1}_2 \norm{Q_0^j}_2 \norm{\grad U^1_0}_{H^{\sigma^\prime-1}} \lesssim \epsilon^3, 
\end{align*}
which is consistent with \eqref{ineq:basic_Q1} for times $t \leq c_0\epsilon^{-1}$ and $c_0$ sufficiently small. 
For times with $t > c_0\epsilon^{-1}$ we use instead (along with Lemma \ref{lem:CoefCtrl}), 
\begin{align*}
S^L & \lesssim \jap{t}^{-2}\norm{Q_0^1}_2 \norm{\grad U_0^j}_{H^{\sigma^\prime-1}} \norm{\grad U^1_0}_{H^{\sigma^\prime-1}} \\ 
& \lesssim c_0^{-1} \epsilon^2 \norm{\grad U_0^j}_{H^{\sigma^\prime-1}} \norm{\grad U^1_0}_{H^{\sigma^\prime-1}} \\ 
& \lesssim c_0^{-2} \epsilon^3 \norm{\grad U^1_0}_{H^{\sigma^\prime-1}}^2 + \epsilon\norm{\grad U_0^j}_{H^{\sigma^\prime-1}}^2, 
\end{align*}  
which, since $c_0^{-1}\epsilon \lesssim \nu$, integrates to be $O(c_0 \epsilon^2)$ by the low frequency controls \eqref{ineq:Boot_LowFreq} and $j \neq 1$.
The treatment of the high frequencies is slightly easier: by Lemma \ref{lem:AAiProd} and Lemma \ref{lem:PELbasicZero}, 
\begin{align*}
S^{H} & \lesssim \norm{\left(A^1 Q_0^1\right)_{>1}}_2 \norm{A^1\left( Q_0^j \partial_j^t U^1_0\right)_{>1}}_2 \\ 
& \lesssim \norm{\grad A^1 Q_0^1}_2 \norm{A^j Q_0^j}_2 \norm{\grad A^1 U^1_0}_2 \\ 
& \lesssim \epsilon\norm{\grad A^1 Q^1_0}_2^2 + \frac{\epsilon}{\jap{t}^{2}}\norm{\grad U_0^1}_2^2 + \epsilon^3 \norm{\grad AC}_2^2. 
\end{align*} 
To see that the middle term is integrable, we use a trick similar that employed above on the low frequency term. For $t \leq c_0 \epsilon^{-1}$ we have by \eqref{ineq:Boot_U01_Low1},
\begin{align*}
\frac{\epsilon}{\jap{t}^{2}}\norm{\grad U_0^1}_2^2 & \lesssim K_{U1}^2\epsilon^3, 
\end{align*}
which is consistent with \eqref{ineq:basic_Q1} for $t \leq c_0 \epsilon^{-1}$. After this point,  
\begin{align*}
\frac{\epsilon}{\jap{t}^{2}}\norm{\grad U_0^1}_2^2 & \lesssim c_0^{-2} \epsilon^3 \norm{\grad U_0^1}_2^2, 
\end{align*}
which integrates to $O(c_0 \epsilon^2)$ by \eqref{ineq:Boot_U01_Low2}.  
This completes the treatment of $NLS1(j,0)$. The $NLS2(i,j,0)$ terms follow similarly and are hence omitted for brevity. 

\subsubsection{Forcing from non-zero frequencies}
In this section we consider interactions of type \textbf{(F)}: the forcing of non-zero frequencies directly back onto $Q_0^1$. 
From \eqref{eq:XavgCanc} and noting the $X$ averages, we get 
\begin{align*} 
\mathcal{F} & = -\int A^1 Q^1 A^1 \left(\partial_Y^t \partial_Y^t \partial_Y^t \left(U^2_{\neq} U^1_{\neq}\right)_{0} + \partial_Y^t \partial_Y^t \partial_Z^t \left(U^3_{\neq} U^1_{\neq}\right)_{0} \right) dV \\
& \quad - \int A^1 Q^1 A^1\left(\partial_Z^t \partial_Z^t \partial_Z^t \left(U^3_{\neq} U^1_{\neq}\right)_{0} + \partial_Z^t \partial_Z^t \partial_Y^t \left(U^2_{\neq} U^1_{\neq}\right)_{0} \right) dV \\
& = F^1 + F^2 + F^3 + F^4. 
\end{align*}
Since all of these contributions are roughly equivalent, we will focus on $F^1$ and $F^2$.
Decompose $F^2$ with a paraproduct; as usual we group contributions where the coefficients appear in low frequency with the remainder:  
\begin{align}  
F^2 & = -\sum_{k\neq 0} \int A^{1}Q^1_0 A_0^{1} \partial_Y\partial_Y\partial_Z\left( \left(U^3_{-k}\right)_{Hi} \left( U^1_k\right)_{Lo}\right)dV \nonumber \\
& \quad - \sum_{k\neq 0} \int A^{1}Q^1_0 A_0^{1} \partial_Y \partial_Y\partial_Z\left( \left(U^3_{-k}\right)_{Lo} \left( U^1_k\right)_{Hi}\right) dV \nonumber \\
& \quad -\sum_{k\neq 0}  \int A^{1}Q^1_0 A_0^{1} (\psi_y)_{Hi} \partial_Y\partial_Y\partial_Z\left( \left(U^3_{-k}\right)_{Lo} \left( U^1_k\right)_{Lo}\right) dV \nonumber \\
& \quad - \sum_{k\neq 0} \int A^{1}Q^1_0 A_0^{1} \partial_Y\left( (\psi_y)_{Hi}\partial_Y\partial_Z\left( \left(U^3_{-k}\right)_{Lo} \left( U^1_k\right)_{Lo}\right)\right) dV \nonumber \\
& \quad - \sum_{k\neq 0} \int A^{1}Q^1_0 A_0^{1} \partial_Y \partial_Y \left( (\psi_z)_{Hi} \partial_Y \left( \left(U^3_{-k}\right)_{Lo} \left( U^1_k\right)_{Lo}\right)\right) dV \nonumber \\
& \quad - F^2_{\mathcal{R},C} \nonumber \\
& = F^2_{HL} + F^2_{LH} + F^2_{C1} + F^2_{C2} + F^2_{C3} + F^2_{\mathcal{R},C}. \label{def:F2ppQ10}
\end{align} 
Turn first to $F^{2}_{HL}$. 
By Lemma \ref{lem:ABasic}, \eqref{ineq:AdelLij} and \eqref{ineq:triQuadHL} we have 
\begin{align*} 
F_{HL}^{2} & \lesssim \frac{\epsilon \jap{t}^{\delta_1}}{\jap{\nu t^3}^{\alpha}}\sum_{l,l^\prime,k \neq 0} \int \abs{A^{1} \widehat{Q^1_0}(\eta,l) A^1_{0}(\eta,l) \frac{\abs{\eta}^2\abs{l}}{k^2 + (l^\prime)^2 + \abs{\xi-kt}^2} \Delta_L \widehat{U^3_{k}}(\xi,l^\prime)_{Hi}} \\ & \hspace{5cm} \times Low(-k,\eta-\xi,l-l^\prime) d\eta d\xi \\ 
& \lesssim \frac{\epsilon \jap{t}^{1+\delta_1}}{\jap{\nu t^3}^{\alpha}}\norm{\left(\sqrt{\frac{\partial_t w}{w}} + \frac{\abs{\grad}^{s/2}}{\jap{t}^{s}}\right) A^{1}Q^1_0}_2 \norm{\left(\sqrt{\frac{\partial_t w}{w}} + \frac{\abs{\grad}^{s/2}}{\jap{t}^s}\right)A^{3} \Delta_L U^3_{\neq}}_2
\\ & \quad + \frac{\epsilon \jap{t}^{\delta_1-1}}{\jap{\nu t^3}^{\alpha}}\norm{\sqrt{-\Delta_L} A^{1}Q^1}_2 \norm{A^{3}\Delta_L U^3_{\neq}}_2,
\end{align*} 
which, after the application of Lemmas \ref{lem:PEL_NLP120neq} and \ref{lem:SimplePEL}, is consistent with \eqref{ineq:basic_Q1}. 
 
Turn next to $F_{LH}^{2}$, for which also by \eqref{ineq:ABasic}, \eqref{ineq:AdelLij} and \eqref{ineq:quadHL} we have 
\begin{align*} 
F_{LH}^{2} & \lesssim \frac{\epsilon \jap{t}^2}{\jap{\nu t^3}^{\alpha}}\norm{\left(\sqrt{\frac{\partial_t w}{w}} + \frac{\abs{\grad}^{s/2}}{\jap{t}^{s}}\right) A^{1}Q^1_0}_2 \norm{\left(\sqrt{\frac{\partial_t w}{w}} + \frac{\abs{\grad}^{s/2}}{\jap{t}^s}\right)A^{1} \Delta_L U^1_{\neq}}_2
\\ & \quad + \frac{\epsilon }{\jap{\nu t^3}^\alpha}\norm{\sqrt{-\Delta_L} A^{1}Q^1}_2 \norm{A^{1}\Delta_L U^1_{\neq}}_2,   
\end{align*} 
which, after the application of the Lemmas \ref{lem:PEL_NLP120neq} and \ref{lem:SimplePEL}, is consistent with \eqref{ineq:basic_Q1}. 

The most difficult coefficient error term in \eqref{def:F2ppQ10} is $F^2_{C3}$, since two derivatives of the coefficients are being taken. 
By Lemma \ref{lem:ABasic}, \eqref{ineq:triQuadHL}, and Lemma \ref{lem:CoefCtrl},
\begin{align*} 
F^2_{C3} & \lesssim \frac{ \epsilon^2 \jap{t}^{\delta_1}}{ \jap{\nu t^3}^{2\alpha}} \sum_{l,l^\prime,k\neq 0} \int \abs{A^{1} \widehat{Q^1_0}(\eta,l) A_0^{1}(\eta,l) \abs{\eta}^2 \widehat{\psi_z}(\xi,l^\prime)} Low(k,\eta-\xi,l-l^\prime) d\eta d\xi \\ 
& \lesssim \epsilon \norm{\grad A^{1}Q^1_0}_2^2 + \frac{\epsilon^3 \jap{t}^{2\delta_1}}{\jap{\nu t^3}^{4\alpha}}\norm{AC}^2_2,
\end{align*} 
which is consistent with \eqref{ineq:basic_Q1} for $c_{0}$ and $\epsilon$ sufficiently small by the bootstrap hypotheses. 
The other coefficient terms in \eqref{def:F2ppQ10}, $F_{C1}^2$ and $F_{C2}^2$ are easier and give similar contributions. Hence, these are omitted for the sake of brevity.
The remainder term in \eqref{def:F2ppQ10}, $F_{\mathcal{R}}^2$, is similarly straightforward and is omitted as well.  
This completes the treatment of $F^2$. The term $F^4$ yields similar contributions with a similar treatment and is hence omitted. 

Turn to the other difficult term, $F^1$. 
Decompose the corresponding term in \eqref{eq:AevoQ10} with a paraproduct; as usual we group contributions where the coefficients appear in low frequency with the remainder:  
\begin{align}  
F^1 & = -\sum_{k\neq 0} \int A^{1}Q^1_0 A_0^{1} \partial_Y\partial_Y\partial_Y\left( \left(U^2_{-k}\right)_{Hi} \left( U^1_k\right)_{Lo}\right)dV \nonumber \\
& \quad - \sum_{k\neq 0} \int A^{1}Q^1_0 A_0^{1} \partial_Y \partial_Y\partial_Y\left( \left(U^2_{-k}\right)_{Lo} \left( U^1_k\right)_{Hi}\right) dV \nonumber \\
& \quad -\sum_{k\neq 0}  \int A^{1}Q^1_0 A_0^{1} (\psi_y)_{Hi} \partial_Y\partial_Y\partial_Y\left( \left(U^2_{-k}\right)_{Lo} \left( U^1_k\right)_{Lo}\right) dV \nonumber \\
& \quad - \sum_{k\neq 0} \int A^{1}Q^1_0 A_0^{1} \partial_Y\left( (\psi_y)_{Hi}\partial_Y\partial_Y\left( \left(U^2_{-k}\right)_{Lo} \left( U^1_k\right)_{Lo}\right)\right) dV \nonumber \\
& \quad - \sum_{k\neq 0} \int A^{1}Q^1_0 A_0^{1} \partial_Y \partial_Y \left( (\psi_y)_{Hi} \partial_Y \left( \left(U^2_{-k}\right)_{Lo} \left( U^1_k\right)_{Lo}\right)\right) dV \nonumber \\
& \quad - F^1_{\mathcal{R},C} \nonumber \\
& = F^1_{HL} + F^1_{LH} + F^1_{C1} + F^1_{C2} + F^1_{C3} + F^1_{\mathcal{R},C}. \label{def:F1ppQ10}
\end{align}
Turn to $F^1_{HL}$ first, which is the most difficult term. 
By \eqref{ineq:ABasic}, \eqref{ineq:AdeYYY} and \eqref{ineq:triQuadHL} we have 
\begin{align*} 
F_{HL}^{1} & \lesssim \frac{\epsilon \jap{t}^{\delta_1}}{\jap{\nu t^3}^{\alpha}} \sum_{l,l^\prime,k \neq 0} \int \abs{A^{1} \widehat{Q^1_0}(\eta,l) A^1_{0}(\eta,l) \frac{\abs{\eta}^3}{k^2 + (l^\prime)^2 + \abs{\xi-kt}^2} \Delta_L \widehat{U^2_{k}}(\xi,l^\prime)_{Hi}} \\ & \hspace{5cm} \times Low(-k,\eta-\xi,l-l^\prime) d\eta d\xi \\ 
& \lesssim \frac{\epsilon \jap{t}^{2+\delta_1}}{\jap{\nu t^3}^{\alpha}}\norm{\left(\sqrt{\frac{\partial_t w}{w}} + \frac{\abs{\grad}^{s/2}}{\jap{t}^{s}}\right) A^{1}Q^1_0}_2 \norm{\left(\sqrt{\frac{\partial_t w}{w}} + \frac{\abs{\grad}^{s/2}}{\jap{t}^s}\right)A^{2} \Delta_L U^2_{\neq}}_2
\\ & \quad + \frac{\epsilon \jap{t}^{\delta_1-1}}{\jap{\nu t^3}^{\alpha}}\norm{\sqrt{-\Delta_L} A^{1}Q^1}_2 \norm{A^{2}\Delta_L U^2_{\neq}}_2,
\end{align*}
 which, after the application of Lemmas \ref{lem:PEL_NLP120neq} and \ref{lem:SimplePEL}, is consistent with \eqref{ineq:basic_Q1}. 
The treatment of $F^1_{LH}$ is similar, except instead of the extra $\jap{t}^{-1}$ coming from the ratio of $A^1$ and $A^2$ as above, it comes from the inviscid damping of the low frequency factor. Hence, we omit the treatment for brevity. 
The coefficient error terms and the remainder terms are treated the same way they are treated above for $F^2$ and hence are also omitted. This completes the treatment of $F^1$. 

The last term to consider is $F^3$, which is slightly different due to the large number of $Z$ derivatives. 
As above, we will expand with a paraproduct but we will apply \eqref{ineq:AdeZZZ} to the leading order terms, which yields contributions of the form
\begin{align*}
F_{HL}^3 + F_{LH}^3 & \lesssim \frac{\epsilon}{\jap{\nu t^3}^{\alpha}}\norm{\sqrt{-\Delta_L} A^{1}Q^1}_2\left( \norm{A^{1}\Delta_L U^1_{\neq}}_2 +  \norm{A^{3}\Delta_L U^3_{\neq}}_2\right), 
\end{align*}
which, after the application of Lemma \ref{lem:SimplePEL} is consistent with \eqref{ineq:basic_Q1}.
The remainder and coefficient errors are analogous to the treatments of previous $F^i$ above. 
We omit the details for brevity. 
  
\subsubsection{Dissipation error terms}
These can be treated in the same manner as the dissipation error terms on $Q^2_0$ were treated in \S\ref{sec:DEQ02}.
We omit the details for brevity. 

\subsection{Low norm estimate on $Q^1_0$, improvement to \eqref{ineq:Boot_Q1Mid}}
This section is to deduce the improvement to \eqref{ineq:Boot_Q1Mid}, which is similar to \eqref{ineq:Boot_Q1Hi1} except at a lower regularity and offering a uniform bound past $t \gtrsim \epsilon^{-1}$.  
The proof is much easier than \eqref{ineq:Boot_Q1Hi1}, due to the lower regularity. 
In what follows denote 
\begin{align*}
A^{S}(t,\grad) := \jap{\grad}^\gamma e^{\lambda(t)\abs{\grad}^s}. 
\end{align*}
Consider from the evolution equation for $Q^1_0$: 
\begin{align} 
\frac{1}{2}\frac{d}{dt}\norm{Q^1_0}_{\G^{\lambda,\gamma}}^2 & \leq \dot{\lambda}\norm{\abs{\grad}^{s/2}Q_0^1}_{\G^{\lambda,\gamma}}^2 -\int A^{S} Q^1_0 A^{S} Q^2_0 dV + \nu \int A^{S} Q^{1}_0 A^{S} \left(\tilde{\Delta_t} Q^1_0\right) dv \nonumber \\ & \quad 
- \int A^{S} Q^1_0 A^{S}\left(\tilde U_0 \cdot \grad Q^1_0\right)  dV - \int A^{S} Q^1_0 A^{S} \left(Q^j_0 \partial_j^t U^1_0 + 2\partial_i^t U^j_0 \partial_{ij}^t U^1_0\right) dV \nonumber \\ 
& \quad - \int A^{S} Q^1_0 A^{S} \left(\tilde U_{\neq} \cdot \grad Q^1_{\neq} \right)_0 dV - \int A^{S} Q^1_0 A^{S} \left(Q^j_{\neq} \partial_j^t U^1_{\neq} + 2\partial_i^t U^j_{\neq} \partial_{ij}^t U^1_{\neq} \right)_0 dV \nonumber \\ 
& = -CK^1_\lambda + LU + \mathcal{D} + \mathcal{T}_0 + NLS + \mathcal{F}. \label{eq:AevoQ10Low}
\end{align} 
Relative to \S\ref{sec:Q1Hi1}, the treatment of the lift-up effect is more precise for long times: using \eqref{ineq:Boot_LowFreq}, there is a $K > 0$ such that 
\begin{align*} 
LU & \lesssim \norm{Q^1_0}_2 \norm{Q^2_0}_2 + \norm{\grad Q^1_0}_{\G^{\lambda,\gamma}} \norm{\grad A^2 Q^2_0}_2 \\
& \leq \frac{K\epsilon}{\jap{\nu t}^\alpha}\norm{Q^1_0}^2_2 + \frac{K\epsilon}{\jap{\nu t}^\alpha} + \frac{\nu}{10} \norm{\grad Q^1_0}^2_{\G^{\lambda,\gamma}} + K \nu^{-2} \left( \nu \norm{\grad A^2 Q^2_0}^2_2\right), 
\end{align*} 
which is consistent with Proposition \ref{prop:Boot} for $c_0$ sufficiently small and $K_{M1}$ sufficiently large. 
The remaining nonlinear terms in \eqref{eq:AevoQ10Low} are very easy: the transport nonlinearity $\mathcal{T}$ can be treated essentially the same as in \S\ref{sec:TransQ20}, the dissipation error terms as in \S\ref{sec:DEQ02}, the $NLS$ terms as in \S\ref{sec:NLPSQ20}, and finally the forcing terms $\mathcal{F}$ can be treated like remainder terms are treated in the high norm estimates due to the regularity gap between \eqref{ineq:Boot_Q1Mid} and the high norm estimates (see a similar argument in \S\ref{sec:Lowg} below).
 We omit the details as they are repetitive (see \S\ref{sec:Lowg} below for a similar proof).  

\section{High norm estimate on $Q^1_{\neq}$} \label{sec:HiQ1neq}
Consider from the evolution equation for $Q^1$: 
\begin{align} 
\frac{1}{2}\frac{d}{dt} \norm{A^{1} Q^1_{\neq}}_2^2  & \leq \dot{\lambda}\norm{\abs{\grad}^{s/2}A^{1} Q^1_{\neq}}_2^2 - \norm{\sqrt{\frac{\partial_t w}{w}} A^{1} Q^1_{\neq}}_2^2 - \norm{\sqrt{\frac{\partial_t w_L}{w_L}} A^{1} Q^1_{\neq}}_2^2 \nonumber \\
 & \quad - \frac{t}{\jap{t}^{2}}\norm{A^{1}Q^1_{\neq}}_2^2 -\frac{(1+\delta_1)}{t} \norm{\mathbf{1}_{t > \jap{\grad_{Y,Z}}} A^{1} Q^1_{\neq}}_2^2 \nonumber \\
& \quad - \int A^{1}Q^1_{\neq} A^1 Q^2_{\neq} dV -2 \int A^{1} Q^1 A^{1} \partial_{YX}^t U^1_{\neq} dV + 2 \int A^{1} Q^1_{\neq} A^{1} \partial_{XX} U^2_{\neq} dV \nonumber \\  
 & \quad + \nu \int A^{1} Q^{1}_{\neq} A^{1} \left(\tilde{\Delta_t} Q^1_{\neq}\right) dV  - \int A^{1} Q^1_{\neq} A^{1}\left( \tilde U \cdot \grad Q^1 \right) dV \nonumber \\ 
& \quad -  \int A^{1} Q^1_{\neq} A^{1} \left[ Q^j \partial_j^t U^1 + 2\partial_i^t U^j \partial_{ij}^t U^1  - \partial_X\left(\partial_i^t U^j \partial_j^t U^i\right) \right] dV \nonumber \\ 
& = -\mathcal{D}Q^1_{\neq} - CK_{L1}^1 - (1+\delta_1)CK_{L2}^1 + LU + LS1 + LP1 \nonumber \\ & \quad\quad + \mathcal{D}_E + \mathcal{T} + NLS1 + NLS2 + NLP,  \label{ineq:Q1HneqEvo}
\end{align}
where as usual 
\begin{align*}
\mathcal{D}_E = \int A^1 Q^1_{\neq} A^1 \left((\tilde{\Delta_t} - \Delta_L)Q^1_{\neq}\right) dV. 
\end{align*}
The main growth is from $LU$ and $LS1$ together. 
The fact that both are acting at once is the reason for the $\jap{t}^{\delta_1}$ loss, seen at the `low' frequencies in $A^1$. 
We define enumerations of the nonlinear terms analogous to those in \eqref{def:Q2Enums} and \eqref{def:Q3Enums}. 

\subsection{Linear stretching term $LS1$} \label{sec:LS1_Hi}
One of the primary difficulties is the linear stretching term $LS1$.
We use essentially the same treatment as in \S\ref{sec:LS30_Hi} except with a small change in the use of $CK_{L1}^1$ and $CK_{L2}^1$. 
First separate into two parts (to be sub-divided further), 
\begin{align*} 
LS1 & = -2\int A^{1} Q^1 A^{1}\partial_X(\partial_Y - t\partial_X) U^1  dV - 2\int A^{1} Q^1 A^{1} \partial_X\left((\psi_y)(\partial_Y - t\partial_X)\right) U^1 dV \\ 
& = LS1^0 + LS1^{C}.   
\end{align*} 

\subsubsection{Treatment of $LS1^C$}
The $LS1^C$ term can be treated in essentially the same manner as the corresponding $LS3^C$ in \S\ref{sec:LS3C}. 
Hence, we omit the details for brevity and conclude
\begin{align*}  
LS1^C & \lesssim \frac{\epsilon^{1/3}}{\jap{\nu t^3}^\alpha}\norm{A^1 Q^1}_2^2 + \frac{\epsilon^{5/3}}{\jap{t}^{2-2\delta_1}\jap{\nu t^3}^\alpha}\norm{A C}_2^2 \\ 
& \quad + c_{0}\norm{\left(\sqrt{\frac{\partial_t w}{w}} + \frac{\abs{\grad}^{s/2}}{\jap{t}^{s}}\right)A^1 Q^1}_2 \norm{\left(\sqrt{\frac{\partial_t w}{w}} + \frac{\abs{\grad}^{s/2}}{\jap{t}^{s}}\right)A^1 \Delta_L U_{\neq}^1}_2 \\ & \quad + \epsilon \norm{\sqrt{-\Delta_L} A^1 Q^1}_2 \norm{A^1 \Delta_L U^1_{\neq}}_2. 
\end{align*}  
By Lemmas \ref{lem:PEL_NLP120neq} and \ref{lem:SimplePEL}, this is consistent with Proposition \ref{prop:Boot} after applying the bootstrap hypotheses and integrating in time. 

\subsubsection{Leading order term $LS1^0$}
This term is treated with a variant of the method used to treat $LS3^0$ in \S\ref{sec:LS30}, however, we need to make a few minor changes due to the slightly different use of the $CK_{Li}$ terms here. 
As in \eqref{eq:LS30} of \S\ref{sec:LS30}, we first expand by writing out $\Delta_t^{-1}$ in terms of $\Delta_L$: 
\begin{align*} 
LS1^0 & =  -2 \int A^{1} Q^1 A^{1}\partial_X(\partial_Y - t\partial_X) \Delta_{L}^{-1} \left[Q^1 - G (\partial_Y - t\partial_X)^2 U^1 \right. \\ & \left. \quad\quad\quad  - 2 \psi_z \partial_Z (\partial_Y - t\partial_X) U^1  - \Delta_t C (\partial_Y - t\partial_X) U^1\right]  dV \\  
& = LS1^{0;0}  + LS1^{0;C1} + LS1^{0;C2} + LS1^{0;C3}. 
\end{align*} 
Consider first the leading order term, $LS1^{0;0}$.  
Most of the contributions on the frequency side are dealt with similarly to the analogous terms in \S\ref{sec:LS30}. We omit these treatments and reduce ourselves to the following for a universal constant $K > 0$
\begin{align*} 
LS1^{0,0} & \leq \frac{\delta_\lambda}{10\jap{t}^{3/2}} \norm{\abs{\grad}^{s/2}A^{1}Q^1_{\neq}}_2^2 + \frac{K}{\delta_\lambda^{\frac{1}{2s-1}} \jap{t}^{3/2}}\norm{A^{1}Q^1_{\neq}}_2^2 + \frac{K}{\kappa}\norm{\sqrt{\frac{\partial_t w}{w}}A^{1} Q^1_{\neq}}_2^2 + LS1^{0;0,LT,Y},
\end{align*}
 where (the terminology is from \S\ref{sec:LS30}),  
\begin{align*} 
LS1^{0;0,LT,Y} = -2 \sum_{k \neq 0} \int \mathbf{1}_{t > 2\abs{\eta}} \mathbf{1}_{t \geq \jap{\eta,l}} \abs{A^{1} \widehat{Q^1_k}(\eta,l)}^2 \frac{k(\eta-kt)}{k^2 + l^2 + \abs{\eta-kt}^2} d\eta.  
\end{align*} 
Write, 
\begin{align*} 
LS1^{0;0,LT,Y} & \leq 2 \sum \int \mathbf{1}_{t > 2\abs{\eta}} \mathbf{1}_{t \geq \jap{\eta,l}} \abs{A^{1} \widehat{Q^1_k}}^2 \frac{1}{\sqrt{1 + \abs{t- \abs{\eta}}^2}} d\eta \\ 
& = (1+\delta_1)CK^1_{L2} + (1-\delta_1)CK_{L1}^1 \\ 
&\quad  - \sum \int \mathbf{1}_{t > 2\abs{\eta}} \mathbf{1}_{t \geq \jap{\eta,l}} \abs{A^{1} \widehat{Q^1_k}}^2 \left(\frac{1+ \delta_1}{t} + \frac{(1-\delta_1)t}{\jap{t}^2} -  \frac{2}{\sqrt{1 + \left(t- \abs{\eta}\right)^2}} \right) d\eta \\  
& = (1+\delta_1)CK^1_{L2} + (1-\delta_1)CK_{L1}^1 + LS1^{0;0,LT,Y}_{\mathcal{R}}. 
\end{align*} 
The remainder term is written
\begin{align*} 
LS1^{0;0,LT,Y}_{\mathcal{R}} & = - \sum_{k \neq 0} \int \mathbf{1}_{t > 2\abs{\eta}} \mathbf{1}_{t \geq \jap{\eta,l}} \abs{A^{1} \widehat{Q^1_{k}}}^2 \left(\frac{2}{t} -  \frac{2}{\sqrt{1 + \left(t- \abs{\eta}\right)^2}} \right) d\eta \\ 
& \quad - \sum_{k \neq 0} \int \mathbf{1}_{t > 2\abs{\eta}} \mathbf{1}_{t \geq \jap{\eta,l}} \abs{A^{1} \widehat{Q^1_{k}}}^2 \left(\frac{(1-\delta_1)t}{\jap{t}^2} - \frac{1- \delta_1}{t}\right) d\eta.  
\end{align*} 
The first term is treated as in \eqref{ineq:LS3Remainder} in \S\ref{sec:LS30} and the second term is controlled by first noting that
\begin{align*} 
\left(\frac{(1-\delta_1)t}{\jap{t}^2} - \frac{1- \delta_1}{t}\right) = \frac{1-\delta_1}{\jap{t}^2 t}\left(t^2 - \jap{t}^2\right) = -\frac{1-\delta_1}{\jap{t}^2 t}. 
\end{align*} 
Therefore we get for a universal $K > 0$, 
\begin{align*} 
LS1^{0;0,LT,Y}_{\mathcal{R}} & \leq \frac{\delta_\lambda}{10\jap{t}^{3/2}} \norm{\abs{\grad}^{s/2} A^{1}Q^1_{\neq}}_2^2 +  \frac{K}{\delta_\lambda^{\frac{1}{2s-1}}\jap{t}^{3/2}} \norm{ A^{1}Q^1_{\neq}}_2^2 + \frac{1-\delta_1}{\jap{t}^2 t} \norm{A^1 Q^1_{\neq}}_2^2, 
\end{align*} 
which now time-integrates to a contribution consistent with Proposition \ref{prop:Boot} under the bootstrap hypotheses for $K_{H1\neq} \gg \exp\left[C\delta_\lambda^{-\frac{1}{2s-1}}\right]$, for some universal constant $C > 0$ (via integrating factors; see \eqref{ineq:LS300LTZ} for another example).      

The error terms $LS1^{0;Ci}$ are treated in a manner similar to the analogous terms in $LS3$ in \S\ref{sec:LS30} and hence the details are omitted for brevity. This completes the treatment of the $LS1$ term. 

\subsection{Lift-up effect term $LU$} \label{sec:LUQhi2}
Were we able to directly use $CK_{L1}^1$ and $CK_{L2}^1$, this would be immediate, however, we need most of these terms to control $LS1$, as seen above in \S\ref{sec:LS1_Hi}.  
Instead we use (also using that $\delta_\lambda < 1$), 
\begin{align*} 
LU & \leq \delta_1 t\jap{t}^{-2}\norm{A^{1} Q^1_{\neq}}^2_2 + \frac{1}{4\delta_1 t}\norm{A^{2} Q^2_{\neq}}_2^2 \\ 
& \leq \delta_1t\jap{t}^{-2}\norm{A^{1} Q^1_{\neq}}^2_2 + \frac{1}{4\delta_1 t^{3/2}}\norm{\jap{\grad}^{1/4}A^2Q^2_{\neq}}_2^2 +  \frac{1}{4\delta_1 t} \norm{\mathbf{1}_{t > \jap{\grad_{Y,Z}}} A^{2}Q^2_{\neq}}_2^2 \\ 
& \leq \delta_1t\jap{t}^{-2}\norm{A^{1} Q^1_{\neq}}^2_2 + \frac{\delta_\lambda}{4\delta_1 t^{3/2}}\norm{\abs{\grad}^{s/2}A^2 Q^2_{\neq}}_2^2 \\ & \quad\quad + \frac{1}{4\delta_1 \delta_\lambda^{\frac{1}{2s-1}} t^{3/2}} \norm{A^2 Q^2_{\neq}}_2^2  +  \frac{1}{4\delta_1 t} \norm{\mathbf{1}_{t > \jap{\grad_{Y,Z}}} A^{2}Q^2_{\neq}}_2^2. 
\end{align*}  
The first term is absorbed by the remaining piece of $CK_{L1}^1$ left over in \eqref{ineq:Q1HneqEvo} from the treatment of $LS1$.
The others are consistent with Proposition \ref{prop:Boot} via \eqref{ineq:Boot_Q2Hi} if $K_{H1\neq}$ is sufficiently large relative to $4\max(\delta_1^{-1}, \delta_1^{-1} \delta_\lambda^{-\frac{1}{2s-1}})$ (note in particular the usefulness of $CK_L^2$). This suffices to treat $LU$.

\subsection{Linear pressure term $LP1$}
The linear pressure term here is easier than $LP3$ treated in \S\ref{ineq:LP3_Hi}. 
From Cauchy-Schwarz and Lemma \ref{lem:dtw}, we get
\begin{align*}
LP1 & \leq 2 \sum_{k \neq 0} \int \abs{A^1\widehat{Q^1_k}(\eta,l) \frac{\abs{k}^2}{k^2 + l^2 + \abs{\eta-kt}^2} A^1 \Delta_L \widehat{U^2_k}(\eta,l)}  d\eta \\ 
& \lesssim \kappa^{-1} \jap{t}^{-1} \norm{\sqrt{\frac{\partial_t w}{w}} A^1 Q^1_{\neq}}_2\norm{\sqrt{\frac{\partial_t w}{w}} A^2\Delta_L U_{\neq}^2}_2 + \jap{t}^{-3}\norm{A^1Q^1}_2\norm{A^2\Delta_L U_{\neq}^2}_2. 
\end{align*}
Therefore for $\kappa$ and $K_{H1\neq}$ sufficiently large and $c_{0}$ sufficiently small, this is consistent with Proposition \ref{prop:Boot} by the bootstrap hypotheses after applying Lemmas \ref{lem:PEL_NLP120neq} and \ref{lem:SimplePEL}.
 
\subsection{Nonlinear pressure $NLP$}
This refers to the nonlinear pressure interactions of types \textbf{(SI)} and \textbf{(3DE)}. 
None of the existing terms here are worse than those appearing in $Q^2$ in \S\ref{sec:NLPQ2} or $Q^3$ in \S\ref{sec:NLP3}. Moreover, on $Q^1$ we are imposing less control (since $A^1$ is significantly weaker than $A^{2,3}$ at high frequencies due to the extra $\jap{t}^{-1}$ in the definition of $A^1$) and the leading derivative is an $X$ derivative, which is generally less dangerous than those associated with $Y$ and $Z$. 
Therefore, the treatment of the $NLP$ contributions here are an easy variant of the treatments in \S\ref{sec:NLPQ2} and \S\ref{sec:NLP3}. Accordingly, the details are omitted for the sake of brevity. 

\subsection{Nonlinear stretching $NLS$}
These terms look more dangerous than the corresponding $NLS$ terms in $Q^{2,3}$ due to the persistent presence of $U^1$, however, this is balanced by the allowed linear growth of $Q^1_{\neq}$ even at high frequencies. 

\subsubsection{Treatment of $NLS1$}
Consider first the $NLS1(j,\neq,0)$ terms. Note $j \neq 1$ due to the zero frequencies. The case $j = 3$ is worse than $j=2$ due to the larger growth of $Q^3$, hence let us just treat this case. 
Expand with a paraproduct and group any terms with coefficients in low frequencies in with the remainder  
\begin{align*} 
NLS1(3,\neq,0) & = -\sum_{k \neq 0} \int A^1 Q^1_k A^1\left( (Q^3_{k})_{Hi} (\partial_Z U_0^1)_{Lo} \right) dV - \sum \int A^1 Q^1_k A^1\left( (Q^3_{k})_{Lo} (\partial_Z U_0^1)_{Hi} \right) dV \\ 
& \quad - \sum_{k \neq 0} \int A^1 Q^1_k A^1\left( (Q^3_{k})_{Lo} (\psi_z)_{Hi} (\partial_Y U_0^1)_{Lo} \right) dV + S_{\mathcal{R},C} \\ 
& = S_{HL} + S_{LH} + S_{C} + S_{\mathcal{R},C}, 
\end{align*}
where $S_{\mathcal{R},C}$ includes the remainders from the paraproduct and the low frequency coefficient terms. 
By \eqref{ineq:AprioriU0} and \eqref{ineq:ABasic}, followed by \eqref{ineq:ratlongtime} and \eqref{ineq:triQuadHL},
\begin{align*}
S_{HL} & \lesssim \max(\epsilon \jap{t},c_0) \sum_{k \neq 0} \int \abs{ A^1 \widehat{Q^1_k}(\eta,l) \jap{t}^{-1} \jap{\frac{t}{\jap{\xi,l^\prime}}}^{1-\delta_1} A^3 \widehat{Q^3_{k}}(\xi,l^\prime) Low(\eta-\xi,l-l^\prime) } d\eta d\xi \\ 
& \lesssim \epsilon \norm{A^1 Q^1_{\neq}}_2 \norm{\sqrt{-\Delta_L} A^3 Q^3}_2, 
\end{align*}
which is consistent with Proposition \ref{prop:Boot} for $c_0$ sufficiently small. 
For $S_{LH}$ we use Lemma \ref{lem:ABasic} and \eqref{ineq:AprioriUneq} followed by \eqref{ineq:triQuadHL}, 
\begin{align*}
S_{LH} & \lesssim \frac{\epsilon \jap{t}^2 }{\jap{\nu t^3}^{\alpha}} \sum_{k} \int \abs{ A^1 \widehat{Q^1_k}(\eta,l) \frac{1}{\jap{t} \jap{\xi,l^\prime}} \jap{\frac{t}{\jap{\xi,l^\prime}}}^{-1-\delta_1} A \widehat{U^1_{0}}(\xi,l^\prime) Low(k,\eta-\xi,l-l^\prime) } d\eta d\xi \\ 
& \lesssim \frac{\epsilon}{\jap{\nu t^3}^{\alpha}}\norm{A^1 Q^1}_2 \norm{A U_0^1}_2,
\end{align*}
which is consistent with Proposition \ref{prop:Boot} for $c_0$ sufficiently small. 
 By a similar argument, 
\begin{align*}
S_{C} & \lesssim \frac{\epsilon^2 \jap{t}}{\jap{\nu t^3}^{\alpha}}\norm{A^1 Q^1}_2 \norm{A C}_2. 
\end{align*}
The remainder terms are similar to the above and are hence omitted for brevity. This completes the treatment of the $NLS1(j,\neq,0)$ terms. 

Next consider the $NLS1(j,0,\neq)$ terms. Here, let us focus on the case $j=1$ (which does not cancel). 
Expand with a paraproduct and group any terms with coefficients in low frequencies in with the remainder  
\begin{align*}
NLS1(1,0,\neq) & = -\sum_{k \neq 0} \int A^1 Q^1_k A^1\left( (Q^1_{0})_{Hi} ( \partial_X U_k^1)_{Lo} \right) dV - \sum_{k \neq 0} \int A^1 Q^1_k A^1\left( (Q^1_{0})_{Lo} ( \partial_X U_k^1)_{Hi} \right) dV + S_{\mathcal{R},C} \\  
& = S_{HL} + S_{LH} + S_{\mathcal{R},C}. 
\end{align*}
From Lemma \ref{lem:ABasic}, \eqref{ineq:AprioriUneq}, and \eqref{ineq:triQuadHL},  
\begin{align*}
S_{HL} & \lesssim \frac{\epsilon \jap{t}^{\delta_1}}{\jap{\nu t^3}^{\alpha}} \norm{A^1 Q^1_{\neq}}_2 \norm{A^1 Q^1_0}_2. 
\end{align*}
Similarly from \eqref{ineq:ABasic} and \eqref{ineq:AprioriU0}, followed by \eqref{ineq:AiPartX} and \eqref{ineq:triQuadHL}, 
\begin{align*}
S_{LH} & \lesssim \max(\epsilon t, c_0) \sum_{k} \int \abs{ A^1 \widehat{Q^1_k}(\eta,l) \frac{\abs{k}}{k^2 + (l^\prime)^2 + \abs{\eta-kt}^2} \Delta_L A^1 \widehat{U^1_{k}}(\xi,l^\prime) Low(\eta-\xi,l-l^\prime) } d\eta d\xi \\ 
& \lesssim c_0 \norm{\left(\sqrt{\frac{\partial_t w}{w}} + \frac{\abs{\grad}^{s/2}}{\jap{t}^{s}}\right) A^1 Q^1}_2\norm{\left(\sqrt{\frac{\partial_t w}{w}} + \frac{\abs{\grad}^{s/2}}{\jap{t}^{s}}\right) \Delta_L A^1 U^1_{\neq}}_2 \\ 
& \quad + \epsilon \norm{A^1 Q^1_{\neq}}_2 \norm{A^1 \Delta_L A^1 U^1_{\neq}}_2, 
\end{align*}
which is consistent with Proposition \ref{prop:Boot} after applying Lemmas \ref{lem:PEL_NLP120neq} and \ref{lem:SimplePEL}. 
The remainder and coefficient error term are treated as usual and are hence omitted for brevity. 
This completes the $NLS1(j,0,\neq)$ terms. 

Finally consider the $NLS1(j,\neq,\neq)$ terms for interactions of type \textbf{(3DE)}.  
All these terms are treated similarly, hence, consider just $j = 3$. 
Expand as above 
\begin{align*}
NLS1(3,\neq,\neq) & = -\sum_{k \neq 0} \int \mathbf{1}_{kk^\prime(k-k^\prime) \neq 0}  A^1 Q^1_k A^1\left( (Q^3_{k^\prime})_{Hi} (\partial_Z U_{k-k^\prime}^1)_{Lo} \right) dV \\
& \quad -\sum_{k \neq 0} \int \mathbf{1}_{kk^\prime(k-k^\prime) \neq 0}  A^1 Q^1_k A^1\left( (Q^3_{k^\prime})_{Lo} (\partial_Z U_{k-k^\prime}^1)_{Hi} \right) dV \\ 
& \quad - \sum_{k \neq 0} \int \mathbf{1}_{kk^\prime(k-k^\prime) \neq 0}  A^1 Q^1_k A^1\left( (Q^3_{k^\prime})_{Lo} (\psi_z)_{Hi} ( (\partial_Y - t\partial_X) U_{k-k^\prime}^1)_{Lo} \right) dV \\ 
& \quad + S_{\mathcal{R},C} \\ 
& = S_{HL} + S_{LH} + S_{C} + S_{\mathcal{R},C}. 
\end{align*}
For $S_{HL}$, we have by \eqref{ineq:AprioriUneq} and \eqref{ineq:ABasic}, 
\begin{align*} 
S_{HL} & \lesssim \frac{\epsilon \jap{t}^{\delta_1}}{\jap{\nu t^3}^{\alpha}}\norm{A^1 Q^1_{\neq}}_2 \norm{A^3 Q^3}_2.
\end{align*} 
For $S_{LH}$ we have \eqref{ineq:Boot_Hi} and \eqref{ineq:ABasic}, 
and \eqref{ineq:triQuadHL}, 
\begin{align*} 
S_{LH} & \lesssim \frac{\epsilon \jap{t}^2}{\jap{\nu t^3}^{\alpha}} \sum_{k} \int \abs{ A^1 \widehat{Q^1_k}(\eta,l) \frac{ \mathbf{1}_{kk^\prime(k-k^\prime) \neq 0}}{\abs{k^\prime} + \abs{l^\prime} + \abs{\eta-k^\prime t}} \Delta_L A^1 \widehat{U^1_{k^\prime}}(\xi,l^\prime)_{Hi} Low(k-k^\prime, \eta-\xi,l-l^\prime) } d\eta d\xi \\ 
& \lesssim \frac{\epsilon \jap{t}^2}{\jap{\nu t^3}^{\alpha}}\norm{A^1 Q^1}_2 \norm{A^1 \Delta_L U_{\neq}^1}_2, 
\end{align*}
which is consistent with Proposition \ref{prop:Boot} for $c_0$ sufficiently small (we could have been more precise here but it was not necessary). 
This completes the $NLS1$ terms. 

\subsubsection{Treatment of $NLS2$} 

Turn to the $NLS2$ terms. These terms are all treated via easy variants of the treatments of the $NLS1$ and $NLP$ terms. They are hence omitted for the sake of brevity. 

\subsection{Transport nonlinearity $\mathcal{T}$} 
This term is treated in the same manner as the corresponding terms in $Q^3$, found in \S\ref{sec:Q3_TransNon}, and in $Q^2$, found in \S\ref{sec:Q2_TransNon}. 
The results are analogous as to those found therein and are hence omitted here for brevity. 

\subsection{Dissipation error terms $\mathcal{D}$}
These terms are treated in the same manner as the corresponding terms in $Q^3$, found in \S\ref{sec:DEneqQ3}.
The results are analogous to those found therein and are hence here omitted for brevity. 

\section{Coordinate system controls} \label{sec:Coord} 
In this section we prove the necessary controls on $C$ and the auxiliary unknown $g$. 

\subsection{High norm estimate on $g$} 
Begin by computing the evolution of $Ag$ from \eqref{def:gPDE2}
\begin{align*} 
\frac{1}{2}\frac{d}{dt}\norm{Ag}_2^2 & = -CK^g_\lambda - CK^g_w - \frac{2}{t}\norm{Ag}_2^2 + \int Ag A\left(\tilde{U}_0\cdot \grad g\right) dV \\ 
& \quad + \int Ag A\left(\tilde{\Delta_t}g\right) dV - \frac{1}{t}\sum_{k \neq 0}\int Ag A \left( U_{-k} \cdot \grad^t U^1_{k} \right) dV \\ 
& = -\mathcal{D}g - CK_L^g + \mathcal{T} + \mathcal{D}_E + \mathcal{F},
\end{align*} 
where
\begin{align*}
\mathcal{D}_E = \int Ag A\left( (\tilde{\Delta_t} - \Delta)g\right) dV. 
\end{align*}

\subsubsection{Transport nonlinearity} \label{sec:TransNong}
Here we are referring to the transport nonlinearity $\mathcal{T}$.
This term is treated in the same manner as \S\ref{sec:TransQ20} and hence is omitted here for brevity.  

\subsubsection{Dissipation error terms} \label{sec:DissErrg}
These terms are treated in the same manner as the corresponding terms in \S\ref{sec:DEQ02}, despite the regularity of $A$ being higher than $A^2$. Using the treatment therein we have 
\begin{align*}
\mathcal{D}_E & \lesssim c_0\nu\norm{\grad A g}_2^2 + \nu\norm{Ag}_2 \norm{\grad AC}_2\norm{\grad Ag}_2 \leq c_0\nu\norm{\grad Ag}_2^2 + c_0^{-1} \epsilon^2 \nu \norm{\grad AC}_2^2. 
\end{align*} 
Note the last term integrates to $O(c_0\epsilon^2)$ by \eqref{ineq:Boot_ACC}. Notice also that the cancellation which eliminates the lower order term in $\Delta_t$ is crucial for this approach. 

\subsubsection{Forcing from non-zero frequencies} \label{sec:g_Forcing}
Similar to the cancellations derived in \eqref{eq:XavgCanc},  by the divergence free condition, 
\begin{align*} 
\mathcal{F} & = -\sum_{k \neq 0}\frac{1}{t}\int A g A \left( \partial_Y^t \left( U^2_{-k} U^1_{k}\right) + \partial_Z^t \left( U^3_{-k} U^1_{k}\right) \right) dV = F_Y + F_Z. 
\end{align*}  
Consider $F_Y$.
Expand with a paraproduct and group terms where the coefficients appear in low frequency with the remainder:  
\begin{align*} 
F_Y & = - \sum_{k \neq 0}\frac{1}{t}\int A g  A \partial_Y \left( (U^2_{-k})_{Lo} (U^1_{k})_{Hi} \right) dV 
- \sum_{k \neq 0}\frac{1}{t}\int A g  A \partial_Y \left( (U^2_{-k})_{Hi} (U^1_{k})_{Lo} \right) dV  \\ 
& \quad - \sum_{k \neq 0}\frac{1}{t}\int A g  A \left( (\psi_y)_{Hi} \partial_Y\left((U^2_{-k})_{Lo} (U^1_{k})_{Lo}\right) \right) dV  + F_{Y;\mathcal{R},C} \\ 
& = F_{Y;LH} + F_{Y;HL} + F_{Y;C} + F_{Y;\mathcal{R},C}. 
\end{align*}   
Consider $F^0_{Y;LH}$ first. By \eqref{ineq:AprioriUneq} and \eqref{ineq:ABasic},  
\begin{align*} 
F_{Y;LH} & \lesssim \frac{\epsilon}{\jap{\nu t^3}^{\alpha}} \sum_{k \neq 0}\frac{1}{t \jap{t}^{2-\delta_1} }\sum_{l,l^\prime} \int \abs{A \hat{g}(\eta,l)} \frac{\jap{t} \abs{\eta} \jap{\eta,l}^2 \jap{\frac{t}{\jap{\xi,l^\prime}}}^{1+\delta_1}}{k^2 + (l^\prime)^2 + \abs{\xi-kt}^2} \\ & \quad\quad\quad \times \abs{A^{1} \Delta_L \widehat{U^1_{k}}(\xi,l^\prime)_{Hi}} Low(-k,\eta-\xi,l-l^\prime) d\eta d\xi.
\end{align*} 
By \eqref{ineq:AdeGen} and \eqref{ineq:triQuadHL} we get
\begin{align*} 
F_{Y;LH} & \lesssim \frac{\epsilon \jap{t}^{1+\delta_1}}{\jap{\nu t^3}^{\alpha}}
 \norm{\left(\sqrt{\frac{\partial_t w}{w}} + \frac{\abs{\grad}^{s/2}}{\jap{t}^s} \right) Ag}_2^2 +  \frac{\epsilon \jap{t}^{1+\delta_1}}{\jap{\nu t^3}^{\alpha}}\norm{\left(\sqrt{\frac{\partial_t w}{w}} + \frac{\abs{\grad}^{s/2}}{\jap{t}^s} \right) A^1 \Delta_L U^1_{\neq}}_2 \\ 
& \quad + \frac{\epsilon}{\jap{t}^{2-2\delta_1} \jap{\nu t^3}^{\alpha}}  \norm{\sqrt{-\Delta_L}Ag}_2 \norm{A^{1} \Delta_L U^1_{\neq}}_2,  
\end{align*} 
which is consistent with Proposition \ref{prop:Boot} by Lemmas \ref{lem:SimplePEL} and \ref{lem:PEL_NLP120neq}. 
Turn next to $F_{Y;HL}$. Similar to $F_{Y;LH}$, we get from \eqref{ineq:AprioriUneq}, \eqref{ineq:ABasic}, \eqref{ineq:AdeGen}, and \eqref{ineq:triQuadHL} we get
\begin{align*} 
F_{Y;HL} & \lesssim \frac{\epsilon \jap{t}^{\delta_1}}{t\jap{\nu t^3}^{\alpha}} \sum_{k \neq 0}\sum_{l,l^\prime} \int \abs{A \hat{g}(\eta,l)} \frac{\abs{\eta} \jap{\eta,l}^2 \jap{\frac{t}{\jap{\xi,l^\prime}}}}{k^2 + (l^\prime)^2 + \abs{\xi-kt}^2} \abs{A^{2} \Delta_L \widehat{U^2_{k}}(\xi,l^\prime)_{Hi}} Low(-k,\eta-\xi,l-l^\prime) dV \\ 
& \lesssim \frac{\epsilon t^{2+\delta_1}}{\jap{\nu t^3}^\alpha}\norm{\left(\sqrt{\frac{\partial_t w}{w}} + \frac{\abs{\grad}^{s/2}}{\jap{t}^s}\right) A g}_2^2 + \frac{\epsilon t^{2+\delta_1}}{\jap{\nu t^3}^\alpha} \norm{\left(\sqrt{\frac{\partial_t w}{w}} + \frac{\abs{\grad}^{s/2}}{\jap{t}^s}\right) A^{2} \Delta_L U^2_{\neq}}_2^2 \\ & \quad + \frac{\epsilon}{t^{1-\delta_1}\jap{\nu t^3}^\alpha} \norm{\sqrt{-\Delta_L}Ag}_2 \norm{A^{2} \Delta_L U^2_{\neq}}_2,   
\end{align*} 
which is consistent with Proposition \ref{prop:Boot}.  
The remainder terms $F_{Y;\mathcal{R},C}$ and coefficient error term $F_{Y;C}$ are similar  to many treatments we have already done, 
so are hence omitted for brevity.  
This completes the treatment of $F_Y$.
The treatment of $F_Z$ is similar (with \eqref{ineq:AdeGen2} instead of \eqref{ineq:AdeGen}).  
The treatment is hence omitted for the sake of brevity. 
This completes the treatment of the forcing terms on $g$, and of the entire high norm estimate on $g$. 

\subsection{Low norm estimate on $g$} \label{sec:Lowg}
Computing the evolution of $\norm{g}_{\G^{\lambda,\gamma}}$ from \eqref{def:gPDE2} (denoting $A^S(t,\grad) = e^{\lambda(t)\abs{\grad}^s} \jap{\grad}^\gamma$), 
\begin{align} 
\frac{1}{2}\frac{d}{dt}\left( t^{4} \norm{g}_{\G^{\lambda,\gamma}}^2\right) & \leq t^4 \dot{\lambda}(t) \norm{\abs{\grad}^{s/2} g}_{\G^{\lambda,\gamma}}^2  - t^{4}\int A^S g A^S\left(\tilde{U}_0\cdot \grad g\right) dV \nonumber \\ 
& \quad + t^{4} \int A^S g A^S \left(\tilde{\Delta_t}g\right) dV - t^{3}\int A^S g A^S \left( U_{\neq} \cdot \grad^t U^1_{\neq} \right)_{0} dV \nonumber \\ 
& = -CK_\lambda^{g,L} + \mathcal{T} + \mathcal{D} + \mathcal{F}. \label{eq:gLowEvo}
\end{align} 
The treatment of the transport nonlinearity $\mathcal{T}$ and the dissipation terms in $\mathcal{D}$ are is essentially the same as in \S\ref{sec:g_Forcing}, which in turn was essentially the same as in \S\ref{sec:HiQ2}. 

It remains to treat $\mathcal{F}$.
As in \S\ref{sec:g_Forcing}, the forcing terms can be re-written via the divergence free condition as 
\begin{align*} 
\mathcal{F} & = - t^{3} \int A^S g A^S \left( \partial_Y^t \left( U^2_{\neq} U^1_{\neq}\right)_0 + \partial_Z^t \left( U^3_{\neq} U^1_{\neq}\right)_0 \right) dV. 
\end{align*} 
By Cauchy-Schwarz, Lemma \ref{lem:GevProdAlg}, Lemma \ref{lem:CoefCtrl}, \eqref{ineq:AprioriUneq}, and Lemma \ref{lem:LossyElliptic}, 
\begin{align*}  
\mathcal{F} & \lesssim t^{3} \norm{g}_{\G^{\lambda,\gamma}}(1 + \norm{C}_{\G^{\lambda,\gamma+1}})\left(\norm{U^2_{\neq}}_{\G^{\lambda,\gamma+1}}\norm{U^1_{\neq}}_{\G^{\lambda,3/2+}} + \norm{U^3_{\neq}}_{\G^{\lambda,3/2+}}\norm{U^1_{\neq}}_{\G^{\lambda,\gamma+1}} \right. \\ & \left. \quad\quad + \norm{U^2_{\neq}}_{\G^{\lambda,3/2+}}\norm{U^1_{\neq}}_{\G^{\lambda,\gamma+1}} + \norm{U^3_{\neq}}_{\G^{\lambda,\gamma+1}}\norm{U^1_{\neq}}_{\G^{\lambda,3/2+}} \right) \\ 
&  \lesssim t^{3} \norm{g}_{\G^{\lambda,\gamma}} \left(\frac{\epsilon^2 \jap{t}^{\delta_1}}{\jap{\nu t^3}^\alpha} \right) \\ 
& \lesssim \frac{\epsilon t}{\jap{\nu t^3}^{\alpha}} t^{4} \norm{g}^2_{\G^{\lambda,\gamma}} + \frac{\epsilon^3 t^{1 + 2\delta_1}}{\jap{\nu t^3}^{\alpha}}, 
\end{align*} 
which is consistent with an improvement to \eqref{ineq:Boot_gLow} for $c_0$ sufficiently small. This completes the proof of \eqref{ineq:Boot_gLow} with constant `2'.    

\subsection{Long time, high norm estimate on $C$: improvement to \eqref{ineq:Boot_ACC}} \label{sec:ACC}
Computing from the evolution equation on $C$, \eqref{def:CReal}, we get 
\begin{align} 
\frac{1}{2}\frac{d}{dt}\norm{AC}_2^2 & \leq \dot{\lambda}\norm{\abs{\grad}^{s/2}A C}_2^2 - \norm{\sqrt{\frac{\partial_t w}{w}}A C}_2^2 + \nu \int A C  A\left(\tilde{\Delta_t} C\right) dV \nonumber  \\ & \quad -\int A C A\left( \tilde U \cdot \grad C \right) dV + \int A C  A g  dV - \int  A C A U_0^2  dV \nonumber \\ 
& = -\mathcal{D}C + \mathcal{D}_E  + \mathcal{T} + \int A C  A g  dV - \int A C A U_0^2  dV, \label{ineq:C1Evo} 
\end{align}
where 
\begin{align*}
\mathcal{D}_E & = \nu \int AC A\left( (\tilde{\Delta_t} - \Delta) C \right) dV.  
\end{align*}

\subsubsection{Linear driving terms}
The main contributions in \eqref{ineq:C1Evo} are the linear driving terms which originate from the lift-up effect.  
Divide the first into low and high frequencies: 
\begin{align*} 
\int A C  A g  dV & = \int \left(A C\right)_{\leq 1} \left(A g\right)_{\leq 1}  dV + \int \left(A C\right)_{>1} \left(A g\right)_{>1}  dV \\ 
& = Lg^L + Lg^H. 
\end{align*} 
The low frequency term is estimated using the decay estimate available on $g$ in \eqref{ineq:Boot_gLow}: 
\begin{align*}
 Lg^L & \lesssim \norm{C}_2 \norm{g}_2 \lesssim \frac{\epsilon}{\jap{t}^{2}}\norm{C}^2_2 + \frac{\epsilon}{\jap{t}^{2}}.  
\end{align*} 
For $\epsilon$ sufficiently small, this is consistent with Proposition \ref{prop:Boot}.   
At the high frequencies we use: 
\begin{align*} 
Lg^H & \leq \frac{\nu}{10} \norm{\grad AC}_2^2 + \frac{5}{2\nu^2}\left(\nu \norm{\grad Ag}_2^2\right). 
\end{align*} 
The first term is absorbed by the dissipation whereas the latter term integrates to $O(\epsilon^2 \nu^{-2}) = O(c_0^2)$ via the bootstrap control \eqref{ineq:Boot_Ag}. 
Hence, for $K_{HC1} \gg 1$, this is consistent with Proposition \ref{prop:Boot}.   

Turn to the next linear term, which is again divided into high and low frequencies
\begin{align*}
-\int A C  A U_0^2  dV & = -\int \left(A C\right)_{\leq 1} \left(A U_0^2\right)_{\leq 1}  dV - \int \left(A C\right)_{>1} \left(A U_0^2\right)_{>1}  dV \\ 
& = LU^L + LU^H. 
\end{align*} 
From \eqref{ineq:Boot_U02_Low} we get for some $K > 0$, 
\begin{align*} 
LU^L & \lesssim \norm{C}_2 \norm{U_0^2}_2  \leq \frac{\nu}{\jap{\nu t}^{\alpha}}\norm{C}^2_2 + \frac{K\epsilon^2}{\nu \jap{\nu t}^{\alpha}}. 
\end{align*} 
For $K_{HC1}$ sufficiently large, the first term is consistent with Proposition \ref{prop:Boot}. 
The second term integrates to $O(K c_0^2)$, which is consistent with Proposition \ref{prop:Boot} provided $K_{HC1}$ is sufficiently large.  
For high frequencies, we have the following by Lemma \ref{lem:PELbasicZero}, 
\begin{align*} 
LU^H & \leq \frac{\nu}{10} \norm{\grad AC}_2^2 + \frac{5}{2\nu^2}\left(\nu \norm{\grad A U_0^2}_2^2\right) 
\leq \frac{\nu}{10} \norm{\grad AC}_2^2 + \frac{K}{2\nu^2}\left(\nu \norm{\grad A^2 Q_0^2}_2^2 + \nu\norm{\grad U_0^2}_2^2\right),  
\end{align*}  
for some $K > 0$ depending only on $s,\sigma$ and $\lambda$. 
The first term is absorbed by the dissipation whereas the latter term integrates to an $O(Kc^2_{0})$ number via the bootstrap controls, so this is consistent with Proposition \ref{prop:Boot} for $K_{HC1}$ sufficiently large. 

\subsubsection{Transport nonlinearity} \label{sec:TransNon_ACC}
We apply an argument similar to that used in \S\ref{sec:TransQ20}.
We omit the details for brevity and conclude
\begin{align*} 
\mathcal{T} & \lesssim \epsilon\norm{\grad AC}_2^2 + \epsilon^{-1} \norm{C}_{\G^{\lambda,\gamma}}^2\left(\frac{\epsilon^2}{\jap{t}^{4}} + \norm{\grad U_0^3}_2^2\right) \\ 
& \lesssim \epsilon\norm{\grad AC}_2^2 + \frac{K_{HC2} \epsilon^3 \abs{\log \epsilon}}{\jap{t}^{2}} + \frac{K_{HC1} c_0^2}{\epsilon \nu} \left(\nu \norm{\grad U_0^3}_2^2\right). 
\end{align*} 
Note by the bootstrap hypotheses, the last term integrates to $O(c_0^3 K_{HC1} K_{U3})$, which is consistent with Proposition \ref{prop:Boot} provided $c_{0}$ is sufficiently small. 

\subsubsection{Dissipation error terms} \label{sec:DissC}
The dissipation error terms are treated with an easy variant of the treatment in \S\ref{sec:DissErrg}. 
We omit the details for brevity and state the result
\begin{align*} 
\mathcal{D}_E &  \lesssim c_{0} \nu \norm{\grad A C}_2^2, 
\end{align*} 
which is then absorbed by the dissipation by choosing $c_{0}$ sufficiently small. 
This completes the high norm improvement to \eqref{ineq:Boot_ACC}.  

\subsection{Shorter time, high norm estimate on $C$: improvement to \eqref{ineq:Boot_ACC2}}
The proof of \eqref{ineq:Boot_ACC2} with constant `2' is essentially the same as \eqref{ineq:Boot_ACC} with a few slight changes. 
From \eqref{def:CReal}, 
\begin{align} 
\frac{1}{2}\frac{d}{dt}\left(\jap{t}^{-2}\norm{A C}_2^2\right) & \leq -\frac{t}{\jap{t}^4}\norm{A C}_2^2 + \jap{t}^{-2}\dot{\lambda}\norm{\abs{\grad}^{s/2}A C}_2^2 - \jap{t}^{-2}\norm{\sqrt{\frac{\partial_t w}{w}}A C}_2^2 \nonumber \\ & \quad + \jap{t}^{-2}\nu \int A C  A \left(\tilde{\Delta_t} C\right) dv \nonumber 
-\jap{t}^{-2}\int A C A\left( \tilde U \cdot \grad C \right) dV \nonumber \\ & \quad 
+ \jap{t}^{-2}\int A C  A g  dV - \jap{t}^{-2}\int  A C A U_0^2  dV \nonumber \\ 
& = -CK_{L}^{C}  + \jap{t}^{-2}\mathcal{D}C + \mathcal{D}_E + \mathcal{T} + \jap{t}^{-2}\int A C  A g  dV - \jap{t}^{-2}\int A C A U_0^2  dV, \label{ineq:CCEvo}
\end{align}
where 
\begin{align*}
\mathcal{D}_E = \jap{t}^{-2} \int AC A\left((\tilde{\Delta_t} - \Delta)C\right) dV. 
\end{align*}

\subsubsection{Linear driving terms}
The main difference between \eqref{ineq:Boot_ACC2} and \eqref{ineq:Boot_ACC} is in the treatment of the linear driving terms, so let us focus there. 
Indeed consider first the linear term involving $g$: 
\begin{align*} 
\jap{t}^{-2}\int A C  A g  dV & \leq \frac{t}{2\jap{t}^4}\norm{AC}_2^2 + \frac{1}{2 t}\norm{Ag}_2^2 \leq \frac{1}{2}CK_{L}^{C} + \frac{1}{4}CK_L^g.  
\end{align*} 
The first is absorbed by the $CK_L^{C}$ term in \eqref{ineq:CCEvo}. The other term is integrable by \eqref{ineq:Boot_Ag} and consistent with Proposition \ref{prop:Boot} 
provided $K_{HC2} \gg 1$. 

The logarithmic loss in \eqref{ineq:Boot_ACC2} is due to the linear term involving $U_0^2$.
Using Lemma \ref{lem:PELbasicZero} we get, for some $K > 0$ depending on $\lambda$ and $s$ (possibly different in each line),  
\begin{align*} 
\jap{t}^{-2}\int A C  A U_0^2   dV & \leq \frac{t}{10\jap{t}^4}\norm{AC}_2^2 + \frac{5}{2t}\norm{A U_0^2}_2^2 \\
& \leq \frac{t}{10\jap{t}^4}\norm{AC}_2^2 + \frac{K}{\jap{t}}\norm{A^2 Q_0^2}_2^2 + \frac{K}{\jap{t}}\norm{U_0^2}_2^2 + \frac{K\epsilon^2}{\jap{t}\jap{\nu t}^{2\alpha}}\norm{AC}_2^2 \\ 
& \leq \frac{t}{10\jap{t}^4}\norm{AC}_2^2 + \frac{4K}{\jap{t}}\epsilon^2 + \frac{4 K_{HC1} K\epsilon^2 c_0^2}{\jap{t}\jap{\nu t}^{2\alpha}}.
\end{align*} 
The first term is absorbed by $CK_L^{C}$ in \eqref{ineq:CCEvo}.  
The latter two terms are integrated until $t \sim \epsilon^{-1}$ 
to deduce the bound \eqref{ineq:Boot_ACC2} on that time-scale provided $K_{HC2}$ is sufficiently large and $c_0$ is chosen sufficiently small.  
For times $t \gtrsim \epsilon^{-1}$ we introduce dissipation to deduce instead for some $K > 0$, 
\begin{align*} 
\jap{t}^{-2}\int A C  A U_0^2   dV & \leq \frac{t}{10\jap{t}^4}\norm{AC}_2^2 + \frac{K}{\jap{t}}\norm{\grad A^2 Q_0^2}_2^2 + \frac{K\epsilon}{\jap{t}}\norm{Q_0^2}_2^2 + \frac{K}{\jap{t}}\norm{U_0^2}_2^2 + \frac{K \epsilon^2}{\jap{t}\jap{\nu t}^{\alpha}}\norm{AC}_2^2 \\ 
& \leq \frac{t}{10\jap{t}^4}\norm{AC}_2^2 + K\epsilon\norm{\grad A^2 Q_0^2}_2^2 + \frac{4K \epsilon^2}{\jap{t}\jap{\nu t}^{2\alpha}} + \frac{K_{HC1} K\epsilon^2 c_0^2}{\jap{t}\jap{\nu t}^{2\alpha}},
\end{align*} 
where the last line followed from the low frequency decay estimates \eqref{ineq:Boot_LowFreq}. This is consistent with \eqref{ineq:Boot_ACC2} provided $K_{HC2}$ is sufficiently large and $c_0$ is sufficiently small.
This completes the treatment of the linear driving terms in estimate \eqref{ineq:Boot_ACC2} 

\subsubsection{Transport nonlinearity }
These are treated in essentially the same fashion as in \S\ref{sec:TransNon_ACC}, which in turn is the same method as that employed in \S\ref{sec:TransQ20}. 
We omit the details for brevity. 

\subsubsection{Dissipation error terms}
These are treated with the same method that was used in \S\ref{sec:DissC}. 
Hence, we omit the details for brevity. This completes the improvement to \eqref{ineq:Boot_ACC2}.

\section{Enhanced dissipation estimates} \label{sec:ED} 
In this section we prove the enhanced dissipation estimates \eqref{ineq:Boot_ED} with constant `2'.  

\subsection{Enhanced dissipation of $Q^3$} \label{sec:ED3}
We begin with $Q^3$.
Computing the time evolution of $\norm{A^{\nu;3}Q^3}_2$ we get
\begin{align} 
\frac{1}{2}\frac{d}{dt}\norm{A^{\nu;3} Q^3}_2^2 & \leq \dot{\lambda}\norm{\abs{\grad}^{s/2}A^{\nu;3} Q^3}_2^2
-\frac{2}{t}\norm{\mathbf{1}_{t > \jap{\grad_{Y,Z}}} {A}^{\nu;3} Q^3}_2^2 - \norm{\sqrt{\frac{\partial_t w_L}{w_L}}A^{\nu;3}Q^3}_2^2 + G^\nu \nonumber \\
& \quad -2 \int A^{\nu;3} Q^3 A^{\nu;3} \partial_{YX}^t U^3 dV + 2 \int A^{\nu;3} Q^3 A^{\nu;3} \partial_{ZX}^t U^2 dV \nonumber \\   
 & \quad + \nu \int A^{\nu;3} Q^{3} A^{\nu;3} \left(\tilde \Delta_t Q^3\right) dv -\int A^{\nu;3} Q^3 A^{\nu;3}\left( \tilde U \cdot \grad Q^3 \right) dV \nonumber \\ 
& \quad - \int A^{\nu;3} Q^3 A^{\nu;3} \left[Q^j \partial_j^t U^3 + 2\partial_i^t U^j \partial_{ij}^t U^3  - \partial_Z^t\left(\partial_i^t U^j \partial_j^t U^i\right) \right] dV \nonumber \\  
& = -\mathcal{D}Q^{\nu;3} - CK_{L}^{\nu;3}  + G^{\nu} + LS3 + LP3 + \mathcal{D}_E + \mathcal{T} + NLS1 + NLS2 + NLP,  \label{ineq:AnuEvo3}
\end{align} 
where we write 
\begin{align*}
\mathcal{D}_E & = \nu \int A^{\nu;3} Q^3 A^{\nu;3}\left(\tilde{\Delta_t}Q^3 - \Delta_L Q^3\right) dV, 
\end{align*} 
and
\begin{align*}
G^\nu = \alpha \int A^{\nu;3} Q^3 \min\left(1, \frac{\jap{\nabla_{Y,Z}}^2}{t^2}\right) e^{\lambda(t)\abs{\grad}^s}\jap{\grad}^\beta\jap{D(t,\partial_Y)}^{\alpha-1} \frac{D(t,\partial_Y)}{\jap{D(t,\partial_Y)}} \partial_t D(t,\partial_Y) Q^3_{\neq} dV. 
\end{align*} 
First, we need to cancel the growing term $G^\nu$ in \eqref{ineq:AnuEvo3} using part of the dissipation term $\mathcal{D}$. 
This is done in the same manner as in \cite{BMV14}; indeed:  
\begin{align*} 
G^{\nu} - \nu \norm{\sqrt{-\Delta_L} A^{\nu;3}Q^3}_2^2  
& \leq \nu \sum_{k \neq 0} \sum_{l} \int \left(\frac{1}{8}t^2\mathbf{1}_{t \geq 2 \abs{\eta}} - \abs{k}^2 - \abs{l}^2 - \abs{\eta-kt}^2\right)   \abs{A^{\nu;3} \widehat{Q^3_k}(\eta,l)}^2 d\eta \\ 
& \leq -\frac{\nu}{8}\norm{\sqrt{-\Delta_L}A^{\nu;3}Q^{3}_{\neq}}_2^2. 
\end{align*}
Next we see how to control the remaining linear and nonlinear contributions. 

\subsubsection{Linear stretching term $LS3$} 
First separate into two parts (to be sub-divided further), 
\begin{align*} 
LS3 & = -2\int A^{\nu;3} Q^3 A^{\nu;3}\partial_X(\partial_Y - t\partial_X) U^3  dV 
- 2\int A^{\nu;3} Q^3 A^{\nu;3}\left(\psi_y\partial_X (\partial_Y - t\partial_X) U^3 \right) dV \\ 
& = LS3^0 + LS3^{C}. 
\end{align*} 
Turn first to $LS3^C$: applying \eqref{ineq:L2L2L1}, \eqref{ineq:AnuiDistri}, Lemma \ref{lem:AnuLossy}, and Lemma \ref{lem:CoefCtrl} (with $\gamma > \beta + 3\alpha  + 4$), 
\begin{align} 
LS3^{C} & \lesssim \norm{\sqrt{-\Delta_L} A^{\nu;3}Q^3}_2\norm{C}_{\G^{\lambda,\beta+3\alpha+4}}\norm{A^{\nu;3}\partial_X U^3}_2 \nonumber \\ 
& \lesssim \norm{\sqrt{-\Delta_L} A^{\nu;3}Q^3}_2\norm{C}_{\G^{\lambda,\beta+3\alpha+4}}\frac{1}{\jap{t}^2}\left(\norm{A^{\nu;3}Q^3}_2 + \norm{A^3Q^3}_2\right) \nonumber \\ 
& \lesssim \epsilon\norm{\sqrt{-\Delta_L} A^{\nu;3}Q^3}^2_2 + \frac{\epsilon}{\jap{t}^2}\left(\norm{A^{\nu;3}Q^3}_2 + \norm{A^3Q^3}_2\right)^2. \label{ineq:ED_LS3C}
\end{align} 
The first term is absorbed by the leading order dissipation in \eqref{ineq:AnuEvo3} for $c_0$ small whereas the second is consistent with Proposition \ref{prop:Boot} provided $\epsilon$ is sufficiently small. 

For $LS3^0$, we can proceed similar to the high norm estimate in \S\ref{sec:LS30}.  
Begin as in \eqref{eq:LS30} by expanding $\Delta_L\Delta_t^{-1}$: 
\begin{align} 
LS3^0 & =  -2\int A^{\nu;3} Q^3 A^{\nu;3}\partial_X(\partial_Y - t\partial_X) \Delta_{L}^{-1} \left[Q^3 - G (\partial_Y - t\partial_X)^2 U^3 - 2 \psi_z \partial_Z(\partial_Y - t\partial_X)U^3 \right. \nonumber \\ & \quad\quad\quad \left. - \Delta_tC (\partial_Y - t\partial_X) U^3\right]  dV \nonumber \\ 
& = LS3^{0;0} + LS3^{0;C1} + LS3^{0;C2} + LS3^{0;C3}. \label{def:LS30nu} 
\end{align}  
The leading order term in \eqref{def:LS30nu} is divided into two contributions
\begin{align*} 
LS3^{0;0} & = -2\sum \int \left[\mathbf{1}_{t \leq 2\abs{\eta}} + \mathbf{1}_{t > 2\abs{\eta}}\right] \abs{A^{\nu;3} \widehat{Q^3_k}}^2 \frac{k(\eta-kt)}{k^2 + l^2 + \abs{\eta-kt}^2} d\eta \\
& = LS3^{0;0,ST} + LS3^{0;0,LT}.  
\end{align*} 
In the short-time regime we may simply apply \eqref{ineq:AnuHiLowSep} and use that $\beta + 3\alpha +2 < \sigma$ to obtain, 
\begin{align*} 
LS3^{0;0,ST} & \lesssim \frac{1}{\jap{t}^2}\norm{A^3 Q^3_{\neq}}^2_2, 
\end{align*} 
which is consistent with Proposition \ref{prop:Boot} for $K_{ED3}$ sufficiently large relative to $K_{H3}$.
The long-time regime is treated in the same manner as $LS3^{0;0,LT}$ is treated in \S\ref{sec:LS30}; hence we omit the treatment and simply state the result: 
\begin{align*} 
LS3^{0;0,LT} & \leq CK^{\nu;3}_{L} + \frac{\delta_\lambda}{10\jap{t}^{3/2}}\norm{\abs{\grad}^{s/2} A^{\nu;3}Q^3}_2^2 +  \frac{K}{\delta_\lambda^{\frac{1}{2s-1}}\jap{t}^{3/2}}\norm{ A^{\nu;3}Q^3}_2^2,
\end{align*} 
where $K$ is a fixed constant. As discussed after \eqref{ineq:LS300LTZ}, this is consistent with Proposition \ref{prop:Boot} 
for small $\delta_\lambda$ and large $K_{ED3}$.

Turn to the first error term in \eqref{def:LS30nu}, which by \eqref{ineq:AnuHiLowSep} and $\beta + 3\alpha + 6 < \gamma$ is controlled via 
\begin{align*} 
LS3^{0;C1} & \lesssim  \frac{1}{\jap{t}^5}\norm{A^{\nu;3}Q^3}_{2} \norm{G}_{\G^{\lambda,\gamma-1}} \norm{\Delta_L U^3_{\neq}}_{\G^{\lambda,\gamma}} + \frac{1}{\jap{t}}\norm{A^{\nu;3} Q^3}_2 \norm{A^{\nu;3} \left(G(\partial_Y - t\partial_X)^2 U^3_{\neq}\right)}_2 \\ 
& \lesssim  \frac{\epsilon}{\jap{t}^2}\norm{A^{\nu;3} Q^3_{\neq}}_2 \norm{A^3 \Delta_L U^3_{\neq}}_2 + \frac{1}{\jap{t}}\norm{A^{\nu;3} Q^3}_2 \norm{A^{\nu;3} \left(G(\partial_Y - t\partial_X)^2 U^3_{\neq}\right)}_2. 
\end{align*} 
To control the latter term we use \eqref{ineq:AnuiDistri} and Lemma \ref{lem:CoefCtrl}, 
\begin{align*} 
\frac{1}{\jap{t}}\norm{A^{\nu;3} Q^3}_2 \norm{A^{\nu;3} \left(G(\partial_Y - t\partial_X)^2 U^3_{\neq}\right)}_2 & \lesssim \frac{1}{\jap{t}}\norm{ A^{\nu;3} Q^3}_2 \norm{C}_{\G^{\lambda,\gamma}} \norm{A^{\nu;3}(\partial_Y - t\partial_X)^2 U_{\neq}^3}_2 \\ 
& \lesssim  \epsilon\norm{\sqrt{-\Delta_L} A^{\nu;3} Q^3}_2 \norm{A^{\nu;3}(\partial_Y - t\partial_X)^2 U_{\neq}^3}_2,  
\end{align*} 
where we have used that $A^{\nu;3}$ includes a projection to non-zero frequencies to add the dissipation. After applying Lemma \ref{lem:AnuLossy_CKnu}, this is consistent with Proposition \ref{prop:Boot} for $c_0$ chosen sufficiently small (absorbing the leading order contributions via the dissipation). 
This completes $LS3^{0;C1}$. 
The other error terms, $LS3^{0;C2}$ and $LS3^{0;C3}$, are treated similarly and yield similar contributions; hence these are omitted for the sake of brevity. 
This completes the treatment of $LS3^0$. 

\subsubsection{Linear pressure term $LP3$} \label{ineq:LP3_ED} 
Begin by separating out the contribution of the coefficients, 
 \begin{align*} 
LP3 & = 2\int A^{\nu;3} Q^3 A^{\nu;3}\partial_X \partial_Z U_{\neq}^2  dV 
 + 2\int A^{\nu;3} Q^3 A^{\nu;3} \left(\psi_z(\partial_Y - t\partial_X)\partial_X U_{\neq}^2\right)  dV \\ 
& = LP3^0 + LP3^C.
\end{align*}                 
By Cauchy-Schwarz and \eqref{def:wL}, 
\begin{align*}
LP3^0 & \leq \frac{1}{2\kappa}\norm{\sqrt{\frac{\partial_t w_L}{w_L}} A^{\nu;3} Q^3_{\neq}}_2^2 + \frac{1}{2\kappa} \norm{\sqrt{\frac{\partial_t w_L}{w_L}} \Delta_L A^{\nu;3} U^2_{\neq}}_2^2. 
\end{align*}
Therefore, by Lemma \ref{lem:AnuLossy_CKnu}, this is consistent with Proposition \ref{prop:Boot} for $c_0$ sufficiently small and $K_{ED3} \gg K_{ED2}$. 

The coefficient error term, $LP3^C$, can be treated in the same manner as $LS3^C$ above in \eqref{ineq:ED_LS3C} and yields similar contributions. 
Hence we omit the treatment for brevity. 
This completes the treatment of the linear pressure term $LP3$. 

\subsubsection{Nonlinear pressure and stretching} \label{sec:NLPS_Q3ED}
Before beginning, note that due to the regularity gap $\beta + 3\alpha + 12 \leq \gamma$ and \eqref{ineq:AnuiDistri},  
the presence of the coefficients from the coordinate transform does not have an important impact. 
Moreover, by Lemma \ref{lem:AnuLossy}, there is not a significant difference between $\partial_Y - t\partial_X$ and $\partial_Z$ derivatives when making most relevant estimates.   Hence, for simplicity we will treat all $NLS$ and $NLP$ terms as if there were no variable coefficients. 

As above, we will enumerate the terms as follows for $i,j \in \set{1,2,3}$ and 
$a,b \in \set{0,\neq}$
\begin{subequations}  \label{def:enumnu}
\begin{align}
NLP(i,j,a,b) & = \int A^{\nu;3} Q^3 A^{\nu;3} \partial_Z^t(\partial_j^t U^i_a \partial_i^t U^j_b  ) dV \\
NLS1(j,a,b) & = -\int A^{\nu;3} Q^3 A^{\nu;3} \left( Q^j_a \partial_j^t U^3_b  \right) dV \\
NLS2(i,j,a,b) & = -2\int A^{\nu;3} Q^3 A^{\nu;3} (\partial_i^t U^j_a \partial_i^t\partial_j^t U^3_b  ) dV.
\end{align}
\end{subequations}
We will use repeatedly the inequalities 
\begin{subequations} 
\begin{align}
A^{\nu; 3} & \lesssim t A^{\nu;1} \\ 
A^{\nu; 3} & \lesssim A^{\nu;2}. 
\end{align}
\end{subequations} 

\paragraph{Treatment of $NLP(i,j,0,\neq)$ terms} 
By \eqref{ineq:AnuiDistri}, 
\begin{align*} 
NLP(i,j,0,\neq) & \lesssim \norm{A^{\nu;3}Q^3}_2 \norm{A^{\nu;3} \jap{\partial_{Z}} \partial^t_i U^j}_2 \norm{\partial_j^t U_0^i}_{\G^{\lambda,\beta+3\alpha+4}}. 
\end{align*} 
We see that the loss of $t$ if $i=1$ on the third factor is balanced by no loss of $t$ on the second (indeed, see Lemma \ref{lem:AnuLossy}). If $i \neq 1$ then there is no loss of $t$ on the last factor 
but a loss of $t$ on the second.
 Moreover, we see that due to the zero $X$ mode on $U^i$, $j \neq 1$.   
Regardless, after Lemma \ref{lem:AnuLossy} we get
\begin{align*} 
NLP(i,j,\neq,0) & \lesssim \epsilon \norm{A^{\nu;3}Q^3}_2\left(\norm{A^{j} Q^j_{\neq}}_2 + \norm{A^{\nu;j}Q^j}_2\right) \\
 & \lesssim \epsilon \norm{\sqrt{-\Delta_L} A^{\nu;3}Q^3}_2\left(\norm{\sqrt{-\Delta_L} A^{j} Q^j_{\neq}}_2 + \norm{\sqrt{-\Delta_L} A^{\nu;j}Q^j}_2\right), 
\end{align*} 
which is subsequently absorbed by the dissipation for $c_0$ small. 
This completes the treatment of $NLP(i,j,\neq,0)$. 

\paragraph{Treatment of $NLS1(j,0,\neq)$ terms}\label{sec:NLS10neq_Q3ED} 
Next turn to the treatment of the $NLS1(i,j,0,\neq)$ terms, which by \eqref{ineq:AnuiDistri} followed by \eqref{ineq:AnuLossyII},  
\begin{align*} 
NLS1(j,0,\neq) & \lesssim \norm{A^{\nu;3}Q^3}_2 \norm{Q^j_0}_{\G^{\lambda,\beta + 3\alpha + 4}}  \norm{A^{\nu;3}\partial_j^t U^3}_2 \lesssim \frac{\epsilon}{\jap{t}}\norm{A^{\nu;3}Q^3}_2\left(\norm{A^{\nu;3}Q^3}_2 + \norm{A^3 Q^3} \right) 
\end{align*} 
which is then absorbed by the dissipation for $c_0$ sufficiently small due to the projection to non-zero frequencies.  

\paragraph{Treatment of $NLS1(j,\neq,0)$ terms} \label{sec:ED3NSL1ijneq0}
Next turn to the treatment of the $NLS1(j,\neq,0)$ terms, which by \eqref{ineq:AnuiDistri} followed by \eqref{ineq:AnuLossyII} (noting that $j \neq 1$), 
\begin{align*} 
NLS1(j,\neq,0) &  \lesssim \norm{A^{\nu;3}Q^3}_2 \norm{A^{\nu;j} Q^j}_2  \norm{U^3_0}_{\G^{\lambda,\beta + 3\alpha + 4}} \lesssim \epsilon \norm{A^{\nu;3}Q^3}_2 \norm{A^{\nu;j} Q^j}_2, 
\end{align*} 
which is subsequently absorbed by the dissipation for $c_0$ sufficiently small due to the projection to non-zero frequencies. 

\paragraph{Treatment of $NLS2(i,j,\neq,0)$ terms} 
Next turn to the treatment of the $NLS2(i,j,\neq,0)$ terms. 
Notice that in this case, neither $i$ nor $j$ can be one. 
Similar to \S\ref{sec:NLS10neq_Q3ED}, we get by \eqref{ineq:AnuiDistri} and \eqref{ineq:AprioriU0},  
\begin{align*}
NLS2(i,j,\neq,0) & \lesssim \frac{\epsilon}{\jap{t}}\norm{A^{\nu;3}Q^3}_2\left(\norm{A^{\nu;j}Q^j}_2 + \norm{A^{j}Q^j_{\neq}}_2 \right), 
\end{align*}
which is then absorbed by the dissipation for $c_0$ small by the restriction to non-zero frequencies. 

\paragraph{Treatment of $NLS2(i,j,0,\neq)$ terms}  
Next turn to the treatment of the $NLS1(i,j,\neq,0)$ terms.
Notice that $i$ cannot be one but $j$ can. 
Further notice that if $j = 1$ then we can gain a power of $t$ on $U^3_{\neq}$ using Lemma \ref{lem:AnuLossy}. 
It follows that 
\begin{align*} 
NLS2(i,j,0,\neq) & \lesssim \epsilon \norm{A^{\nu;3}Q^3}_2\left(\norm{A^{3} Q^3_{\neq}}_2 + \norm{A^{\nu;3}Q^3}_2 \right). 
\end{align*}

\paragraph{Treatment of $NLP(i,j,\neq,\neq)$} \label{sec:NLPneqneq_Q3ED}
Notice that we will lose a power of $t$ from $A^1$ if $j$ or $i$ is one, but in this case we would lose one less power of $t$ in Lemma \ref{lem:AnuLossy} due to the lack of $Z$ or $Y$ derivatives. 
Hence regardless of the combination of $i$ and $j$, we will gain at least one power of $t$. 
Therefore, from \eqref{ineq:AnuiDistriDecay},  
\begin{align*} 
NLP(i,j,\neq,\neq) 
& \lesssim \frac{t^2}{\jap{\nu t^3}^\alpha} \norm{A^{\nu;3}Q^3}_2 \left(\norm{A^{\nu;3}\partial^t_{Z} \partial_i^t U^j}_2\norm{ A^{\nu;3}\partial_{j}^t U^i}_2 + \norm{A^{\nu;3}\partial_{i}^t U^j}_2\norm{A^{\nu;3}\partial_{Z}^t \partial_j^t U^i}_2\right) \\ 
& \lesssim \frac{\epsilon^2 \jap{t}}{\jap{\nu t^3}^\alpha} \norm{A^{\nu;3}Q^3}_2, 
\end{align*} 
which is consistent with Proposition \ref{prop:Boot}  for $\epsilon$ sufficiently small. 

\paragraph{Treatment of $NLS1(j,\neq,\neq)$} \label{sec:NLS1neqneq_Q3ED}
These terms are all treated in essentially the same manner. 
From \eqref{ineq:AnuiDistriDecay} and \eqref{ineq:AnuLossyII} (again using that a loss from $j = 1$ is balanced by a gain on the second factor), 
\begin{align*}
NLS1(j,\neq,\neq) & \lesssim \frac{\jap{t}^2 }{\jap{\nu t^3}^{\alpha}} \norm{A^{\nu;3} Q^3}_2\norm{A^{\nu;3} Q^j}_2 \norm{A^{\nu;3} \partial_j^t U^3}_2 \\ 
& \lesssim \frac{\jap{t} }{\jap{\nu t^3}^{\alpha}} \norm{A^{\nu;3} Q^3}_2\norm{A^{\nu;j} Q^j}_2\left(\norm{A^{\nu;3} Q^3}_2 + \norm{A^{3} Q^3_{\neq}}_2 \right),  
\end{align*}
which is consistent with Proposition \ref{prop:Boot} for $\epsilon$ sufficiently small.

\paragraph{Treatment of $NLS2(i,j,\neq,\neq)$} \label{sec:NLS2neqneq_Q3ED}
The treatment of $NLS2$ is essentially the same as $NLP$: using again that the loss is at most $t^3$ regardless of $i$ and $j$, we get from \eqref{ineq:AnuiDistriDecay} and Lemma \ref{lem:AnuLossy}:   
\begin{align*} 
NLS2(i,j,\neq,\neq) & \lesssim \frac{t^2}{\jap{\nu t^3}^\alpha} \norm{A^{\nu;3}Q^3}_2 \norm{A^{\nu;3}\partial_{i}^t U^j}_2 \norm{ A^{\nu;3}\partial_{ij}^t U^3}_2 \\ 
& \lesssim \frac{ \jap{t}}{\jap{\nu t^3}^\alpha} \norm{A^{\nu;3}Q^3}_2 \left(\norm{A^{\nu;3} Q^3}_2 + \norm{A^{3} Q^3_{\neq}}_2 \right)\left(\norm{A^{\nu;j} Q^j}_2 + \norm{A^{j} Q^j_{\neq}}_2 \right), 
\end{align*} 
which is consistent with Proposition \ref{prop:Boot} for $\epsilon$ sufficiently small. 

\subsubsection{Transport nonlinearity} \label{sec:Trans_ED_Q3}
Divide the transport nonlinearity via: 
\begin{align*} 
\mathcal{T} & = -\int A^{\nu;3}Q^3 A^{\nu;3}_k\left(\tilde U_0 \cdot \grad_{Y,Z} Q^3_{\neq}\right) dV - \int A^{\nu;3}Q^3 A^{\nu;3}_k\left(\tilde U_{\neq} \cdot \grad_{Y,Z} Q^3_{0}\right) dV \\ & \quad -  \int A^{\nu;3}Q^3 A^{\nu;3}_k\left(\tilde U_{\neq} \cdot \grad Q^3_{\neq}\right) dV \\ 
& = \mathcal{T}_{0\neq} + \mathcal{T}_{\neq 0}  + \mathcal{T}_{\neq \neq}. 
\end{align*}   
Consider first $\mathcal{T}_{0\neq}$. 
By \eqref{ineq:AnuiDistri} and $\abs{\eta} \leq \abs{\eta-kt} + \abs{kt} \leq \jap{t}\left(\abs{\eta-kt} + \abs{k}\right)$, 
\begin{align*}  
\mathcal{T}_{0\neq} & \lesssim \norm{A^{\nu;3}Q^3}_2\left(\norm{g}_{\G^{\lambda,\beta + 3\alpha+2}}\jap{t}\norm{\sqrt{-\Delta_L} A^{\nu;3}Q^3}_2 + \norm{U_0^3}_{\G^{\lambda,\beta+3\alpha+2}}\norm{\sqrt{-\Delta_L} A^{\nu;3}Q^3}_2\right) \\ 
& \lesssim \epsilon \norm{\sqrt{-\Delta_L} A^{\nu;3}Q^3}_2^2,  
\end{align*}
where the last line followed by \eqref{ineq:AprioriU0} and \eqref{ineq:Boot_gLow} (and $\gamma \geq \beta+3\alpha + 2$). This is subsequently absorbed by the dissipation for $c_0$ sufficiently small. 

Turn next to $\mathcal{T}_{\neq 0}$. By \eqref{ineq:AnuiDistri} and Lemma \ref{lem:AnuLossy} we have  
\begin{align*}
\mathcal{T}_{\neq 0} & \lesssim \epsilon\norm{A^{\nu;3}Q^3}_2\left(\norm{A^{\nu;3}U^2}_2 + \norm{A^{\nu;3}U^3}_2\right) \\ 
& \lesssim \frac{\epsilon}{\jap{t}^2}\norm{A^{\nu;3}Q^3}_2\left(\norm{A^{\nu;2}Q^2}_2 + \norm{A^2Q^2_{\neq}}_2 + \norm{A^{\nu;3}Q^3}_2 + \norm{AQ^3_{\neq}}\right)
\end{align*}
which is hence consistent with Proposition \ref{prop:Boot}.  

Turn next to $\mathcal{T}_{\neq}$, which is written
\begin{align*} 
\mathcal{T}_{\neq \neq} & = \int A^{\nu;3}Q^3_{\neq} A^{\nu;3}\left(\begin{pmatrix}U^1_{\neq} \\ (1 + \psi_y) U^2_{\neq} + \psi_zU^3_{\neq} \\ U^3_{\neq} \end{pmatrix}  \cdot \begin{pmatrix} \partial_X Q^3_{\neq} \\ (\partial_{Y} - t\partial_X)Q_{\neq}^3 \\ \partial_Z Q_{\neq}^3 \end{pmatrix}\right) dV. 
\end{align*}  
By Cauchy-Schwarz, \eqref{ineq:AnuiDistri}, Lemma \ref{lem:CoefCtrl} and \eqref{ineq:AnuiDistriDecay}, we get
\begin{align*} 
\mathcal{T}_{\neq \neq}  & \lesssim \norm{A^{\nu;3}Q^3}_2 \frac{\jap{t}^2}{\jap{\nu t^3}^\alpha} \left(\norm{A^{\nu;3}U^2}_2 + \norm{A^{\nu;3}U^3}_2 + \norm{A^{\nu;3}U^1}_2\right)\norm{\sqrt{-\Delta_L} A^{\nu;3} Q^3}_{2}. 
\end{align*} 
Applying Lemma \ref{lem:AnuLossy} and \eqref{ineq:Boot_ED} with \eqref{ineq:Boot_Hi} gives
\begin{align*} 
\mathcal{T}_{\neq \neq} & \lesssim \norm{A^{\nu;3}Q^3}_2 \frac{1}{\jap{\nu t^3}^\alpha} \left( \jap{t}\norm{A^{\nu;1}Q^1}_2 + \jap{t}\norm{A^1 Q^1_{\neq}}_2 \right. \\ 
& \left. \quad  + \norm{A^{\nu;2}Q^2}_2 + \norm{A^2 Q^2_{\neq}}_2 + \norm{A^{\nu;3}Q^3}_2 + \norm{A^3 Q^3_{\neq}}_2\right) \norm{\sqrt{-\Delta_L} A^{\nu;3} Q^3}_{2} \\ 
 & \lesssim  \frac{\epsilon \jap{t}}{\jap{\nu t^3}^\alpha} \norm{A^{\nu;3}Q^3}_2\norm{\sqrt{-\Delta_L} A^{\nu;3} Q^3}_{2} \\ 
& \lesssim \frac{\epsilon \jap{t}^2}{\jap{\nu t^3}^\alpha} \norm{A^{\nu;3}Q^3}_2^2 + \epsilon\norm{\sqrt{-\Delta_L} A^{\nu;3} Q^3}_{2}^2,
\end{align*} 
which is consistent with Proposition \ref{prop:Boot} for $\epsilon$ and $c_0$ sufficiently small using $\nu \geq \epsilon c_0^{-1}$.  

\subsubsection{Dissipation error terms} \label{sec:DE_ED_Q3}
The dissipation error terms are treated easily as in \cite{BMV14} using \eqref{ineq:AnuiDistri} together with the regularity gap between $A^{\nu;3}$ and the coefficient control in \eqref{ineq:Boot_LowC}. 
We hence omit the treatment for brevity and simply state the result: 
\begin{align*} 
\mathcal{D}_E & \lesssim c_{0}\nu \norm{\sqrt{-\Delta_L}A^{\nu;3}Q^3}_2^2. 
\end{align*} 

\subsection{Enhanced dissipation of $Q^2$} 
The enhanced dissipation of $Q^2$ is deduced in a manner very similar to $Q^3$, however, since we are imposing more control on $Q^2$, some nonlinear interactions must be handled with more precision. 
On the other hand, the evolution of $Q^2$ lacks the troublesome linear terms that are present in $Q^3$ and $Q^1$. 

Computing the time evolution of $\norm{A^{\nu;2}Q^2}_2$ we get
\begin{align} 
\frac{1}{2}\frac{d}{dt}\norm{A^{\nu;2} Q^2}_2^2 & \leq \dot{\lambda}\norm{\abs{\grad}^{s/2}A^{\nu;2} Q^2}_2^2 -\frac{\delta_1}{t}\norm{\mathbf{1}_{t > \jap{\grad_{Y,Z}}} A^{\nu;2} Q^2}_2^2 - \norm{\sqrt{\frac{\partial_t w_L}{w_L}} A^{\nu;2}Q^2}_2^2 + G^\nu  \nonumber \\ %
 & \quad + \nu \int A^{\nu;2} Q^{2} A^{\nu;2} \left(\tilde{\Delta_t} Q^2\right) dV -\int A^{\nu;2} Q^2 A^{\nu;2}\left( \tilde U \cdot \grad Q^2 \right) dV \nonumber \\ 
& \quad - \int A^{\nu;2} Q^2 A^{\nu;2} \left[\left(Q^j \partial_j^t U^2\right) + 2\partial_i^t U^j \partial_{ij}^t U^2  - \partial_Y^t\left(\partial_i^t U^j \partial_j^t U^i\right) \right] dV \nonumber \\  
& = -\mathcal{D}Q^{\nu;2} - \delta_1 CK_L^{\nu;2} + \mathcal{D}_E + \mathcal{T} + NLS1 + NLS2 + NLP,  \label{ineq:AnuEvo2}
\end{align} 
where as in \S\ref{sec:ED3}, we write 
\begin{align*}
\mathcal{D}_E & = \nu \int A^{\nu;2} Q^2 A^{\nu;2}\left(\tilde{\Delta_t}Q^2 - \Delta_L Q^2\right) dV, 
\end{align*} 
and
\begin{align*}
G^\nu = \alpha \int A^{\nu;2} Q^2 \min\left(1, \frac{\jap{\grad_{Y,Z}}^{\delta_1}}{t^{\delta_1}}\right) e^{\lambda(t)\abs{\grad}^s}\jap{\grad}^\beta\jap{D(t,\partial_v)}^{\alpha-1} \frac{D(t,\partial_v)}{\jap{D(t,\partial_v)}} \partial_t D(t,\partial_v) Q^2_{\neq} dV. 
\end{align*} 
As in \S\ref{sec:ED3}, we have
\begin{align*} 
-\nu \norm{\sqrt{-\Delta_L}A^{\nu;2} Q^{2}_{\neq}}_2^2 + G^{\nu}& \leq -\frac{\nu}{8}\norm{\sqrt{-\Delta_L}A^{\nu;2} Q^{2}}_2^2. 
\end{align*}

\subsubsection{Nonlinear pressure and stretching}
In this section we treat $NLS1$, $NLS2$ and $NLP$. 
As in \S\ref{sec:NLPS_Q3ED}, for simplicity we will treat all $NLS$ and $NLP$ terms as if there were no variable coefficients. 
As above, we will enumerate the terms as follows for $i,j \in \set{1,2,3}$ and 
$a,b \in \set{0,\neq}$
\begin{subequations}  \label{def:enumnu}
\begin{align}
NLP(i,j,a,b) & = \int A^{\nu;2} Q^2 A^{\nu;2} \partial_Y^t(\partial_j^t U^i_a \partial_i^t U^j_b  ) dV \\
NLS1(j,a,b) & = -\int A^{\nu;2} Q^2 A^{\nu;2} \left( Q^j_a \partial_j^t U^2_b  \right) dV \\
NLS2(i,j,a,b) & = -2\int A^{\nu;2} Q^2 A^{\nu;2} (\partial_i^t U^j_a \partial_i^t\partial_j^t U^2_b  ) dV.
\end{align}
\end{subequations}
We will use repeatedly the inequalities 
\begin{subequations} 
\begin{align}
A^{\nu; 2} & \lesssim t^{2}A^{\nu;1} \\ 
A^{\nu; 2} & \lesssim t^{2-\delta_1}A^{\nu;3}. 
\end{align}
\end{subequations} 

\paragraph{Treatment of $NLP(i,j,0,\neq)$ terms} \label{sec:NLP0neq_Q2ED}
All of these terms can be treated with a variant of the same general argument. 
Via integration by parts, the projection to non-zero frequencies and \eqref{ineq:AnuiDistri} we get, 
\begin{align*} 
NLP(i,j,0,\neq) & \lesssim \norm{\sqrt{-\Delta_L}A^{\nu;2}Q^2}_2 \norm{A^{\nu;2}\left( \partial_i^t U^j_{\neq} \partial_j^tU_0^i\right)}_2 \lesssim \norm{\sqrt{-\Delta_L}A^{\nu;2}}_2 \norm{A^{\nu;2}\partial_i^t U^j}_2 \norm{U_0^i}_{\G^{\lambda, \beta + 3\alpha + 5}}. 
\end{align*} 
With \eqref{ineq:AnuLossyED} in mind, the power of $t$ lost from the derivatives and $U_0^1$ together is at most one and the powers of $t$ lost from the possibility that $j=3$ is at most an additional two, so at worst we get from \eqref{ineq:AnuiDistri}, \eqref{ineq:AnuLossyED}, 
and \eqref{ineq:AprioriU0}, 
\begin{align*}
NLP(i,j,0,\neq)  & \lesssim \epsilon \norm{\sqrt{-\Delta_L}A^{\nu;2}Q^2 }_2\left(\norm{\sqrt{-\Delta_L} A^{\nu;j} Q^j} + \norm{\sqrt{-\Delta_L} A^{j} Q^j_{\neq}}_2\right),  
\end{align*} 
which is then absorbed by the dissipation (note that $j \neq 1$ by the skew structure of the nonlinearity). 

\paragraph{Treatment of $NLS1(j,0,\neq)$ terms} \label{sec:NLS10neq_Q2ED}
These terms are straightforward by \eqref{ineq:AnuiDistri}, \eqref{ineq:Boot_Hi}, and \eqref{ineq:AnuLossyII}; we omit the details and conclude
\begin{align*}
NLS1(j,0,\neq) & \lesssim \epsilon\norm{A^{\nu;2}Q^2}_2\left(\norm{A^{\nu;2}Q^2}_2 + \norm{A^{2}Q^2_{\neq}}_2\right), 
\end{align*}
which is absorbed by the dissipation for $c_0$ sufficiently small. 

\paragraph{Treatment of $NLS1(j,\neq,0)$ terms} \label{sec:NLS1neq0_Q2ED}
Note that $j \neq 1$. Hence, the worst possibility is $j = 3$, where $2-\delta_1$ powers of time are lost. However, these may be recovered by \eqref{ineq:AnuHiLowSep2}. Indeed, from \eqref{ineq:AnuiDistri}, \eqref{ineq:AprioriU0}, and \eqref{ineq:AnuHiLowSep2}, 
\begin{align*}
NLS1(j,\neq,0) & \lesssim \epsilon\left(\norm{\sqrt{-\Delta_L} A^{\nu;2}Q^2}_2 + \norm{\sqrt{-\Delta_L} A^{2}Q^2_{\neq}}_2\right)\left(\norm{\sqrt{-\Delta_L} A^{\nu;j}Q^j}_2 + \norm{\sqrt{-\Delta_L}A^{j}Q^j_{\neq}}_2\right),
\end{align*}
which is subsequently absorbed by the dissipation.

\paragraph{Treatment of $NLS2(i,j,0,\neq)$ terms} 
These are treated similar to the analogous $NLS1$ terms in \S\ref{sec:NLS10neq_Q2ED}, yielding the following 
\begin{align*} 
NLS2(i,j,0,\neq) & \lesssim \epsilon\norm{A^{\nu;2}Q^2}\left(\norm{A^{\nu;2} Q^2} + \norm{A^2 Q^2_{\neq}}_2\right), 
\end{align*} 
which is absorbed by the dissipation.  

\paragraph{Treatment of $NLS2(i,j,\neq,0)$ terms} 
Notice that in this case, $j \neq 1$.  
First, by \eqref{ineq:AnuiDistri}, 
\begin{align*} 
NLS2(i,j,\neq,0) & \lesssim \norm{A^{\nu;2}Q^2}_2 \norm{\sqrt{-\Delta_L}A^{\nu ;2} U^j_{\neq}}_2 \norm{U_0^2}_{\G^{\lambda,\beta+3\alpha + 7}}.
\end{align*} 
Note then that we lose at most $2-\delta_1$ powers of time from $A^3$ if $j=3$, however these are recovered by Lemma \ref{lem:AnuLossy}: 
\begin{align*} 
NLS2(i,j,\neq,0) & = \epsilon\norm{A^{\nu;2}Q^2}_2\left(\norm{\sqrt{-\Delta_L} A^{\nu;j}Q^j}_2 + \norm{\sqrt{-\Delta_L} A^{j}Q_{\neq}^j}_2 \right), 
\end{align*} 
which is absorbed by the dissipation for $c_{0}$ sufficiently small. 

\paragraph{Treatment of $NLP(i,j,\neq,\neq)$} \label{sec:NLPneqneq_Q2ED}
Turn next to the nonlinear pressure interactions of two non-zero frequencies. 
By integration by parts and \eqref{ineq:AnuiDistriDecay}, we have 
\begin{align*} 
NLP(i,j,\neq,\neq) & \lesssim \norm{\sqrt{-\Delta_L}A^{\nu;2}Q^2}_2 \norm{A^{\nu;2}(\partial_j^t U^i_{\neq} \partial_i^tU^j_{\neq})}_2 \\ 
& \lesssim \norm{\sqrt{-\Delta_L}A^{\nu;2}Q^2}_2 \frac{\jap{t}^{\delta_1}}{\jap{\nu t^3}^\alpha} \left(\norm{\jap{\grad}^{2-\beta}A^{\nu;2} \partial_j^t U^i_{\neq}}_2 \norm{A^{\nu;2}\partial_i^tU^j_{\neq}}_2 \right. \\ 
 & \quad\quad \left. + \norm{A^{\nu;2} \partial_j^t U^i_{\neq}}_2 \norm{\jap{\grad}^{2-\beta}A^{\nu;2}\partial_i^tU^j_{\neq}}_2 \right). 
\end{align*}
Each combination of $i$ and $j$ can be treated in a rather similar manner, each time using \eqref{ineq:AnuiDistri} and Lemma \ref{lem:AnuLossy} (either \eqref{ineq:AnuLossyII} or \eqref{ineq:AnuLossyED} depending on the case). 
The case $NLP(1,3,\neq,\neq)$ turns out to be the hardest, and it is this case which precipitates the loss of $t^{\delta_1}$ in $Q^2$ (and hence ultimately the slightly slower than the rate of inviscid damping predicted by the linear theory in \eqref{ineq:u2damping}). 
Hence, let us simply focus on this case and omit the others for brevity. 
In this case, using the extra $\jap{\frac{t}{\jap{\grad_{Y,Z}}}}^{\delta_1}$ in the definition of $A^{\nu;2}$ in \eqref{def:Anu} and \eqref{ineq:AnuLossyED} (actually we do not need to use the full product rule in this case, but it is necessary for the $i = j = 3$ case so let us demonstrate it here):   
\begin{align*} 
NLP(1,3,\neq,\neq) & \lesssim \norm{\sqrt{-\Delta_L}A^{\nu;2}Q^2}_2 \frac{\jap{t}^{4}}{\jap{\nu t^3}^\alpha} \left(\norm{\jap{\grad}^{2-\beta}A^{\nu;1} \partial_Z^t U^1_{\neq}}_2 \norm{A^{\nu;3}\partial_X U^3_{\neq}}_2 \right. \\ 
 & \quad\quad \left. + \norm{A^{\nu;1} \partial_Z^{t} U^1_{\neq}}_2 \norm{\jap{\grad}^{2-\beta}A^{\nu;3}\partial_X U^3_{\neq}}_2 \right) \\ 
& \lesssim \norm{\sqrt{-\Delta_L}A^{\nu;2}Q^2}_2 \frac{\jap{t}^{4}}{\jap{\nu t^3}^\alpha} \left((1 + \epsilon t^2)\norm{A^{\nu;1} U^1_{\neq}}_2 \norm{A^{\nu;3}\partial_X U^3_{\neq}}_2 \right. \\ 
 & \quad\quad \left. + \norm{A^{\nu;1} \partial_Z^{t} U^1_{\neq}}_2 \norm{A^{\nu;3} U^3_{\neq}}_2 \right) \\
 & \lesssim \norm{\sqrt{-\Delta_L}A^{\nu;2}Q^2}_2 \frac{1}{\jap{\nu t^3}^{\alpha-1}} \left( \left(\norm{A^{\nu;1}Q^1}_2 + \norm{A^1 Q^1_{\neq}}_2\right)\left(\norm{A^{\nu;3}Q^3}_2 + \norm{A^3 Q^3_{\neq}}_2\right)    \right. \\  
& \quad\quad \left. + \left(\norm{\sqrt{-\Delta_L} A^{\nu;1}Q^1}_2 + \norm{\sqrt{-\Delta_L} A^1 Q^1_{\neq}}_2\right)\left(\norm{A^{\nu;3}Q^3}_2 + \norm{A^3 Q^3_{\neq}}_2\right) \right) \\ 
& \lesssim \epsilon\norm{\sqrt{-\Delta_L}A^{\nu;2}Q^2}_2 \left(\norm{\sqrt{-\Delta_L} A^{\nu;1}Q^1}_2 + \norm{\sqrt{-\Delta_L}A^{1} Q^1_{\neq}}_2  \right), 
\end{align*}  
which is then absorbed by the dissipation. 
The other terms can be treated with a simple variation or easier arguments (none of the others require the extra power of $t^{\delta_1}$ in \eqref{def:Anu} but the $i=j=3$ term depends more crucially on the product rule being employed above). 

\paragraph{Treatment of $NLS1(j,\neq,\neq)$} \label{sec:NLS1neqneq_Q2ED}
By \eqref{ineq:AnuiDistriDecay} and \eqref{ineq:AnuLossyII}, 
\begin{align*}
NLS1(j,\neq,\neq) & \lesssim \frac{\jap{t}^{\delta_1}}{\jap{\nu t^3}^{\alpha}}\norm{A^{\nu;2}Q^2}_2\norm{A^{\nu;2} Q^j}_2 \norm{A^{\nu;2} \partial_j^t U^2_{\neq}}_2 \\ 
& \lesssim \frac{\jap{t}}{\jap{\nu t^3}^{\alpha}}\norm{A^{\nu;2}Q^2}_2\norm{A^{\nu;j} Q^j}_2\left(\norm{A^{\nu;2} Q^2}_2 + \norm{A^{2} Q^2_{\neq}}_2\right),
\end{align*}
which is consistent with Proposition \ref{prop:Boot}. 

\paragraph{Treatment of $NLS2(i,j,\neq,\neq)$} \label{sec:NLS2neqneq_Q2ED}
For these terms we again apply \eqref{ineq:AnuiDistriDecay} to deduce 
\begin{align*} 
NLS2(i,j,\neq,\neq)  & \lesssim \norm{A^{\nu;2}Q^2}_2 \frac{\jap{t}^{\delta_1}}{\jap{\nu t^3}^\alpha} \norm{ A^{\nu;2} \partial_i^t U^j_{\neq}}_2 \norm{A^{\nu;2}\partial_{ij}^tU^2_{\neq}}_2. 
\end{align*} 
The most problematic term is $j = 3$, $i = 2$; in this case we apply \eqref{ineq:AnuLossyED} and \eqref{ineq:AnuLossyII},  
\begin{align*} 
NLS2(2,3,\neq,\neq)  & \lesssim \frac{\epsilon \jap{t}^{\delta_1}}{\jap{\nu t^3}^\alpha} \norm{ A^{\nu;2}Q^2}_2 \left(\norm{\sqrt{-\Delta_L} A^{\nu;3}Q^3}_2  + \norm{\sqrt{-\Delta_L} A^{3}Q^3_{\neq}}_2\right); 
\end{align*} 
the other cases can be treated similarly and are hence omitted for brevity. 
This completes the treatment of all of the nonlinear pressure and stretching terms. 

\subsubsection{Transport nonlinearity} \label{sec:Trans_ED_Q2}
These terms are treated similar to \S\ref{sec:Trans_ED_Q3}; we just briefly sketch the differences here. 
 Write the transport nonlinearity as 
\begin{align*} 
\mathcal{T} & = -\int A^{\nu;2}Q^2_{\neq} A^{\nu;2}\left(\tilde U_0 \cdot \grad_{Y,Z} Q^2_{\neq}\right) dV - \int A^{\nu;2}Q^2_{\neq} A^{\nu;2}\left(\tilde U_{\neq} \cdot \grad_{Y,Z} Q^2_{0}\right) dV \\ & \quad - \int A^{\nu;2}Q^2_{\neq} A^{\nu;2}_k\left(\tilde U_{\neq} \cdot \grad Q^2_{\neq}\right) dV \\ 
& = \mathcal{T}_{0\neq} + \mathcal{T}_{\neq0}  + \mathcal{T}_{\neq \neq}. 
\end{align*}  
As in \S\ref{sec:Trans_ED_Q3}, \eqref{ineq:AnuiDistri} together with the bootstrap hypotheses imply 
\begin{align*} 
\mathcal{T}_{0\neq} & \lesssim \epsilon\norm{\sqrt{-\Delta_L} A^{\nu;2}Q^2}^2_2. 
\end{align*} 
Similarly, 
\begin{align*} 
\mathcal{T}_{\neq 0} & \lesssim \epsilon\norm{A^{\nu;2}Q^2}^2_2\left(\norm{A^{\nu;2}Q^2}_2 + \norm{A^2Q^2_{\neq}}_2 + \norm{A^{\nu;3}Q^3}_2 + \norm{A^3Q^3_{\neq}}\right), 
\end{align*} 
which is absorbed by the dissipation for $c_0$ sufficiently small. 

For $\mathcal{T}_{\neq \neq}$ we get from \eqref{ineq:AnuiDistriDecay} and \eqref{ineq:AnuLossyII}, 
\begin{align*} 
\mathcal{T}_{\neq \neq}  & \lesssim \norm{A^{\nu;2}Q^2}_2 \frac{\jap{t}^{\delta_1}}{\jap{\nu t^3}^\alpha} \left(\norm{A^{\nu;2}U^1}_2 + \norm{A^{\nu;2}U^2}_2 + \norm{A^{\nu;2}U^3}_2\right)\norm{\sqrt{-\Delta_L} A^{\nu;2} Q^2}_{2} \\
& \lesssim \norm{A^{\nu;2}Q^2}_2 \frac{\jap{t}^{\delta_1}}{\jap{\nu t^3}^\alpha} \left(\jap{t}^2\norm{A^{\nu;1}U^1}_2 + \norm{A^{\nu;2}U^2}_2 + \jap{t}^{2-\delta_1}\norm{A^{\nu;3}U^3}_2\right)\norm{\sqrt{-\Delta_L} A^{\nu;2} Q^2}_{2} \\
& \lesssim \norm{A^{\nu;2}Q^2}_2 \frac{\epsilon \jap{t}^{\delta_1}}{\jap{\nu t^3}^\alpha}\norm{\sqrt{-\Delta_L} A^{\nu;2} Q^2}_{2} \\
& \lesssim \epsilon \norm{\sqrt{-\Delta_L} A^{\nu;2} Q^2}_{2}^2 + \frac{\epsilon \jap{t}^{2\delta_1}}{\jap{\nu t^3}^{2\alpha}} \norm{A^{\nu;2}Q^2}^2_2,
\end{align*} 
which completes the treatment of $\mathcal{T}_{\neq \neq}$. 

\subsubsection{Dissipation error terms}
As in \S\ref{sec:DE_ED_Q3}, these terms are treated in the same manner as the analogous terms in \cite{BMV14} and absorbed by the dissipation; the details are omitted for brevity.

\subsection{Enhanced dissipation of $Q^1$}
As in the high norm estimates on $Q^1_{\neq}$ in \S\ref{sec:HiQ1neq}, we need to deal with the issue caused by the lift-up effect and the linear stretching term. 
Computing the time evolution of $\norm{A^{\nu;1}Q^1}_2$, we get
\begin{align} 
\frac{1}{2}\frac{d}{dt} \norm{A^{\nu;1} Q^1}_2^2 & \leq \dot{\lambda}\norm{\abs{\grad}^{s/2}A^{\nu;1} Q^1}_2^2  + G^\nu -\norm{\sqrt{\frac{\partial_t w_L}{w_L}} A^{\nu;1} Q^1}_2^2 \nonumber \\ 
& \quad - \frac{t}{\jap{t}^{2}}\norm{A^{\nu;1}Q^1}_2^2 -\frac{(1+\delta_1)}{t} \norm{\mathbf{1}_{t > \jap{\grad_{Y,Z}}} A^{\nu;1} Q^1}_2^2 \nonumber \\
& \quad - \int A^{\nu;1}Q^1 A^{\nu;1} Q^2 dV -2 \int A^{\nu;1} Q^1 A^{\nu;1} \partial_{YX}^t U^1 dV \nonumber \\
 & \quad + 2 \int A^{\nu;1} Q^1 A^{\nu;1} \partial_{XX} U^2 dV  + \nu\int A^{\nu;1} Q^{1} A^{\nu;1} \left(\tilde{\Delta_t} Q^1\right) dv \nonumber \\ 
& \quad -\int A^{\nu;1} Q^1 A^{\nu;1}\left( \tilde U \cdot \grad Q^1 \right) dv \nonumber \\ 
& \quad -\int A^{\nu;1} Q^1 A^{\nu;1} \left[\left(Q^j \partial_j^t U^1\right) + 2\partial_i^t U^j \partial_{ij}^t U^1  - \partial_X\left(\partial_i^t U^j \partial_j^t U^i\right) \right] dv \nonumber \\ 
& = -\mathcal{D}Q^{\nu;1} + G^\nu - CK_{L1}^{\nu;1} - (1+\delta_1) CK_{L2}^{\nu;1} \nonumber \\ & \quad + LU + LS1 + LP1  + \mathcal{D}_E + \mathcal{T} + NLS1 + NLS2 + NLP.  \label{ineq:AnuEvo1}
\end{align} 
where $G^\nu$ is analogous to the corresponding term in \eqref{ineq:AnuEvo3}. 
As in \S\ref{sec:ED3}, $G^\nu$ is absorbed by using the dissipation. 
Note that for $i \in \set{2,3}$, 
\begin{align}
A^{\nu; 1} & \lesssim A^{\nu;i}. 
\end{align}

\subsubsection{Linear terms} 
The treatment of the $LU$ term is analogous to the treatment used in the improvement to \eqref{ineq:Boot_Q1Hi2} in \S\ref{sec:LUQhi2}. 
We omit the details for brevity and conclude that there is some constant $K > 0$ such that
\begin{align*} 
LU & \leq \delta_1t\jap{t}^{-2}\norm{A^{\nu;1} Q^1}^2_2 + \frac{\delta_\lambda}{4\delta_1 t^{3/2}}\norm{\abs{\grad}^{s/2}A^{\nu;2} Q^2}_2^2 + \frac{K}{\delta_1 \delta_\lambda^{\frac{1}{2s-1}} t^{3/2}}\norm{A^{\nu;2} Q^2}_2^2 +  \frac{K}{\delta_1 t} \norm{\mathbf{1}_{t > \jap{\grad_{Y,Z}}} A^{\nu;2}Q^2}_2^2,  
\end{align*}
which is consistent with Proposition \ref{prop:Boot} provided $K_{ED1} \gg K_{ED2} \max(\delta_1^{-2}, \delta_\lambda^{\frac{1}{2s-1}} \delta_1^{-1})$.  
The treatment of $LS1$ can be made analogous to the $LS3$ term treated above in \S\ref{sec:ED3} along with the $t^{\delta_1}$ tweak used in the improvement of \eqref{ineq:Boot_Q1Hi2}. We omit the details and conclude, for some constant $K > 0$,  
\begin{align*} 
LS1 & \leq (1+\delta_1)CK^{\nu;1}_{L2} + (1-\delta_1)CK_{L1}^{\nu;1} + K\epsilon \norm{\sqrt{-\Delta_L} A^{\nu;1}Q^1}^2_2 + \frac{K}{\jap{t}^2}\norm{A^1 Q^1_{\neq}}_2^2 + \frac{K\epsilon}{\jap{t}^2}\norm{A^{1}\Delta_L U^1_{\neq}}_2^2 \\ & \quad + \frac{\delta_\lambda}{10\jap{t}^{3/2}}\norm{\abs{\grad}^{s/2} A^{\nu;1}Q^1}_2^2 +  \frac{K}{\delta_\lambda^{\frac{1}{2s-1}}\jap{t}^{3/2}} \norm{ A^{\nu;1}Q^1}_2^2, 
\end{align*} 
which (after Lemma \ref{lem:SimplePEL}) is consistent with Proposition \ref{prop:Boot} for $K_{ED1}$ chosen large relative to $K_{H1\neq}$ and $K_{ED2}$. 
Next consider the linear pressure term $LP1$.  

We may directly apply Lemma \ref{lem:AnuLossy} to deduce 
\begin{align*} 
LP1  \leq 2\norm{A^{\nu;1}Q^1}_2 \norm{\partial_{XX} A^{\nu;1}U^2_{\neq}}_2 &\lesssim \jap{t}^{-3} \norm{A^{\nu;1}Q^1}_2\left(\norm{A^{\nu;2}Q^2_{\neq}}_2 + \norm{A^{2}Q^2_{\neq}}_2\right) \\ 
& \lesssim \frac{1}{\jap{t}^{3}}\norm{A^{\nu;1}Q^1}^2_2 + \frac{1 + K_{ED2}}{\jap{t}^{3}}\epsilon^2, 
\end{align*} 
which is consistent with Proposition \ref{prop:Boot} provided $K_{ED1} \gg K_{ED2}$. 

\subsubsection{Nonlinear pressure and stretching} 
These terms are treated in essentially the same manner as in \S\ref{sec:NLPS_Q3ED}, however, we sketch some of the similarities and differences briefly. 
Recall the enumeration
As above, we will enumerate the terms as follows for $i,j \in \set{1,2,3}$ and 
$a,b \in \set{0,\neq}$
\begin{subequations}  \label{def:enumnu}
\begin{align}
NLP(i,j,a,b) & = \int A^{\nu;1} Q^1 A^{\nu;1} \partial_X(\partial_j^t U^i_a \partial_i^t U^j_b  ) dV \\
NLS1(j,a,b) & = -\int A^{\nu;1} Q^1 A^{\nu;1} \left( Q^j_a \partial_j^t U^1_b  \right) dV \\
NLS2(i,j,a,b) & = -2\int A^{\nu;1} Q^1 A^{\nu;1} (\partial_i^t U^j_a \partial_i^t\partial_j^t U^1_b  ) dV.
\end{align}
\end{subequations}

\paragraph{Treatment of $NLP(i,j,0,\neq)$ terms} \label{sec:NLP0neq_Q1ED}
Notice that in this case, $j \neq 1$. From \eqref{ineq:AnuiDistri}, Lemma \ref{lem:AnuLossy}, and \eqref{ineq:AprioriU0}, 
\begin{align*}                                
NLP(i,j,\neq,0) & \lesssim  \norm{A^{\nu;1}Q^1}_2 \norm{A^{\nu;1} \partial_X\partial_i^t U^j}_2\norm{U_0^i}_{\G^{\lambda,\beta + 3\alpha + 5}} \lesssim  \epsilon\norm{A^{\nu;1}Q^1}_2 \left(\norm{A^{\nu;j} Q^j}_2 + \norm{A^{j}Q^j_{\neq}}_2\right),  
\end{align*} 
which is subsequently absorbed by the dissipation. 

\paragraph{Treatment of $NLS1(j,0,\neq)$ terms} \label{sec:NLS10neq_Q1ED} 
From \eqref{ineq:AnuiDistri}, Lemma \ref{lem:AnuLossy}, and \eqref{ineq:AprioriU0}, 
\begin{align*}
NLS1(j,0,\neq) &  \lesssim \norm{A^{\nu;1}Q^1}_2 \norm{Q^j_0}_{\G^{\lambda,\beta + 3\alpha + 5}}  \norm{\partial_j^t A^{\nu;1}U^1}_2
& \lesssim \frac{\epsilon}{\jap{t}}\norm{A^{\nu;1}Q^1}_2\left(\norm{A^{\nu;1}Q^1}_2 + \norm{A^{1}Q^1_{\neq}}_2 \right).
\end{align*}
which is absorbed by the dissipation. 

\paragraph{Treatment of $NLS1(j,\neq,0)$ terms} \label{sec:NLS1neq0_Q1ED} 
Note that $j \neq 1$. 
From \eqref{ineq:AnuiDistri}, Lemma \ref{lem:AnuLossy}, \eqref{ineq:AnuHiLowSep2}, and \eqref{ineq:AprioriU0}, 
\begin{align*}
NLS1(j,\neq,0) &  \lesssim \norm{A^{\nu;1}Q^1}_2 \norm{A^{\nu;1} Q^j}_2\norm{U_0^1}_{\G^{\lambda,\beta + 3\alpha + 5}} \\ 
&  \lesssim \epsilon\left(\norm{\sqrt{-\Delta_L} A^{\nu;1}Q^1}_2 + \norm{A^{1}Q^1_{\neq}}_2 \right)\left(\norm{A^{\nu;j} Q^j}_2 + \norm{A^j Q^j_{\neq}}_2\right),
\end{align*}
which is absorbed by the dissipation. 

\paragraph{Treatment of $NLS2(i,j,\neq,0)$ terms} \label{sec:NLS2neq0_Q1ED} 
From \eqref{ineq:AnuiDistri}, Lemma \ref{lem:AnuLossy}, and \eqref{ineq:AprioriU0}, we have (noting that both $j \neq 1$ and $i \neq 1$):
\begin{align*} 
NLS2(i,j,\neq,0) & \lesssim \norm{A^{\nu;1}Q^1}_2 \norm{A^{\nu;1} \partial_i^t U^j_{\neq}}_2 \norm{U_0^1}_{\G^{\lambda,\beta + 3\alpha +6}} \lesssim \epsilon \norm{A^{\nu;1}Q^1}_2 \left(\norm{A^{\nu;j}Q^j}_2 + \norm{A^{j}Q^j_{\neq}}_2\right), 
\end{align*} 
which is then absorbed by the dissipation. 

\paragraph{Treatment of $NLS2(i,j,0,\neq)$ terms} \label{sec:NLS20neq_Q1ED} 
From \eqref{ineq:AnuiDistri}, Lemma \ref{lem:AnuLossy}, and \eqref{ineq:AprioriU0}. we have 
\begin{align*} 
NLS2(i,j,0,\neq)& \lesssim \norm{A^{\nu;1}Q^1}_2 \norm{A^{\nu;1} \partial_{ij}^t U^1}_2 \norm{U_0^j}_{\G^{\lambda,\beta+3\gamma+5}} \lesssim \epsilon \norm{A^{\nu;1}Q^1}_2 \left(\norm{A^{\nu;1}Q^1}_2 + \norm{A^{1}Q^1_{\neq}}_2\right), 
\end{align*} 
which is then absorbed by the dissipation. 

\paragraph{Treatment of $NLP(i,j,\neq,\neq)$, $NLS1(i,j,\neq,\neq)$, and $NLS2(i,j,\neq,\neq)$}
The nonlinear terms involving two non-zero frequencies can all be treated in essentially the same manner as 
in $Q^3$ in \S\ref{sec:NLPneqneq_Q3ED} and \S\ref{sec:NLS1neqneq_Q3ED}. 
We treat the representative example $NLS2(2,3,\neq,\neq)$. By \eqref{ineq:AnuiDistriDecay},  
\begin{align*} 
NLS2(2,3,\neq,\neq) & \lesssim \frac{\jap{t}^{2+\delta_1}}{\jap{\nu t^3}^\alpha}\norm{A^{\nu;1}Q^1}_2 \norm{A^{\nu;1}\partial_Y^t U^3}_2 \norm{A^{\nu;1}\partial_{YZ}^t U^1}_2 \\ 
& \lesssim \frac{\epsilon t^{1+\delta_1}}{\jap{\nu t^3}^\alpha}\norm{A^{\nu;1}Q^1}_2\left( \norm{A^{\nu;1}Q^1}_2 + \norm{A^1 Q^1_{\neq}}_2 \right) \left(\norm{A^{\nu;3}Q^3}_2 + \norm{A^3 Q^3_{\neq}}_2 \right),
\end{align*}
which is consistent with Proposition \ref{prop:Boot} for $\epsilon$ sufficiently small. 

\subsubsection{Transport nonlinearity}
The transport nonlinearity, $\mathcal{T}$ in \eqref{ineq:AnuEvo1}, can be treated in the same manner as the transport nonlinearity in \S\ref{sec:Trans_ED_Q2}. 
We omit the details for brevity.

\subsubsection{Dissipation error terms}   
The dissipation error terms can be treated in same manner as those in \S\ref{sec:DE_ED_Q3} and hence we omit the details for brevity. This completes the enhanced dissipation estimate on $Q^1$. 

\section{Sobolev estimates} \label{sec:LowNrmVel}

In this section we improve the $H^{\sigma^\prime}$ (and $L^4$) estimates in \eqref{ineq:Boot_LowFreq}.   
Since these estimates are on quantities which are independent of $X$ and are in finite regularity (so it is easy to handle compositions), we may 
deduce them in the original $(y,z)$ variables and then transfer them back to $(Y,Z)$ variables (see \S\ref{sec:RegCont} for more discussion on changing coordinate systems). 
This is also useful for taking advantage of the divergence free condition in the form \eqref{ineq:specq} below. 

First, recall Lemma \ref{lem:intermedSob}. 
From \eqref{eq:u0i} and the equation satisfied by $q^2_0$ (from \eqref{def:qi}):
\begin{align} 
\partial_t q_0^2 + (u_0^2,u_0^3)^T \cdot \grad q_0^2 = -q^j_0 \partial_j u^2_0 + \partial_y\left(\partial_i u^j_0 \partial_j u^i_0\right) - 2\partial_{i} u^j_0 \partial_{ij}u^2_0 + \mathcal{F}_q + \nu \Delta q^2_0, \label{eq:q02} 
\end{align} 
where,  
analogous to \eqref{eq:XavgCanc} and \eqref{eq:Fbaru}, we derive
\begin{align*} 
\mathcal{F}_q = \left(\partial_i \partial_i \partial_j \left( \bar{u}^j_{\neq} \bar{u}_{\neq}^2\right)_{0} - \partial_{y} \partial_j \partial_i \left(\bar{u}_{\neq}^i \bar{u}^j_{\neq}\right)_{0} \right)\mathbf{1}_{i \neq 1, j \neq 1}. 
\end{align*} 
In what follows denote 
\begin{align*}
u_0 := \begin{pmatrix} u_0^2 \\ u_0^3 \end{pmatrix}.  
\end{align*}

\subsection{Improvement of \eqref{ineq:Boot_Q02_Low}}
In this section we improve \eqref{ineq:Boot_Q02_Low} from \eqref{eq:q02} by making an $L^2$ estimate on $q_0^2$ and then transforming this to an estimate on $Q_0^2$. 
The first thing to note is the spectral gap estimate
\begin{align} 
\|q_0^2 \|_{H^{\sigma'}} + \|u_0^2 \|_{H^{\sigma'}} + \| \nabla \partial_z u_0^2 \|_{H^{\sigma'}} + \| \partial_z u_0^2 \|_{H^{\sigma'}} \lesssim \norm{\grad q_0^2}_{H^{\sigma^\prime}} \label{ineq:specq}
\end{align} 
following from the divergence free condition on $u$. 
This kind of estimate would normally imply exponential decay, however, the decay of $\mathcal{F}$ is only polynomial. 
From \eqref{eq:q02} we have
\begin{align*} 
\frac{1}{2}\frac{d}{dt} \norm{ q_0^2}_{H^{\sigma^\prime}}^2 & = -\nu \norm{\grad q_0^2}_{H^{\sigma^\prime}}^2 - \int \jap{\grad}^{\sigma^\prime} q_0^2 \jap{\grad}^{\sigma^\prime} \left( u_0 \cdot \grad q_0^2\right) dy dz \\
 & \quad + \int \jap{\grad}^{\sigma^\prime} q_0^2 \jap{\grad}^{\sigma^\prime} \left[-q^j_0 \partial_j u^2_0 + \partial_y\left(\partial_i u^j_0 \partial_j u^i_0\right) - 2\partial_{i} u^j_0 \partial_{ij}u^2_0\right] dy dz \\ & \quad + \int \jap{\grad}^{\sigma^\prime} q_0^2 \jap{\grad}^{\sigma^\prime}  \mathcal{F}_q dy dz.  
\end{align*}  
Write
\begin{align*} 
-\int \jap{\grad}^{\sigma^\prime} q_0^2 \jap{\grad}^{\sigma^\prime} \left( u_0 \cdot \grad q_0^2\right) dy dz & = -\int \jap{\grad}^{\sigma^\prime} q_0^2 \jap{\grad}^{\sigma^\prime} \left( u_0^2 \partial_y q_0^2\right) dy dz -\int \jap{\grad}^{\sigma^\prime} q_0^2 \jap{\grad}^{\sigma^\prime} \left( u_0^3 \partial_z q_0^3\right) dy dz \\  
& = T_y + T_z.
\end{align*} 
For $T_y$ we use the algebra property of $H^{\sigma^\prime}$, \eqref{ineq:uzAPriori}, and \eqref{ineq:specq},   
\begin{align*} 
T_y & \lesssim \epsilon\norm{\grad  q_0^2}_{H^{\sigma^\prime}}^2 + \epsilon^{-1}\norm{q_0^2}_{H^{\sigma'}}^2\norm{u^2_0}_{{H^{\sigma'}}}^2 \lesssim \epsilon\norm{\grad  q_0^2}_{H^{\sigma^\prime}}^{2}, 
\end{align*}  
which is consistent with Proposition \ref{prop:Boot} for $c_{0}$ sufficiently small.
For $T_z$ we use the fact that one of the first two factors must have non-zero frequency in $z$ (this was also used in \S\ref{sec:TransQ20}), 
\begin{align*}
T_z & \lesssim \norm{u_0^3}_{H^{\sigma^\prime}} \norm{\grad q_0^2}_{H^{\sigma^\prime}}^2 + \norm{q_0^2}_2 \norm{\grad u_0^3}_{H^{\sigma^\prime}} \norm{\grad q_0}_{H^{\sigma^\prime}} \\ 
& \lesssim \epsilon \norm{\grad q_0^2}_{H^{\sigma^\prime}}^2 +  \frac{\epsilon}{\jap{\nu t}^{2\alpha}}\norm{\grad u_0^3}_{H^{\sigma^\prime}}^2,
\end{align*}  
which is consistent with Proposition \ref{prop:Boot}. 
From the algebra property and \eqref{ineq:Xyzubds} we get
\begin{align*} 
\int \jap{\grad}^{\sigma^\prime}q_0^2 \jap{\grad}^{\sigma^\prime} \mathcal{F}_q dy dz & \lesssim \frac{\epsilon^2}{\jap{\nu t^3}^{2\alpha}} \norm{q_0^2}_{H^{\sigma^\prime}}. 
\end{align*} 

Turn to the $NLS1$ term,
\begin{align*}
-\int \jap{\grad}^{\sigma^\prime}q_0^2 \jap{\grad}^{\sigma^\prime} \left(q_0^j \partial_j u^2_0\right) dy dz
& = -\int \jap{\grad}^{\sigma^\prime}q_0^2 \jap{\grad}^{\sigma^\prime} \left(\Delta u_0^3 \partial_z u^2_0\right) dy dz - \int \jap{\grad}^{\sigma^\prime}q_0^2 \jap{\grad}^{\sigma^\prime} \left(q_0^2 \partial_y u^2_0\right) dy dz \\
& = S_z + S_y. 
\end{align*}
To treat $S_y$ we use the algebra property and \eqref{ineq:uzAPriori} to obtain
\begin{align*}
S_y &\lesssim \norm{q_0^2}^2_{H^{\sigma^\prime}}\norm{\grad u_0^2}_{H^{\sigma^\prime}} \lesssim \frac{\epsilon}{\jap{\nu t}^\alpha}\norm{q_0^2}^2_{H^{\sigma^\prime}}. 
\end{align*}
For $S_z$ we use the integration by parts
\begin{align*}
S_z & = \int \jap{\grad}^{\sigma^\prime}\grad q_0^2 \cdot \left(\grad \jap{\grad}^{\sigma^\prime} u_0^3\right)  \partial_z u^2_0 dy dz + \int \jap{\grad}^{\sigma^\prime}q_0^2 \left(\grad \jap{\grad}^{\sigma^\prime} u_0^3\right) \cdot \partial_z \grad u^2_0 dy dz 
  \\ & \quad -\int \jap{\grad}^{\sigma^\prime}q_0^2\left(\jap{\grad}^{\sigma^\prime} \left(\Delta u_0^3 \partial_z u^2_0\right) - \left(\Delta \jap{\grad}^{\sigma^\prime} u_0^3\right)  \partial_z u^2_0\right) dy dz \\ 
& = S_z^1 + S_z^2 + S_z^3.   
\end{align*}
To control the first two terms, we use \eqref{ineq:uzAPriori} (and \eqref{ineq:specq}), to deduce 
\begin{align*}
S_z^1 + S_z^2 & \lesssim \norm{\grad q_0^2}_{H^{\sigma^\prime}}\norm{\grad u_0^3}_{H^{\sigma^\prime}}\norm{\partial_z u_0^2}_{H^{\sigma^\prime}} + \norm{q_0^2}_{H^{\sigma^\prime}}\norm{\grad u_0^3}_{H^{\sigma^\prime}}\norm{\grad \partial_z u_0^2}_{H^{\sigma^\prime}} \\
& \lesssim \epsilon \left( \norm{\grad q_0^2}_{H^{\sigma^\prime}}^2 + \|q_0^2 \|^2_{H^{\sigma'}} + \| \partial_z u_0^2 \|_{H^{\sigma'}}^2 \right) \\ 
& \lesssim \epsilon \norm{\grad q_0^2}_{H^{\sigma^\prime}}^2. 
\end{align*}
Treating the commutator in the $S_z^3$ term is by now classical and, in particular, by using that for $\abs{\eta,l} \approx \abs{\xi,l^\prime}$,  
\begin{align*}
\jap{\eta,l}^{\sigma^\prime} - \jap{\xi,l^\prime}^{\sigma^\prime} \lesssim \abs{\eta-\xi,l-l^\prime} \jap{\xi,l^\prime}^{\sigma^\prime-1},  
\end{align*}
we have 
\begin{align*}
S_z^3 & \lesssim \norm{q_0^2}_{H^{\sigma^\prime}}\norm{\grad u_0^3}_{H^{\sigma^\prime}}\norm{\partial_z u_0^2}_{H^{\sigma^\prime}} \lesssim \epsilon \|q_0^2 \|^2_{H^{\sigma'}} + \frac{\epsilon}{\jap{\nu t}^{2\alpha}} \| \grad u_0^3 \|_{H^{\sigma'}}^2,
\end{align*}
which is sufficient. 
For the $NLP$ term, by integration by parts and \eqref{ineq:uzAPriori},
\begin{align*} 
\int \jap{\grad}^{\sigma^\prime} q_0^2 \jap{\grad}^{\sigma^\prime}\partial_y\left(\partial_i u^j_0 \partial_j u^i_0\right) dy dz & \lesssim \norm{\grad q_0^2}_{H^{\sigma^\prime}}\left(\norm{\partial_z u_0^3}_{H^{\sigma^\prime}}^2 + \norm{\partial_y u_0^3}_{H^{\sigma^\prime}}\norm{\partial_z u_0^2}_{H^{\sigma^\prime}} + \norm{\partial_y u_0^2}^2_{H^{\sigma^\prime}}\right) \\ 
& \lesssim \epsilon \norm{\grad q_0^2}_{H^{\sigma^\prime}}^2 + \frac{\epsilon}{\jap{\nu t}^{2\alpha}}\norm{\grad u_0^3}_{H^{\sigma^\prime}}^2.  
\end{align*} 
For the second $NLS$ term we do not need to integrate by parts and instead use the algebra property and the definition of $q_0^2$: 
\begin{align*} 
-\int \jap{\grad}^{\sigma^\prime} q_0^2 \jap{\grad}^{\sigma^\prime}\left(\partial_i u^j_0 \partial_i\partial_j u^2_0\right) dy dz & \lesssim \norm{q_0^2}^2_{H^{\sigma^\prime}}\left(\norm{\grad u^2_0}_{H^{\sigma^\prime}} + \norm{\grad u_0^3}_{H^{\sigma^\prime}}\right) \\ 
& \lesssim \epsilon \norm{ q_0^2}^2_{H^{\sigma^\prime}} + \frac{\epsilon}{\jap{\nu t}^{2\alpha}}\norm{\nabla u_0^3}^2_{H^{\sigma^\prime}}. 
\end{align*}

Putting the above estimates together now yields the stated decay estimate by standard methods via \eqref{ineq:specq}. 
The details are omitted for brevity and we simply conclude that for $c_0$ chosen sufficiently small we have 
\begin{align}
\norm{q_0^2}_{H^{\sigma^\prime}} \leq \frac{3 \epsilon}{2 \jap{\nu t}^{\alpha}}.   \label{ineq:q02yzest}
\end{align} 
It follows by Sobolev composition that for $c_0$ sufficiently small,  $\norm{Q_0^2}_{H^{\sigma^\prime}} \leq (1+ Kc_0)\norm{q_0^2}_{H^{\sigma^\prime}}$ for some universal $K > 0$, and hence for $c_0$ chosen sufficiently small, the improvement to \eqref{ineq:Boot_Q02_Low} follows.  

\subsection{Improvement of \eqref{ineq:Boot_partZU03_Low}}
Note that $\int \partial_z u_0^3 dz = 0$ and hence we get the spectral gap 
\begin{align} 
\norm{\partial_z u_0^3}_{H^{\sigma^\prime}} \lesssim \norm{\grad \partial_z u_0^3}_{H^{\sigma^\prime}}. \label{ineq:specgu3} 
\end{align} 
From \eqref{eq:u0i} we get
\begin{align*} 
\frac{1}{2}\frac{d}{dt} \norm{\partial_z u_0^3}_{H^{\sigma^\prime}}^2 & = -\nu \norm{\grad \partial_z u_0^3}_{H^{\sigma^\prime}}^2 - \int \jap{\grad}^{\sigma^\prime} \partial_z u_0^3 \jap{\grad}^{\sigma^\prime} \partial_z \left( u_0 \cdot \grad u_0^3\right) dy dz \\ & \quad + \int \jap{\grad}^{\sigma^\prime} \partial_z u_0^3 \jap{\grad}^{\sigma^\prime} \partial_{zz} \Delta^{-1}(\partial_i u_0^j \partial_j u^3_0) dy dz + \int \jap{\grad}^{\sigma^\prime} \partial_z u_0^3 \jap{\grad}^{\sigma^\prime}  \partial_z \mathcal{F}^3 dy dz.  
\end{align*} 
The transport term is handled via 
\begin{align*} 
-\int \jap{\grad}^{\sigma^\prime} \partial_z u_0^3 \jap{\grad}^{\sigma^\prime} \partial_z \left( u_0 \cdot \grad u_0^3\right) dy dz & = -\int \jap{\grad}^{\sigma^\prime} \partial_z u_0^3 \jap{\grad}^{\sigma^\prime} \left( \partial_z u_0 \cdot \grad u_0^3\right) dy dz  \\ 
& \quad -\int \jap{\grad}^{\sigma^\prime} \partial_z u_0^3 \jap{\grad}^{\sigma^\prime} \left( u_0 \cdot \grad \partial_z u_0^3\right) dy dz \\ 
& = T_1 + T_2. 
\end{align*} 
For $T_1$ we get from the algebra property 
\begin{align*} 
T_1 & \lesssim \norm{\partial_z u_0^3}_{H^{\sigma^\prime}} \norm{\grad u_0^3}_{H^{\sigma^\prime}}  \left(\norm{\partial_z u^2_0}_{H^{\sigma^\prime}} + \norm{\partial_z u^3_0}_{H^{\sigma^\prime}} \right) \\ 
& \lesssim \epsilon\norm{\partial_z u_0^3}_{H^{\sigma^\prime}}^2 + \frac{\epsilon}{\jap{\nu t}^{2\alpha}} \norm {\grad u_0^3}_{H^{\sigma^\prime}}^2. 
\end{align*} 
For $T_2$ we may instead simply absorb with the dissipation (using \eqref{ineq:specgu3}) by adding another derivative to the first factor: 
\begin{align*} 
T_2 & \lesssim \norm{\partial_z u_0^3}_{H^{\sigma^\prime}}^2 \norm{u_0}_{H^{\sigma^\prime}}\norm{\nabla \partial_z u_0^3}_{H^{\sigma^\prime}}^2 \lesssim \epsilon \norm{\nabla \partial_z u_0^3}_{H^{\sigma^\prime}}^2,
\end{align*} 
which will be sufficient to treat the transport term. 

From \eqref{ineq:Xyzubds} we get
\begin{align*} 
\int \jap{\grad}^{\sigma^\prime} \partial_z u_0^3 \jap{\grad}^{\sigma^\prime} \partial_z \mathcal{F}^3 dy dz & \lesssim \frac{\epsilon^2}{\jap{\nu t^3}^{2\alpha}} \norm{\partial_z u_0^3}_{H^{\sigma^\prime}}. 
\end{align*} 

The last difficulty is from the pressure. 
By the algebra property,
\begin{align*} 
\int \jap{\grad}^{\sigma^\prime} \partial_z u_0^3 \jap{\grad}^{\sigma^\prime} \partial_{zz} \Delta^{-1}(\partial_i u^j_0 \partial_j u^i_0) dy dz & \lesssim \norm{\partial_z u_0^3}_{H^{\sigma^\prime}}\left(\norm{\partial_z u^2_0 \partial_y u^3_0}_{H^{\sigma^\prime}} + \norm{(\partial_y u^2_0)^2}_{H^{\sigma^\prime}} + \norm{(\partial_z u^3_0)^2}_{H^{\sigma^\prime}} \right) \\ 
& \lesssim \frac{\epsilon}{\jap{\nu t}^{\alpha}} \norm{\partial_z u_0^3}_{H^{\sigma^\prime}}\norm{\grad u^3_0}_{H^{\sigma^\prime}} + \frac{\epsilon^3}{\jap{\nu t}^{3\alpha}} \\ 
& \lesssim \epsilon \norm{\partial_z u_0^3}_{H^{\sigma^\prime}}^2 + \frac{\epsilon}{\jap{\nu t}^{2\alpha}} \norm{\grad u^3_0}_{H^{\sigma^\prime}} + \frac{\epsilon^3}{\jap{\nu t}^{3\alpha}},  
\end{align*} 
which is consistent with Proposition \ref{prop:Boot} for $c_0$ sufficiently small by \eqref{ineq:specgu3}. 
Putting the above estimates together now yields the stated decay estimate using \eqref{ineq:specgu3} via standard methods.  
The details are omitted for brevity.  

\subsection{Improvement of \eqref{ineq:Boot_U02_Low}}
Due to the divergence-free condition, it follows that 
\begin{align*}
\widehat{u^2_0}(\eta,l) = -\frac{\widehat{q^2_0}(\eta,l)\mathbf{1}_{l \neq 0}}{\abs{\eta}^2 + \abs{l}^2}. 
\end{align*}
Therefore, it follows from \eqref{ineq:q02yzest} that 
\begin{align*}
\norm{u_0^2}_{H^{\sigma^\prime}} \leq \frac{3 \epsilon}{2 \jap{\nu t}^{\alpha}}, 
\end{align*}
and so for $c_0$ sufficiently small, the improvement to \eqref{ineq:Boot_U02_Low} follows by Sobolev composition as above.

\subsection{Improvement of \eqref{ineq:Boot_U01_Low1} and \eqref{ineq:Boot_U01_Low2}}
The improvement of \eqref{ineq:Boot_U01_Low1} and \eqref{ineq:Boot_U01_Low2} will be implied by the two following estimates: 
\begin{subequations} \label{ineq:U01LowReal}
\begin{align} 
\norm{U_0^1}_{H^{\sigma^\prime}}^2 + \nu \int_1^t \norm{\grad U_0^1(\tau)}_{H^{\sigma^\prime}}^2 d\tau & \leq 2K_{U12} c_0^2 \label{ineq:U011} \\ 
\norm{U_0^1}_{H^{\sigma^\prime}} & \leq 2K_{U1}\epsilon \jap{t}. \label{ineq:U012}
\end{align} 
\end{subequations} 
Both are straightforward (especially since the nonlinear pressure vanishes from \eqref{eq:u0i} for $u_0^1$); the only difference is in the treatment of the lift up effect term.  
Indeed, from \eqref{eq:u0i} we get
\begin{align} 
\frac{1}{2}\frac{d}{dt} \norm{u_0^1}_{H^{\sigma^\prime}}^2 & = -\nu \norm{\grad u_0^1}_{H^{\sigma^\prime}}^2 - \int \jap{\grad}^{\sigma^\prime} u_0^1 \jap{\grad}^{\sigma^\prime} \left( u_0 \cdot \grad u_0^1\right) dy dz \nonumber \\ & \quad - \int\jap{\grad}^{\sigma^\prime} u_0^1 \jap{\grad}^{\sigma^\prime}u_0^2 dy dz  + \int \jap{\grad}^{\sigma^\prime} u_0^1 \jap{\grad}^{\sigma^\prime} \mathcal{F}^1 dy dz.  \label{ineq:LowU01Evo}
\end{align} 
The forcing terms are dealt with as above and are hence omitted. 
To treat the transport terms, we use the slight variation of the proof used in \S\ref{sec:TransQ20} above. 
Indeed, note 
\begin{align*}
- \int \jap{\grad}^{\sigma^\prime} u_0^1 \jap{\grad}^{\sigma^\prime} \left( u_0 \cdot \grad u_0^1\right) dy dz & = - \int \jap{\grad}^{\sigma^\prime} u_0^1 \jap{\grad}^{\sigma^\prime} \left( u_0^2 \partial_y u_0^1\right) dy dz \\ 
& \quad - \int \jap{\grad}^{\sigma^\prime} u_0^1 \jap{\grad}^{\sigma^\prime} \left( u_0^3 \partial_z u_0^1\right) dy dz \\ 
& = T_y + T_z. 
\end{align*}  
For $T_y$ we may use the algebra property 
\begin{align*} 
T_y & \lesssim \norm{u_0^1}_{H^{\sigma^\prime}}\norm{\grad u_0^1}_{H^{\sigma^\prime}}\norm{u_0^2}_{H^{\sigma^\prime}} 
\lesssim \epsilon\norm{\grad u_0^1}_{H^{\sigma^\prime}}^2 + \frac{\epsilon}{\jap{\nu t}^{2\alpha}}\norm{u_0^1}^2_{H^{\sigma^\prime}}, 
\end{align*}
which is consistent with \eqref{ineq:U01LowReal} for $c_0$ sufficiently small. 
For $T_z$ note that due to the $\partial_z$, at least one of the other two factors must have a non-zero frequency in $z$, which implies
\begin{align*} 
T_z & \lesssim \norm{u_0^1}_{H^{\sigma^\prime}}\norm{\grad u_0^1}_{H^{\sigma^\prime}}\norm{\partial_z u_0^3}_{H^{\sigma^\prime}} + \norm{u_0^3}_{H^{\sigma^\prime}}\norm{\grad u_0^1}^2_{H^{\sigma^\prime}} \\ 
& \lesssim \epsilon\norm{\grad u_0^1}_{H^{\sigma^\prime}}^2 + \frac{\epsilon}{\jap{\nu t}^{2\alpha}}\norm{u_0^1}^2_{H^{\sigma^\prime}}, 
\end{align*}  
which is again consistent with \eqref{ineq:U01LowReal} for $c_0$ sufficiently small. 

If one is deducing \eqref{ineq:U011} then the lift-up effect term can be bounded for some universal $K > 0$ by
\begin{align*} 
- \int\jap{\grad}^{\sigma^\prime} u_0^1 \jap{\grad}^{\sigma^\prime}u_0^2 dy dz & \leq \norm{u_0^1}_{H^{\sigma^\prime}}\norm{u_0^2}_{H^{\sigma^\prime}} 
\leq \frac{\nu}{\jap{\nu t}^\alpha}\norm{u_0^1}_{H^{\sigma^\prime}}^2 + \frac{K\epsilon^2}{\nu \jap{\nu t}^{\alpha}}. 
\end{align*} 
The former term is treated via an integrating factor whereas the latter term is consistent with \eqref{ineq:U011}, both for $K_{U12}$  sufficiently large. 

If one is deducing \eqref{ineq:U012}, then first we multiply \eqref{ineq:LowU01Evo} on both sides by $\jap{t}^{-2}$ and use
\begin{align*} 
\frac{1}{2 \jap{t}^{2}}\frac{d}{dt} \norm{u_0^1}_{H^{\sigma^\prime}}^2 = \frac{1}{2}\frac{d}{dt}\left( \jap{t}^{-2} \norm{u_0^1}_{H^{\sigma^\prime}}^2\right) + \frac{t}{\jap{t}^4}  \norm{u_0^1}_{H^{\sigma^\prime}}^2. 
\end{align*}  
Then the lift up effect term is treated via 
\begin{align*} 
- \jap{t}^{-2}\int\jap{\grad}^{\sigma^\prime} u_0^1 \jap{\grad}^{\sigma^\prime}u_0^2 dy dz & \leq \jap{t}^{-2}\norm{u_0^1}_{H^{\sigma^\prime}}\norm{u_0^2}_{H^{\sigma^\prime}} \leq 4\jap{t}^{-2}\epsilon \norm{u_0^1}_{H^{\sigma'}}. 
\end{align*} 
The nonlinear terms are treated by an easy variant of the treatment applied in the proof of \eqref{ineq:U011}.  From here, one applies the super-solution method described in \S\ref{sec:Q1Hi1}. We omit the remaining details for brevity. 
 
\subsection{Improvement of \eqref{ineq:Boot_U03_Low}}
This estimate is a slight variant of the proof of \eqref{ineq:Boot_U01_Low1}.
From \eqref{eq:u0i} we get
\begin{align*} 
\frac{1}{2}\frac{d}{dt} \norm{u_0^3}_{H^{\sigma^\prime}}^2 & = -\nu \norm{\grad u_0^3}_{H^{\sigma^\prime}}^2 - \int \jap{\grad}^{\sigma^\prime} u_0^3 \jap{\grad}^{\sigma^\prime} \left( u_0 \cdot \grad u_0^3\right) dy dz \\ & \quad + \int \jap{\grad}^{\sigma^\prime} u_0^3 \jap{\grad}^{\sigma^\prime} \partial_{z} \Delta^{-1} \partial_i \partial_j \left(u^j_0 u^i_0\right) \mathbf{1}_{i\neq1,j\neq 1} dy dz \\ & \quad + \int \jap{\grad}^{\sigma^\prime} u_0^3 \jap{\grad}^{\sigma^\prime} \mathcal{F}^3 dy dz.  
\end{align*} 
The treatments of the transport nonlinearity and forcing terms are essentially the same as for the improvement of \eqref{ineq:Boot_U01_Low1} and \eqref{ineq:Boot_U01_Low2}.
Let us briefly comment on how to treat the pressure nonlinearity. 
Denote this contribution 
\begin{align*} 
P = \int \jap{\grad}^{\sigma^\prime} u_0^3 \jap{\grad}^{\sigma^\prime} \partial_{z} \Delta^{-1} \partial_i \partial_j \left(u^j_0 u^i_0\right) \mathbf{1}_{i\neq1,j\neq 1}  dy dz. 
\end{align*}  
Notice that if $i = j = 3$ then at least one of the factors of $u_0^3$ must have non-zero $z$ frequency. we integrate by parts and use the algebra property: 
\begin{align} 
P & \lesssim \norm{\grad u_0^3}_{H^{\sigma^\prime}}\left(\norm{u^2_0 u^3_0}_{H^{\sigma^\prime}} + \norm{u^2_0 u^2_0}_{H^{\sigma^\prime}} + \norm{u^3_0}_{H^{\sigma^\prime}}\norm{ \partial_z u^3_0}_{H^{\sigma^\prime}} \right) \nonumber \\  
& \lesssim \norm{\grad u_0^3}_{H^{\sigma^\prime}}\left(\frac{\epsilon^2}{\jap{\nu t}^\alpha} + \epsilon\norm{\grad u_0^3}_{H^{\sigma^\prime}}\right) \nonumber \\ 
& \lesssim \epsilon\norm{\grad u_0^3}_{H^{\sigma^\prime}}^2 + \frac{\epsilon^3}{\jap{\nu t}^{2\alpha}}. \label{ineq:P1U03trick} 
\end{align} 
This suffices for the improvement of \eqref{ineq:Boot_U03_Low} as the first term is absorbed the dissipation and the latter integrates to $O(c_0\epsilon^2)$.   

\subsection{Improvement of  \eqref{ineq:L43} and \eqref{ineq:L41}}

These both follow by a straightforward $L^4$ energy estimate on \eqref{eq:u0i}, together with the 2D Gagliardo-Nirenberg inequality 
\begin{align} 
\norm{f}_4^2 & \lesssim \norm{f}_2 \norm{\grad f^2}_2^{1/2}. \label{ineq:GNS}
\end{align} 
From \eqref{eq:u0i} and computations similar to the above Sobolev scale estimates (using \eqref{ineq:uzAPriori} and \eqref{ineq:Xyzubds} of course), one derives
\begin{align*}
\frac{1}{4}\frac{d}{dt}\norm{u_0^3}_4^4 & \lesssim -\nu \norm{\grad (u_0^3)^2}_2^{2} + \frac{\epsilon^5}{\jap{\nu t}^\alpha} + \frac{\epsilon^5}{\jap{\nu t^3}^\alpha}. 
\end{align*}
We remark that this kind of good decay is only possible due to the strong decay estimates obtained above and the special structure of the nonlinearity. 
By \eqref{ineq:GNS} and \eqref{ineq:uzAPriori} there holds
\begin{align*} 
\frac{d}{dt}\left(\jap{\nu t} \norm{u_0^3}_4^4\right) & \lesssim \nu\norm{u_0^3}_4^4 - \nu \jap{\nu t} \epsilon^{-4} \norm{u_0^3}_4^{8} + \frac{\epsilon^5}{\jap{\nu t}^{\alpha-1}} + \frac{\epsilon^5}{\jap{\nu t^3}^{\alpha-1}}. 
\end{align*}
By comparing $\jap{\nu t}  \norm{u_0^3}_4^4$ against the super solution 
\begin{align*}
Y(t) = K\epsilon^4 + \int_1^t \frac{\epsilon^5}{\jap{\nu t}^{\alpha-1}} + \frac{\epsilon^5}{\jap{\nu t^3}^{\alpha-1}} d\tau,
\end{align*}
for $K$ chosen sufficiently large (independently of $\nu$, $\epsilon$, and $K_B$ of course)
we conclude that 
\begin{align*}
\norm{u_0^3}_4 \lesssim \frac{\epsilon}{\jap{\nu t}^{1/4}}, 
\end{align*}
from which the improvement to \eqref{ineq:L43} follows for $K_L$ sufficiently large by composition (note that since \eqref{ineq:L43} is not used in the bootstrap argument, it is really deduced a posteriori). 
We may improve \eqref{ineq:L41} in essentially the same way except for the lift-up effect term,  which is treated as in \eqref{ineq:U012} above. 
The details are omitted for brevity.

\section*{Acknowledgments}
The authors would like to thank the following people for helpful discussions: Margaret Beck, Steve Childress, Michele Coti Zelati, Bruno Eckhardt, Pierre-Emmanuel Jabin, Susan Friedlander, Yan Guo, Alex Kiselev, Nick Trefethen, Mike Shelley, Vladimir Sverak, Vlad Vicol, and Gene Wayne. 
The authors would like to especially thank Tej Ghoul for encouraging us to focus our attention on finite Reynolds number questions. 
The work of JB was in part supported by NSF Postdoctoral Fellowship in Mathematical Sciences DMS-1103765 and NSF grant DMS-1413177, the work of PG was in part supported by a Sloan fellowship and the NSF grant DMS-1101269, while the work of NM was in part supported by the NSF grant DMS-1211806. 

\appendix

\section{Fourier analysis conventions, elementary inequalities, and Gevrey spaces} \label{apx:Gev}
For $f(x,y,z)$ (or $(X,Y,Z)$), we define the Fourier transform $\hat{f}_k(\eta,l)$ where $(k,\eta,l) \in \Integer \times \Real \times \Integer$ and the inverse Fourier transform via 
\begin{align*} 
\hat{f}_k(\eta,l) & = \frac{1}{(2\pi)^{3/2}}\int_{\Torus \times \Real \times \Torus} e^{-i x k - iy\eta - ilz} f(x,y,z) dx dy dz \\ 
f(x,y,z) & = \frac{1}{(2\pi)^{3/2}}\sum_{k,l \in \Integer} \int_{\Real} e^{i x k + iy\eta + izl} \hat{f}_k(\eta,l) d\eta. 
\end{align*} 
With these conventions note, 
\begin{align*} 
\int f(x,y,z) \overline{g}(x,y,z) dx dy dz & = \sum_{k}\int \hat{f}_k(\eta,l) \overline{\hat{g}}_{k}(\eta,l) d\eta \\ 
\widehat{fg} & = \frac{1}{(2\pi)^{3/2}}\hat{f} \ast \hat{g} \\ 
(\widehat{\grad f})_k(\eta,l) & = (ik,i\eta,il)\hat{f}_k(\eta,l). 
\end{align*}
The paraproducts are defined in \S\ref{sec:paranote} using the Littlewood-Paley dyadic decomposition.  
Here we fix conventions and review the basic properties, see e.g. \cite{BCD11} for more details. 
Let $\psi \in C_0^\infty(\Real_+;\Real_+)$ be such that $\psi(\xi) = 1$ for $\xi \leq 1/2$ and $\psi(\xi) = 0$ for $\xi \geq 3/4$ and define $\rho(\xi) = \psi(\xi/2) - \psi(\xi)$, supported in the range $\xi \in (1/2,3/2)$. 
Then we have the partition of unity for $\xi > 0$, 
\begin{align*} 
1 = \sum_{M \in 2^\Integers} \rho(M^{-1}\xi), 
\end{align*}  
where we mean that the sum runs over the dyadic integers $M = ...,2^{-j},...,1/4,1/2,1,2,4,...,2^{j},...$
 and we define the cut-off $\rho_M(\xi) = \rho(M^{-1}\xi)$, each supported in $M/2 \leq \xi \leq 3M/2$. 
For $f \in L^2(\Torus \times \Real \times \Torus)$ we define
\begin{align*} 
f_{M}  = \rho_M(\abs{\grad})f, \quad\quad\quad f_{< M}  = \sum_{K \in 2^{\Integers}: K < M} f_K, 
\end{align*}
which defines the decomposition (in the $L^2$ sense)  
\begin{align*} 
f = \sum_{M \in 2^\Integers} f_M.  
\end{align*}
There holds the almost orthogonality and the approximate projection property 
\begin{subequations} \label{ineq:LPOrthoProject}
\begin{align} 
\norm{f}^2_2 & \approx \sum_{M \in 2^{\Integers}} \norm{f_M}_2^2 \\
 \norm{f_M}_2 & \approx  \norm{(f_{M})_{\sim M}}_2, 
\end{align}
\end{subequations}
where we make use of the notation 
\begin{align*} 
f_{\sim M} = \sum_{K \in 2^{\Integer}: \frac{1}{C}M \leq K \leq CM} f_{K}, 
\end{align*}
for some constant $C$ which is independent of $M$.
Generally the exact value of $C$ which is being used is not important; what is important is that it is finite and independent of $M$. 
Similar to \eqref{ineq:LPOrthoProject} but more generally, if $f = \sum_{j} D_j$ for any $D_j$ with $\frac{1}{C}2^{j} \subset \textup{supp}\, D_j \subset C2^{j}$ it follows that 
\begin{align} 
\norm{f}^2_2 \approx_C \sum_{j \in \Integers} \norm{D_j}_2^2. \label{ineq:GeneralOrtho}
\end{align}

Next we give some basic inequalities which are useful for working in Gevrey class and related norms. 
The first are versions of Young's inequality. 
\begin{lemma}
Let $f(\xi),g(\xi) \in L_\xi^2(\Real^d)$, $\jap{\xi}^\sigma h(\xi) \in L_\xi^2(\Real^d)$ and $\jap{\xi}^\sigma b(\xi) \in L_\xi^2(\Real^d)$  for $\sigma > d/2$, 
Then we have 
\begin{align} 
\norm{f \ast h}_2 & \lesssim_{\sigma, d} \norm{f}_2\norm{\jap{\cdot}^\sigma h}_2, \label{ineq:L2L1}  \\
\int \abs{f(\xi) (g \ast h)(\xi)} d\xi & \lesssim_{\sigma,d} \norm{f}_2\norm{g}_2\norm{\jap{\cdot}^\sigma h}_2 \label{ineq:L2L2L1} \\ %
\int \abs{f(\xi) (g \ast h \ast b) (\xi)} d\xi & \lesssim_{\sigma,d} \norm{f}_2\norm{g}_2\norm{\jap{\cdot}^\sigma h}_2\norm{\jap{\cdot}^\sigma b}_2. \label{ineq:L2L2L1L1}  
\end{align}
Further iterates are applied for higher order nonlinear terms in Lemma \ref{lem:ParaHighOrder} and are similar to \eqref{ineq:L2L2L1L1} but are omitted here.   
\end{lemma}
The next set of inequalities show that one can often gain on the index of regularity when comparing frequencies which are not too far apart (provided $0 < s < 1$).
This is crucial for doing effective paradifferential calculus in Gevrey regularity. 
\begin{lemma}
Let $0 < s < 1$, $x,y>0$, and $K>1$. 
\begin{itemize} 
\item[(i)] There holds
\begin{align} 
\abs{x^s - y^s} \leq s \max(x^{s-1},y^{s-1})\abs{x-y}. \label{ineq:TrivDiff}
\end{align}
so that if $|x-y|<\frac{x}{K}$,
\begin{align} 
\abs{x^s - y^s} \leq \frac{s}{(K-1)^{1-s}}\abs{x-y}^s. \label{lem:scon}
\end{align} 
Note $\frac{s}{(K-1)^{1-s}} < 1$ as soon as $s^{\frac{1}{1-s}} + 1 < K$. 
\item[(ii)] There holds 
\begin{align} 
\abs{x + y}^s \leq \left(\frac{\max(x,y)}{x+y}\right)^{1-s}\left(x^s + y^s\right), \label{lem:smoretrivial}
\end{align}  
so that, if $\frac{1}{K}y \leq x \leq Ky$,
\begin{align} 
\abs{x + y}^s \leq \left(\frac{K}{1 + K}\right)^{1-s}\left(x^s + y^s\right). \label{lem:strivial}
\end{align} 
\end{itemize}
\end{lemma}
Gevrey and Sobolev regularities can be related with the following two inequalities:  
\begin{itemize}
\item[(i)] For all $x \geq 0$, $\alpha > \beta \geq 0$, $C,\delta > 0$, 
\begin{align} 
e^{Cx^{\beta}} \leq e^{C\left(\frac{C}{\delta}\right)^{\frac{\beta}{\alpha - \beta}}} e^{\delta x^{\alpha}};  \label{ineq:IncExp}
\end{align}
\item[(ii)] For all $x \geq 0$, $\alpha,\sigma,\delta > 0$, 
\begin{align} 
e^{-\delta x^{\alpha}} \lesssim \frac{1}{\delta^{\frac{\sigma}{\alpha}} \jap{x}^{\sigma}}. \label{ineq:SobExp}
\end{align}
\end{itemize}
Together these inequalities show that for $\alpha > \beta \geq 0$, $\norm{f}_{\mathcal{G}^{C,\sigma;\beta}} \lesssim_{\alpha,\beta,C,\delta,\sigma} \norm{f}_{\mathcal{G}^{\delta,0;\alpha}}$. 

The last lemma concerns composition and is somewhat subtle; see e.g. the Appendix of \cite{BM13} for a proof.
Notice the regularity loss associated with composition, which is an issue associated with infinite regularity classes.  

\begin{lemma}[Composition in Gevrey class] \label{lem:GevComp}
Let $s \in (0,1)$, $\lambda > 0$, $\delta > 0$ and $f,g$ be given Gevrey class functions with $\norm{\grad g}_{\infty} \leq c_g < 1$. 
Then, 
\begin{align} 
\norm{f \circ (Id + g) }_{\G^{\lambda}} \lesssim_{c_g,\delta} \norm{f}_{\G^{\lambda + \nu + \delta}},  
\end{align}
with 
\begin{align*}
\nu = K_{cp}\norm{g}_{\G^{\lambda}},  
\end{align*} 
for some constant $K_{cp} = K_{cp}(s)$. 
\end{lemma}  

\section{Definition and analysis of norms} \label{sec:def_nrm}  

\subsection{Definition and analysis of $w$} \label{sec:defw}
There are several constraints we need to satisfy when defining $w$. Not only does the multiplier defined by the $w$ have to be approximate supersolutions of the toy model in \S\ref{sec:Toy}, but it also has to be always getting weaker in time and be locally Lipschitz continuous in time and frequency. 
As mentioned above, the multipliers we use are variants of those used in \cite{BM13,BMV14}, so we draw on these constructions.  

We first begin by defining $\bar{w}(t,\eta)$, which is used to construct $w(t,\eta)$. 
In what follows fix $k,\eta > 0$; we will see that the norms do not depend on the sign of $k$ and $\eta$. 
Further, recall the definitions in \S\ref{sec:Notation}. 
The multiplier is built backwards in time, which makes resonance counting easier. 
Let $t \in I_{k,\eta}$. Let $\bar{w}(t,\eta)$ be a non-decreasing function of time with $\bar{w}(t,\eta) = 1  $ for $t \geq  2\eta $. 
For $ k \geq 1$, we assume that $\bar{w}(t_{k-1,\eta})  $ was computed.  
To compute $\bar{w}$ on the interval $I_{k,\eta} $, we use the behavior predicted by the toy model in \eqref{def:stablesuper}.  
For a parameter $\kappa > 2$ fixed sufficiently large depending on a universal constant determined by the proof,
for $k=1,2,3,..., E(\sqrt{\eta}) $, we define
\begin{subequations} \label{def:wNR}
\begin{align}
\bar{w}(t,\eta) &=   \Big( \frac{k^2}{\eta}   \left[ 1 + b_{k,\eta} |t-\frac{\eta}k | \right]   \Big)^{\kappa}  \bar{w} (t_{k-1,\eta}),  \quad& 
  \quad  \forall t \in  I^R_{k,\eta} =  \left[ \frac{\eta}k ,t_{k-1,\eta}  \right], \\ 
\bar{w}(t,\eta) &=   \Big(1 + a_{k,\eta} |t-\frac{\eta}k |   \Big)^{-1-\kappa}  \bar{w} \left(\frac{\eta}k\right),  \quad& 
  \quad \forall  t \in  I^L_{k,\eta} =  \left[ t_{k,\eta}  , \frac{\eta}k   \right].  
\end{align} 
\end{subequations}
The constant $b_{k,\eta}  $   is chosen to ensure that $ \frac{k^2}{\eta}   \left[ 1 + b_{k,\eta} |t_{k-1,\eta}-\frac{\eta}k | \right]  =1$, hence for $k \geq2$, we have 
\begin{align} \label{bk} 
 b_{k,\eta} = \frac{2(k-1)}{k} \left(1 - \frac{k^2}{\eta} \right)
\end{align} 
and $b_{1,\eta} = 1 - 1/\eta$. 
Similarly,  $a_{k,\eta}$ is chosen to ensure $ \frac{k^2}{\eta}\left[ 1 + a_{k,\eta} |t_{k,\eta}-\frac{\eta}k | \right] = 1$, which implies
\begin{align} \label{ak} 
 a_{k,\eta} = \frac{2(k+1)}{k} \left(1 - \frac{k^2}{\eta} \right). 
\end{align} 
Hence, we have $ \bar{w}(\frac{\eta}k) = \bar{w} (t_{k-1,\eta})  \Big( \frac{k^2}{\eta} \Big)^{\kappa}$  
and  $\bar{w} ( t_{k,\eta} ) = \bar{w} (t_{k-1,\eta})  \Big( \frac{k^2}{\eta} \Big)^{1+ 2\kappa}$. 
For earlier times $[0, t_{E(\sqrt{\eta}),\eta }] $, we take $\bar{w}$ to be constant. 
Next, we will impose additional losses in time on $\bar{w}$:  
\begin{align}
w(t,\eta) = \bar{w}(t,\eta) \exp\left[-\kappa \int_{t}^\infty \mathbf{1}_{\tau \leq 2\sqrt{\eta}} d\tau  - \kappa \int_{t}^\infty \mathbf{1}_{\sqrt{\abs{\eta}} \leq \tau \leq 2\abs{\eta}} \frac{\abs{\eta}}{\tau^2} d\tau \right]. \label{def:wextraloss} 
\end{align}
The following lemma is essentially Lemma 3.1 in \cite{BM13} and  shows that $w(t,\eta)^{-1}$ loses some fixed radius of Gevrey-2 regularity. The proof is omitted for brevity.  
\begin{lemma} \label{lem:totalGrowthw} 
There is a constant $\mu$ (depending on $\kappa$) and a constant $p > 0$ such that for all $\abs{\eta} > 1$, we have 
\begin{align*} 
\frac{1}{w(t,\eta)} \leq \frac{1}{w(1,\eta)} & \sim \eta^{-p} e^{\frac{\mu}{2} \sqrt{\eta} },  
\end{align*} 
where `$\sim$' is in the sense of asymptotic expansion (up to a multiplicative constant) as $\eta \rightarrow \infty$. 
\end{lemma} 

The following lemma is from \cite{BM13}, and shows how to use the well-separation of critical times. 
\begin{lemma} \label{lem:wellsep}
Let $\xi,\eta$ be such that there exists some $K \geq 1$ with $\frac{1}{K}\abs{\xi} \leq \abs{\eta} \leq K\abs{\xi}$ and let $k,n$ be such that $t \in I_{k,\eta}$ and $t \in I_{n,\xi}$  (note that $k \approx n$).  
Then at least one of following holds:
\begin{itemize} 
\item[(a)] $k = n$ (almost same interval); 
\item[(b)] $\abs{t - \frac{\eta}{k}} \geq \frac{1}{10 K}\frac{\abs{\eta}}{k^2}$ and $\abs{t - \frac{\xi}{n}} \geq \frac{1}{10 K}\frac{\abs{\xi}}{n^2}$ (far from resonance);
\item[(c)] $\abs{\eta - \xi} \gtrsim_K \frac{\abs{\eta}}{\abs{n}}$ (well-separated). 
\end{itemize}
\end{lemma}

The next lemma tells us how to take advantage of the time derivative of $w$ and hence the $CK_w$ terms. 

\begin{lemma}[Time derivatives near the critical times] \label{lem:dtw}
If $1 \lesssim t \leq 2\sqrt{\eta}$, then there holds 
\begin{align} 
\frac{\partial_t w(t,\eta)}{w(t,\eta)} & \approx \kappa.  
\end{align}
If we instead have $t \in \I_{k,\eta}$ for some $k$, then the following holds
\begin{align} 
\frac{\partial_t w(t,\eta)}{w(t,\eta)} & \approx \frac{\kappa}{1 + \abs{\frac{\eta}{k} - t}} + \frac{\kappa \abs{\eta}}{t^2} \label{dtw}
\end{align}
\end{lemma} 

The next lemma is a variant of [Lemma 3.4, \cite{BM13}]. 
It is important for estimating nonlinear terms where we need to be able to compare $CK_w$ multipliers of different frequencies. 
The proof is essentially verbatim from \cite{BM13} so it is omitted for the sake of brevity. 

\begin{lemma} \label{lem:WtFreqCompare}
\begin{itemize}
\item[(i)] For $t \gtrsim 1$, and $\eta,\xi$ such that $t < 2  \min( \abs{\xi},  \abs{\eta})    $,
\begin{align} \label{dtw-xi} 
 \frac{\partial_t w(t,\eta)}{w(t,\eta)}\frac{w(t,\xi)}{\partial_t w(t,\xi)}
 \lesssim \jap{\eta - \xi}^2
\end{align}
\item[(ii)] For all $t \gtrsim 1$, and $\eta,\xi$ such that for some $K \geq 1$, $\frac{1}{K}\abs{\xi} \leq \abs{\eta} \leq K\abs{\xi}$,  
\begin{align}
\sqrt{\frac{\partial_t w(t,\xi)}{w(t,\xi)}} \lesssim_K \left[\sqrt{\frac{\partial_t w(t,\eta)}{w(t,\eta)}} + \frac{\abs{\eta}^{s/2}}{\jap{t}^{s}}\right]\jap{\eta-\xi}^2. \label{ineq:partialtw_endpt}  
\end{align}
\end{itemize}
\end{lemma}

The next lemma is an easy variant of the analogous lemma in \cite{BM13}. 

\begin{lemma}[Ratio estimates for nonlinear interactions] \label{lem:wRat}
There exists a $K > 0$ such that for all $\eta,\xi$, 
\begin{align*}
\frac{w(t,\eta)}{w(t,\xi)} & \lesssim e^{K\abs{\eta-\xi}^{1/2}}. 
\end{align*} 
\end{lemma} 
\begin{proof} 
The proof that $\frac{\bar w(t,\eta)}{\bar w(t,\xi)} \lesssim e^{C|\eta-\xi|^{1/2}}$ for some $C>0$ is non-trivial but follows mutatis mutandis the proof of Lemma 3.5 in~\cite{BM13}.
What remains is  
\begin{align*} 
W_1(t,\eta) =  \exp\left[-\kappa \int_{t}^\infty \mathbf{1}_{\tau \leq 2\sqrt{|\eta|}} d\tau  - \kappa \int_{t}^\infty \mathbf{1}_{\sqrt{\abs{\eta}} \leq \tau \leq 2\abs{\eta}} \frac{\abs{\eta}}{\tau^2} d\tau \right],  
\end{align*} 
for which it is straightforward to show 
$\frac{W_1(t,\eta)}{W_1(t,\xi)} \lesssim e^{C|\eta-\xi|^{1/2}}$.  
\end{proof}

\subsection{The design and analysis of $w_L$} \label{sec:Nmult}
The multiplier $w_L$ is introduced to deal with the linear pressure term in the $Q^3$ equation. 
We define $w_L$ such that it solves the following (we may without loss of generality use the same parameter $\kappa$) 
\begin{subequations}  \label{def:wL}
\begin{align} 
\partial_tw_L(t,k,\eta,l) & = \kappa \frac{\abs{k}\jap{l} }{k^2 + l^2 + \abs{\eta - kt}^2} w_L(t,k,\eta,l) \quad\quad t \geq 1 \\ 
w_L(1,k,\eta,l) & = 1. 
\end{align}
\end{subequations} 
Since the following holds uniformly in $k,l,\eta$: 
\begin{align} 
\int_0^\infty \frac{\abs{k} \jap{l}}{k^2 + l^2 + \abs{\eta - kt}^2} dt \lesssim 1, \label{ineq:unifN}
\end{align}
the multiplier $w_L$ is $O(1)$ and hence will have very little effect on most estimates. 
To see \eqref{ineq:unifN}: 
\begin{align*} 
\int_0^\infty \frac{\abs{k} \jap{l} }{k^2 + l^2 + \abs{\eta - kt}^2} dt &  \lesssim  \frac{\abs{k} \jap{l}}{k^2 + l^2} \int_0^{\infty} \frac{1}{1 + \frac{k^2}{k^2 + l^2}\abs{\frac{\eta}{k} - t}^2} dt  \lesssim \frac{\abs{k} \jap{l}}{k^2 + l^2} \sqrt{\frac{k^2 + l^2}{k^2}} \int_{-\infty}^\infty \frac{1}{1 + s^2} ds \lesssim 1.  
\end{align*}  

\section{Elliptic estimates} \label{sec:Elliptic} 
In this section, we group and discuss all of the necessary ``elliptic'' estimates on $\Delta_t^{-1}$. 
While $\Delta_t$ itself is not elliptic, it is the representation of $\Delta$ in these new coordinates
and it is precisely the ellipticity of $\Delta$ that is the origin of inviscid damping. 

\subsection{Lossy estimates} \label{sec:Lossy}
First, there is the equivalent of the lossy elliptic lemma from \cite{BM13}.
The proof is analogous and is hence omitted here. 
\begin{lemma}[Lossy elliptic lemma] \label{lem:LossyElliptic}
If $C$ satisfies the bootstrap assumptions~\eqref{ineq:Boot_CgHi}, then for $c_0$ sufficiently small, for any function~$\phi$,  and $a \leq \sigma$,
\begin{align*} 
\norm{\Delta^{-1}_t \phi_{\neq}}_{\G^{\lambda,a-2}} \lesssim \frac{1}{\jap{t}^2}\norm{\phi_{\neq}}_{\G^{\lambda,a}} . 
\end{align*} 
\end{lemma}  

We also need a version of the enhanced dissipation lossy elliptic lemma from \cite{BMV14}: 
\begin{lemma}[Lossy elliptic lemma II] \label{lem:AnuLossy}
If $C$ satisfies the bootstrap assumptions~\eqref{ineq:Boot_CgHi}, then for $c_0$ sufficiently small, for any function $\phi$, and $\gamma^\prime = \beta + 3\alpha + 5$,  
\begin{subequations} \label{ineq:AnuLossyII} 
\begin{align} 
\norm{ A^{\nu;i} \Delta^{-1}_t \phi}_{2} + \norm{\partial_X A^{\nu;i} \Delta^{-1}_t \phi}_{2} & \lesssim \frac{1}{\jap{t}^2}\left(\norm{A^{\nu;i} \phi}_2 + \jap{t}^{-3}\norm{\phi_{\neq}}_{\G^{\lambda,\gamma^\prime}} \right) \\ 
\norm{ \partial_Z A^{\nu;i} \Delta^{-1}_t \phi}_{2} + \norm{ (\partial_Y - t \partial_X) A^{\nu;i} \Delta^{-1}_t \phi}_{2} & \lesssim \frac{1}{\jap{t}} \left(\norm{A^{\nu;i} \phi}_2 + \jap{t}^{-3}\norm{\phi_{\neq}}_{\G^{\lambda,\gamma^\prime}} \right) \\ 
\norm{\partial_{m}^t \partial_n^t A^{\nu;i} \Delta_t^{-1} \phi}_2 & \lesssim \frac{1}{\jap{t}^{b}}\left(\norm{A^{\nu;i} \phi}_2 + \jap{t}^{-3}\norm{\phi_{\neq}}_{\G^{\lambda,\gamma^\prime}} \right), 
\end{align} 
\end{subequations}
where $b = 0$ if $n,m \neq 1$, $b = 1$ if exactly one of $m$ or $n$ equals one, and $b = 2$ if $m = n = 1$. 
Moreover, 
\begin{subequations} \label{ineq:AnuLossyED} 
\begin{align} 
\norm{ A^{\nu;i} \Delta^{-1}_t \phi}_{2} + \norm{\partial_X A^{\nu;i} \Delta^{-1}_t \phi}_{2} & \lesssim \frac{1}{\jap{t}^3}\left(\norm{ \sqrt{-\Delta_L}A^{\nu;i} \phi}_2 + \jap{t}^{-3}\norm{\phi_{\neq}}_{\G^{\lambda,\gamma^\prime}} \right) \\ 
\norm{ \partial_Z A^{\nu;i} \Delta^{-1}_t \phi}_{2} + \norm{ (\partial_Y - t \partial_X) A^{\nu;i} \Delta^{-1}_t \phi}_{2} & \lesssim \frac{1}{\jap{t}^2} \left(\norm{\sqrt{-\Delta_L} A^{\nu;i} \phi}_2 + \jap{t}^{-3}\norm{\phi_{\neq}}_{\G^{\lambda,\gamma^\prime}} \right) \\ 
\norm{\partial_{m}^t \partial_n^t A^{\nu;i} \Delta_t^{-1} \phi}_2 & \lesssim \frac{1}{\jap{t}^{1+b}}\left(\norm{\sqrt{-\Delta_L} A^{\nu;i} \phi}_2 + \jap{t}^{-3}\norm{\phi_{\neq}}_{\G^{\lambda,\gamma^\prime}} \right). 
\end{align} 
\end{subequations}
Finally, we have 
\begin{align} 
\norm{A^{\nu;i} \Delta_L \Delta^{-1}_t \phi}_2 & \lesssim \norm{A^{\nu;i} \phi}_2. \label{ineq:PEL_CKnuIII} 
\end{align}  
\end{lemma}  
\begin{proof}
The proof is similar to the corresponding lossy elliptic lemma in \cite{BMV14}. 
For all $i$ the proof is basically the same, so concentrate on $i = 3$. Moreover, the proofs of all the inequalities are essentially the same, so we consider just the $\partial_Z$ estimate in \eqref{ineq:AnuLossyII}. 

First, by \eqref{ineq:AnuHiLowSep}, 
\begin{align*} 
\norm{\partial_Z A^{\nu;3} \Delta_t^{-1} \phi}_2 & \lesssim \frac{1}{\jap{t}^5}\norm{\Delta_L \Delta_t^{-1}\phi}_{\G^{\lambda,\beta+3\alpha+4}} + \frac{1}{\jap{t}}\norm{A^{\nu;3} \Delta_L \Delta_t^{-1}\phi}_2. 
\end{align*}
Control of the first term follows from Lemma \ref{lem:LossyElliptic}, so turn to the second term. 
First, 
\begin{align}
\Delta_L \Delta_t^{-1} \phi & = \phi - G(\partial_Y - t\partial_X)^2 \Delta_t^{-1} \phi - 2\psi_z (\partial_Y - t\partial_X) \partial_Z \Delta_t^{-1} \phi - \Delta_t C(\partial_Y - t\partial_X) \Delta_t^{-1}\phi  \nonumber \\ 
& = \phi + \mathcal{E}_1 + \mathcal{E}_2 + \mathcal{E}_3  \label{def:Deltinv_AnuLoss} 
\end{align} 
Applying $A^{\nu;3}$ to both sides yields 
\begin{align}
\norm{A^{\nu;3}\Delta_L \Delta_t^{-1} \phi}_2 & \lesssim \norm{A^{\nu;3}\phi}_2 + \sum_{j=1}^3 \norm{A^{\nu;3}\mathcal{E}_j}_2\label{ineq:trivDeltainv}
\end{align}
Then, by \eqref{ineq:AnuiDistri}, and Lemma \ref{lem:CoefCtrl}, 
\begin{align*}
\sum_{j = 1}^3\norm{A^{\nu;3}\mathcal{E}_j}_2 \lesssim \left(\norm{G}_{\G^{\lambda,\beta+3\alpha + 3}} + \norm{\psi_z}_{\G^{\lambda,\beta+3\alpha + 3}} + \norm{\Delta_t C}_{\G^{\lambda,\beta+3\alpha + 3}}\right)\norm{A^{\nu;3} \Delta_L \Delta_t^{-1}\phi}_2 & \lesssim c_0\norm{A^{\nu;3} \Delta_L \Delta_t^{-1}\phi}_2,
\end{align*}
which for $c_0$ sufficiently small can be then absorbed on the LHS of \eqref{ineq:trivDeltainv}. 
This completes the treatment of \eqref{ineq:AnuLossyII}. 

Consider \eqref{ineq:AnuLossyED}. 
Note by \eqref{ineq:AnuHiLowSep2} that, 
\begin{align*} 
\norm{\partial_Z A^{\nu;3} \Delta_t^{-1} \phi}_2 & \lesssim \frac{1}{\jap{t}^6}\norm{\Delta_L \Delta_t^{-1}\phi}_{\G^{\lambda,\beta + 3\alpha + 5}} + \frac{1}{\jap{t}^2}\norm{\sqrt{-\Delta_L} A^{\nu;3} \Delta_L \Delta_t^{-1}\phi}_2.   
\end{align*}
Then we consider \eqref{def:Deltinv_AnuLoss} and apply $\sqrt{-\Delta_L}A^{\nu;3}$ to both sides. 
The result follows as in \eqref{ineq:AnuLossyII} above, except also applying \eqref{ineq:TriTriv} to move the $\sqrt{-\Delta_L}$ past the coefficients. We omit the details for brevity. 

Finally, the proof of \eqref{ineq:PEL_CKnuIII} is straightforward by the above techniques and is hence omitted. 
\end{proof}

Next we consider the interaction of $A^{\nu;i}$ with the $w_L$ multiplier.

\begin{lemma}[$CK^\nu_{wL}$ elliptic lemma] \label{lem:AnuLossy_CKnu}
Under the bootstrap hypotheses, for $c_0$ sufficiently small we have for any function~$\phi$, 
\begin{align}
\norm{\sqrt{\frac{\partial_t w_L}{w_L}} A^{\nu;i} \Delta_L \Delta^{-1}_t \phi}_{2} & \lesssim \norm{\sqrt{\frac{\partial_t w_L}{w_L}} A^{\nu;i} \phi}_2. \label{ineq:PEL_CKnuII} 
\end{align}   
\end{lemma} 
\begin{proof} 
The proof of \eqref{ineq:PEL_CKnuII} follows similar to Lemma \ref{lem:AnuLossy} by applying
\begin{align*} 
\frac{\abs{k}^{1/2} \jap{l}^{1/2} }{\abs{k,l,\eta-kt}} & = \frac{\jap{l}^{1/2} \abs{k,l^\prime,\xi-kt}}{\jap{l^\prime}^{1/2} \abs{k,l,\eta-kt}} \frac{\abs{k}^{1/2} \jap{l^\prime}^{1/2} }{\abs{k,l^\prime,\xi-kt}} \\ 
& \lesssim \jap{l-l^\prime, \eta-\xi}^{3/2} \frac{\abs{k}^{1/2} \jap{l^\prime}^{1/2} }{\abs{k,l^\prime,\xi-kt}}.  
\end{align*} 
\end{proof}

\subsection{Precision lemmas} \label{sec:PEL}
All of the so-called `precision elliptic lemmas' (PEL) are variations on the common theme of using $\Delta_L^{-1}$ as an approximate inverse.  

The first PEL puts $U_0^i$ in the high norm, producing low frequency error terms. 
\begin{lemma}[Zero mode PEL] \label{lem:PELbasicZero}
If $C$ satisfies the bootstrap assumptions~\eqref{ineq:Boot_CgHi}, then for $c_0$ sufficiently small, 
\begin{subequations} \label{ineq:AU0PEL} 
\begin{align} 
\jap{t}^{-2}\norm{A U_0^1}_2^2 & \lesssim \norm{A^1 Q^1_0}_2^2 +\jap{t}^{-2} \norm{U_0^1}_2^2 + \epsilon^2\norm{AC}_2^2 \\ 
\norm{A U_0^2}_2^2 & \lesssim \norm{A^2 Q^2_0}_2^2 + \norm{U_0^2}_2^2 + \frac{\epsilon^2}{\jap{\nu t}^{2\alpha}} \norm{AC}_2^2 \\
\norm{A U_0^3}_2^2 & \lesssim \norm{A^3 Q^3_0}_2^2 + \norm{U_0^3}_2^2 + \epsilon^2 \norm{AC}_2^2. \label{ineq:AU03PEL} 
\end{align} 
\end{subequations} 
Moreover we have
\begin{subequations} \label{ineq:gradAU0i}  
\begin{align} 
\jap{t}^{-2}\norm{\grad A U_0^1}_2^2 & \lesssim \norm{\grad A^1 Q^1_0}_2^2 +\jap{t}^{-2} \norm{\grad U_0^1}_2^2 + \epsilon^2 \norm{\nabla AC}_2^2 \\ 
\norm{\grad A U_0^2}_2^2 & \lesssim \norm{\grad A^2 Q_0^2}_2^2 + \norm{\grad U_0^2}_2^2 + \frac{\epsilon^2}{\jap{\nu t}^{2\alpha}}\norm{\grad AC}_2^2. \label{ineq:gradAU02_PEL} \\
\norm{\grad A U_0^3}_2^2 & \lesssim \norm{\grad A^3 Q^3_0}_2^2 + \norm{\grad U_0^3}_2^2 + \epsilon^2 \norm{\nabla AC}_2^2. 
\end{align}
\end{subequations} 
\end{lemma}
\begin{proof}
The proofs of \eqref{ineq:AU0PEL} are all the same so we only focus on \eqref{ineq:AU03PEL} for simplicity.   
We will separate low and high frequencies:
\begin{align*} 
\norm{A U_0^3}_2^2 & = \norm{ \left(A U_0^3\right)_{\leq 1}}_2^2 + \norm{ \left(AU_0^3\right)_{>1}}_2^2.
\end{align*} 
For the low frequency term we have
\begin{align*} 
\norm{\left(A  U_0^3\right)_{\leq 1}}_2^2 & \lesssim \norm{U_0^3}_2^2. 
\end{align*} 
For the high frequencies we write: 
\begin{align} 
\left(\Delta \Delta_t^{-1} Q_0^3\right)_{> 1} & = (Q_0^3)_{> 1} - \left( (G\partial_{YY} + 2\psi_z\partial_{ZY})U^3_0\right)_{>1} - \left(\Delta_t C \partial_Y U^3_0\right)_{>1} \nonumber \\ 
& = Q_0^3 + \mathcal{E}_1 + \mathcal{E}_2, \label{ineq:0ModeApproxInv} 
\end{align}  
and denote the multiplier
\begin{align*} 
\mathcal{M} = A \Delta^{-1}.    
\end{align*}
We apply this multiplier to both sides of \eqref{ineq:0ModeApproxInv} and deduce 
\begin{align}
\norm{\mathcal{M}\left(\Delta \Delta_t^{-1}Q_0^3\right)_{> 1}}_2 & \lesssim \norm{\mathcal{M} Q_0^3}_2 + \sum_{j =1}^2 \norm{\mathcal{M}\mathcal{E}_j}_2. \label{ineq:MDelDeltQ}
\end{align}
The error terms will be either absorbed on the LHS or will yield the contribution from the coefficients in \eqref{ineq:AU0PEL}. 
Consider $\mathcal{E}_1$ first. We expand with a paraproduct 
\begin{align*} 
\mathcal{M}\mathcal{E}_1 & = \mathcal{M}\left((G)_{Hi} \partial_{YY} + 2(\psi_z)_{Hi}\partial_{ZY})(U^3_0)_{Lo}\right)_{>1} + \mathcal{M}\left((G)_{Lo} \partial_{YY} + 2(\psi_z)_{Lo}\partial_{ZY})(U^3_0)_{Hi}\right)_{>1} + \mathcal{M}\mathcal{E}_{1;\mathcal{R}} \\ 
& = \mathcal{M}\mathcal{E}_{1;C} + \mathcal{M}\mathcal{E}_{1;U} + \mathcal{M}\mathcal{E}_{1;\mathcal{R}}. 
\end{align*} 
When the coefficient is in high frequency, we have from \eqref{ineq:AprioriU0}, 
\begin{align} 
\mathcal{M}(\eta,l)\mathcal{E}_{1;C} & \lesssim \epsilon \int \frac{1}{\abs{\eta,l}^{2}} A(\eta,l) \mathbf{1}_{\abs{\eta,l} > 1} \left(\widehat{G}(\xi,l^\prime)_{Hi} + \widehat{\psi_z}(\xi,l^\prime)_{Hi}\right) Low(\eta-\xi,l-l^\prime) d\xi. \label{ineq:GhiCoef}
\end{align} 
(it is only the size of the low frequency factor that changes between the different cases in \eqref{ineq:AU0PEL}). 
Therefore, by Lemma \ref{lem:ABasic} and  \eqref{ineq:quadHL}, 
\begin{align} 
\norm{\mathcal{M}(\eta,l)\mathcal{E}_{1;C}}_2^2 & \lesssim \epsilon^2 \norm{A C}_2^2,  
\end{align} 
which suffices for $\mathcal{E}_{1;C}$. 
For $\mathcal{E}_{1;U}$ we use a slightly different treatment: 
\begin{align*} 
\mathcal{M}(\eta,l)\mathcal{E}_{1;U} & \lesssim c_0\int \frac{1}{\abs{\eta,l}^{2}} A(\eta,l) \mathbf{1}_{\abs{\eta,l} > 1}  \abs{\widehat{\jap{\grad}^2U_0^3}(\xi,l^\prime)_{Hi}} Low(\eta-\xi,l-l^\prime) d\xi  \\ 
& \lesssim c_0\int \frac{1}{\abs{\eta,l}^{2}} A(\eta,l) \mathbf{1}_{\abs{\eta,l} > 1}\left[\mathbf{1}_{\abs{\xi,l^\prime} > 1} + \mathbf{1}_{\abs{\xi,l^\prime} \leq 1}\right]  \abs{\widehat{\jap{\grad}^2 U_0^3}(\xi,l^\prime)_{Hi}} Low(\eta-\xi,l-l^\prime) d\xi \\ 
& = \mathcal{M}(\eta,l)\mathcal{E}_{1;U}^H + \mathcal{M}(\eta,l)\mathcal{E}_{1;U}^L,  
\end{align*} 
From \eqref{ineq:quadHL},
\begin{align*} 
\norm{\mathcal{M}(\eta,l)\mathcal{E}_{1;U}^H}_2^2 \lesssim c_{0}^2\norm{\mathcal{M}\left(\Delta U_0^3\right)_{>1}}_2^2,  
\end{align*} 
which will be absorbed on the LHS of the estimate \eqref{ineq:MDelDeltQ}. Similarly, 
\begin{align*} 
\norm{\mathcal{M}(\eta,l)\mathcal{E}_{1;U}^L}_2^2 \lesssim c_{0}^2\norm{U_0^3}_2^2,  
\end{align*} 
which appears on the RHS of \eqref{ineq:AU03PEL}. 
Next turn to the remainder term $\mathcal{E}_{1;\mathcal{R}}$, which by applying \eqref{ineq:AARemainderBasic} and \eqref{ineq:quadR},
\begin{align*} 
\norm{\mathcal{M}(\eta,l)\mathcal{E}_{1;\mathcal{R}}}_2^2 & \lesssim c_{0}^2\norm{\mathcal{M}(\Delta U_0^3)_{>1}}_2^2 + c_{0}^2\norm{U_0^3}_2^2.  
\end{align*}  
Hence, everything on the RHS either also appears on the RHS of \eqref{ineq:AU03PEL} or is absorbed on the LHS of \eqref{ineq:MDelDeltQ} for $c_0$ sufficiently small. 
This completes the treatment of $\mathcal{E}_1$. The treatment of $\mathcal{E}_2$ is very similar, note that we have an extra derivative in \eqref{ineq:GhiCoef}. 
Hence, this treatment is omitted for brevity. 
This completes the proof of \eqref{ineq:AU03PEL}; the remaining inequalities in \eqref{ineq:AU0PEL} are the same and the proofs are omitted.  

The proof of \eqref{ineq:gradAU0i} is essentially the same as \eqref{ineq:AU0PEL} after noting that the coefficients $\psi_z$ and $\psi_y$ always have factors of $\grad C$ (see Lemma \ref{lem:CoefCtrl}).
\end{proof}

The next PEL is the one most analogous to Proposition 2.4 in \cite{BM13}.  

\begin{lemma}[$CK$ PEL] \label{lem:PEL_NLP120neq}
Let $\phi$ be given such that $\norm{\phi}_{\G^{\lambda}}\lesssim \epsilon \jap{t}^{b} \jap{\nu t^3}^{-a}$ for some $a \geq 0$ and $b \geq 0$. 
Then for $c_0$ and $\epsilon$ sufficiently small, and provided that $C$ satisfies the bootstrap assumptions~\eqref{ineq:Boot_CgHi},
\begin{subequations} \label{ineq:PEL_NLP120neq}
\begin{align} 
\norm{\left(\sqrt{\frac{\partial_t w}{w}} + \frac{\abs{\grad}^{s/2}}{\jap{t}^s}\right) A^{i} \Delta_L \Delta_t^{-1} \phi_{\neq}}_2^2 
  & \lesssim \norm{\left(\sqrt{\frac{\partial_t w}{w}} + \frac{\abs{\grad}^{s/2}}{\jap{t}^s}\right)A^{i} \phi_{\neq}}_2^2 \nonumber \\ 
 & \quad + \frac{\epsilon^2 \jap{t}^{2b-2} }{\jap{\nu t^3}^{2a}} \norm{\left(\sqrt{\frac{\partial_t w}{w}} + \frac{\abs{\grad}^{s/2}}{\jap{t}^s} \right)A C}_2^2.   \label{ineq:PEL_NLP120neq1} 
\end{align} 
\end{subequations}
\end{lemma}
\begin{proof} 
We will prove the result in the case $i = 2$; both of the other cases are very similar. 
Write $P = \Delta_t^{-1}\phi_{\neq}$ 
\begin{align} 
\Delta_L P & = \phi_{\neq} - G(\partial_Y - t\partial_X)^2 P - 2\psi_z(\partial_Y - t\partial_X) \partial_Z P  - \Delta_tC (\partial_Y - t\partial_X) P \nonumber \\ 
& = \phi_{\neq} + \mathcal{E}_1 + \mathcal{E}_2 + \mathcal{E}_3. \label{def:PkPEL}
\end{align} 
Analogous to the proof of [Proposition 2.4 \cite{BM13}], we will apply the multiplier
\begin{align*}
\left(\sqrt{\frac{\partial_t w}{w}} + \frac{\abs{\grad}^{s/2}}{\jap{t}^s}\right)A^{2} = \mathcal{M}
\end{align*} 
to both sides of \eqref{def:PkPEL} and estimate the terms on the RHS. 
Hence we get
\begin{align} 
\norm{\mathcal{M}\Delta_L P}_2^2 \lesssim \norm{\mathcal{M}\phi_{\neq}}_2^2 + \sum_{i = 1}^3 \norm{\mathcal{M}\mathcal{E}_i}_2^2. \label{ineq:DeltaPk} 
\end{align}
The leading order term on the LHS of \eqref{ineq:DeltaPk} appears on the RHS of \eqref{ineq:PEL_NLP120neq} so it remains to control the error terms.  

Turn to the first error term and expand with a paraproduct
\begin{align*} 
\mathcal{M}\mathcal{E}_1 & = \mathcal{M}G_{Hi}(\partial_Y - t\partial_X)^2 P_{Lo} + \mathcal{M}G_{Lo}(\partial_Y - t\partial_X)^2 P_{Hi} + \mathcal{E}_{1;\mathcal{R}} \\ 
& = \mathcal{E}_{1;C} + \mathcal{E}_{1;P} + \mathcal{E}_{1;\mathcal{R}}. 
\end{align*}
By \eqref{ineq:ABasic}, \eqref{ineq:dtwBasic}, \eqref{ineq:quadHL}, and Lemma \ref{lem:CoefCtrl} it follows that
\begin{align*} 
\norm{\mathcal{M}\mathcal{E}_{1;P}}_2^2 & \lesssim c_{0}^2\norm{\mathcal{M}\Delta_L P}_2^2, 
\end{align*} 
which can hence be absorbed on the LHS of \eqref{ineq:DeltaPk} by choosing $c_{0}$ sufficiently small.
Via \eqref{ineq:ARemainderBasic}, \eqref{ineq:dtwBasic}, \eqref{ineq:quadR}, and Lemma \ref{lem:CoefCtrl} we also have
\begin{align*} 
\norm{\mathcal{M}\mathcal{E}_{1;\mathcal{R}}}_2^2 & \lesssim c_{0}^2\norm{\mathcal{M}\Delta_L P}_2^2,
\end{align*} 
which, as above, is absorbed on the LHS of \eqref{ineq:DeltaPk} by choosing $c_{0}$ sufficiently small.

Consider next $\mathcal{E}_{1;C}$ for which, by the hypotheses and Lemma \ref{lem:LossyElliptic}, we have
\begin{align*} 
\mathcal{M}\mathcal{E}_{1;C} & \lesssim \frac{\epsilon \jap{t}^{b}}{\jap{\nu t^3}^{a}}\sum_{l} \int_\xi \mathcal{M}(k,\eta,l) \abs{\widehat{G}(\xi,l^\prime)_{Hi}} Low(k,\eta-\xi,l-l^\prime) d\xi; 
\end{align*}  
the extra $\jap{t}^2$ from $(\partial_Y - t\partial_X)^2$ was canceled by the $\Delta_t^{-1}$ in the definition of $P$ and Lemma \ref{lem:LossyElliptic}.  

From Lemma \ref{lem:ABasic} and \eqref{ineq:dtwBasic} it follows that 
\begin{align*} 
|\mathcal{M}\mathcal{E}_{1;C}| & \lesssim \frac{\epsilon \jap{t}^{b}}{\jap{\nu t^3}^{a}}\sum_{l} \int_\xi \left(\sqrt{\frac{\partial_t w(\xi)}{w(\xi)}} + \frac{\abs{\xi,l^\prime}^{s/2}}{\jap{t}^s} \right) 
\frac{1}{\jap{\xi,l^\prime}^2 \jap{\frac{t}{\jap{\xi,l^\prime}}}} A(\xi,l^\prime)\left| \widehat{G}(\xi,l^\prime)_{Hi} \right| Low(\eta-\xi,l-l^\prime) d\xi. 
\end{align*}  
Therefore, by \eqref{ineq:quadHL} and Lemma \ref{lem:CoefCtrl},
\begin{align*} 
\norm{\mathcal{M}\mathcal{E}_{1;C}}_2^2 & \lesssim \frac{\epsilon^2 \jap{t}^{2b-2}}{\jap{\nu t^3}^{2a}} \norm{\left(\sqrt{\frac{\partial_t w}{w}}  + \frac{\abs{\grad}^{s/2}}{\jap{t}^s} \right) A \jap{\grad}^{-1} G}_2^2 \\ 
& \lesssim \frac{\epsilon^2 \jap{t}^{2b-2}}{\jap{\nu t^3}^{2a}} \norm{\left(\sqrt{\frac{\partial_t w}{w}} + \frac{\abs{\grad}^{s/2}}{\jap{t}^s} \right) AC}_2^2.
\end{align*}
As this contribution appears on the RHS of \eqref{ineq:PEL_NLP120neq}, we are done with this term. 

The error terms $\mathcal{E}_{2}$ and $\mathcal{E}_{3}$ can also be separated using the paraproduct decomposition. With obvious notations, $\mathcal{E}_{2,P}$, $\mathcal{E}_{3,P}$, $\mathcal{E}_{2,\mathcal R}$, $\mathcal{E}_{3,\mathcal R}$ can be treated similarly to $\mathcal{E}_{1,P}$ and $\mathcal{E}_{1,\mathcal{R}}$ above. The second coefficient error term $\mathcal{E}_{2;C}$ is easier than $\mathcal{E}_{1;C}$ since the number of derivatives on $C$ is the same but the low frequency factor is one power of $t$ better.
The third error term $\mathcal{E}_{3;C}$ is treated analogously to $\mathcal{E}_{1;C}$ since the extra derivative on $C$ is absorbed at the cost of a power of time, but the low frequency factor is a power of $t$ better.   
We omit the details for brevity and conclude. 
\end{proof}

The next PEL is simpler than Lemma \ref{lem:PEL_NLP120neq} and is very frequently used when studying terms which are lower order in powers of $t$. 
 
\begin{lemma}[Zero order PEL] \label{lem:SimplePEL} 
Let $\phi$ be given such that $\norm{\phi}_{\G^{\lambda}} \lesssim \epsilon\jap{t}^{b} \jap{\nu t^3}^{-a}$ for $a,b \geq 0$ and suppose $C$ satisfies the bootstrap hypotheses.
Then for $c_0$ and $\epsilon$ sufficiently small, under the bootstrap hypotheses we have for all $i \in \set{1,2,3}$,  
\begin{align}
\norm{A^{i} \Delta_L \Delta_t^{-1} \phi_{\neq}}_2^2 & \lesssim \norm{A^{i} \phi_{\neq}}_2^2 + \frac{\epsilon^2 \jap{t}^{2b-2} }{\jap{\nu t^3}^{2a}}\norm{A C}_2^2, \label{ineq:SimplePEL}
\end{align} 
\end{lemma} 
\begin{proof} 
The proof is essentially the same as Lemma \ref{lem:PEL_NLP120neq} so it is omitted for the sake of brevity. 
\end{proof} 

The next PEL is the primary tool for treating the linear pressure term $LP3$ in the $Q^3$ equation. 

\begin{lemma}[PEL for $CK_{wL}$] \label{lem:QPELpressureI}
Let $\phi$ be given such that $\norm{\phi}_{\G^{\lambda}} \lesssim \epsilon \jap{t}^b \jap{\nu t^3}^{-a}$ for $a,b \geq 0$ and suppose $C$ satisfies the bootstrap hypotheses.
Then for $c_0$ and $\epsilon$ sufficiently small, there holds
\begin{align} 
\norm{\sqrt{\frac{\partial_t w_L}{w_L}} A^{3} \Delta_{L} \Delta_t^{-1} \phi_{\neq}}_2^2  & \lesssim \norm{\sqrt{\frac{\partial_t w_L}{w_L}} A^{3} \phi_{\neq}}_2^2 + \frac{\epsilon^{2}\jap{t}^{2b-2}}{\jap{\nu t^3}^{2a}} \norm{\left(\sqrt{\frac{\partial_t w}{w}} + \frac{\abs{\grad}^{s/2}}{\jap{t}^s}\right)AC}_2^2. \label{ineq:PELpressureI}
\end{align} 
\end{lemma}
\begin{proof} 
As in the proof of Lemma \ref{lem:PEL_NLP120neq}, denote $P = \Delta_t^{-1} \phi_{\neq}$ and write
\begin{align} 
\Delta_L P & = \phi_{\neq} - G(\partial_Y - t\partial_X)^2 P - 2\psi_z(\partial_Y - t\partial_X) \partial_Z P  - \Delta_t C(\partial_Y - t\partial_X) P \nonumber \\ 
& = \phi_{\neq} + \mathcal{E}_1 + \mathcal{E}_2 + \mathcal{E}_3. \label{def:CKNPk} 
\end{align} 
As in Lemma \ref{lem:PEL_NLP120neq}, we will deduce \eqref{ineq:PELpressureI} by applying $\sqrt{\frac{\partial_tw_L}{w_L}}A^3$ to both sides and estimating \eqref{def:CKNPk}. 
Expand each term with a paraproduct in the usual way: 
\begin{align*} 
\mathcal{E}_i & = \mathcal{E}_{i;C} + \mathcal{E}_{i;P} + \mathcal{E}_{i;\mathcal{R}}.   
\end{align*}  
From \eqref{ineq:dtNBasic}, \eqref{ineq:ABasic}, \eqref{ineq:ARemainderBasic}, \eqref{ineq:quadHL}, and \eqref{ineq:quadR}, one can show
\begin{align*} 
\norm{\sqrt{\frac{\partial_t w_L}{w_L}}A^3 \mathcal{E}_{i;P}}_2^2 + \norm{\sqrt{\frac{\partial_t w_L}{w_L}} A^3 \mathcal{E}_{i;\mathcal{R}}}_2^2 & \lesssim c_{0}^2\norm{\sqrt{\frac{\partial_t w_L}{w_L}} A^{3} \Delta_{L} \Delta_t^{-1} \phi_{\neq}}_2^2,  
\end{align*} 
which is then absorbed on the LHS of the estimate on \eqref{def:CKNPk} for $c_{0}$ chosen sufficiently small. 

Turn next to the coefficient error terms and denote
\begin{align*} 
\mathcal{M}(k,\eta,l) = \sqrt{\frac{\partial_t w_L(k,\eta,l)}{w_L(k,\eta,l)}} A^{3}_k(\eta,l)  
\end{align*} 
Notice that by Lemma \ref{lem:dtw}, it follows that (using that $1/2 + s/2 > s$),  
\begin{align} 
\sqrt{\frac{\partial_t w_L}{w_L}} = \frac{\abs{k}^{1/2}\jap{l}^{1/2}}{\left(k^2 + l^2 + \abs{\eta-kt}^2\right)^{1/2}}  \lesssim \abs{k}^{1/2} \left(\frac{1}{\jap{t}^{s}}\abs{k,\eta,l}^{s/2} + \sqrt{\frac{\partial_t w(t,\eta)}{w(t,\eta)}}\right). \label{ineq:ZLPZtriv} 
\end{align} 
For $\mathcal{E}_{1;C}$ we have by Lemma \ref{lem:ABasic}, \eqref{ineq:ZLPZtriv}, \eqref{ineq:dtwBasic} (the low frequency control comes from the hypotheses and Lemma \ref{lem:LossyElliptic}),  
\begin{align*}
\mathcal{M}\mathcal{E}_{1;C} & \lesssim \frac{\epsilon \jap{t}^{b}}{\jap{\nu t^3}^{a}} \sum_{l} \int_\xi \left(\frac{1}{\jap{t}^{s}}\abs{\xi,l^\prime}^{s/2} + \sqrt{\frac{\partial_t w(t,\xi)}{w(t,\xi)}}\right)  \frac{1}{\jap{\xi,l^\prime}\jap{t}} \abs{A \hat{G}(\xi,l^\prime)_{Hi}} Low(k,\eta-\xi,l-l^\prime) d\xi,
\end{align*}
which by \eqref{ineq:quadHL} and Lemma \ref{lem:CoefCtrl} implies 
\begin{align*}
\norm{\mathcal{M}\mathcal{E}_{1;C}}^2_2 & \lesssim \frac{\epsilon^{2}\jap{t}^{2b-2}}{\jap{\nu t^3}^{2a}} \norm{\left(\sqrt{\frac{\partial_t w}{w}} + \frac{\abs{\grad}^{s/2}}{\jap{t}^{s}}\right)AC}_2^2,
\end{align*}
which appears on the RHS of \eqref{ineq:PELpressureI}. This completes the treatment of $\mathcal{E}_{1;C}$. 

The treatment of $\mathcal{E}_{2;C}$ is similar (though one power of time better). The treatment of $\mathcal{E}_{3;C}$ is precisely analogous to $\mathcal{E}_{1;C}$: in the terms where the coefficients are in high frequency, the extra derivative means we cannot gain a power of time from Lemma \ref{lem:ABasic}, however, this is balanced by the fact that the lower frequency factor is one power of $t$ better than in $\mathcal{E}_{1;C}$. The details are omitted for brevity. 
\end{proof}

\bibliographystyle{plain} \bibliography{eulereqns_vlad,IDnLD}

\def\cprime{$'$} \def\cprime{$'$}
\begin{thebibliography}{10}

\bibitem{Alinhac01}
S.~Alinhac.
\newblock The null condition for quasilinear wave equations in two space
  dimensions {I}.
\newblock {\em Invent. Math.}, 145(3):597--618, 2001.

\bibitem{BDT95}
Jeffrey~S Baggett, Tobin~A Driscoll, and Lloyd~N Trefethen.
\newblock A mostly linear model of transition to turbulence.
\newblock {\em Physics of Fluids (1994-present)}, 7(4):833--838, 1995.

\bibitem{BT97}
Jeffrey~S Baggett and Lloyd~N Trefethen.
\newblock Low-dimensional models of subcritical transition to turbulence.
\newblock {\em Physics of Fluids (1994-present)}, 9(4):1043--1053, 1997.

\bibitem{BCD11}
H.~Bahouri, J.-Y. Chemin, and R.~Danchin.
\newblock {\em Fourier analysis and nonlinear partial differential equations},
  volume 343 of {\em Grundlehren der Mathematischen Wissenschaften [Fundamental
  Principles of Mathematical Sciences]}.
\newblock Springer, Heidelberg, 2011.

\bibitem{Bajer2001}
K.~Bajer, A.P. Bassom, and A.D. Gilbert.
\newblock Accelerated diffusion in the centre of a vortex.
\newblock {\em Journal of Fluid Mechanics}, 437:395--411, 2001.

\bibitem{BM95}
N.J. Balmforth and P.J. Morrison.
\newblock Normal modes and continuous spectra.
\newblock {\em Annals of the New York Academy of Sciences}, 773(1):80--94,
  1995.

\bibitem{BMT13}
N.J. Balmforth, P.J. Morrison, and J.-L. Thiffeault.
\newblock Pattern formation in {Hamiltonian} systems with continuous spectra; a
  normal-form single-wave model.
\newblock {\em preprint}, 2013.

\bibitem{GilbertBassom98}
A.P. Bassom and A.D. Gilbert.
\newblock The spiral wind-up of vorticity in an inviscid planar vortex.
\newblock {\em J. Fluid Mech.}, 371:109--140, 1998.

\bibitem{BeckWayne11}
M~Beck and C~E Wayne.
\newblock Metastability and rapid convergence to quasi-stationary bar states
  for the two-dimensional {Navier--Stokes} equations.
\newblock {\em Proc. Royal Soc. of Edinburgh: Sec. A Mathematics},
  143(05):905--927, 2013.

\bibitem{BGM15II}
J.~Bedrossian, P.~Germain, and N.~Masmoudi.
\newblock Dynamics near the subcritical transition of the {3D Couette flow II:
  Above} threshold.
\newblock {\em arXiv preprint}, 2015.

\bibitem{BMM13}
J.~Bedrossian, N.~Masmoudi, and C.~Mouhot.
\newblock Landau damping: paraproducts and {Gevrey} regularity.
\newblock {\em {arXiv:1311.2870}}, 2013.

\bibitem{BMV14}
J.~Bedrossian, N.~Masmoudi, and V.~Vicol.
\newblock Enhanced dissipation and inviscid damping in the inviscid limit of
  the {Navier-Stokes} equations near the {2D Couette} flow.
\newblock {\em To appear in {Arch. Rat. Mech. Anal.}}, 2014.

\bibitem{BM13}
Jacob Bedrossian and Nader Masmoudi.
\newblock Inviscid damping and the asymptotic stability of planar shear flows
  in the {2D Euler} equations.
\newblock {\em Publications math{\'e}matiques de l'IH{\'E}S}, pages 1--106,
  2013.

\bibitem{BernoffLingevitch94}
A.J. Bernoff and J.F. Lingevitch.
\newblock Rapid relaxation of an axisymmetric vortex.
\newblock {\em Phys. Fluids}, 6(3717), 1994.

\bibitem{Bony81}
J.M. Bony.
\newblock Calcul symbolique et propagation des singularit\'es pour les
  \'equations aux d\'eriv\'ees partielles non lin\'aires.
\newblock {\em Ann.Sc.E.N.S.}, 14:209--246, 1981.

\bibitem{bottin98}
S.~Bottin, O.~Dauchot, F.~Daviaud, and P.~Manneville.
\newblock Experimental evidence of streamwise vortices as finite amplitude
  solutions in transitional plane {Couette} flow.
\newblock {\em Physics of Fluids}, 10:2597, 1998.

\bibitem{BouchetMorita10}
F.~Bouchet and H.~Morita.
\newblock Large time behavior and asymptotic stability of the {2D Euler} and
  linearized {Euler} equations.
\newblock {\em Physica D}, 239:948--966, 2010.

\bibitem{Briggs70}
R.J. Briggs, J.D. Daugherty, and R.H. Levy.
\newblock Role of {Landau} damping in crossed-field electron beams and inviscid
  shear flow.
\newblock {\em Phys. Fl.}, 13(2), 1970.

\bibitem{CagliotiMaffei98}
E.~Caglioti and C.~Maffei.
\newblock Time asymptotics for solutions of {Vlasov-Poisson} equation in a
  circle.
\newblock {\em J. Stat. Phys.}, 92(1/2), 1998.

\bibitem{CerfonEtAl13}
A.J. Cerfon, J.P. Freidberg, F.I. Parra, and T.A. Antaya.
\newblock Analytic fluid theory of beam spiraling in high-intensity cyclotrons.
\newblock {\em Phys. Rev. {ST} Accel. Beams}, 16(024202), 2013.

\bibitem{Chapman02}
S.J. Chapman.
\newblock Subcritical transition in channel flows.
\newblock {\em Journal of Fluid Mechanics}, 451:35--98, 2002.

\bibitem{Christodoulou1986}
D.~Christodoulou.
\newblock Global solutions of nonlinear hyperbolic equations for small initial
  data.
\newblock {\em Comm. Pure Appl. Math.}, 39(2):267--282, 1986.

\bibitem{ConstantinEtAl08}
P.~Constantin, A.~Kiselev, L.~Ryzhik, and A.~Zlato{\v{s}}.
\newblock Diffusion and mixing in fluid flow.
\newblock {\em Ann. of Math. (2)}, 168(2):643--674, 2008.

\bibitem{Craik1971}
Alex~DD Craik.
\newblock Non-linear resonant instability in boundary layers.
\newblock {\em Journal of Fluid Mechanics}, 50(02):393--413, 1971.

\bibitem{Daviaud92}
F~Daviaud, J~Hegseth, and P~Berg{\'e}.
\newblock Subcritical transition to turbulence in plane {Couette} flow.
\newblock {\em Physical review letters}, 69(17):2511, 1992.

\bibitem{DrazinReid81}
P.G. Drazin and W.H. Reid.
\newblock {\em Hydrodynamic stability}.
\newblock Cambridge U. Press, Cambridge, 1981.

\bibitem{DubrulleNazarenko94}
B~Dubrulle and S~Nazarenko.
\newblock On scaling laws for the transition to turbulence in uniform-shear
  flows.
\newblock {\em EPL (Europhysics Letters)}, 27(2):129, 1994.

\bibitem{DuguetEtAl2010}
Y.~Duguet, P.~Schlatter, and D.~S. Henningson.
\newblock Formation of turbulent patterns near the onset of transition in plane
  {Couette} flow.
\newblock {\em J. of Fluid Mechanics}, 650:119--129, 2010.

\bibitem{EllingsenPalm75}
T~Ellingsen and E~Palm.
\newblock Stability of linear flow.
\newblock {\em Physics of Fluids}, 18:487, 1975.

\bibitem{Elofsson1999}
Per~A Elofsson, Mitsuyoshi Kawakami, and P~Henrik Alfredsson.
\newblock Experiments on the stability of streamwise streaks in plane
  {Poiseuille} flow.
\newblock {\em Physics of Fluids (1994-present)}, 11(4):915--930, 1999.

\bibitem{FaouRousset14}
E.~Faou and F.~Rousset.
\newblock Landau damping in {Sobolev} spaces for the {Vlasov-HMF} model.
\newblock {\em arXiv:1403.1668}, 2014.

\bibitem{FoiasTemam89}
C.~Foias and R.~Temam.
\newblock Gevrey class regularity for solutions of the {Navier-Stokes}
  equations.
\newblock {\em J. Funct. Anal.}, 87:359--369, 1989.

\bibitem{Gebhardt1994}
T.~Gebhardt and S.~Grossmann.
\newblock Chaos transition despite linear stability.
\newblock {\em Physical Review E}, 50(5):3705, 1994.

\bibitem{Gevrey18}
M.~Gevrey.
\newblock Sur la nature analytique des solutions des \'equations aux
  d\'eriv\'ees partielles. {P}remier m\'emoire.
\newblock {\em Ann. Sci. \'Ecole Norm. Sup. (3)}, 35:129--190, 1918.

\bibitem{Gilbert88}
A.D. Gilbert.
\newblock Spiral structures and spectra in two-dimensional turbulence.
\newblock {\em J. Fluid Mech.}, 193:475--497, 1988.

\bibitem{glassey94}
Robert Glassey and Jack Schaeffer.
\newblock Time decay for solutions to the linearized vlasov equation.
\newblock {\em Transport Theory and Statistical Physics}, 23(4):411--453, 1994.

\bibitem{glassey95}
Robert Glassey and Jack Schaeffer.
\newblock On time decay rates in landau damping.
\newblock {\em Communications in partial differential equations},
  20(3-4):647--676, 1995.

\bibitem{GrenierGuoNguyen14a}
E.~Grenier, Y.~Guo, and T.~Nguyen.
\newblock Spectral instability of characteristic boundary layer flows.
\newblock {\em {arXiv:1406.3862}}, 06 2014.

\bibitem{HLJ93}
Dan~S Henningson, Anders Lundbladh, and Arne~V Johansson.
\newblock A mechanism for bypass transition from localized disturbances in
  wall-bounded shear flows.
\newblock {\em Journal of Fluid Mechanics}, 250:169--207, 1993.

\bibitem{HofJuelMullin2003}
B~Hof, A~Juel, and T~Mullin.
\newblock Scaling of the turbulence transition threshold in a pipe.
\newblock {\em Physical review letters}, 91(24):244502, 2003.

\bibitem{Hormander1990}
Lars H{\"o}rmander.
\newblock The {Nash-Moser} theorem and paradifferential operators.
\newblock {\em Analysis, et cetera}, pages 429--449, 1990.

\bibitem{HwangVelazquez09}
H.J. Hwang and J.J.L. Vela{\'z}quez.
\newblock On the existence of exponentially decreasing solutions of the
  nonlinear {Landau} damping problem.
\newblock {\em Indiana Univ. Math. J}, pages 2623--2660, 2009.

\bibitem{Kelvin87}
Lord Kelvin.
\newblock Stability of fluid motion-rectilinear motion of viscous fluid between
  two parallel plates.
\newblock {\em Phil. Mag.}, (24):188, 1887.

\bibitem{Klainerman1986}
S.~Klainerman.
\newblock The null condition and global existence to nonlinear wave equations.
\newblock {\em Lect. Appl. Math}, 23:293--326, 1986.

\bibitem{klebanoff1962}
PS~Klebanoff, KD~Tidstrom, and LM~Sargent.
\newblock The three-dimensional nature of boundary-layer instability.
\newblock {\em Journal of Fluid Mechanics}, 12(01):1--34, 1962.

\bibitem{KreissEtAl94}
G.~Kreiss, A.~Lundbladh, and D.~Henningson.
\newblock Bounds for threshold amplitudes in subcritical shear flows.
\newblock {\em Journal of Fluid Mechanics}, 270:175--198, 1994.

\bibitem{KukavicaVicol09}
I.~Kukavica and V.~Vicol.
\newblock On the radius of analyticity of solutions to the three-dimensional
  {E}uler equations.
\newblock {\em Proc. Amer. Math. Soc.}, 137(2):669--677, 2009.

\bibitem{landahl80}
MT~Landahl.
\newblock A note on an algebraic instability of inviscid parallel shear flows.
\newblock {\em Journal of Fluid Mechanics}, 98(02):243--251, 1980.

\bibitem{Landau46}
L.~Landau.
\newblock On the vibration of the electronic plasma.
\newblock {\em J. Phys. USSR}, 10(25), 1946.

\bibitem{LatiniBernoff01}
M.~Latini and A.J. Bernoff.
\newblock Transient anomalous diffusion in {Poiseuille} flow.
\newblock {\em Journal of Fluid Mechanics}, 441:399--411, 2001.

\bibitem{LemoultEtAl2012}
G.~Lemoult, J.-L. Aider, and J.E. Wesfreid.
\newblock Experimental scaling law for the subcritical transition to turbulence
  in plane {Poiseuille} flow.
\newblock {\em Physical Review E}, 85(2):025303, 2012.

\bibitem{LevermoreOliver97}
D.~Levermore and M.~Oliver.
\newblock Analyticity of solutions for a generalized {Euler} equation.
\newblock {\em J. Diff. Eqns.}, 133:321--339, 1997.

\bibitem{LiLin11}
Y.C. Li and Z.~Lin.
\newblock A resolution of the {S}ommerfeld paradox.
\newblock {\em SIAM J. Math. Anal.}, 43(4):1923--1954, 2011.

\bibitem{Liefvendahl2002}
M.~Liefvendahl and G.~Kreiss.
\newblock Bounds for the threshold amplitude for plane {Couette} flow.
\newblock {\em Journal of Nonlinear Mathematical Physics}, 9(3):311--324, 2002.

\bibitem{LinZeng11}
Z.~Lin and C.~Zeng.
\newblock Inviscid dynamic structures near {Couette} flow.
\newblock {\em Arch. Rat. Mech. Anal.}, 200:1075--1097, 2011.

\bibitem{LHRS94}
Anders Lundbladh, Dan~S Henningson, and Satish~C Reddy.
\newblock Threshold amplitudes for transition in channel flows.
\newblock In {\em Transition, turbulence and combustion}, pages 309--318.
  Springer, 1994.

\bibitem{Lundgren82}
T.S. Lundgren.
\newblock Strained spiral vortex model for turbulent fine structure.
\newblock {\em Phys. of Fl.}, 25:2193, 1982.

\bibitem{MajdaBertozzi}
A.~Majda and A.~L. Bertozzi.
\newblock {\em Vorticity and Incompressible Flow}.
\newblock Cambridge University Press, 2002.

\bibitem{MalmbergWharton68}
J.~Malmberg, C.~Wharton, C.~Gould, and T.~O'Neil.
\newblock Plasma wave echo.
\newblock {\em Phys. Rev. Lett.}, 20(3):95--97, 1968.

\bibitem{MontgomeryKallenback1997}
M.~T. Montgomery and R.~J. Kallenbach.
\newblock A theory for vortex {Rossby-waves} and its application to spiral
  bands and intensity changes in hurricanes.
\newblock {\em Quarterly Journal of the Royal Meteorological Society},
  123(538):435--465, 1997.

\bibitem{Morrison98}
P.~J. Morrison.
\newblock Hamiltonian description of the ideal fluid.
\newblock {\em Rev. Modern Phys.}, 70(2):467--521, 1998.

\bibitem{MouhotVillani11}
C.~Mouhot and C.~Villani.
\newblock On {Landau} damping.
\newblock {\em Acta Math.}, 207:29--201, 2011.

\bibitem{Mullin2011}
T~Mullin.
\newblock Experimental studies of transition to turbulence in a pipe.
\newblock {\em Annual Review of Fluid Mechanics}, 43:1--24, 2011.

\bibitem{Nirenberg72}
L.~Nirenberg.
\newblock An abstract form of the nonlinear {Cauchy-Kowalewski} theorem.
\newblock {\em J. Diff. Geom.}, 6:561--576, 1972.

\bibitem{Nishida77}
T.~Nishida.
\newblock A note on a theorem of {Nirenberg}.
\newblock {\em J. Diff. Geom.}, 12:629--633, 1977.

\bibitem{Nishioka1975}
M~Nishioka, Y~Ichikawa, et~al.
\newblock An experimental investigation of the stability of plane {Poiseuille}
  flow.
\newblock {\em Journal of Fluid Mechanics}, 72(04):731--751, 1975.

\bibitem{Orr07}
W.~Orr.
\newblock The stability or instability of steady motions of a perfect liquid
  and of a viscous liquid, {Part I}: a perfect liquid.
\newblock {\em Proc. Royal Irish Acad. Sec. A: Math. Phys. Sci.}, 27:9--68,
  1907.

\bibitem{Orszag80}
Steven~A Orszag and Lawrence~C Kells.
\newblock Transition to turbulence in plane {Poiseuille} and plane {Couette}
  flow.
\newblock {\em Journal of Fluid Mechanics}, 96(1):159--205, 1980.

\bibitem{Rayleigh80}
Lord Rayleigh.
\newblock On the {S}tability, or {I}nstability, of certain {F}luid {M}otions.
\newblock {\em Proc. London Math. Soc.}, S1-11(1):57, 1880.

\bibitem{ReddySchmidEtAl98}
S.C. Reddy, P.J. Schmid, J.S. Baggett, and D.S. Henningson.
\newblock On stability of streamwise streaks and transition thresholds in plane
  channel flows.
\newblock {\em J. of Fluid Mechanics}, 365:269--303, 1998.

\bibitem{Reynolds83}
O~Reynolds.
\newblock An experimental investigation of the circumstances which determine
  whether the motion of water shall be direct or sinuous, and of the law of
  resistance in parallel channels.
\newblock {\em Proc. R. Soc. Lond.}, (35):84, 1883.

\bibitem{RhinesYoung83}
P.B. Rhines and W.R. Young.
\newblock How rapidly is a passive scalar mixed within closed streamlines?
\newblock {\em Journal of Fluid Mechanics}, 133:133--145, 1983.

\bibitem{Romanov73}
VA~Romanov.
\newblock Stability of plane-parallel couette flow.
\newblock {\em Functional analysis and its applications}, 7(2):137--146, 1973.

\bibitem{Ryutov99}
D.D. Ryutov.
\newblock Landau damping: half a century with the great discovery.
\newblock {\em Plasma physics and controlled fusion}, 41(3A):A1, 1999.

\bibitem{SchecterEtAl00}
D.A. Schecter, D.~Dubin, A.C. Cass, C.F. Driscoll, and I.M.~Lansky et. al.
\newblock Inviscid damping of asymmetries on a two-dimensional vortex.
\newblock {\em Phys. Fl.}, 12, 2000.

\bibitem{SchmidHenningson2001}
Peter~J Schmid and Dan~S Henningson.
\newblock {\em Stability and transition in shear flows}, volume 142.
\newblock Springer, 2001.

\bibitem{SmithMontgomery1995}
G.~B. Smith and M.~T. Montgomery.
\newblock Vortex axisymmetrization: Dependence on azimuthal wave-number or
  asymmetric radial structure changes.
\newblock {\em Quarterly Journal of the Royal Meteorological Society},
  121(527):1615--1650, 1995.

\bibitem{Tillmark92}
N.~Tillmark and P.H. Alfredsson.
\newblock Experiments on transition in plane {Couette} flow.
\newblock {\em J. Fluid Mech.}, 235:89--102, 1992.

\bibitem{Trefethen2005}
Lloyd~Nicholas Trefethen and Mark Embree.
\newblock {\em Spectra and pseudospectra: the behavior of nonnormal matrices
  and operators}.
\newblock Princeton University Press, 2005.

\bibitem{TTRD93}
L.N. Trefethen, A.E. Trefethen, S.C. Reddy, and T.A. Driscoll.
\newblock Hydrodynamic stability without eigenvalues.
\newblock {\em Science}, 261(5121):578--584, 1993.

\bibitem{Vanneste02}
J.~Vanneste.
\newblock Nonlinear dynamics of anisotropic disturbances in plane {C}ouette
  flow.
\newblock {\em SIAM J. Appl. Math.}, 62(3):924--944 (electronic), 2002.

\bibitem{VMW98}
J.~Vanneste, P.J. Morrison, and T.~Warn.
\newblock Strong echo effect and nonlinear transient growth in shear flows.
\newblock {\em Physics of Fluids}, 10:1398, 1998.

\bibitem{Waleffe95}
Fabian Waleffe.
\newblock Transition in shear flows. nonlinear normality versus non-normal
  linearity.
\newblock {\em Physics of Fluids (1994-present)}, 7(12):3060--3066, 1995.

\bibitem{Yaglom12}
A.M. Yaglom.
\newblock {\em Hydrodynamic Instability and Transition to Turbulence}, volume
  100.
\newblock Springer, 2012.

\bibitem{Young14}
B.~Young.
\newblock Landau damping in relativistic plasmas.
\newblock {\em {arXiv:1408.2666}}, 2014.

\bibitem{YuDriscoll02}
J.H. Yu and C.F. Driscoll.
\newblock Diocotron wave echoes in a pure electron plasma.
\newblock {\em {IEEE} Trans. Plasma Sci.}, 30(1), 2002.

\bibitem{YuDriscollONeil}
J.H. Yu, C.F. Driscoll, and T.M. O`Neil.
\newblock Phase mixing and echoes in a pure electron plasma.
\newblock {\em Phys. of Plasmas}, 12(055701), 2005.

\bibitem{Zillinger2014}
C.~Zillinger.
\newblock Linear inviscid damping for monotone shear flows.
\newblock {\em arXiv preprint arXiv:1410.7341}, 2014.

\end{thebibliography}

\end{document}